\documentclass[11pt, usletter, reqno, twoside, makeidx, alphabetic]{amsart}
\usepackage[margin=3cm, marginpar=3cm]{geometry}
\usepackage{tikz-cd}

\usepackage{mathabx}
\usepackage{amssymb}
\usepackage{braket}
\usepackage{mathrsfs}
\usepackage{ifthen}
\usepackage{graphicx}

\usepackage{comment} 
\usepackage{mleftright}

 \usepackage[all]{xy}
 \usepackage{amscd} 
 \usepackage{amsrefs}
 \usepackage{color}
 \usepackage{enumitem}
\usepackage{tikz-cd}
\usetikzlibrary{calc}

\usepackage{thmtools} 
\usepackage{thm-restate}

\usepackage{hyperref}
\usepackage{cleveref}

\tikzset{curve/.style={settings={#1},to path={(\tikztostart)
    .. controls ($(\tikztostart)!\pv{pos}!(\tikztotarget)!\pv{height}!270:(\tikztotarget)$)
    and ($(\tikztostart)!1-\pv{pos}!(\tikztotarget)!\pv{height}!270:(\tikztotarget)$)
    .. (\tikztotarget)\tikztonodes}},
    settings/.code={\tikzset{quiver/.cd,#1}
        \def\pv##1{\pgfkeysvalueof{/tikz/quiver/##1}}},
    quiver/.cd,pos/.initial=0.35,height/.initial=0}
    
\newlist{steps}{enumerate}{1}
\setlist[steps, 1]{label = Step \arabic*:}
\usepackage[fontsize=10]{scrextend}
\usepackage{verbatim,amsmath,amsthm,amsfonts,amssymb,latexsym, mathtools,color, mathabx}

\makeatletter
\DeclareRobustCommand\widecheck[1]{{\mathpalette\@widecheck{#1}}}
\def\@widecheck#1#2{%
   \setbox\z@\hbox{\m@th$#1#2$}%
   \setbox\tw@\hbox{\m@th$#1%
      \widehat{%
         \vrule\@width\z@\@height\ht\z@
         \vrule\@height\z@\@width\wd\z@}$}%
   \dp\tw@-\ht\z@
   \@tempdima\ht\z@ \advance\@tempdima2\ht\tw@ \divide\@tempdima\thr@@
   \setbox\tw@\hbox{%
      \raise\@tempdima\hbox{\scalebox{1}[-1]{\lower\@tempdima\box\tw@}}}%
   {\ooalign{\box\tw@ \cr \box\z@}}}
\makeatother

\theoremstyle{theorem}
\newtheorem{thm}{\bf Theorem}[section]
\newtheorem{lem}[thm]{\bf Lemma}
\newtheorem{cor}[thm]{\bf Corollary}
\newtheorem{prop}[thm]{\bf Proposition}

\theoremstyle{definition}
\newtheorem{defn}[thm]{\bf Definition}
\theoremstyle{remark}
\newtheorem{rem}[thm]{\bf Remark}

\newtheorem{ex}[thm]{\bf Example}

\crefformat{thm}{\textnormal{Theorem~#2#1#3}}
\Crefformat{thm}{\textnormal{Theorem~#2#1#3}}

\crefformat{lem}{\textnormal{Lemma~#2#1#3}}
\Crefformat{lem}{\textnormal{Lemma~#2#1#3}}

\crefformat{cor}{\textnormal{Corollary~#2#1#3}}
\Crefformat{cor}{\textnormal{Corollary~#2#1#3}}

\crefformat{ex}{\textnormal{Example~#2#1#3}}
\Crefformat{ex}{\textnormal{Example~#2#1#3}}

\crefformat{rem}{\textnormal{Remark~#2#1#3}}
\Crefformat{rem}{\textnormal{Remark~#2#1#3}}

\crefformat{defn}{\textnormal{Definition~#2#1#3}}
\Crefformat{defn}{\textnormal{Definition~#2#1#3}}

\def\gr{\mathrm{gr}}

\newcommand{\Z}{\mathbb{Z}}

\newcommand{\Q}{\mathbb{Q}}

\newcommand{\A}{\mathcal A}
\newcommand{\B}{\mathcal B}
\newcommand{\G}{\mathcal G}

\newcommand{\pr}{\text{pr}}

\def\spinc{\text{spin}^c}
\newcommand{\N}{\mathbb N}
\newcommand{\R}{\mathbb R}

\newcommand{\ctext}[1]{\raise0.2ex\hbox{\textcircled{\scriptsize{#1}}}}%\ctext{}?????????????????

\def\ker{\operatorname{ker}}

\def\inte{\operatorname{int}}
\def\id{\operatorname{id}}

\def\ind{\operatorname{ind}}

\newcommand{\mbar}[1]{{\ooalign{\hfil#1\hfil\crcr\raise.167ex\hbox{--}}}}

\def\a{\mathfrak{a}}
\def\b{\mathfrak{b}}

\def\gr{\operatorname{gr}}

\def\wt{\widetilde}

\def\f{\mathbb{F}}
\def\fix{\mathrm{Fix}}

\tikzset{
contains/.style = {draw=none,"\in" description,sloped}
}

\def\un{\underline}
\def\ov{\overline}
\def\mf{\mathfrak}
\newcommand{\Wh}{\mathit{Wh}}

\newcommand\CP{\smash{\mathbb{CP}^2}}
\newcommand\oCP{\smash{\ov{\mathbb{CP}}}^2}
\newcommand\RP{\smash{\mathbb{RP}^2}}
\newcommand\U{U}

\newcommand{\eC}{\smash{\widetilde{\mathcal{C}}}}
\newcommand{\ima}{\operatorname{im}}

\newcommand{\SU}{\mathit{SU}}

\newcommand{\cS}{\mathcal{S}}

\def\CF {\mathit{CF}}
\def\HF {\mathit{HF}}

\newcommand \CFm {\CF^-}
\newcommand \HFm {\HF^-}

\newcommand{\CFK}{\mathcal{CFK}}

\newcommand\Cis{\widetilde{\mathcal C}_{i}}
\newcommand\Cist{\widetilde{\mathcal C}_{i,top}}

\def\Inv{\mathfrak{I}}
\newcommand{\K}{\mathfrak{K}}
\title{Exotic disks and  singular instanton Floer homology}
\author{Irving Dai}
\address{Department of Mathematics, The University of Texas at Austin, Austin, Texas, 78712, USA }
\email{irving.dai@math.utexas.edu}
\author{Abhishek Mallick}
\address{Department of Mathematics, Dartmouth College, 29 N Main St, Hanover, New Hampshire, 03755, USA}
\email{abhishek.mallick@dartmouth.edu}
\author{Masaki Taniguchi}
\address{Department of Mathematics, Graduate School of Science, Kyoto University, Kitashirakawa Oiwake-cho, Sakyo-ku, Kyoto 606-8502, Japan}
\email{taniguchi.masaki.7m@kyoto-u.ac.jp}
\begin{document}

\maketitle

%\tableofcontents
\begin{abstract}
We show that singular instanton Floer homology with the Chern--Simons filtration can be used to produce exotic pairs of slice disks. We moreover construct a strongly invertible $\Z$-slice knot for which any symmetric pair of $\Z$-disks are exotic, and remain exotic after stabilizing by $n \mathbb{CP}^2$ or $n\oCP$ (or by standard $n\mathbb{RP}^2$ or $-n\mathbb{RP}^2$) for any $n$. Our methods apply more generally to stabilization by any simply connected definite manifold, or by any number of exotic embedded projective planes of the same sign. We also provide an example of a strongly invertible knot which is $\Z$-slice and equivariantly slice, but not equivariantly $\Z$-slice. Along the way, we partially compute various symmetry actions on the singular instanton Floer complexes of two-bridge knots via an explicit analysis of their traceless \(\SU(2)\)-character varieties.
\end{abstract}

%Our work gives several computational techniques

%We also provide several computational techniques for involutive singular instanton complexes of
%strongly invertible knots, including a connected sum formula and a method for relating the induced involution on singular instanton homology to its action on the traceless \(\SU(2)\)-character variety.
\section{Introduction}\label{sec:1}

Exotic surfaces form a central area of study in low-dimensional topology. Equivariant knots and manifolds have provided an important source of examples in several recent papers on constructing and detecting these phenomena. In the present work, we develop a suite of instanton-theoretic invariants aimed at studying strongly invertible knots and concordances between them. We use these to produce exotic disks for Whitehead doubles and investigate stabilization questions for symmetric pairs of disks by $\CP$ and (standard) $\RP$. More generally, our techniques can also be used to study stabilizations by simply connected definite manifolds and by embedded projective planes with exterior fundamental group $\Z_2$. We also leverage the Chern--Simons filtration to explore the role of the fundamental group in equivariant knot theory by distinguishing equivariant sliceness from equivariant $\Z$-sliceness. Our invariants are defined using the action of knot symmetries on the filtered $U(1)$-equivariant singular instanton complex of Daemi and Scaduto \cite{DS19}. We provide several methods for partially computing these actions, including two connected sum formulas and techniques involving the direct analysis of traceless \(\SU(2)\)-character varieties.

As observed in \cite{DMS}, there is a curious connection between equivariant knot theory and the formation of exotic disks. Recall that a strongly invertible knot $(K, \tau)$ is \textit{equivariantly slice} if it admits a (smooth) slice disk $D$ fixed by the obvious extension $\tau_{B^4}$ of $\tau$ over $B^4$. In \cite{DMS}, a slightly more general notion was introduced: given a slice disk $D$, one can instead ask for $D$ and its symmetric copy  $\tau_{B^4}(D)$ to be smoothly isotopic rel boundary. If $K$ admits such a $D$, then we say $K$ is \textit{isotopy-equivariantly slice}. Note that if $K$ is a strongly invertible slice knot which is \textit{not} isotopy-equivariantly slice, then any slice disk $D$ generates a pair of non-isotopic disks $D$ and $\tau_{B^4}(D)$ for $K$. In particular, if $K$ is $\Z$-slice, then choosing \textit{any} smooth $\Z$-disk and considering its pushforward by $\tau_{B^4}$ gives a pair of exotic disks via work of Conway and Powell \cite{conway2021characterisation}. Obstructing isotopy-equivariant sliceness can thus be used to establish or re-establish examples of exotic pairs \cite{DMS}, such as those from \cite{hayden}.%\marginpar{ID: cut or extend this sentence?} %[In addition to (in principle) producing further exotic pairs of disks, the advantage of this point of view is that obstructions to isotopy-equivariant sliceness can be defined in terms of the knot only, rather than making reference to a specific pair of surfaces. This is useful in the context of invariants (such as instanton Floer theory) for which calculations involving surfaces or bordisms are scarce.]

\subsection{Symmetries and exotic disks}\label{sec:1.1}

We first apply this principle to produce exotic disks for Whitehead doubles. Given a strongly invertible knot $(K, \tau)$, we construct an instanton-theoretic invariant 
\[
r_s(K, \tau) \in [0, \infty ], \text{ defined for each } s \in [-\infty, 0]
\]
using the action of $\tau$ on the filtered $U(1)$-equivariant singular instanton complex of Daemi and Scaduto \cite{DS19}. Our $r_s$-invariant not only obstructs the isotopy-equivariant sliceness of $(K, \tau)$, but also obstructs the isotopy-equivariant sliceness of the Whitehead double of $K$. Note that if $K$ is strongly invertible, then its $i$-clasped Whitehead double $\Wh_i(K)$ is also strongly invertible; and if $K$ is slice, then $\Wh_i(K)$ is $\Z$-slice as observed in \cite[Proposition 2.1]{GHKP}.\footnote{There are two ways of forming the equivariant Whitehead double; see Example~\ref{ex:whitehead}. Our results apply to both.} Our first theorem thus gives a simple condition on an invariant intrinsic to $(K, \tau)$
which guarantees the existence of an exotic pair of slice disks for $\Wh_i(K)$.

\begin{restatable}{thm}{introwhitehead}\label{thm:intro-whitehead}
Let $(K,\tau)$ be a strongly invertible slice knot. If $r_{- 1/8}(K, \tau ) < \infty$ and $r_{0}(-K, -\tau ) = \infty$, then no pair of symmetric disks for $\Wh_i(K)$ are smoothly isotopic rel boundary for any $i \in \N$. In particular, $\Wh_i(K)$ bounds an exotic pair of disks.
\end{restatable}

Theorem~\ref{thm:intro-whitehead} adds to the array of Floer- and Khovanov-theoretic methods for producing exotic slice surfaces; see for example \cite{JMZ, HS, DMS}. Here, we make use of singular instanton Floer theory as initially introduced by Kronheimer and Mrowka \cite{kronheimer2011knot, KM13} and further developed by Daemi and Scaduto \cite{DS19, DS20}. To the best of the authors' knowledge, this is the first time that singular instanton Floer theory has been explicitly invoked in the detection of exotic disks.\footnote{The proof of Theorem~\ref{thm:intro-whitehead} is somewhat indirect and proceeds by taking branched covers. Although singular instanton Floer theory is used in an essential manner, we do not quite claim that the knot cobordism maps in singular instanton theory detect exotic pairs of disks.} Theorem~\ref{thm:intro-whitehead} is similar to recent results of Guth, Hayden, Kang, and Park \cite{GHKP}; see our discussion in Section~\ref{sec:GHKP} below. Further properties of our $r_s$-invariant are given in Theorem~\ref{thm:rsproperties}.

We also use Theorem~\ref{thm:intro-whitehead} to construct examples of knots which are not isotopy-equivariantly slice, and for which this non-isotopy persists under stabilization by any number of copies of $\CP$ or standard $\RP$ of the same sign. Note that any two exotic surfaces become smoothly isotopic after mixed stabilizations by $\CP$ and $\oCP$ \cite{Perron84, Perron86, Quinn86}; or, alternatively, mixed stabilizations by standard $\RP$ and $-\RP$ \cite{Kam89, ConwayPowell23}. See \cite{BS16, AKMRS19, LM21} for work on external stabilization by \(\CP\), \(\oCP\), or \(S^2\times S^2\). Here, recall that an embedding of $\RP$ has sign given by whether the self-intersection is $+2$ or $-2$.

\begin{restatable}{thm}{stabilization}\label{CP2RP2stabilization}
There exists a $\Z$-slice knot $K$ with a strong inversion $\tau$ such that if $D$ and $\tau_{B^4}(D)$ are any symmetric pair of slice disks for $K$ in $B^4$, then 
\begin{enumerate}[noitemsep]
    \item\label{item:staba} $D$ and $\tau_{B^4}(D)$ are not isotopic rel boundary in $B^4 \# n \CP$ for any $n \in \Z$; and,
    \item\label{item:stabb} $D \# n \RP$ and $\tau_{B^4}(D) \# n \RP$ are not isotopic rel boundary in $B^4$ for any $n \in \Z$.
    \end{enumerate}
Here, $n$ may be positive or negative, so that the stabilization may be carried out with any number of copies of $\CP$ $($or standard $\RP$$)$ of the same sign.
In particular, any symmetric pair of $\Z$-disks for $K$ gives an exotic pair that survives stabilization by $n\CP$ $($or standard $n\RP$$)$ for any $n \in \Z$.
\end{restatable}

Theorem~\ref{CP2RP2stabilization} is similar to the problem of obstructing a homology sphere from bounding a definite manifold of either sign. This was resolved using filtered instanton theory in \cite{NST19} and studied in the equivariant setting by the authors and their collaborators in \cite{ADMT}. Note that Floer- and Khovanov-theoretic techniques are generally sensitive to the sign of stabilization. For example, in \cite{nahmkhovanov} it is shown that (enhanced) Khovanov homology can obstruct exotic pairs of disks from being isotopic rel boundary up to stabilization by $n \RP$ for one sign of $n$, but that taking the connected sum with an $\RP$ of the opposite sign causes the (enhanced) Khovanov map to vanish. We thus stress that Theorem~\ref{CP2RP2stabilization} provides pairs of exotic disks which persist under any number of stabilizations of \textit{either} sign. Examples of individual pairs of such disks can also be constructed by other means; see Remark~\ref{rem:alternative}. However, Theorem~\ref{CP2RP2stabilization} applies simultaneously to \textit{all} pairs of symmetric disks for $K$.

\subsection{Equivariant sliceness in definite manifolds}

The definition of isotopy-equivariance may be relaxed significantly by allowing $B^4$ to be replaced with a general $4$-manifold $W$ and $\tau_{B^4}$ to be replaced with any self-diffeomorphism of $W$ restricting to $\tau$ on $\partial W = S^3$. Our invariants are capable of obstructing isotopy equivariance even in this broader setting.

\begin{restatable}{thm}{introonesign}\label{thm:intro-onesign}
Let $(K, \tau)$ be a strongly invertible knot with $\sigma(K) = 0$. If $r_0(K, \tau)< \infty $, then $(K, \tau)$ is not isotopy-equivariantly $H$-slice in any negative definite manifold $W$ with $H_1(W, \Z_2) = 0$. Likewise, if $r_0(-K, -\tau)< \infty$, then $(K, \tau)$ is not isotopy-equivariantly $H$-slice in any positive definite manifold $W$ with $H_1(W, \Z_2) = 0$.
\end{restatable}

Here, we say that $(K, \tau)$ is isotopy-equivariantly $H$-slice if it admits an $H$-slice disk which is isotopy equivariant in the sense discussed above. Although Theorem~\ref{thm:intro-onesign} bears some similarity to Theorem~\ref{CP2RP2stabilization}\eqref{item:staba}, we do not establish Theorem~\ref{CP2RP2stabilization}\eqref{item:staba} as a corollary of Theorem~\ref{thm:intro-onesign}, as the authors do not have an example where simultaneously $r_0(K, \tau) < \infty$ and $r_0(-K, - \tau) < \infty$. Instead, the two signs of Theorem~\ref{CP2RP2stabilization} are approached via different methods: stabilization of one sign is dealt with using the invariants of this paper, while stabilization of the other sign is studied using Heegaard Floer-theoretic techniques from \cite{DHM20, DMS}.

We also use the Chern--Simons filtration to explore the role of the fundamental group in equivariant knot theory by distinguishing isotopy-equivariant sliceness from isotopy-equivariant $\Z$-sliceness. This follows the broad principle that filtered instanton theory is capable of (specifically) obstructing simply connected bordism and $\Z$-slice surfaces; see for example \cite{T87, D18, NST19, DISST22}. Constraints on simply connected equivariant bordism were obtained by the authors and their collaborators in \cite{ADMT}. Here, we use the methods of this paper to establish the following:

\begin{restatable}{thm}{introZ}\label{thm:intro-Z}
There exists a strongly invertible knot which is $\Z$-slice and equivariantly slice, but not equivariantly (or even isotopy-equivariantly) $\Z$-slice. 
\end{restatable}

Explicitly, we show that if $(K, \tau)$ is any strongly invertible slice knot satisfying the hypotheses of Theorem~\ref{thm:intro-whitehead}, then setting
\[
K_0 = \Wh_1(K) \# - \Wh_1(K) \quad \text{and} \quad \tau_0 = \tau \# -\tau
\]
gives an example of Theorem~\ref{thm:intro-Z}. Note that any knot as in Theorem~\ref{thm:intro-Z} must be isotopy-equivariantly $\Z$-slice in the \textit{topological} category, since any two $\Z$-disks are necessarily isotopic rel boundary \cite{conway2021characterisation}.  Thus the non-isotopy claim of Theorem~\ref{thm:intro-Z} is inherently a phenomenon in the smooth category.

In fact, we show that such knots give a strengthening of the exotic pairs afforded by Theorem~\ref{CP2RP2stabilization}:
\begin{restatable}{thm}{introD}\label{thm:intro_definite}
    There exists a strongly invertible $\Z$-slice knot for which any symmetric pair of $\Z$-disks are exotic, and remain exotic after either of the following:
\begin{enumerate}[noitemsep]
    \item\label{definite1} External stabilization by any closed, simply connected definite $4$-manifold.
    \item\label{definite2} Internal stabilization by any number of embedded projective planes of the same sign, where each projective plane has exterior fundamental group $\Z_2$.
\end{enumerate}
\end{restatable}

While \eqref{definite1} is only a theoretical improvement -- as currently we do not know any examples of exotic simply connected definite manifolds -- we may apply \eqref{definite2} to conclude that our exotic pairs survive stabilization by the exotic projective planes provided by Miyazawa \cite{Miyazawa23}. Although Theorem~\ref{thm:intro_definite} enhances the exotic disks conclusion provided by Theorem~\ref{CP2RP2stabilization}, the two theorems are slightly different and are established in different ways. Theorem~\ref{CP2RP2stabilization} is a general claim about stabilizations which applies to all pairs of symmetric disks, with $\Z$-disks forming a special case. In contrast, the non-isotopy results of Theorem~\ref{thm:intro_definite} apply specifically to symmetric pairs of $\Z$-disks. Indeed, as discussed in Theorem~\ref{thm:intro-Z}, the example of Theorem~\ref{thm:intro_definite} is equivariantly slice (but not equivariantly $\Z$-slice). 

The proof of Theorem~\ref{thm:intro_definite} additionally gives another example of an equivariant homology sphere which bounds no simply connected equivariant definite manifold; the first such constructions were presented in \cite{ADMT}. These results motivate \cite[Problem 3.74(b)]{k3} of the recent Kirby problem list, which asks whether there is an equivariant homology sphere which does not bound any equivariant definite manifold of either sign.

%$such that any symmetric pair of $\Z$-disks is exotic after stabilization by any closed definite simply connected 4-manifold or any $n %\mathbb{RP}^2$ stabilization with $\pi_1(S^4\setminus \mathbb{RP}^2)=\Z_2$.$

\subsection{Linear independence results}

We now give several applications of our work to the structure of various equivariant concordance and cobordism groups. The first of these is to the theory of corks. In \cite{DHM20}, it was observed that if $(K, \tau)$ is a strongly invertible slice knot, then $1/n$-surgery along $K$ frequently produces a strong cork in the sense of \cite{LRS18, DHM20}. In \cite{ADMT}, the authors and their collaborators used filtered instanton theory to study linear independence results for corks formed by $1/n$-surgeries on a fixed knot. In particular, it was shown that $1/n$-surgeries on $\ov{9}_{46}$ give an infinite linearly independent family of corks, extending the well-known case of the Akbulut-Mazur cork when $n = 1$. Here, we use our $r_s$-invariant to provide a general criterion on a strongly invertible knot $K$ which guarantees the linear independence of equivariant $1/n$-surgeries along $K$.\footnote{This claim is only interesting when our $1/n$-surgeries along $K$ bound homology balls, as occurs if (for example) $K$ is slice. However, our proof applies more generally to the case where $\sigma(K) = 0$.}

\begin{restatable}{thm}{corkindependent}\label{cork_independent}
Let $(K, \tau)$ be a strongly invertible knot with $\sigma(K) = 0$. If $\smash{r_{-1/8} (K, \tau)< \infty }$, then any nontrivial linear combination of $\smash{\{S^3_{1/n}(K)\}_{n\geq 1}}$ yields a strong cork, where each $\smash{S^3_{1/n}(K)}$ is equipped with the surgered involution induced from $\tau$.
\end{restatable}

Our next application is to the \textit{isotopy-equivariant concordance group} $\smash{\Cis}$, which we define in Section~\ref{sec:iso-eq}. This is a straighforward generalization of the strongly invertible concordance group $\smash{\wt{\mathcal{C}}}$ of Sakuma \cite{Sa86}. Our construction admits an obvious locally flat topological analogue $\smash{\Cist}$ and a map from the smooth to the topological categories. We also have a forgetful homomorphism from $\smash{\Cis}$ to the usual smooth concordance group. We denote these by
\[
\mathfrak F:\Cis \rightarrow \Cist \quad \text{and} \quad F:\Cis \rightarrow \mathcal C,
\]
respectively. The following claim gives a large family of strongly invertible knots which are smoothly slice and isotopy-equivariantly slice in the topological category, but not isotopy-equivariantly slice in the smooth category. 

\begin{restatable}{thm}{introkernel}\label{thm:intro-kernel}
There exists an infinite family of nontorsion elements
\[
        x_k\in \ker\mathfrak F\cap\ker F\subset \Cis
        \quad \text{for} \quad k \in \N
\]
satisfying the following property. Consider any element $y$ in the span of the $x_k$, written as a word
\[
y=x_{k_1}^{a_1}x_{k_2}^{a_2}\cdots x_{k_m}^{a_m}, \quad \text{where the $k_i$ are not necessarily distinct.}
\]
If, when considered as an element in the free group formally generated by the $x_k$, the abelianization of this word is nontrivial, then $y$ is a nontrivial element of \(\ker\mathfrak F\cap\ker F\).
\end{restatable} 

Theorem~\ref{thm:intro-kernel} is an attempt to obtain an isotopy-equivariant analogue of the work of Endo \cite{E95}, who showed that the usual smooth concordance group admits a $\Z^\infty$-subgroup spanned by topologically slice knots. Unfortunately, this is difficult to establish in our setting due to the (possible) non-abelian nature of $\smash{\Cis}$. (See \cite{Pri22} for the non-commutativity of the usual equivariant concordance group.) In Theorem~\ref{thm:intro-kernel}, we present an infinite family of elements in \(\ker\mathfrak F\cap\ker F\) that exhibits a weak form of linear independence: any linear combination of the $x_i$ which is \textit{coefficient-nontrivial} when written as a word is indeed nontrivial. Here, a \textit{coefficient-nontrivial} word in a free group is one whose abelianization is nontrivial. This is enough to show (for example) that no $x_k$ is a linear combination of the other $x_j$.

\begin{comment}
    where the $k_i$ are not necessarily distinct. 

Let \(F_\infty=\langle X_1,X_2,\ldots\rangle\)
be the free group, and let \(X_i\mapsto x_i\). For any word
\[
        y=X_{i_1}^{a_1}X_{i_2}^{a_2}\cdots X_{i_m}^{a_m}\in F_\infty,
        \qquad a_j\in \Z\setminus\{0\},
\]
if its abelianization $
        [y]_{\mathrm{ab}}
        =
        \sum_{j=1}^m a_j[X_{i_j}]
        \in (F_\infty)_{\mathrm{ab}}\cong \Z^\infty$
is non-zero, then the image of \(y\) is non-trivial in
\(\ker\mathfrak F\cap\ker F\).
\end{comment}

\subsection{Computations of symmetry actions}
From the point of view of generating exotic phenomena, one of the principal advantages of leveraging equivariant knot theory is that obstructions to isotopy-equivariant sliceness can be defined in terms of $K$ and $\tau$ only, rather than referencing a specific pair of surfaces. This is useful when using invariants (such as instanton Floer theory) for which calculations involving surfaces or bordisms are scarce. However, even in the non-equivariant setting, determining the singular instanton complex of a knot is still difficult, with the main source of computations coming from the case of two-bridge knots discussed in \cite{DS19}. We thus give a brief description of the computational techniques developed in this paper that will allow us to establish our examples.

Our first result in this direction consists of two partial connected sum formulas which we prove in Theorems~\ref{conn sum} and \ref{conn sum2} (see also Theorems~\ref{enrich_conn sum} and \ref{enrich_conn sum2}). These will allow us to bootstrap from simple calculations which are determined for formal reasons to more complex examples. For instance, we frequently consider the connected sum of four trefoils
\[
(T_{2,3} \# T_{2,3}^r) \# (- T_{2,3} \# - T_{2,3}) \quad \text{with} \quad \tau_f \# (-\tau_0 \#-\tau_0).
\]
Here, $\tau_f$ is the flipping involution on $T_{2,3} \# T_{2,3}^r$ which interchanges the two factors, while $\tau_0$ is the unique strong inversion on the trefoil; see Example~\ref{ex:sumfour}. The action of $\tau_f$ can be computed from a connected sum formula tailored to the flipping involution, while the action of $\tau_0$ is the identity for formal reasons. We show that combining these via the equivariant connected sum formula of Theorem~\ref{conn sum} gives an example satisfying the hypotheses of Theorem~\ref{thm:intro-whitehead}. Similar knots were considered in \cite{DMS, GHKP}.

Our second (and more novel) computational technique proceeds via a direct analysis of \(\SU(2)\)-representation varieties. Recall that the singular instanton complex of a two-bridge knot $K_{p, q}$ is partially computed in terms of the $\SU(2)$-character variety of its branched cover $\Sigma_2(K_{p, q}) = L(p, q)$. This uses the equivariant ADHM construction developed in \cite{ADHM78, furuta1990invariant, austin1995equivariant, DS19}. For such knots, we thus begin by understanding the induced actions of our symmetries on this representation variety. We show that this provides a partial calculation of the action of $\tau$, which for certain $p$ and $q$ can be used to help construct interesting new examples. For instance, we partially compute the actions of two distinct involutions $\tau_0$ and $\tau_1$ on $K_{45, 26} = 9_{23}$ and show that the connected sum
\[
K_{45, 26} \# - K_{45, 26} \quad \text{with} \quad \tau_1 \# - \tau_0
\]
satisfies the hypotheses of Theorem~\ref{cork_independent}. To the best of the authors' knowledge, this is the first time that an explicit analysis of \(\SU(2)\)-representation varieties has had any direct bearing on the construction of corks.\footnote{In fact, the authors conjecture that the same example also satisfies the hypotheses of Theorem~\ref{thm:intro-whitehead}, but we do not have sufficient control over the relevant calculation to conclude this.}
%The second is [REF]. For example, we can take an example as $K=K_{45, 26} \# - K_{45, 26}$ equipped with 
%$\tau = \tau_1 \# - \tau_0$ , where $K_{45, 26}$ is the two bridge knot of the parameter $(45, 26)$, which is identified with $9_{23}$, and $\tau_1$ and $\tau_0$ are different strong inversions on $K_{45, 26}$. 

\subsection{Comparison with other work and technical difficulties}
We close by discussing connections with existing literature and highlighting some technical subtleties encountered throughout the paper.
% This Floer homology and its variants have spectral sequences from Khovanov homology and their variants \cite{KM11, kronheimer2021instantons}. These homologies are called {\it framed singular instanton homology} as they use the {\it admissible} orbifold bundles to avoid difficulties coming from the quotient singularities. 
\subsubsection{Singular instanton Floer theory} Originally defined by Kronheimer and Mrowka \cite{kronheimer2011knot, KM13}, singular instanton Floer theory for knots is a natural Floer-theoretic extension of the singular Donaldson invariants \cite{KM93,kronheimer1995gauge,kronheimer1997obstruction} associated to embedded surfaces in 4-manifolds. (See also \cite{collin1999instanton}.) In \cite{DS19}, Daemi and Scaduto introduced an equivariant refinement of this construction which associates to any knot $K$ a $U(1)$-equivariant complex $\smash{\wt{C}(K)}$ called an {\it $\smash{\cS}$-complex}. This is a singular analogue of the $SO(3)$-equivariant instanton Floer homology described in \cite{austin1996equivariant, Do02, Fr02, D18, Mike19}. They moreover constructed homomorphisms $h$ and $h^f$ from the concordance group $\mathcal{C}$ into the \textit{local equivalence group of $\cS$-complexes} $\Theta^{\cS}$ and \textit{local equivalence group of enriched $\cS$-complexes} $\Theta^{\mf E}$, respectively. Here, the latter is obtained from the former by retaining the information of the Chern--Simons filtration. See \cite{DS20, DISST22} for applications of these invariants. The maps $h$ and $h^f$ fit into the commutative diagram

\[\begin{tikzcd}
	{\mathcal{C}} & {\Theta^{\mathfrak E}} & {\Theta^S}
	\arrow[from=1-1, to=1-1]
	\arrow["{\mathrm{forget}}",from=1-2, to=1-3]
	\arrow["{h^f}", from=1-1, to=1-2]
	\arrow["h"', curve={height=18pt}, from=1-1, to=1-3]
\end{tikzcd}\]

In this paper, we define involutive versions of these local equivalence theories, with corresponding involutive local equivalence sets $\Theta^{\cS}_{eq}$ and $\Theta^{\mf E}_{eq}$. These will come with maps $h_\tau$ and $h^f_\tau$ out of the strongly invertible concordance group $\smash{\wt{\mathcal{C}}}$. Our $r_s$-invariant factors through $\smash{h^f_\tau}$. However, unlike $h$ and $\smash{h^f}$, we do \textit{not} claim that $h_\tau$ and $\smash{h^f_\tau}$ are homomorphisms. The maps $h_\tau$ and $h^f_\tau$ fit into the commutative diagram

\[\begin{tikzcd}
	{\widetilde{\mathcal{C}}} & {\Theta^{\mathfrak{E}}_{eq}} & {\Theta^S_{eq}} \\
	{\mathcal{C}} & {\Theta^{\mathfrak E}} & {\Theta^S}
	\arrow["{h^f_\tau}"', from=1-1, to=1-2]
	\arrow["{h_\tau}", curve={height=-18pt}, from=1-1, to=1-3]
	\arrow[from=1-1, to=2-1]
	\arrow["{\mathrm{forget}}"',from=1-2, to=1-3]
	\arrow[from=1-2, to=2-2]
	\arrow[from=1-3, to=2-3]
	\arrow["{h^f}", from=2-1, to=2-2]
	\arrow["h"', curve={height=18pt}, from=2-1, to=2-3]
	\arrow["{\mathrm{forget}}",from=2-2, to=2-3]
\end{tikzcd}\]

\subsubsection{Decorated bordism and weak local equivalence} A significant complication in our theory arises from the fact that $h_\tau$ and $h^f_\tau$ are not group homomorphisms. The reason for this is rather subtle, and is related to why our connected sum formulas are only partial. In Daemi and Scaduto's formalism \cite{DS19}, a knot cobordism $\Sigma$ from $K_1$ to $K_2$ only induces a morphism of $\cS$-complexes in the presence of a path $\gamma$ on $\Sigma$ connecting $K_1$ to $K_2$. This is similar to Zemke's graph TQFT formalism for knot Floer homology \cite{Zemkefunctoriality}, which is functorial only under \textit{decorated} knot cobordisms. In the equivariant setting, the most obvious requirement would thus be for an equivariant knot cobordism to come equipped with an equivariant (or isotopy-equivariant) path. However, for many applications, requiring the existence of such a path is inconvenient. Indeed, it turns out that for the pair-of-pants cobordism used to establish the usual connected sum formula in singular instanton theory, we do \textit{not} have an obvious choice of path data which makes the cobordism map equivariant.

We deal with this subtlety by introducing a new formalism called \textit{weak equivalence}. This a slightly strange notion in which our morphisms are only required to homotopy commute with the action of $\tau$ up to a controlled error term modeled on the possible contributions of different paths. Our connected sum formulas will thus only hold up to weak homotopy equivalence in this sense. Throughout, we often replace $\smash{\Theta^{\cS}_{eq}}$ and $\smash{\Theta^{\mf E}_{eq}}$ with their weak local equivalence analogues $\smash{\Theta^{\cS}_{w.eq}}$ and $\smash{\Theta^{\mf E}_{w.eq}}$; as we discuss in Section~\ref{sec:localequivalenceset}, these are \textit{not} groups. While the $r_s$-invariant factors through weak local equivalence, analyzing connected sums will require a delicate mix of the usual and weak formalisms. The authors believe that a more systematic understanding of decorations in the instanton TQFT would allow us to resolve some of these difficulties. In particular, the fact that the Sarkar basepoint-moving map \cite{sarkarbasepoint} is understood in knot Floer homology \cite{zemkebasepoint} is central to resolving similar subtleties with the equivariant connected sum formula in the knot Floer setting \cite{DMS}. 

\subsubsection{Heegaard Floer theory}\label{sec:GHKP} Theorem~\ref{thm:intro-whitehead} is reminiscent of recent work of Guth, Hayden, Kang and Park \cite{GHKP}, who showed that if two slice disks are distinguished by their maps on knot Floer homology, then their Whitehead doubled disks are likewise so distinguished. Their proof uses the bordered Heegaard Floer package to understand the behavior of knot Floer cobordism maps under satellite operations. As instanton Floer theory does not currently have a bordered counterpart, an approach analogous to \cite{GHKP} is not feasible in the instanton setting. 

Theorem~\ref{thm:intro-whitehead} also differs from the setup of \cite{GHKP} in the sense that we show \textit{any} pair of symmetric disks for $\Wh_i(K)$ are not isotopic, including pairs of disks that do not arise as Whitehead doubles of disks for $K$. Our proof of Theorem~\ref{thm:intro-whitehead} thus proceeds via a more topological argument involving branched covers that works broadly across different Floer theories. Indeed, the same argument as in the proof of Theorem~\ref{thm:intro-whitehead} gives a similar statement for the $\underline{V}^{\iota\tau}_{0}$-invariant defined in \cite{DMS}; see Remark~\ref{rem:heegaardfloer}.

\begin{rem}
Throughout, we work over $\Z_2$ and deal with the case of strongly invertible knots. As pointed out in \cite[Remark 6.8]{ADMT}, these are linked together: when studying symmetry actions of order $p$, it is generally best to use a coefficient ring whose characteristic is \textit{not} coprime to $p$. We study order-two symmetries largely for convenience and due to the fact that this is the most common case of interest. However, the authors believe that with sufficient care and the use of $\Z$-coefficients, the methods of this paper can be extended to study symmetries of any order.
\end{rem}

\subsection*{Organization}
In \Cref{sec:2}, we collect various topological preliminaries on strongly
invertible knots, isotopy-equivariant concordance, and branched covers. In
\Cref{sec:instantonreview}, we review Daemi and Scaduto's
$U(1)$-equivariant singular instanton theory for knots. We develop an involutive refinement of this formalism in \Cref{sec:involutive}. In \Cref{section:Two connected sum formula}, we
prove two partial connected sum formulas. In
\Cref{sec:CS-filtration}, we discuss the Chern--Simons filtration and
define our involutive $r_s$-invariant. In \Cref{sec:applications}, we prove all of the theorems discussed in the introduction with the exception of Theorem~\ref{CP2RP2stabilization} and give our central example of Theorem~\ref{thm:intro-whitehead}. In
\Cref{section:Computations and Examples}, we provide some partial computations of various symmetry actions for two-bridge knots using their traceless $\SU(2)$-character varieties.
Finally, in \Cref{sec:proof-applications}, we prove Theorem~\ref{CP2RP2stabilization}.

\subsection*{Acknowledgements} The authors would like to thank John Baldwin, Inanc Baykur, Keegan Boyle, Ali Daemi, Mike Miller Eismeier, Danny Ruberman, Taketo Sano, and Chris Scaduto for helpful conversations. The first author was partially supported by NSF DMS-1902746 and NSF DMS-2303823. The third author was partially supported by JSPS KAKENHI Grant Number 22K13921.

\section{Topological preliminaries}\label{sec:2}

We review some topological background regarding strongly invertible knots and branched covers.

\subsection{Equivariant bordism}\label{sec:2.1} We first set up some basic terminology regarding equivariant bordism. See \cite[Section 2]{DHM20} for further discussion.

\begin{defn}\label{def:eqsphere}
An \textit{equivariant homology sphere} is a pair $(Y, \tau)$, where $Y$ is an integer homology sphere and $\tau$ is an orientation-preserving involution on $Y$ with non-empty fixed point set.
\end{defn}

The fixed point set $\fix(\tau)$ of such a $\tau$ is necessarily a knot in $Y$; see \cite[Section 2]{My81}. 

%\begin{rem}
%We may of course consider the general case of a $3$-manifold equipped with a finite-order symmetry (or even an arbitrary self-diffeomorphism). However, the invariants of this paper are only defined for integer homology spheres. Likewise, while $\tau$ being an involution is not essential, it turns out that this is helpful when working with coefficients over $\Z_2$; see \cite[Remark 6.8]{ADMT}. We thus make these restrictions up front for the sake of concreteness.
%\end{rem}

\begin{defn}\label{def:eqbordism}
Let $(Y_1, \tau_1)$ and $(Y_2, \tau_2)$ be two equivariant homology spheres. An \textit{equivariant cobordism} from $(Y_1, \tau_1)$ to $(Y_2, \tau_2)$ is a pair $(W, \tau_W)$, where $W$ is a cobordism from $Y_1$ to $Y_2$ and $\tau_W$ is a self-diffeomorphism of $W$ with $\tau_W|_{Y_i} = \tau_i$ for each $i$. 
\end{defn}

We refer to $\tau_W$ as an extension of the $\tau_i$ over $W$. If $W$ is a homology cobordism, we say that $(W, \tau_W)$ is an equivariant homology cobordism. Quite frequently, we will deal with the case where the incoming end of $W$ is empty, in which case we say that $(Y, \tau)$ bounds the equivariant $4$-manifold $(W, \tau_W)$. We stress that $\tau_W$ is \textit{not} required to be an involution, although this makes the terminology of Definition~\ref{def:eqbordism} slightly misleading. Considering the case where $W$ is a $\Z_2$-homology ball gives the notion of a \textit{strong cork}, as discussed in \cite{LRS18, DHM20}. See \cite[Section 2]{DHM20} for further discussion of strong corks and the equivariant homology cobordism group. 

\begin{defn}
Let $(Y, \tau)$ be an equivariant homology sphere. We say that $(Y, \tau)$ is a \textit{strong cork} if $\tau$ does not extend (as a diffeomorphism) over any $\Z_2$-homology ball bounded by $Y$. In order to make this condition non-vacuous, we sometimes require that $Y$ bound at least one $\Z_2$-homology ball (or even a contractible manifold), although we do not do this in the statement of Theorem~\ref{cork_independent}.
\end{defn}

Occasionally, we will need to restrict the action of $\tau_W$ on homology:

\begin{defn}\label{def:homologyfixing}
let $(W, \tau_W)$ be an equivariant cobordism. We say that $\tau_W$ is \textit{homology-fixing} if it acts as $\id$ on $H_2(W, \Q)$ and \textit{homology-reversing} if it acts as $- \id$ on $H_2(W, \Q)$.
\end{defn}

\subsection{Strongly invertible knots}\label{sec:2.2}

We begin with the definition of a strongly invertible knot. Our discussion here generalizes the usual formalism of e.g.\ \cite{Sa86} to knots in other homology spheres.

\begin{defn}\label{def:involutiveknot}
A \textit{strongly invertible knot} is a triple $(Y, K, \tau)$, where $(Y, \tau)$ is an equivariant homology sphere and $K$ is a knot in $Y$ such that $\tau(K) = K$ and $K \cap \fix(\tau)$ consists of exactly two points.
\end{defn}

The condition that $\fix(\tau)$ intersects $K$ in two points shows that $\tau$ \textit{reverses} orientation on $K$.

\begin{defn}\label{def:sakumaequivalent}
Two strongly invertible knots $(Y_1, K_1, \tau_1)$ and $(Y_2, K_2, \tau_2)$ are said to be \textit{equivariantly diffeomorphic} (or \textit{Sakuma equivalent}) if there exists a diffeomorphism $f \colon Y_1 \rightarrow Y_2$ which sends $K_1$ to $K_2$ and intertwines $\tau_1$ and $\tau_2$; that is, $f \circ \tau_1 = \tau_2 \circ f$.
\end{defn}

If we refer to $(K, \tau)$ with no mention of $Y$, the reader should assume $Y = S^3$. In this case, $\fix(\tau)$ is an unknot \cite{Wa69}. We thus have that up to Sakuma equivalence, every strongly invertible knot in $S^3$ may be regarded as a strongly invertible knot in $S^3$ equipped with the standard involution given by rotation around a fixed unknot. 

\begin{defn}\label{def:eqconc}
Let $K_1$ and $K_2$ be two strongly invertible knots in $S^3$ equipped with the standard involution $\tau$. An \textit{equivariant knot cobordism} from $K_1$ to $K_2$ in $[0, 1] \times S^3$ is a knot cobordism $S \subset [0, 1] \times S^3$ from $K_1$ to $K_2$ such that $\tau_{[0, 1] \times S^3}(S) = S$, where $\tau_{[0, 1] \times S^3} = \id \times \tau$.
\end{defn}

If $S$ is a cylinder, we say that $K_1$ and $K_2$ are equivariantly concordant. We have an obvious definition of equivariant sliceness using the standard extension $\tau_{B^4}$ of $\tau$ over $B^4$. However, in this paper we will mainly deal with the following fundamental generalization of Definition~\ref{def:eqconc}:

\begin{defn}\label{def:isoeqconc}
Let $(Y_1, K_1, \tau_1)$ and $(Y_2, K_2, \tau_2)$ be two strongly invertible knots. An \textit{isotopy-equivariant knot cobordism} from $K_1$ to $K_2$ in $W$ is triple $(W, S, \tau_W)$, where $(W, \tau_W)$ is an equivariant cobordism from $(Y_1, \tau_1)$ to $(Y_2, \tau_2)$ and $S \subset W$ is a knot cobordism from $K_1$ to $K_2$ such that $\tau_W(S)$ is isotopic to $S$ rel boundary.
\end{defn}

Once again, if $S$ is a cylinder, then we refer to $S$ as an isotopy-equivariant knot concordance in $W$, and so on. Note that we have relaxed Definition~\ref{def:eqconc} in several ways: by replacing $[0, 1] \times S^3$ with a general $W$, replacing $\tau_{[0, 1] \times S^3}$ with a general self-diffeomorphism $\tau_W$, and allowing isotopy-invariance rather than strict invariance. In fact, the last condition is somewhat superfluous in light of the second:

%As with the notion of an equivariant cobordism, we do not require $\tau_W$ here to be an involution.
\begin{lem}\label{lem:isotopy}
Let $(Y_1, K_1, \tau_1)$ and $(Y_2, K_2, \tau_2)$ be two strongly invertible knots. A knot cobordism $(W, S)$ is isotopy-equivariant if and only if $S$ is fixed setwise by some extension $\tau_W'$ of the $\tau_i$ over $W$.
\end{lem}
\begin{proof}
The reverse direction is tautological, since invariance is a special case of isotopy-invariance. For the forwards direction, let $\tau_W$ be the self-diffeomorphism afforded by the definition of isotopy-equivariance. Then $\tau_W(S)$ is ambiently isotopic rel boundary to $S$. The endpoint of this isotopy gives a self-diffeomorphism $g$ of $W$ which is the identity on the $Y_i$. But then taking $\tau_W' = g \circ \tau_W$ gives a self-diffeomorphism of $W$ which restricts to $\tau_i$ on the ends and under which $S$ is strictly invariant.
\end{proof}

%Once again, if $S$ is a cylinder, then we refer to $S$ as an isotopy-equivariant knot concordance in $W$, and so on. 
It may thus seem slightly silly to refer to two knots as being isotopy-equivariantly concordant, rather than (for example) diffeomorphism-equivariantly concordant. The distinction between isotopy-equivariance and equivariance is only meaningful if the extension $\tau_W$ is fixed, or if we require $\tau_W$ to be an involution. However, we continue to use the language of isotopy-equivariance, since we will often have in mind the case that $\tau_W = \tau_{B^4}$ or $\tau_{[0, 1] \times S^3}$.

Isotopy-equivariance also provides the conceptually clearest framework for generating non-isotopic surfaces, as discussed in the introduction. Indeed, let $(Y, K, \tau)$ be a strongly invertible knot which is not isotopy-equivariantly slice in any $\Z_2$-homology ball. Suppose, however, that $K$ is in fact (non-equivariantly) slice in some particular $\Z_2$-homology ball $W$. Let $D$ be any such slice disk and let $\tau_W$ be any extension of $\tau$ over $W$. Then we trivially have that $D$ and $\tau_W(D)$ are not smoothly isotopic rel boundary in $W$. We refer to $D$ and $\tau_W(D)$ as a symmetric pair of slice disks for $K$ in $W$. 

The simplest application of this is when $K$ is a strongly invertible slice knot in $S^3$ and $W = B^4$. If $K$ is as in the previous paragraph, then any symmetric pair of slice disks $D$ and $\tau_{B^4}(D)$ are not smoothly isotopic rel boundary. Such examples are of special interest when $D$ and $\tau_{B^4}(D)$ are known to be topologically isotopic; e.g., when the complement of $D$ has fundamental group $\Z$. This has been used by the first two authors and their collaborators to produce new methods for generating exotic pairs of slice disks \cite{DMS}. While $W = B^4$ will usually be the case of interest, we stress the methods of this paper allow for any choice of $\Z_2$-homology ball $W$ and any extension $\tau_W$.

\subsection{Isotopy-equivariant concordance group}\label{sec:iso-eq}
In this section, we define the \textit{isotopy-equivariant concordance group}. This is a straightforward generalization of Sakuma's strongly invertible concordance group; see \cite{Sa86} or \cite[Section 2]{BI22}.

\begin{defn}
A \emph{direction} on $(Y, K,\tau)$ is a choice of oriented half-axis: that is, an orientation on the fixed-point axis
\[
\gamma=\fix(\tau)\subset S^3
\]
together with a choice of connected component of $\gamma \smallsetminus K$. This determines an ordering of $\gamma \cap K = \{p_-, p_+\}$. Viewing the direction as an oriented arc, $p_-$ is the tail/initial point of the direction, while $p_+$ is the head/final point.
\end{defn}

A choice of direction is necessary to unambiguously define the connected sum of two strongly invertible knots.

\begin{defn}\label{def:connectedsum}
Given two directed strongly invertible knots $(Y_1, K_1, \tau_1)$ and $(Y_2, K_2, \tau_2)$, we form their \textit{equivariant connected sum} by taking the connected sum at the head of the direction of $K_1$ and the tail of the direction of $K_2$. This is done in such a way so that the new fixed-point axis is the oriented connected sum of $\fix(\tau_1)$ and $\fix(\tau_2)$. To define a direction, we let the new half-axis be the union of the two original half-axes, with orientation inherited from both.
\end{defn}

\begin{rem}
In general, the equivariant connected sum depends on the directions of the summands and the order that the connected sum is formed. However, in this paper we will usually gloss over this subtlety. In Section~\ref{sec:localequivalenceset}, we discuss the fact that certain invariants in this paper have a (theoretical) dependence on a choice of direction. However, for our applications, we will either use coarser invariants that are insensitive to this data, or we will utilize algebraic arguments that work equally well regardless of how our connected sums are formed or our directions are chosen.
\end{rem}

We now construct the isotopy-equivariant concordance group. For concreteness, we consider only knots in $S^3$ and concordances in $[0, 1] \times S^3$, although the definition easily generalizes.

\begin{defn}
Let $(K, \tau)$ and $(K', \tau')$ be directed strongly invertible knots. We say that $(K, \tau)$ and $(K', \tau')$ are \emph{(directed) isotopy-equivariantly concordant} if there exists a concordance $C\subset [0,1]\times S^3$ from $K$ to $K'$, a self-diffeomorphism $\smash{\widehat{\tau}} \colon [0, 1] \times S^3 \rightarrow [0, 1] \times S^3$, and a concordance $\Gamma \subset [0,1]\times S^3$ from $\fix(\tau)$ to $\fix(\tau')$, such that:
\begin{enumerate}[noitemsep]
\item\label{diec1} $\widehat{\tau}$ is an extension of $\tau$ and $\tau'$ over $[0, 1] \times S^3$; 
\item\label{diec2} $(C, \Gamma)$ and $(\widehat{\tau}(C), \widehat{\tau}(\Gamma))$ are pairwise smoothly isotopic rel boundary; and,
\item\label{diec3} The intersection $C \cap \Gamma$ consists of a pair of arcs that divides $\Gamma$ into two components. We require that the chosen half-axes for $K$ and $K'$ lie in the same connected component of $\Gamma \smallsetminus C$.
\end{enumerate} 
\end{defn}

Here, the directions on $K$ and $K'$ are used twice: firstly, $\Gamma$ is required to be an oriented concordance from $\fix(\tau)$ to $\fix(\tau')$; and secondly, we have condition \eqref{diec3} that the chosen half-axes lie in the same component of $\Gamma \smallsetminus C$. It is clear that the two arcs afforded by \eqref{diec3} necessarily run from the head of $K$ to the head of $K'$, and from the tail of $K$ to the tail of $K'$. Note that since $\widehat{\tau}$ is an isotopy, \eqref{diec3} also holds with $C$ and $\Gamma$ replaced with $\widehat{\tau}(C)$ and $\widehat{\tau}(\Gamma)$. We draw a schematic of \eqref{diec3} in Figure~\ref{fig:annulus}.

Once again, in reality we may replace the isotopy-equivariance condition in \eqref{diec2} with strict equivariance. Indeed, if $\smash{\widehat{\tau}}$ is a self-diffeomorphism as in \eqref{diec1} and $g$ is the endpoint of the isotopy in \eqref{diec2}, then the composition $\smash{g \circ \widehat{\tau}}$ will be an extension of $\tau$ and $\tau'$ which fixes each of $C$ and $\Gamma$ setwise. This necessarily implies that the two arcs of \eqref{diec3} are both fixed by $\smash{g \circ \widehat{\tau}}$. Thus \eqref{diec2} is only meaningful if we have in mind a particular or restricted $\smash{\widehat{\tau}}$ in \eqref{diec1}.

\begin{figure}[h!]
\includegraphics[scale = 0.7]{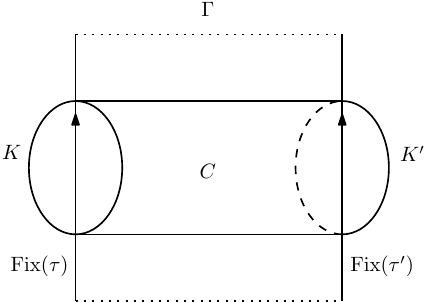}
	\caption{The concordance $C$ from $K$ to $K'$, together with the concordance $\Gamma$ from $\fix(\tau)$ to $\fix(\tau')$. We draw $C$ as a cylinder and draw (part of) $\Gamma$ as the vertical sheet. In this picture, $C \cap \Gamma$ consists of a pair of arcs that divides $\Gamma$ into two rectangles. The component containing the half-axes for $K$ and $K'$ is visible as the central rectangle.}
	\label{fig:annulus}
\end{figure}

\begin{defn}
Define the \emph{isotopy-equivariant concordance group} by
\[
\Cis = \{\text{all directed strongly invertible knots}\}/\sim_{ieq},
\]
where $\sim_{ieq}$ is the equivalence relation on the set of all directed strongly invertible knots given by (directed) isotopy-equivariant concordance. 
The group operation is induced by the equivariant connected sum of Definition~\ref{def:connectedsum}. The identity element is given by the strongly invertible unknot (with any choice of direction). The concordance inverse of $(K, \tau)$ is given by $(-K, - \tau)$. This is obtained by mirroring $K$ and taking the mirrored involution $-\tau$; the direction on $(-K, -\tau)$ is obtained by pushing forward the original direction under the mirror and then reversing orientation on the fixed-point axis. 
\end{defn}

It is straightforward to check that isotopy-equivariant concordance is compatible with equivariant connected sum, especially given the equivalence between isotopy-equivariance and diffeomorphism-equivariance discussed above. It will also be helpful to have the following re-formulation of (directed) isotopy-equivariant concordance. Note that in the following claim, the notion of isotopy-equivariant sliceness has no condition on any direction data. 

\begin{comment}
\begin{lem}
Isotopy-equivariant concordance is compatible with equivariant connected sum.
\end{lem}

\begin{proof}
Let $(C,\Gamma,\widehat\tau)$ be a directed isotopy-equivariant concordance
from $(K_0,\tau_0)$ to $(K_1,\tau_1)$, and let
$(C',\Gamma',\widehat\sigma)$ be one from $(J_0,\sigma_0)$ to
$(J_1,\sigma_1)$.  We take the connected sum of the concordances along the
chosen directed components of $\Gamma\setminus C$ and
$\Gamma'\setminus C'$.  This gives a concordance 
$C\# C' \subset [0,1]\times S^3$ 
from \(K_0\# J_0\) to \(K_1\# J_1\), together with an annulus
$
        \Gamma\#\Gamma'$
connecting the fixed axes of \(\tau_0\#\sigma_0\) and
\(\tau_1\#\sigma_1\).  Using the standard equivariant local model near the
chosen summing arcs, the diffeomorphisms \(\widehat\tau\) and
\(\widehat\sigma\) glue to an orientation-preserving diffeomorphism
$
        \widehat\tau\#\widehat\sigma:
        [0,1]\times S^3\to [0,1]\times S^3$
restricting to \(\tau_0\#\sigma_0\) and \(\tau_1\#\sigma_1\) on the two
boundary components. By assumption, the pairs \((C,\Gamma)\) and
\((\widehat\tau(C),\widehat\tau(\Gamma))\), and similarly
\((C',\Gamma')\) and
\((\widehat\sigma(C'),\widehat\sigma(\Gamma'))\), are smoothly isotopic rel
boundary.  Taking the connected sum of these ambient isotopies gives a smooth
ambient isotopy rel boundary from $
        (C\# C',\Gamma\#\Gamma')$
to $
        ((\widehat\tau\#\widehat\sigma)(C\# C'),
        (\widehat\tau\#\widehat\sigma)(\Gamma\#\Gamma')).$
\end{proof}
\end{comment}

\begin{lem}
Let $(K,\tau)$ and $(K',\tau')$ be directed strongly invertible knots. Then $(K,\tau)$ and $(K',\tau')$ are (directed) isotopy-equivariantly concordant if and only if
\[
        (K,\tau)\#-(K',\tau')
\]
is isotopy-equivariantly slice.
\end{lem}

\begin{proof}
Suppose first that $(K,\tau)$ and $(K',\tau')$ are (directed)
isotopy-equivariantly concordant. Let $C$ and $\Gamma$ be the concordances afforded by the definition; as discussed above, we may assume that $C$ and $\Gamma$ are in fact fixed setwise by some extension $\psi$ of $\tau$ and $\tau'$. Let $\alpha_+$ be the arc of $C \cap \Gamma$ running from the head of $K$ to the head of $K'$. After a further isotopy supported in a small
neighborhood of \(\alpha_+\), we may assume \(\psi\) preserves a tubular
neighborhood \(N(\alpha_+)\) of $\alpha_+$. The complement $W:=([0,1]\times S^3)\smallsetminus \operatorname{int}N(\alpha_+)$ is diffeomorphic to \(B^4\) and $
        D:=C\cap W$
is a slice disk for the connected sum \(K\# -K'\).  Under the identification \(W\cong B^4\), the restriction \(\psi|_W\) gives an extension of
the strong inversion \(\tau\# -\tau'\) on \(K\# -K'\), and \(\psi|_W(D)=D\).
Thus \((K,\tau)\#-(K',\tau')\) is isotopy-equivariantly slice.

Conversely, suppose that \((K,\tau)\#-(K',\tau')\) is isotopy-equivariantly slice. Let \(D\subset B^4\) be a slice disk for \(K\#-K'\), and let \(\Phi:B^4\to B^4\) be an orientation-preserving diffeomorphism extending the boundary strong inversion \(\tau\#-\tau'\), such that \(D\) and \(\Phi(D)\) are smoothly ambiently isotopic rel boundary. We undo the connected sum at the chosen directed fixed point. More precisely, attach to \(B^4\) the standard equivariant four-dimensional \(3\)-handle \(H^3=D^3\times D^1\) along the connected-sum neck on \(\partial B^4\). The resulting manifold is diffeomorphic to \([0,1]\times S^3\), with boundary components identified with the two original copies of \(S^3\). On the surface level, we attach the standard two-dimensional \(1\)-handle \(h^1=D^1\times D^1\subset H^3\) to \(D\) along the two small arcs of \(\partial D\) which come from the connected-sum band. Thus \(C:=D\cup h^1\) is a concordance from \(K\) to \(K'\). The standard equivariant \(3\)-handle also contains the standard trace of the fixed axis as the fixed point set of the involution on the 3-handle. Namely, its equivariant local model contains a properly embedded two-dimensional \(1\)-handle \(h_\Gamma\cong D^1\times D^1\subset H^3\) whose boundary attaches to the two pieces of the connected-sum axis \(\fix(\tau)\#-\fix(\tau')\). After the \(3\)-handle is attached, this axis is split back into the two axes \(\fix(\tau)\) and \(\fix(\tau')\). We define \(\Gamma\) to be the annulus obtained from these two axis pieces together with \(h_\Gamma\). Thus \(\Gamma\subset [0,1]\times S^3\) is an annulus from \(\fix(\tau)\) to \(\fix(\tau')\). In the standard local model, \(h^1\) and \(h_\Gamma\) meet along the core arc corresponding to the chosen directed fixed point. The diffeomorphism \(\Phi\) extends over the standard equivariant \(3\)-handle by the standard model, giving an orientation-preserving diffeomorphism \(\widehat\tau:[0,1]\times S^3\to [0,1]\times S^3\) restricting to \(\tau\) and \(\tau'\) on the two boundary components. The ambient isotopy rel boundary between \(D\) and \(\Phi(D)\), together with the standard isotopy on the added \(3\)-handle and on the two-dimensional \(1\)-handle data, gives an ambient isotopy rel boundary between the pairs \((C,\Gamma)\) and \((\widehat\tau(C),\widehat\tau(\Gamma))\). Therefore \((K,\tau)\) and \((K',\tau')\) are directed isotopy-equivariantly concordant.
\end{proof}

We have two natural homomorphisms
\[
        \widetilde{\mathcal C}
        \xrightarrow{\;\omega\;}
        \Cis
        \xrightarrow{\;F\;}
        \mathcal C . 
\] The first map $\omega$ is the \emph{weakening homomorphism}; it sends an equivariant concordance class to the same directed strongly invertible knot regarded up to isotopy-equivariant concordance. The second map $F$ is the \emph{forgetful homomorphism}; it forgets the strong inversion and the direction. Indeed, every equivariant concordance is an isotopy-equivariant concordance, and every isotopy-equivariant concordance gives an ordinary concordance after forgetting all data except for $C$. Since both maps are compatible with equivariant connected sum and ordinary connected sum, they are group homomorphisms. 

We also have a locally flat topological version of $\smash{\Cis}$. We define $\smash{\Cist}$ in the same way as $\smash{\Cis}$, except that $C$ and $\Gamma$ are required only to be locally flat topological annuli in $[0, 1] \times S^3$, the ambient diffeomorphism $\widehat\tau$ is replaced by an ambient homeomorphism, and the isotopy of pairs is required only to be a locally flat topological isotopy rel boundary. There is a natural homomorphism
\[
        \mathfrak{F}:\Cis\longrightarrow
        \Cist,
\]
from the smooth to the topological category.

\subsection{Branched covers}\label{sec:2.4}

We now give an important discussion of strongly invertible knots and their branched covers. We will use the following claims throughout the paper: 

\begin{lem}[see also Proposition 12 of \cite{BI22}]\label{lem:liftbranched1}
Let $(Y, K, \tau)$ be a strongly invertible knot and $\Sigma_2(K)$ be the double cover of $Y$, branched over $K$. Then $\tau$ lifts to exactly two involutions on $\Sigma_2(K)$, which differ by composition with the branching action.
\end{lem}
\begin{proof}
We follow the proof given in \cite[Proposition 12]{BI22}. Using the equivariant tubular neighborhood theorem, let $N(K)$ be a closed tubular neighborhood of $K$ which is fixed by $\tau$ and on which $\tau$ acts as a bundle map. We first lift $\tau$ on the complement of $\inte N(K)$. This can be done due to some basic algebraic topology, as discussed in the proof of \cite[Proposition 12]{BI22}. Explicitly, let $X = Y - \inte N(K)$ and let $\smash{\wt{X}}$ be the double cover corresponding to the unique nontrivial homomorphism
\[
\pi_1(X) \rightarrow H_1(X, \Z) \cong \Z \rightarrow \Z_2.
\]
Here, when calculating $\pi_1$, we choose a basepoint on $\fix(\tau)$. To show that $\tau$ lifts to a self-diffeomorphism of $\smash{\wt{X}}$, consider the diagram
\[
\begin{tikzcd}
	{\widetilde{X}} & {\widetilde{X}} \\
	X & X
	\arrow[dashed, from=1-1, to=1-2]
	\arrow["\pi"', from=1-1, to=2-1]
	\arrow["\pi", from=1-2, to=2-2]
	\arrow["\tau", from=2-1, to=2-2]
\end{tikzcd}
\]
By the lifting criterion, it suffices to show that the image of $\smash{(\tau \circ \pi)_* \colon \pi_1(\wt{X}) \rightarrow \pi_1(X)}$ lies in the image of $\smash{\pi_* \colon \pi_1(\wt{X}) \rightarrow \pi_1(X)}$. Indeed, note that $\ima \pi_*$ is the unique index-two subgroup of $\pi_1(X)$. (To prove uniqueness, note that any index-two subgroup is the kernel of some nontrivial homomorphism from $\pi_1(X)$ to $\Z_2$. On the other hand, there is clearly only one such homomorphism, since every such homomorphism must factor through $H_1(X, \Z) \cong \Z$.) Since $\tau_*(\ima \pi_*)$ is another index-two subgroup, we have $\wt{\tau}(\ima \pi_*) = \ima \pi_*$. Once $\tau$ has been lifted over $X$, we lift $\tau$ over $N(K)$ explicitly, as in the proof of \cite[Proposition 12]{BI22}.
\end{proof}

\begin{defn}
By abuse of notation, we denote both lifts of $\tau$ also by $\tau$. Generally speaking, any statement that applies to one of these lifts will apply to both of them. In cases where we must specify a particular lift, we will describe it explicitly and refer to the other one by $q \circ \tau$, where $q$ is the branching action. 
\end{defn}

\begin{lem}[see also Proposition 12 of \cite{BI22}]\label{lem:liftbranched2}
Suppose $(Y, K, \tau)$ is isotopy-equivariantly slice in some $\Z_2$-homology ball. Then $(\Sigma_2(K), \tau)$ bounds a $\Z_2$-homology ball over which $\tau$ extends.\footnote{Strictly speaking, to comply with Definition~\ref{def:eqsphere}, we should assume that $\Sigma_2(K)$ is an integer homology sphere. While Lemma~\ref{lem:liftbranched2} obviously applies in the case that $\Sigma_2(K)$ is a rational homology sphere, we will only apply Lemma~\ref{lem:liftbranched2} in the setting where $\Sigma_2(K)$ is an integer homology sphere.}
\end{lem}

\begin{proof}
We follow the proof given in \cite[Proposition 12]{BI22}. Let $(Y, K, \tau)$ be isotopy-equivariantly slice in some $\Z_2$-homology ball $W$ and let $(W, D, \tau_W)$ be the associated slicing data. By definition, this means that $\tau_W(D)$ is isotopic to $D$ rel boundary. As usual, the endpoint of this isotopy gives a self-diffeomorphism $g \colon W \rightarrow W$ such that
\[
f = g \circ \tau_W \colon W \rightarrow W
\]
fixes $D$ setwise and acts as $\tau$ on $Y = \partial W$.

The proof is now essentially the same as that of Lemma~\ref{lem:liftbranched1}. However, it is not quite true that $f$ must fix a tubular neighborhood of $D$: since $f$ is now a general diffeomorphism, rather than a finite-order automorphism of $W$, we cannot apply the equivariant tubular neighborhood theorem. Instead, we claim that we can isotope $f$ rel boundary to a self-diffeomorphism $f' \colon W \rightarrow W$ that fixes some closed tubular neighborhood $N(K)$ and acts on $N(K)$ as a bundle map. This follows from the uniqueness of tubular neighborhoods (see for example \cite[Chapter 4, Theorem 6.5]{Hirsch}) and the isotopy extension theorem. Indeed, simply fix any closed tubular neighborhood $N(D)$ of $D$, and observe that $f(N(D))$ and $N(D)$ are both closed tubular neighborhoods, which must then be (ambiently) isotopic. To see that this may be done rel boundary, note that near $Y = \partial W$, we may assume $f$ acts like $\tau$. The equivariant tubular neighborhood theorem thus already gives a tubular neighborhood near $\partial F$ preserved by $f$.

Using the same argument as in Lemma~\ref{lem:liftbranched1}, lift $f'$ to a self-diffeomorphism of the double cover $\Sigma_2(D)$ of $W$, branched over $D$. Here, note that there is again a unique nontrivial homomorphism
\[
\pi_1(W \smallsetminus D) \rightarrow H_1(W \smallsetminus D, \Z) \cong \Z \oplus \mathrm{Torsion}\rightarrow \Z_2
\]
using the fact that $H_1(W \smallsetminus D, \Z_2) \cong \Z_2$. The two lifts of $f'$ act as the two lifts of $\tau$ on $\Sigma_2(K) = \partial \Sigma_2(D)$. Hence our claim holds regardless of which lift of $\tau$ to $\Sigma_2(K)$ we consider.
\end{proof}

\subsection{Basic examples}\label{sec:2.5}
We now give some useful examples of strongly invertible knots $(Y, K, \tau)$. 

\begin{ex}\label{ex:unknot}
Our first example is rather silly: given any equivariant homology sphere $(Y, \tau)$, let $U$ be a strongly invertible unknot defined in a small ball intersecting $\fix(\tau)$. Clearly, all such choices of $U$ are the same up to Sakuma equivalence. 
\end{ex}

While the invariants of this paper are defined for strongly invertible knots $(Y, K, \tau)$, we will thus frequently treat the invariant associated to $(Y, U, \tau)$ as an invariant of the equivariant homology sphere $(Y, \tau)$. We have the following extremely obvious lemma. 

\begin{lem}\label{lem:obvious}
Let $(Y, U, \tau)$ be a small strongly invertible unknot. Then $(Y, \tau)$ bounds an equivariant $\Z_2$-homology ball if and only if $(Y, U, \tau)$ is isotopy-equivariantly slice in a $\Z_2$-homology ball.
\end{lem}
\begin{proof}
If $(Y, \tau)$ bounds an equivariant $\Z_2$-homology ball $(W, \tau_W)$, then we obtain an equivariant slice disk for $U$ simply by pushing the obvious Seifert disk for $U$ into $W$ and using the fact that in a collar neighborhood of $Y = \partial W$, we may assume $\tau_W$ acts like $\tau$. The converse direction is even more tautological, since the existence of $(W, \tau)$ is part of the definition.
\end{proof}

In light of Lemmas~\ref{lem:liftbranched2} and \ref{lem:obvious}, we thus have two methods for proving that a strongly invertible knot $(Y, K, \tau)$ is not isotopy-equivariantly slice. Roughly speaking, these are:
\begin{enumerate}[noitemsep]
\item Obstruct this directly using equivariant concordance invariants associated to $(Y, K, \tau)$; or, 
\item Show that the branched cover does not bound an equivariant $\Z_2$-homology ball using invariants associated to $(\Sigma_2(K), U, \tau)$. (Here, we assume $\Sigma_2(K)$ is an integer homology sphere.)
\end{enumerate}
We will actually use the second method to prove several of our key results. The reader may thus wonder whether it is truly necessary to develop the theory for strongly invertible knots, as opposed to equivariant homology spheres, as was done in \cite{ADMT}. As we shall see, however, we will usually not be able to compute the invariants for $(\Sigma_2(K), U, \tau)$ directly. Instead, we give (in the cases of interest) a topological argument that bounds the invariants of $(\Sigma_2(K), U, \tau)$ in terms of the invariants of another strongly invertible knot. It is this auxiliary computation which will require the usage of the theory for nontrivial knots.

\begin{ex}\label{ex:sumfour}
We will frequently use the following construction. For any knot $J \subset S^3$, we may form the strongly invertible knot $(J \# J^r, \tau_f)$, where $\tau_f$ is the involution which interchanges the two summands $J$ and $J^r$, as in Figure~\ref{fig:trefoils1}. We refer to $\tau_f$ as the \textit{flipping} or \textit{swapping} strong inversion. 
\begin{figure}[h!]
\includegraphics[scale = 0.7]{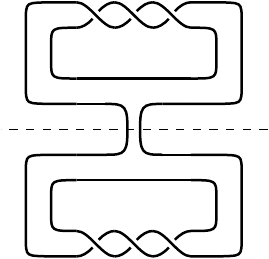}
	\caption{The flip map $\tau_f$ on $T_{2,3} \# T_{2,3}^r$.}
	\label{fig:trefoils1}
\end{figure}

Suppose moreover that $J$ itself is strongly invertible with some strong inversion $\tau_0$. We may then consider the strongly invertible knot $(2J \# - 2J, \tau_f \# -2\tau_0)$, as displayed in Figure~\ref{fig:trefoils2}. The strong inversion $\tau = \tau_f \# -2\tau_0$ acts as the flip map on the first two summands, while we apply the strong inversion $\tau_0$ of $J$ to the latter two summands. Technically, we should write $J \# J^r \# -2J$, but since $J$ is strongly invertible, $J$ and $J^r$ are isotopic. Note that $2J \# -2J$ is (non-equivariantly) slice. We will often consider the case 
\[
(K = 2T_{2,3} \# -2T_{2,3}, \tau = \tau_f \# -2\tau_0),
\] 
as shown in Figure~\ref{fig:trefoils2}. Similar knots were used in \cite{DMS, GHKP}.
\begin{figure}[h!]
\includegraphics[scale = 0.7]{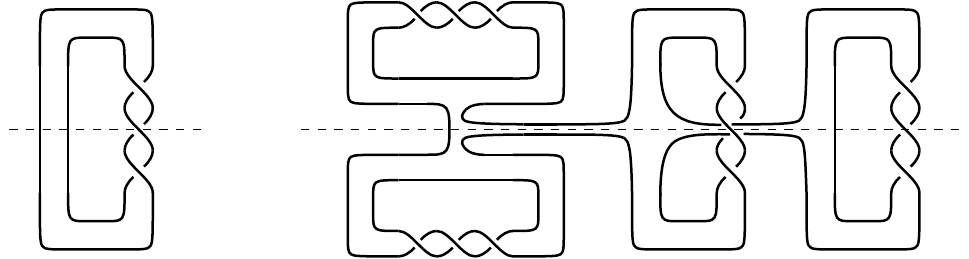}
	\caption{Left: the unique strong inversion $(-)\tau_0$ on the trefoil $(-)T_{2,3}$. Right: the slice knot $K = 2T_{2,3} \# -2T_{2,3}$, equipped with the strong inversion $\tau = \tau_f \# -2\tau_0$.}
	\label{fig:trefoils2}
\end{figure}
\end{ex}

\begin{ex}\label{ex:whitehead}
Let $(K, \tau)$ be a strongly invertible knot in $S^3$. Then we may form the (multiply-clasped) Whitehead double $\Wh_i(K)$ of $K$, as in Figure~\ref{fig:whitehead0}. This inherits a strong inversion, which by abuse of notation we also denote by $\tau$. In fact, there are two distinct ways to form the equivariant Whitehead double, corresponding to placing the clasp at either one of the two fixed points of $\tau$ on $K$. It is convenient to equip $K$ with a direction, in which case we think of the clasp as either being placed at the foot or the head of the direction. Note that if $K$ is smoothly slice (in $B^4$), then $\Wh_i(K)$ is smoothly $\Z$-slice; see \cite[Proposition 2.1]{GHKP}. 
\begin{figure}[h!]
  \centering
  \includegraphics[scale=0.9]{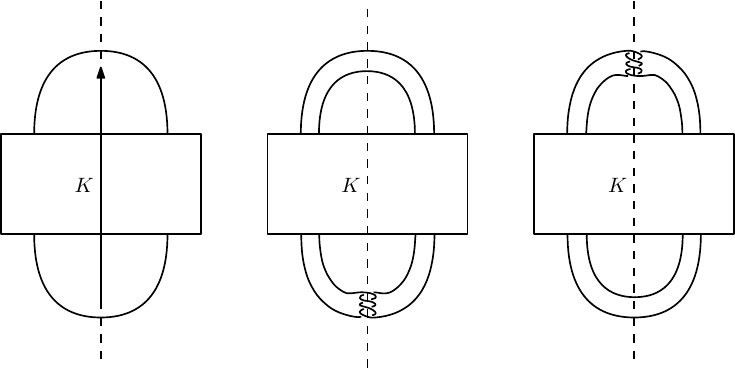}

	\caption{The $i$-clasped Whitehead double $\Wh_i(K)$ for $i = 2$. Left: the original strongly invertible knot $K$, with fixed points drawn at the top and bottom of $K$. The arrow represents a choice of direction. Middle: placing the clasp at the foot of the indicated direction. Right: placing the clasp at the head of the indicated direction.}
	\label{fig:whitehead0}
\end{figure}

Now take the branched double cover $\Sigma_2(\Wh_i(K))$. We wish to identify the involution on $\Sigma_2(\Wh_i(K))$ afforded by Lemma~\ref{lem:liftbranched1}. As is well-known, for any knot $K$ we have a diffeomorphism 
\begin{equation}\label{eq:branchedident}
\Sigma_2(\Wh_i(K)) \cong S^3_{1/(2i)}(K \# K^r).
\end{equation}
We claim that:
\begin{enumerate}[noitemsep]
\item The branching action $q$ on $\Sigma_2(\Wh_i(K))$ is identified with the involution on $\smash{S^3_{1/(2i)}(K \# K^r)}$ induced by the flip map interchanging $K$ and $K^r$; and,
\item The lift of $\tau$ to $\Sigma_2(\Wh_i(K))$ is identified with the involution on $\smash{S^3_{1/(2i)}(K \# K^r)}$ induced by the strong inversion $\tau \# \tau^r$ on $K \# K^r$.
\end{enumerate}
Strictly speaking, of course, there are two lifts of $\tau$ to $\Sigma_2(\Wh_i(K))$: one is given by the involution described above, and the other is given by composing this with $q$. There is an additional subtlety in this discussion due to the fact that we have two ways of forming the equivariant Whitehead double. To understand this and to justify the above assertions, we recall how \eqref{eq:branchedident} is derived.

First think of the clasp of the Whitehead double as sitting inside a small 3-ball, which is indicated in green on the left of Figure~\ref{fig:whitehead1}. Consider performing the tangle replacement indicated in the middle of Figure~\ref{fig:whitehead1}. This replaces the clasp in such a way so that the resulting knot is the unknot. Draw a blue reference arc connecting the two strands of the new tangle. We then isotope the black knot so that it appears as a local unknot; this is done by deforming the green ball along the original knot $K$. Note that the blue reference arc becomes a copy of the $K$ minus a small arc.

\begin{figure}[h!]
  \centering
  \includegraphics[scale = 0.9]{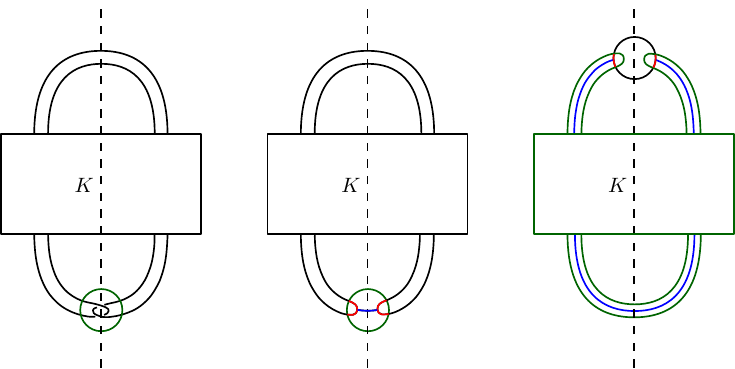}
  \caption{Left: a Whitehead double with the clasp drawn inside a small green $3$-ball. Middle: replacing the clasp with a trivial tangle so that the result is an unknot. Right: isotoping the green $3$-ball along $K$ so that the black knot becomes a local unknot.}\label{fig:whitehead1}
\end{figure}

We now obtain $\Sigma_2(\Wh_i(K))$ by taking the branched double cover over the black unknot (which gives $S^3$) and performing Dehn surgery along the lift of the blue reference arc. To do this, first draw the black unknot as a horizontal axis passing over the point at $\infty$. Lifting the blue reference arc then gives $K \# K^r$, as displayed in Figure~\ref{fig:whitehead2}. It is clear that this procedure is equivariant and that:
\begin{enumerate}[noitemsep]
\item The branching action is identified with the involution induced by interchanging $K$ and $K^r$.
\item The lift of $\tau$ is identified with the involution induced by $\tau \# \tau^r$.
\end{enumerate}
The reader may refer to \cite{dai2024rank} for further discussion and to verify that the surgery coefficient is indeed $1/(2i)$.

\begin{figure}[h!]
  \centering
  \includegraphics[scale = 0.9]{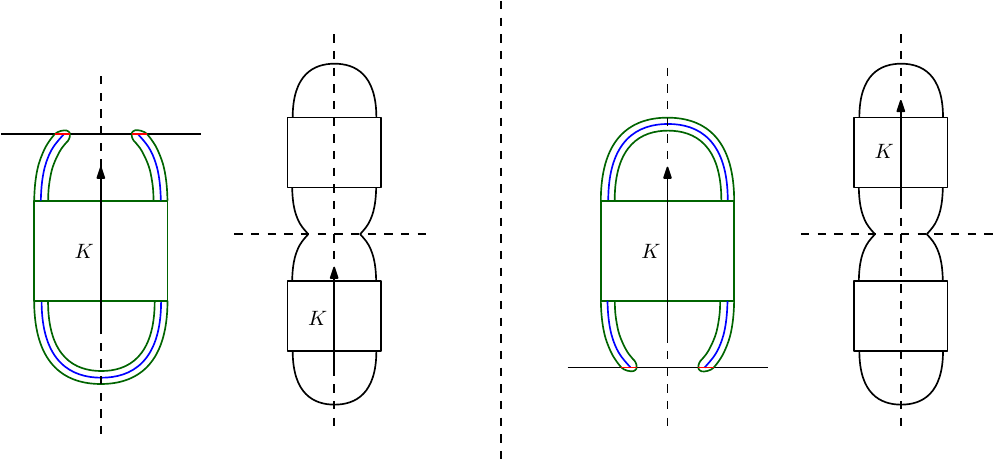}
  \caption{Left: the result of placing the clasp at the foot of the direction. Right: the result of placing the clasp at the head of the direction.}\label{fig:whitehead2}
\end{figure}

 Note, however, that the formation of $K \# K^r$ depends on which fixed point is used to form the clasp. Indeed, in Figure~\ref{fig:whitehead1} we have drawn the clasp at the bottom of $K$; this means that the black unknot appears at the top of $K$ and leads to the left-hand side of Figure~\ref{fig:whitehead2}. The equivariant connected sum in this case is indeed $(K \# K^r, \tau \# \tau^r)$. Here, $K^r$ is formed by rotating $K$ (together with the original direction) by $\pi$ around the horizontal axis, and then reversing orientation and reversing the direction. However, if the clasp is introduced at the top of $K$, then the black unknot appears at the bottom of $K$. This leads to the right-hand side of Figure~\ref{fig:whitehead2}, which corresponds to the equivariant connected sum $(K^r \# K, \tau^r \# \tau)$. Although superficially similar, there is no reason to believe that $K \# K^r$ and $K^r \# K$ are the same. Note that the former is constructed by taking the connected sum at the head of the original direction on $K$, while the latter is constructed by taking the connected sum at the foot.
\end{ex}

\begin{rem}
Despite the fact that there are two ways of forming the equivariant Whitehead double, we will usually gloss over this distinction and simply write $\Wh_i(K)$ without specifying which fixed point is used to form the clasp. All claims in this paper should be treated as applying to both constructions of $\Wh_i(K)$. Indeed, in each of our examples, our algebraic arguments will apply equally well to $(K^r \# K, \tau^r \# \tau)$ as to $(K \# K^r, \tau \# \tau^r)$.
\end{rem}

\begin{ex}\label{ex:whitehead2}
The most important case of Example~\ref{ex:whitehead2} will be when $K$ is the connected sum of four trefoils described in Example~\ref{ex:sumfour}. This is illustrated in Figure~\ref{fig:trefoils3}. 

\begin{figure}[h!]
\includegraphics[scale = 0.75]{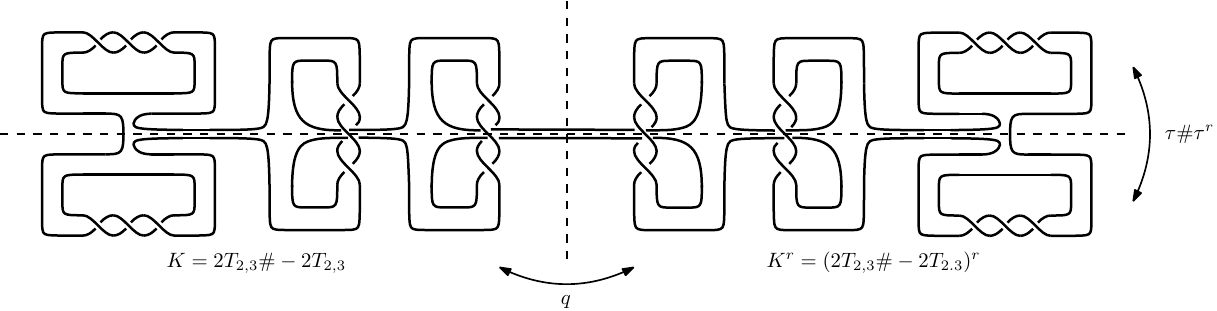}
	\caption{Illustration of strong inversions on $K \# K^r$ in the case $K = 2T_{2,3} \# -2T_{2,3}$, equipped with the involution $\tau = \tau_f \# -2\tau_0$. The involution $\tau \# \tau^r$ is given by rotation about the horizontal axis; this corresponds to one of the lifts of $\tau$ on $\Wh_i(K)$ to $\Sigma_2(\Wh_i(K))$. The involution corresponding to the branching action $q$ is given by rotation about the vertical axis.}
	\label{fig:trefoils3}
\end{figure}
\end{ex}

\section{Singular instanton theory}\label{sec:instantonreview}

In this section, we give a rather lengthy review of the singular instanton knot Floer formalism of Daemi and Scaduto \cite{DS19}. We first discuss the algebraic structure of instanton knot Floer homology, including the tensor product formula and the under-bar/upper-bar formalism. We then sketch the geometric construction of singular instanton theory. We pay particular attention to the role of the basepoint in the instanton knot complex as well as the dependence of cobordism maps on a choice of path. The reader who is familiar with the contents of \cite{DS19} may safely skip to Section~\ref{sec:involutive}. 

\subsection{Instanton $\mathcal{S}$-complexes}

We begin with the fundamental notion of an $\cS$-complex. Throughout, we work over the ring $R=\Z_2[T, T^{-1}]$ of Laurent polynomials in $T$. We write $\Lambda = T^{2} + T^{-2}$.
%The conjugate of a polynomial $p\in R$ will be denoted $\ov{p}$; this replaces $T$ with $T^{-1}$.
\begin{defn}\label{S-comp}
An {\it $\mathcal{S}$-complex} over $R$ is a triple $(\wt{C}, \wt{d}, \chi)$, where 
\begin{enumerate}[noitemsep]
    \item $\wt{C}$ is a free, finitely generated, $\Z_4$-graded $R$-module; and \item $\widetilde d$ and $\chi$ are endomorphisms of $\wt{C}$
\end{enumerate}
such that the following hold:
\begin{enumerate}[noitemsep]	
\item[(i)] $\gr(\widetilde d) = -1$ and $\gr(\chi) = 1$;
\item[(ii)] $\widetilde{d}^2=0$, $\chi^2=0$, and $\chi \widetilde d+\widetilde d \chi=0$; and
\item[(iii)] $H(\wt{C}, \chi) = \ker \chi /\ima \chi$
is isomorphic to a copy of $R$ concentrated in grading zero.
\end{enumerate}
Multiplication by $R$ is required to preserve grading. We refer to $\widetilde d$ as the \textit{differential}. The {\it{trivial $\mathcal{S}$-complex}} is given by $\wt C=R$ with $\wt d= \chi=0$.
\end{defn}

%It will often be useful to reduce $\wt{C}$ to data defined on a smaller complex. 

We think of $\cS$-complexes explicitly as follows. Fix any submodule $C$ of $\wt{C}$ which maps isomorphically onto $\ima \chi$. (For example, choose a basis for $\ima \chi$, take a preimage of each basis element, and form the span of these preimages.) This gives a splitting 
\[
\wt{C} = C \oplus \ker \chi. 
\]
Using the fact that $\ker \chi /\ima \chi \cong R$ is free, we may further split $\ker \chi = \ima \chi \oplus R$ by choosing any representative which generates $\ker \chi /\ima \chi$. It will often be convenient to identify $\ima \chi$ with $C$ via the inverse of the isomorphism $\chi \colon C_* \rightarrow (\ima \chi)_{*+1}$. We thus obtain a splitting of $R$-modules 
\begin{equation}\label{eq:splitting}
\wt C_* = C_* \oplus (\ima \chi)_* \oplus R \cong C_* \oplus C_{*-1} \oplus R.
\end{equation}
The action of $\chi$ maps the first summand isomorphically onto the second and vanishes on the second and third summands. We refer to the third summand $R$ as the {\it{reducible summand}}.

\begin{defn}
Denote the components of $\wt{d}$ with respect to a splitting as in \eqref{eq:splitting} by
\begin{equation}\label{eq:dtildedef}
\wt{d} = 
\begin{pmatrix}
d & 0 & 0\\
v & d  & \delta_2\\
\delta_1 & 0 &0 
\end{pmatrix}.
\end{equation}
Using the fact that $\wt{d}$ commutes with $\chi$, it is straightforward to check the indicated components vanish.  
\end{defn}

\begin{rem}\label{rem:reducetoC}
The map $v$ in \eqref{eq:dtildedef} is the component of $\wt{d}$ from $C$ to $\ima \chi$. However, we may also regard $v$ as an endomorphism of $C$ by postcomposing this component with the inverse of the isomorphism $\chi \colon C \rightarrow \ima \chi$. Likewise, we may regard $\delta_2$ as a map from $R$ to $C$ by taking the component of the differential going from $R$ to $\ima \chi$, and composing this with the inverse of the isomorphism $\chi \colon C \rightarrow \ima \chi$. From this point of view, $v$ and $\delta_2$ have grading $-2$. It is thus possible to reduce the data of the full $\cS$-complex to the data of a collection of maps only involving $C$ and $R$. The reader can easily translate the condition that the differential squares to zero into a collection of conditions involving $d$, $v$, $\delta_1$, and $\delta_2$. Indeed, squaring \eqref{eq:dtildedef} and examining each component, we see that $\smash{\wt{d}^2 = 0}$ is equivalent to
\[
d^2 = 0, \quad vd + dv + \delta_2 \delta_1 = 0, \quad d \delta_2 = 0, \quad \text{and} \quad \delta_1 d = 0.
\]
\end{rem}

\begin{rem}\label{rem:ambiguity}
It will be clear from context whether to interpret $v$ as a map from $C$ to $\ima \chi$ or as a map from $C$ to itself. We will sometimes even write $v$ to mean the the map from $(\ima \chi)_*$ itself obtained by passing through the isomorphism $(\ima \chi)_* \cong C_{*-1}$. Similar statements hold for $\delta_1$ and $\delta_2$.  
\end{rem}

We usually give a basis for $\wt{C}$ by choosing a splitting and giving a basis for $C$. Our convention will be to denote basis elements of $C$ by $\mf{a}_i$ and the corresponding basis elements of $\ima \chi$ by $\un{\mf{a}}_i = \chi(\mf{a}_i)$, so 

\begin{equation}\label{eq:basis}
\wt{C} = \langle \mf{a}_i\rangle \oplus \langle \un{\mf{a}}_i \rangle \oplus R.
\end{equation}

\begin{ex} \label{ex:trefoil}
    The chain complex $\wt{C}$ associated to the trefoil $(S^3, T_{2,3})$ is given by 
    \[
    \wt{C} (T_{2,3} ) = \langle \rho \rangle \oplus \langle \underline{\rho}  \rangle \oplus R, 
    \]
    where $\rho$ has grading one and $\underline{\rho}$ has grading two. The map $\chi$ sends $\rho$ to $\underline{\rho}$ and is zero otherwise.
    The differential is given by 
    \begin{equation*}
\wt{d} = 
\begin{pmatrix}
0 & 0 & 0\\
0 & 0   & 0\\
T^2 + T^{-2} & 0 &0 
\end{pmatrix}. 
\end{equation*}
%This complex is displayed in Figure~\ref{fig:trefoilcomplex}.
In this case, all components of $\wt{d}$ except $\delta_1$ are zero. In general, we often denote the generator of $R$ by $\theta$. Thus, in this case we have $\delta_1 \rho = \Lambda \theta$.
\end{ex}

\begin{comment}
%We will often visually represent such a chain complex as in Figure~\ref{fig:trefoilcomplex}.
\begin{figure}[h!]
%\[
\begin{tikzcd}
	{\underline{\rho}} & \rho & \theta
	\arrow[dashed, from=1-2, to=1-1]
	\arrow["\Lambda", from=1-2, to=1-3]
\end{tikzcd}
%\]
\caption{Schematic representation of $\wt{C}$. We use $\theta$ to denote the generator of the reducible summand. The action of $\chi$ is represented by the dashed arrow, while the solid arrow gives the action of the differential. The solid arrow is decorated with a coefficient of $\Lambda = T^2 + T^{-2}$.}
\label{fig:trefoilcomplex}
\end{figure}
\end{comment}

We now define morphisms of $\cS$-complexes. 

\begin{defn}\label{def:Smorphism}
A \textit{morphism} $\wt{\lambda}$ between $(\widetilde C, \widetilde d, \chi )$ and $(\widetilde C', \widetilde d', \chi' )$ is a grading-preserving $R$-module map $\wt{\lambda}\colon \wt{C} \rightarrow \wt{C}'$ satisfying
\[
\wt \lambda \wt d = \wt d' \wt \lambda \quad \text{and} 
\quad \wt \lambda \chi = \chi' \wt \lambda.
\]
With respect to splittings of the domain and codomain, we denote the components of $\wt{\lambda}$ by
\begin{equation}\label{eq:morphismdecomposition}
\wt{\lambda} = 
\begin{pmatrix}
\lambda & 0 & 0\\
\mu & \lambda  & \Delta_2\\
\Delta_1 & 0 & \eta 
\end{pmatrix}.
\end{equation}
 Note that $\eta$ is a map from $R$ to $R$ and is thus multiplication by some fixed element of $R$, which we also denote by $\eta$.
\end{defn}

\begin{defn}\label{def:Shomotopy}
A $\textit{homotopy}$ $\wt H$ between $\wt\lambda$ and $\wt\lambda'$ is an $R$-module map $\wt H:\widetilde C\to \widetilde C'$ satisfying
\[
\wt\lambda + \wt\lambda' = \wt H \wt d + \wt d' \wt H \text{ and } \chi' \wt H + \wt H \chi = 0.
\]
With respect to splittings of the domain and codomain, we denote the components of $\wt{H}$ by
\begin{equation}\label{eq:homotopydecomposition}
\wt{H} = \begin{pmatrix}
H & 0 & 0\\
\mu^H & H & \Delta_2^{H}\\
\Delta_1^{H} & 0 & 0
\end{pmatrix}.
\end{equation}
\end{defn}

\begin{rem}
Similar statements as Remarks~\ref{rem:reducetoC} and \ref{rem:ambiguity} hold for the components of \eqref{eq:morphismdecomposition} and \eqref{eq:homotopydecomposition}. For example, we will sometimes view $\mu$ as going between $C$ and $C'$ by postcomposing with the inverse of the isomorphism $\chi' \colon C' \rightarrow \ima \chi'$. The conditions of Definitions~\ref{def:Smorphism} and \ref{def:Shomotopy} can likewise be translated into a set of equations involving maps between $C$, $C'$, and $R$.
\end{rem}

Note that the splitting of an $\mathcal{S}$-complex is not canonical. While most $\mathcal{S}$-complexes we encounter will come with a specified splitting, and in fact will be constructed by giving some $(C, d, v, \delta_1, \delta_2)$, it should be noted that the maps $v$, $\delta_1$, and $\delta_2$ are splitting-dependent. (The map $d$ is \textit{not} splitting-dependent, as $\ima \chi$ is a canonical submodule which is taken to itself by the differential.) Likewise, the different components of morphisms and homotopies are splitting-dependent, except for the components $\lambda$ and $H$. For future reference, it will be useful to explicitly understand the extent to which these components can change depending on the chosen splitting. 

The most concrete way to understand this is to choose a basis as in \eqref{eq:basis} and enumerate the effects of basis changes which still result in a basis with the same structure as \eqref{eq:basis}. In order to preserve the form of this splitting, not all basis changes are allowed. For example, a basis element of $\ima \chi$ can only be replaced with a linear combination of other basis elements of $\ima \chi$, since $\ima \chi$ is defined in a canonical manner. In contrast, we may add a multiple of a basis element from $\ima \chi$ or $R$ to a basis element of $C$, as adding such an element does not change the fact that $C$ maps isomorphically onto $\ima \chi$. Finally, we may add a multiple of a basis element from $\ima \chi$ to a basis element of $R$, since this does not change the resulting class in $\ker \chi / \ima \chi$.

\begin{rem}\label{rem:filtration}
We may succinctly view such basis changes as preserving the filtration
\[
\ima \chi \subset \ker \chi \subset \wt{C}.
\]
Note that $C \cong \wt{C}/\ker \chi$ and $R \cong \ker \chi/\ima \chi$. 
\end{rem}

Placing this in a slightly more general context, any map which is filtered with respect to the filtration of Remark~\ref{rem:filtration} has the following form when written as a $3 \times 3$ block matrix:
\begin{equation}\label{eq:matrixsplitting}
\begin{matrix}
& 
\begin{matrix} C & \ima \chi & R \end{matrix} \\
\begin{matrix} C\\ \ima \chi\\ R \end{matrix} & 
\begin{pmatrix}
* & \ \ \ 0 & \ \ \ 0\\  
* & \ \ \ * & \ \ \ *\\
* & \ \ \ 0 & \ \ \ * 
\end{pmatrix}.
\end{matrix}
\end{equation}
In particular, this holds for any map which commutes with $\chi$. We refer to the columns of such a matrix as being divided into the first-third, middle-third, and final-third columns (even though the last of these consists of only one column); and likewise for the rows.
\begin{itemize}[noitemsep]
\item Performing a basis change within $C$ (or, similarly, within $\ima \chi$ or $R$) pre- or post-composes each component of our map by a basis change in the usual sense. 
\item Consider adding an element in $\ima \chi$ or $R$ to an element of $C$. If performed in the domain, this adds a middle- or final-third column to a first-third column. If performed in the codomain, this adds a first-third row to a middle- or final-third row.
\item Consider adding an element in $\ima \chi$ to an element of $R$. If performed in the domain, this adds a middle-third column to a final-third column. If performed in the codomain, this adds a final-third row to a second-third row.
\end{itemize}
The following facts will be extremely important in later sections.

\begin{lem}\label{lem:matrixsplitting}
For maps of $\cS$-complexes which are filtered with respect to Remark~\ref{rem:filtration}, we have:
\begin{enumerate}[noitemsep]
\item Under a filtered basis change, the diagonal blocks are never altered by the addition of any other nonzero block.
\item Under a filtered basis change, the entries of the $(2, 1)$-block are never added to the entries of any other block.
\item The composition of two maps of the form \eqref{eq:matrixsplitting} is again of the form \eqref{eq:matrixsplitting}.
\end{enumerate}
\end{lem}
\begin{proof}
The first two claims follow from the discussion preceding the lemma. The last claim follows from direct inspection, or from the fact that the composition of two filtered maps is filtered.
\end{proof}

Actually, Lemma~\ref{lem:matrixsplitting} is rather trivial: from the point of view of Remark~\ref{rem:filtration}, the more natural order to write the splitting \eqref{eq:splitting} is as $C \oplus R \oplus \ima \chi$. Then \eqref{eq:matrixsplitting} becomes block lower-triangular and what was previously the $(2, 1)$-block moves to the bottom-left corner. This block should be thought of as the component which lies deepest in filtration, in the sense that it drops the three-step filtration by two levels. We have chosen a different ordering of the splitting due to the importance of the irreducible summand $R$ and for consistency with \cite{DS19}.

We now give various notions of equivalence for $\cS$-complexes. The notion of homotopy equivalence is straightforward: we simply require the maps and homotopies in question to intertwine $\chi$ and $\chi'$.

\begin{defn}
Two $\cS$-complexes are \textit{homotopy equivalent} if there exist morphisms between them whose compositions are chain homotopic to the identity via a homotopy as in Definition~\ref{def:Shomotopy}. 
\end{defn}

More important, however, will be to understand how morphisms behave with respect to reducible summands. This is the content of the next definition.

\begin{defn} \label{height i morphism}
A morphism $\wt \lambda : \widetilde C\to \widetilde{C}'$ is called \textit{local} if its $\eta$-component is a unit; or, equivalently, if the induced map
\[
\tilde{\lambda} : H(\widetilde C, \chi) \cong R \to H(\widetilde C', \chi') \cong R
\]
given by multiplication by $\eta$ is an isomorphism. We say $\smash{\wt{C}}$ and $\smash{\wt{C}'}$ are \textit{locally equivalent} if there exist local morphisms between them in both directions.
\end{defn}

Local morphisms capture the behavior of certain kinds of cobordism maps (for example, homology concordances), as we discuss presently. 

\subsection{Under-bar and upper-bar complexes}\label{sec:under-upper}
The $\cS$-complex formalism will be necessary when developing formulas for tensor products and duals. However, for other applications we often only require partial knowledge of the $\cS$-complex. It will thus sometimes be useful to throw out some of the components of the differential in order to define a simpler invariant. 

\begin{defn}\label{def:underbarupperbar}
Given an $\mathcal{S}$-complex $(\widetilde C, \widetilde d, \chi )$, let
\[
( \underline{C} = \ker \chi, \ \underline{d} = \widetilde d|_{\ker \chi} ) \quad \text{and} \quad (\overline{C} = \wt{C}/\ima \chi, \ \overline{d} =[\wt{d} ]   ).
\]
These are called the \textit{under-bar} and \textit{upper-bar} complexes, respectively. 
\end{defn}

The under-bar complex has underlying module given by the last two components $C_{* - 1}\oplus R$ of \eqref{eq:splitting}, while the upper-bar complex can be thought of as having underlying module given by the first and last components. With respect to these identifications, we think of $\un{d}$ and $\ov{d}$ as given by the $2 \times 2$ submatrices
\[
\un{d} = 
\begin{pmatrix}
d  & \delta_2\\
0 &0 
\end{pmatrix} \quad \text{and} \quad
\ov{d} = 
\begin{pmatrix}
d & 0\\
\delta_1 &0 
\end{pmatrix}.
\]
As before, we can think of these as being filtered with respect to the two-step filtrations
\[
\ima \chi \subset \un{C} = \ker \chi \quad \text{and} \quad \ker \chi/\ima \chi \subset \ov{C} = \wt{C}/\ima \chi,
\]
respectively. Note that the under-bar complex remembers the information of $d$ and $\delta_2$ and the upper-bar complex remembers the information of $d$ and $\delta_1$. Thus, $\un{C}$ (for example) is obtained by taking a copy of $C$ together with $R$, and recording the information of $d$ and $\delta_2$ (as a map from $R$ to $C$).

The definitions of the previous section apply without significant change to $\un{C}$ and $\ov{C}$-complexes, replacing the three-step filtration of Remark~\ref{rem:filtration} with the appropriate two-step filtration in each case. This gives the definition of a morphism between two $\un{C}$- or $\ov{C}$-complexes as one which is filtration-preserving; i.e., having the form
\[
\un{\lambda} = 
\begin{pmatrix}
\lambda  & \Delta_2\\
0 &\eta 
\end{pmatrix} \quad \text{or} \quad
\ov{\lambda} = 
\begin{pmatrix}
\lambda & 0\\
\Delta_1 &\eta
\end{pmatrix},
\]
respectively. Chain homotopies, homotopy equivalences, local maps, and local equivalences are defined similarly. It is clear that a morphism of $\cS$-complexes induces a morphism of $\un{C}$- and $\ov{C}$-complexes, which is local if and only if the original morphism is local. 

In the case of the under-bar complex, we frequently think about local morphisms in terms of \textit{$\theta$-supported cycles}, as in \cite[Definition 3.10]{ADMT}:

\begin{defn}\label{def:thetasupported}
We say that a cycle $x \in \un{C}$ is a \textit{$\theta$-supported cycle} if the coefficient of $\theta$ in $x$ is a unit, where $\theta$ is the generator of $R$.
\end{defn}

It follows from the discussion preceding Lemma~\ref{lem:matrixsplitting} that the definition of a $\theta$-supported cycle is independent of the choice of splitting. It is easily checked that a local morphism of $\un{C}$-complexes maps $\theta$-supported cycles to $\theta$-supported cycles, and that $\un{C}$ admits a $\theta$-supported cycle if and only if there exists a local map from the trivial complex $R$ into $\un{C}$.
%Once again, a $\theta$-supported cycle is a cycle whose $\theta$-coefficient is a unit. The inclusion of a $\theta$-supported cycle in $\un{C}$ is a $\theta$-supported cycle in $\smash{\wt{C}}$, while a $\theta$-supported cycle in $\smash{\wt{C}}$ quotients to a $\theta$-supported cycle in $\ov{C}$.

%These complexes will be useful when not all components of the differential can be calculated. A morphism of $\cS$-complexes clearly induces morphisms of under-bar and upper-bar complexes, which is local if the original morphism is local.

%Note that $\un{C}$ and $\ov{C}$ still carry an induced action of $\chi$. A \textit{morphism} between two under-bar or two upper-bar complexes is an $R$-module map that commutes with the differential and the action of $\chi$. In both cases, we say that this is \textit{local} if it induces an isomorphism on $\ker \chi/\ima \chi$. 

\subsection{Tensor products and duals}\label{sec:tensorbasis}

We now discuss tensor products and duals of $\cS$-complexes. The definitions are straightforward.

\begin{defn}
Given two $\mathcal{S}$-complexes $(\wt C,\wt d, \chi)$ and $(\wt C', \wt d', \chi')$, their tensor product is given by
\[
	(\wt C^\otimes,\wt d^\otimes, \chi^\otimes)=(\wt C\otimes \wt C', d\otimes 1 + 1 \otimes d' , \chi\otimes 1 + 1 \otimes \chi').
\]
 \end{defn}

Since the procedure for finding a splitting is slightly involved, we discuss this here. Suppose we have splittings of $\wt{C}$ and $\wt{C}'$. We wish to find a submodule $C^\otimes$ of $\wt{C} \otimes \wt{C}'$ such that
\[ 
(\wt{C}\otimes \wt{C}')_* = C^\otimes_\ast \oplus \ima \chi^\otimes \oplus R \cong C^\otimes_\ast \oplus C^\otimes_{\ast-1} \oplus R.
\]
To do this, fix bases $\wt{C} = \langle \mf{a}_i \rangle \oplus \langle \un{\mf{a}}_i \rangle \oplus R$ and $\wt{C}' = \langle \mf{b}_i \rangle \oplus \langle \un{\mf{b}}_i \rangle \oplus R$ as in \eqref{eq:basis}. Set
\begin{equation}
\begin{split}
\ima \chi^\otimes &= \langle \underline{\mathfrak{a}}_i\otimes \mathfrak{b}_j + \mathfrak{a}_i\otimes \underline{\mathfrak{b}}_j,\quad  \underline{\mathfrak{a}}_i \otimes \underline{\mathfrak{b}}_j, \quad\underline{\mathfrak{a}}_i \otimes 1, \quad1\otimes \underline{\mathfrak{b}}_j \rangle \\
C^\otimes &= \langle {\mathfrak{a}}_i\otimes \mathfrak{b}_j  ,\quad  {\mathfrak{a}}_i \otimes \underline{\mathfrak{b}}_j, \quad {\mathfrak{a}}_i \otimes 1, \quad1\otimes {\mathfrak{b}}_j \rangle.
\end{split}\label{eq:tensorbasis}
\end{equation}
Clearly, the four types of basis elements we have given for $C^\otimes$ map isomorphically onto the four types of basis elements for $\ima \chi^\otimes$ via $\chi^\otimes$. Together with $R = \langle 1 \otimes 1 \rangle$, this gives a basis for the tensor product complex. As in \eqref{eq:dtildedef}, we denote the components of the differential with respect to this splitting by
\[
\begin{pmatrix}
d^\otimes & 0 & 0 \\ 
v^\otimes  & d^\otimes & \delta^\otimes _2 \\ 
\delta^\otimes_1& 0 & 0 \\ 
\end{pmatrix}.
\]
We claim that with respect to the bases of $C^\otimes$ and $\ima \chi^\otimes$ given above, we have
\[
d^\otimes =\begin{pmatrix}
d\otimes 1 + 1 \otimes d' &0&0&0\\
v \otimes 1 + 1 \otimes v' & d\otimes 1 + 1 \otimes d' & 1 \otimes \delta_2' & \delta_2 \otimes 1\\
1 \otimes \delta_1' &0&d \otimes 1& 0\\\delta_1\otimes 1& 0 & 0 & 1 \otimes d'\\
\end{pmatrix},
\]
while
\[
	v^\otimes =\begin{pmatrix}
v\otimes 1 &0& 0& \delta_2\otimes 1\\
0 & v\otimes 1 & 0 &0\\
0 & 0& v \otimes 1 & 0\\
0 & \delta_1\otimes 1& 0  & 1 \otimes v'\\
\end{pmatrix},
\] 
and
\[
\delta_1^\otimes= \begin{pmatrix} 0, 0, \delta_1 \otimes 1, 1 \otimes \delta_1'\end{pmatrix} \quad \text{and} \quad \delta_2^\otimes = \begin{pmatrix}0,0,\delta_2 \otimes 1,1 \otimes \delta_2'\end{pmatrix}^T.
\]

Here, we treat $d^\otimes$, $v^\otimes$, $\delta_1^\otimes$, and $\delta_2^\otimes$ as maps from $C^\otimes$ to itself or maps between $C^\otimes$ and $R$. Within each matrix, however, $v$ sometimes refers to a map from $C$ to $\ima \chi$, a map from $C$ to $C$, or a map from $\ima \chi$ to $\ima \chi$; see Remark~\ref{rem:ambiguity}. Similar statements hold for the other maps appearing above. We likewise write $1$ to mean the identity map from $C$ to itself, the identity map from $\ima \chi$ to itself, or (confusingly) the isomorphism $\chi$ from $C$ to $\ima \chi$. In order to correctly interpret the entries of the above matrix, the reader must keep in mind the form of the incoming and outgoing basis elements. For example, according to \eqref{eq:tensorbasis}, the $(2, 4)$-entry $\delta_2 \otimes 1$ of $d^\otimes$ is a map from basis elements of the form $1 \otimes \mf{b}_j$ to basis elements of the form $\smash{\mf{a}_i \otimes \un{\mf{b}}_j}$. It is thus clear that the $\delta_2$ in this entry must be interpreted as a map from $R$ into $C$, while $1$ must be interpreted as the isomorphism $\chi'$ sending $\mf{b}_j$ to $\un{\mf{b}}_j$. In contrast, \eqref{eq:tensorbasis} shows that the $(1, 4)$-entry $\delta_2 \otimes 1$ of $\smash{v^\otimes}$ is a map from basis elements of the form $1 \otimes \mf{b}_j$ to basis elements of the form $\mf{a}_i \otimes \mf{b}_j$. In this case, the $\delta_2$ is thus interpreted as a map from $R$ into $C$ and $1$ is the identity map on $C'$.

We briefly indicate how to carry out the above computation. Consider a basis element $\mf{a} \otimes \mf{b}$. Then
\begin{align*}
\wt{d}^\otimes(\mf{a} \otimes \mf{b}) &= (\wt{d} \mf{a}) \otimes \mf{b} + \mf{a} \otimes (\wt{d}' \mf{b}) \\
=& (d\mf{a} + \un{v \mf{a}} + \delta_1 \mf{a}) \otimes \mf{b} + \mf{a} \otimes (d' \mf{b} + \un{v' \mf{b}} + \delta_1' \mf{b}) \\
=& (d\mf{a}) \otimes \mf{b} + \mf{a} \otimes (d'\mf{b}) + (\un{v \mf{a}}) \otimes \mf{b} + \mf{a} \otimes (\un{v' \mf{b}}) + (\delta_1 \mf{a}) \otimes \mf{b} + \mf{a} \otimes (\delta_1' \mf{b}) \\
=& (d\mf{a}) \otimes \mf{b} + \mf{a} \otimes (d'\mf{b}) + 
((\un{v \mf{a}}) \otimes \mf{b} + (v \mf{a}) \otimes \un{\mf{b}}) + \\
& (v \mf{a}) \otimes \un{\mf{b}} + \mf{a} \otimes (\un{v' \mf{b}})
+ (\delta_1 \mf{a}) \otimes \mf{b} + \mf{a} \otimes (\delta_1' \mf{b}).
\end{align*}
Here, the last equality comes from re-writing $(\un{v \mf{a}}) \otimes \mf{b} = ((\un{v \mf{a}}) \otimes \mf{b} + (v \mf{a}) \otimes \un{\mf{b}}) + (v \mf{a}) \otimes \un{\mf{b}}$. We do this because $(\un{v \mf{a}}) \otimes \mf{b}$ is not one of the basis elements from \eqref{eq:tensorbasis}. Note that now the third and fourth terms combined lie in $\ima \chi^\otimes$, while all other terms are linear combinations of basis elements from $C^\otimes$.

The first and second terms above now give the $(1, 1)$-entry of $d^\otimes$. The fifth and sixth terms give the $(2, 1)$-entry, with the caveat that the maps $v$, $v'$, and $1$ appearing in the $(2, 1)$-entry $v \otimes 1 + 1 \otimes v'$ must be correctly interpreted. Explicitly, the $(2, 1)$-entry represents a map from basis elements of the form $\mf{a}_i \otimes \mf{b}_j$ to basis elements of the form $\mf{a}_i \otimes \un{\mf{b}}_j$. Hence it is clear that $v$ should be interpreted as a map from $C$ to $C$, the first instance of $1$ represents the isomorphism $\chi'$ taking $\mf{b}_j$ to $\un{\mf{b}}_j$, the second instance of $1$ represents the identity map on $C$, and finally $v'$ is regarded as a map from $C'$ to $\ima \chi'$. Likewise, the eighth term gives the $(3, 1)$-entry, and the seventh term gives the $(4, 1)$-entry. This completes the computation of the first column of $d^\otimes$.

We are thus left with the sum of the third and fourth terms. This is a linear combination of basis elements (of the first kind) of $\ima \chi^\otimes$. Composing with the inverse of the isomorphism $\chi^\otimes \colon C^\otimes \rightarrow \ima \chi^\otimes$, we see that as a map from $C^\otimes$ to itself, we have $v^\otimes(\mf{a} \otimes \mf{b}) = (v \mf{a}) \otimes \mf{b}$. This yields the first column of $v^\otimes$. The rest of the computation is similar.

\begin{rem}
    Note that our choice of splitting is not unique and is, in some sense, asymmetric: we could equally well have chosen $\un{\mf{a}}_i \otimes \mf{b}_j$ in place of ${\mathfrak{a}}_i \otimes \underline{\mathfrak{b}}_j$ in our basis for $C^\otimes$. This is why our expression for $v^\otimes$ is slightly asymmetric. Over a ring in which $2$ is invertible, we can give a symmetric splitting, as discussed in \cite{DISST22}.  
\end{rem}

\begin{rem}
It is important to note that the under-bar complex of a tensor product is \textit{not} the tensor product of under-bar complexes, and likewise for upper-bar complexes. Indeed, our formula for $d^\otimes$ utilizes all four components $d$, $v$, $\delta_1$, and $\delta_2$ of the two factors. The differential on $C^\otimes$ thus cannot be determined from the restricted information contained in $\un{C}$ or $\ov{C}$ of each of the two factors individually.    
\end{rem}

 \begin{defn}
      Given an $\cS$-complex $(\wt C,\wt d, \chi)$, its dual $(\wt C^\dagger, \wt d^\dagger , \chi^\dagger)$ is defined by $\wt C_*^\dagger = \text{Hom}(\wt C_{-*},R)$, with the induced dual differential and dual $\chi$-map.
 \end{defn}

 Note that if we have chosen a splitting for $\smash{\wt{C}}$, then the two copies of $C$ exchange roles after dualizing. (That is, if $\chi \colon C \rightarrow \ima \chi$, then $\chi^\dagger \colon (\ima \chi)^\dagger \rightarrow C^\dagger$.) Indeed, the $\delta_1$-component of $\smash{\wt{d}^\dagger}$ is $\smash{\delta_2^\dagger}$, and the $\delta_2$-component is $\smash{\delta_1^\dagger}$. It follows from this that the dual of a $\un{C}$-complex is not a $\un{C}$-complex; instead the $\un{C}$-complex of a dual is the dual of the $\ov{C}$-complex of the original.

\subsection{Geometric constructions of knot complexes}\label{sec:geometricconstruction}

We now review the geometric construction of singular instanton Floer theory. Let $Y$ be an integer homology $3$-sphere and $K\subset Y$ a knot. In \cite{DS19}, Daemi and Scaduto used singular instanton gauge theory to construct a $\Z/4$-graded $\mathcal{S}$-complex associated to $(Y,K)$, denoted
\begin{equation}\label{eq:scomplexofknotdef}
	\widetilde C(Y,K) = C_*(Y,K)\oplus C_{*-1}(Y,K)\oplus R.
\end{equation}
We summarize their construction in what follows, ignoring technicalities. See \cite{DS19} for further details. 

\subsubsection{The chain complex}
We begin with the construction of the chain complex. This follows the usual setup in instanton Floer theory: we consider the space of $\SU(2)$ connections modulo gauge, and form the Floer homology associated to the Chern-Simons functional on this space. However, the new ingredient is that our connections are in fact defined only on $Y \ \setminus \ K$ and are singular near $K$. It is also possible to view these as equivariant connections on the branched cover of $K$, although we do not emphasize this point of view.

Fix an orientation of $K$. Let $\check{Y}$ be the orbifold whose singular points are $K$ with cone angle $\pi$, and fix an orbifold metric $g$. Consider the space $\A(Y, K)$ of singular $\SU(2)$ connections on $Y\ \setminus \ K$ which are compatible with a preferred reduction near $K$. Precisely, $\A(Y, K)$ is given by
\[
\{ A_0+ a : a \in \Omega^1(\check{Y})\otimes \SU(2) \}.
\]
Here, $A_0$ is an $\SU(2)$-model singular connection locally described by
\[
\begin{pmatrix}
    \frac{1}{4} i d\theta&  0 \\
    0 & -  \frac{1}{4} i d\theta
\end{pmatrix}
\]
on a small neighborhood of $K$, where $\theta$ denotes the angle normal coordinate of a neighborhood of $K$. The space $\Omega^1(\check{Y})\otimes \SU(2)$ is the set of $\SU(2)$-valued orbifold 1-forms on $\check{Y}$, where we take the Sobolev completion with respect to the orbifold metric. The $\SU(2)$-gauge group $\G(Y, K)$ is defined as the orbifold automorphism group preserving $\A(Y, K)$, again with a certain completion. We write
\[
\B(Y, K) = \A(Y, K)/ \G(Y, K) \quad \text{and} \quad \B^\ast(Y, K) = \A^\ast(Y, K)/ \G(Y, K),
\]
where $\A^\ast(Y, K)$ denotes the subspace of irreducible  connections. 

\begin{rem}\label{rem:orientationindependence}
The orientation on $K$ is used to define the coordinate $\theta$, which is constructed so that $\theta$ is increasing in the direction of the canonically-oriented meridian of $K$. The model connection $A_0$ splits the background $\SU(2)$-bundle into $L \oplus L^*$ near $K$, where $A_0|_L$ has holonomy $i$ and $A_0|_{L^*}$ has holonomy $-i$ around small oriented meridians of $K$. Note that the model connection $A_0$ for $K$ becomes a model connection for $K^r$ only after interchanging the roles of $L$ and $L^*$.
\end{rem}

We also consider the limiting holonomy along an oriented longitude
\[
\mathrm{hol}_{(Y, K)} : \mathcal{A}^* (Y, K) \to \R.
\]
Given a singular connection $\mf a$, this is defined as the asymptotic holonomy $\operatorname{hol}_{\lambda_K}(\mf a|_{L})$ of the restriction of $\mf a$ to $L$ along an oriented longitude $\lambda_K$. It is important to note that $\mathrm{hol}_{(Y, K)}$ does \textit{not} depend on the orientation of $K$. Indeed, if we reverse orientation, then we replace $\lambda_K$ with $- \lambda_K$, but we also interchange the roles of $L$ and $L^*$ as discussed in Remark~\ref{rem:orientationindependence}. Thus
\[
\mathrm{hol}_{(Y, K^r)} (\mf a) = \operatorname{hol}_{-\lambda_K}(\mf a|_{L^*}) = \operatorname{hol}_{\lambda_K}(\mf a|_{L}) = \mathrm{hol}_{(Y, K)} (\mf a).
\]
Applying a gauge transformation changes the value of $\mathrm{hol}_{(Y, K)}$ by an integer. Hence $\mathrm{hol}_{(Y, K)}$ descends to a map
\[
\mathrm{hol}_{(Y, K)} : \mathcal{B}^* (Y, K) \to S^1
\]
which induces a local system $\pi_1(\mathcal{B}^* (Y, K)) \to \Z$. Finally, we denote the abelian flat connection by $\theta$. This corresponds to the representation 
\begin{align*}
&\pi_1 (S^3 \ \setminus \ K) \to H_1(S^3 \ \setminus \ K) \to \Z \to \SU(2), \\
&m_K \mapsto 
\begin{pmatrix}
i &0\\
0&-i \end{pmatrix} \in \SU(2)
\end{align*}
where $m_K$ is a small meridian of $K$.  

We are now ready to define the singular instanton knot complex. The $\Z/4$-graded chain complex $(C(Y,K),d)$ is roughly the Morse--Floer complex with respect to a perturbed (singular) Chern--Simons functional $\operatorname{cs}$, having discarded the abelian flat traceless representation $\theta$. (The perturbation data $\pi$ is chosen to ensure that transversality holds.) To define the underlying $R$-module of $C(Y,K)$, we introduce a single $R = \Z_2[T, T^{-1}]$-summand for each irreducible critical point of the Chern-Simons functional $\operatorname{cs}$ perturbed by $\pi$:
\[
\mathrm{cs} + f_\pi: \mathcal{B}^* (Y, K) \to S^1.
\]
This summand is formally shifted by $T^{\mathrm{hol}_{(Y, K)}(\mf \a)}$. We write
\[
	C(Y,K)  = \bigoplus_{[\mf a]\in \mathfrak{C}^{*}_{\pi}(Y,K)} R\cdot T^{\mathrm{hol}_{(Y, K)}(\mf \a)} \mf [\a].
\]
The coefficient $\smash{T^{\mathrm{hol}_{(Y, K)}(\mf \a)}}$ is only introduced for bookkeeping; note that choosing a different representative of $[\mf a]$ changes this by an integer power of $T$. We think of $\smash{R\cdot T^{\mathrm{hol}_{(Y, K)}(\mf \a)} \mf [\a]}$ as a copy of $R$ which is generated (non-canonically) by choosing any representative of $[\a]$.

For each critical point $[\mathfrak{a}]$, considering the spectral flow with respect to $[\theta]$ defines a $\Z_4$ grading
\[
\ind_{\Z_4} ([\mathfrak{a}]) \in \Z_4.
\]
This grading does not depend on the orientation of $K$.
The differential $d$ counts (perturbed) instantons on $\R\times Y$ that have a prescribed singularity type along $\R\times K$ with respect to the product orbifold metric $dt^2 + g$. 
For any two points $[\mathfrak{a}]$ and $[\mathfrak{b}] \in \mathfrak{C}^{*}_{\pi}(Y,K)$ such that $\ind ([\a]) - \ind ([\b]) =1$, we define
\[
d([\a]) := \sum_{z \in \pi_1(\B([\a], [\b])) } \# \check{M} _{z}([\a], [\b])_0\cdot  T^{\nu(A)} [\b], 
\]
where 
\begin{itemize}[noitemsep]
    \item For any $[\a]$ and $[\b]$, we write $\pi_1(\B([\a], [\b]))$ to denote the set of homotopy classes of paths from $[\a]$ to $[\b]$ in $\B(Y, K)$.
    \item For any homotopy class $z$, we write $\check{M}_{z}([\a], [\b])_i$ to denote the $i$-dimensional component of the moduli space of perturbed singular instantons in the class $z$, quotiented out by the $\R$-action. Here, we suppress writing the data of the orbifold metric $g$ and holonomy perturbation $\pi$.
    These live over $(\R\times Y, \R\times K)$ with the product orbifold metric $dt^2 + g$.
    \item For any $[A] \in \check{M}_z ([\a], [\b])_i$, we write $\nu(A)$ for the integral
    \[
  \nu(A)=   \frac{i}{\pi} \int_{\R \times K } \Omega_A 
    \]
    where $\Omega_A$ is the 2-form on $\R \times K$ defined by\footnote{Even though $A$ is singular, its curvature extends over the singular locus.} 
    \begin{align}\label{eq:split}
    F_A |_{\R \times K} = \begin{pmatrix}
        \Omega_A & 0 \\ 
        0 & -\Omega_A 
    \end{pmatrix} .
    \end{align}
    This integral only depends on the homotopy class $z$ and satisfies
    \[
    \nu(A) \equiv \mathrm{hol}_{(Y, K)}(\mf b) - \mathrm{hol}_{(Y, K)}(\mf a) \bmod \Z.
    \]
    Re-writing the equation for the differential with the bookkeeping coefficients explicitly included thus gives
    \[
d(T^{\mathrm{hol}_{(Y, K)}(\mf a)}[\a]) := \sum_{z \in \pi_1(\B([\a], [\b])) } \# \check{M} _{z}([\a], [\b])_0\cdot  T^{(\nu(A) + \mathrm{hol}_{(Y, K)}(\mf a) - \mathrm{hol}_{(Y, K)}(\mf b))} (T^{\mathrm{hol}_{(Y, K)}(\mf b)}[\b]), 
\]
where the exponent $\nu(A) + \mathrm{hol}_{(Y, K)}(\mf a) - \mathrm{hol}_{(Y, K)}(\mf b)$ of $T$ is an integer for a representative $[A] \in \check{M} _{z}([\a], [\b])_0$.  Note that once again, the orientation of $K$ enters into the definition of $\nu(A)$ in two places. Firstly, the orientation of $K$ determines the orientation of the surface $\R \times K$ over which to carry out the integral. Secondly, as in Remark~\ref{rem:orientationindependence}, the orientation of $K$ determines which of the two components of $F_A|_{\R \times K}$ in \eqref{eq:split} to integrate. Reversing the orientation of $K$ negates both of these; hence $\nu(A)$ is unchanged under orientation reversal.
\end{itemize}  

The maps $\delta_1,\delta_2$ in \eqref{eq:dtildedef} involve counting singular instantons on $\R\times Y$ with one reducible limit $\theta$:
\begin{align*}
& \delta_1([\a]) := \sum_{z \in \pi_1([\a], [\theta]) } \# \check{M} _{z}([\a], [\theta])_0\cdot  T^{\nu(A)}[\theta]\\
&\delta_2([\theta]) := \sum_{z \in \pi_1([\theta], [\b]) } \# \check{M} _{z}([\theta], [\b])_0\cdot  T^{\nu(A)} [\b]
\end{align*}
Finally, to define the map $v$, let us fix a basepoint $x_0$ on $K$. We then let
\begin{align*}
v([\a]) := \sum_{z \in \pi_1([\a], [\b])} \# (\check{M}_{z}([\a], [\b])_1 \cap {H}^{-1} (h))\cdot  T^{\nu(A)} [\b],
\end{align*}
where
\[
H \colon  \check{M}_{z}([\a], [\b]) \to S^1
\]
is the modified holonomy along the path $\R \times \{x_0\}$ and $h$ is a regular value of $H$; see \cite[Section 3.3]{DS19}. The definition of $v$ thus depends on $K$ as a based knot; we discuss this further presently.

\begin{defn}
Let $K$ be a knot in a homology $3$-sphere $Y$ and fix a basepoint $x_0$ on $K$. Define
\begin{equation*}\label{eq:scomplexofknotdeflocal}
	\widetilde C_*(Y,K, x_0) = C_*(Y,K, x_0)\oplus C_{*-1}(Y,K, x_0)\oplus R,
\end{equation*}
where the reducible summand $R$ is formally generated by $[\theta]$. We construct the differential by using the components $d$, $\delta_1$, $\delta_2$, and $v$ defined above. In \cite[Theorem 3.34]{DS19}, it is shown that this is an $\mathcal{S}$-complex whose homotopy type is an invariant of $(Y,K, x_0)$. In particular, it is independent of the choice of orbifold metric and holonomy perturbation. 
\end{defn}

In fact, the homotopy type of this $\cS$-complex is independent of the choice of $x_0$; hence we will sometimes suppress $x_0$ from the notation and write $\smash{\wt{C}(Y, K)}$. Note that $x_0$ enters into the construction only through the definition of $v$.

\subsubsection{Orientation reversal and concordance inverses}\label{sec:orientationreversal}
As remarked in \cite[Section 2.2]{DS20}, the instanton complex of $(Y, K)$ does not depend on the orientation of $K$. Since this will be central to the paper, we make the isomorphism explicit. 

\begin{thm}\label{orientation_rev} 
For each Morse--Smale perturbation $\pi$ for $(Y,K)$, we have an isomorphism of $\cS$-complexes 
\[
\alpha: \wt{C} (Y , K^r, x_0, \pi ) \to  \wt{C} (Y , K, x_0, \pi ).
\]
\end{thm}
\begin{proof}
For a fixed orientation of $K$, the configuration space $\mathcal{B}(Y, K)$ and orbifold Chern--Simons functional $\mathrm{cs}$ are defined are defined using a model connection $A_K$, where $A_K$ is locally given by
    \[
    A_K =
\begin{pmatrix}
    \frac{1}{4} i d\theta&  0 \\
    0 & -  \frac{1}{4} i d\theta 
\end{pmatrix},
\]
for a meridional coordinate $\theta$ compatible with the orientation of $K$. The configuration space and Chern--Simons functional for $K^r$ are likewise defined using a model connection $A_{K^r}$. This has the same form as above, but with $\theta$ replaced by a coordinate $\theta^r$ compatible with the oriented knot $K^r$. Now, clearly $\theta^r = -\theta$. Hence
\[
A_{K^r} = 
\begin{pmatrix}
   - \frac{1}{4} i d\theta&  0 \\
    0 &   \frac{1}{4} i d\theta 
\end{pmatrix},
\]
and we see that $A_K$ and $A_{K^r}$ are not the same, nor are they related any orbifold $\SU(2)$-gauge transformation. However, $A_K$ and $A_{K^r}$ are related by the constant bundle isomorphism  
\[
g_0:= \begin{pmatrix}
    0& 1 \\ 
    -1 & 0 
    \end{pmatrix}
\]
over $Y$ between the orbifold $\SU(2)$-bundle for $K$ and the orbifold $\SU(2)$-bundle for $K^r$. Using this, we obtain a bijection between the configuration spaces
\[
\iota : \mathcal{B}(Y, K) \to \mathcal{B}(Y, K^r) \quad \text{where} \quad [\a] \mapsto [g_0^{-1} \a  g_0 + g_0^{-1} d  g_0] =  [g_0^{-1} \a  g_0]. 
\]
The values of the orbifold Chern--Simons functional are invariant under $\iota$. The above map between configuration spaces also appears in the discussion of the flip symmetry in \cite[Section 2.3]{DS19}; we use the notation $\iota$ even though our map is not quite the self-map $\iota$ given in \cite[Section 2.3]{DS19}.

Written in the language of $\iota$, our discussion of orientation reversal in the previous section shows 
\[
\mathrm{hol}_{\lambda_K} ([\a]) = \mathrm{hol}_{\lambda_{K^r}} ([\iota \a]) 
\]
where $\lambda_{K^r} = -\lambda_K$ is the oriented longitude of $K^r$. It is the clear that $\iota$ defines a module isomorphism between $C(Y , K, x_0, \pi )$ and $C(Y , K^r, x_0, \pi )$ via the correspondence
    \[
    R\cdot T^{\mathrm{hol}_{(Y, K)}(\mf \a)} \mf [\a]  \to R\cdot T^{\mathrm{hol}_{(Y, K^r)}(\mf \iota \a)} \mf [\iota\a]
    \]
    obtained from the bijection 
    \[
    \iota|_{\mathfrak{C}^{*}_{\pi}(Y,K)} : \mathfrak{C}^{*}_{\pi}(Y,K) \to \mathfrak{C}^{*}_{\pi}(Y,K^r) . 
    \]
    We define the map to be the identity on the reducible summand. A similar discussion regarding $\nu(A)$ shows that this is a chain isomorphism.
\end{proof}

Theorem~\ref{orientation_rev} is distinct from the formula which compares $(Y, K)$ with $(Y^r, K^r)$. As discussed in \cite{DS19}, the latter is given by:

\begin{thm}\label{thm:concinv}
For each Morse--Smale perturbation $\pi$ for $(Y,K)$, we have an isomorphism of $\cS$-complexes
\[
\wt{C} (Y^r , K^r, x_0, \pi) \cong \wt{C} (Y , K, x_0, \pi )^\dagger
\]
\end{thm}

In the case that $Y = S^3$, we of course identify $(Y^r, K^r, x_0)$ with the concordance inverse $(S^3, -K, x_0)$.

\subsubsection{Cobordism maps}\label{cobordism map}
We now turn to the cobordism maps. Let $(W,S):(Y,K)\to (Y',K')$ be a cobordism of pairs. Here, $W$ is an oriented cobordism from $Y$ to $Y'$ and $S$ is a connected, oriented surface cobordism from $K$ to $K'$ embedded in $W$. 

\begin{itemize}[noitemsep]
    \item Define the {\it topological energy} of $(W,S)$ by
\begin{equation}\label{eq:kappareducible}
	\kappa(W, S) = \min_{z \in H^2(W; \Z) } - \left( z + \frac{1}{4} \operatorname{PD}[S] \right)^2.
\end{equation}
The elements $ z \in H^2(W; \Z)$ minimizing \eqref{eq:kappareducible} are called the {\it minimizing classes}.

\item Consider the trivial $\SU(2)$-bundle over $W$. The {\it{index}} of the corresponding orbifold Atiyah--Hitchin--Singer complex is given by
\[{\rm ind}(W, S)=8\kappa(W, S)+\frac{3}{2}(\sigma(W)+\chi(W))+\frac{1}{2}S\cdot S+ \chi(S)+\sigma(Y,K)-\sigma(Y',K')-1.\]
 This is always an odd integer.

\item For future reference, define
\begin{equation}\label{eta}
	\eta(W, S)=\sum_{z \in H^2(W; \Z) \; {\rm minimizing\  classes } }(-1)^{z^{2}}T^{2z  \cdot S } . 
\end{equation}
\end{itemize}

It will be helpful to briefly consider how the above quantities behave under orientation reversal of $S$. It is clear that reversing orientation on $S$ leaves $\kappa(W, S)$ unchanged: if $z$ was a minimizing class previously, then $-z$ is now a minimizing class with the same minimizing value, since $\operatorname{PD}[S]$ is replaced by $-\operatorname{PD}[S]$. Hence $\ind(W, S)$ is also unchanged. Finally, $\eta(W, S)$ is unchanged since $z^2$ and $z \cdot S$ are unchanged under negating both $z$ and $S$. 

It will also be important for us to understand the interpretation of the minimizing classes. As discussed in \cite[Section 2.3]{DS20}, the set of minimizing classes is in bijection with the set of singular abelian reducible connections over $(W, S)$. Such a connection $A$ has curvature of the form
\[
F_A= \begin{pmatrix}
    \Omega_A & 0 \\
    0 & -\Omega_A 
\end{pmatrix}
\]
near $S$. The correspondence sends $A$ to the first Chern class $z$ of the line bundle determined by this splitting. However, as before, the splitting actually consists of a pair of dual line bundles $L$ and $L^*$; we select which  factor to use by computing the holonomy around an oriented meridian of $S$. If the orientation of $S$ is reversed, $A$ is thus sent to $-z$. Hence the \textit{same} abelian reducible connection corresponds to the minimizing class $z$ in the context of $(W, S)$, but $-z$ if the orientation of $S$ is reversed.

%Taking the first Chern class of the [REF] gives $z$. As in our discussion of $\nu(A)$, reversing orientation on $S$ 

\begin{comment}
{\color{blue} This $z$ is the first Chern class of a singular abelian reducible $A$ with respect to the $(1,1)$-block of the decomposition
; $F_A= \begin{pmatrix}
    \Omega_A & 0 \\
    0 & -\Omega_A 
\end{pmatrix}$, which comes from the compatibility between the local form of $A$ and the orientation of $S$.
Note that if we change the orientation of $S$, this decomposition is flipped, so $z$ is replaced with $-z$.
}

{\color{blue} This again is independent of the choices of orientations of $S$ as $\kappa(W, S)$ is.} 

{\color{blue}
Again, for $-S$, since the local model of singular connection will change, we need to replace $ z$ with $-z$, this change makes $\eta$ independent of the orientations of $S$.  }
\end{comment}

\begin{defn}\label{defn:negdefcob} 
We say that $(W, S)$ is a \textit{(height zero) negative definite pair} if:

\begin{itemize}[noitemsep]
\item[(i)]$b_{1}(W)=b^{+}(W)=0$; and,
\item[(ii)] $\ind (W, S) = -1$.
\end{itemize} 
We say a negative definite pair is {\it strong} if $\eta (W, S)$ is invertible in $R$.\footnote{This is not quite as general as the definition of a negative definite pair in \cite{DS19}. There, the authors allow $\ind(E) \geq -1$, with larger indices corresponding to negative definite pairs of greater height. Also, they consider general $U(2)$-bundles. Here we restrict ourselves to the case $\ind(E) = -1$ (i.e., height zero) and $\SU(2)$-bundles since this will suffice for our applications. Roughly speaking, the condition of being height zero guarantees that the cobordism map is grading-preserving and has the potential to map $\theta$ to $\theta$.}
\end{defn}

Now suppose $(W, S)$ is a negative definite pair from $(Y, K)$ to $(Y', K')$. We fix the trivial $\SU(2)$-bundle $E$.
Choose basepoints $x_0$ and $x_0'$ on $K$ and $K'$, and let $\gamma$ be a properly embedded path on $S$ from $x_0$ to $x_0'$. Fix the following:
\begin{itemize}[noitemsep]
\item orbifold metrics $g$ and $g'$ for $(Y, K)$ and $(Y', K')$, 
\item an orbifold metric $\tilde{g}$ for 
\[
(W^+, S^+):= (( -\infty, -1]\times Y \cup W \cup [0, \infty ) \times Y , ( -\infty, -1]\times K \cup S \cup [0, \infty ) \times K)
\]
which coincides with $dt^2 + g$ and $dt^2+ g'$ on the ends, 
\item Morse--Smale holonomy perturbations $\pi$ and $\pi'$ of the Chern--Simons functionals of $(Y, K)$ and $(Y', K')$, 
\item a holonomy perturbation $\tilde{\pi}$ of the instantons for $(W, S)$ so that the moduli spaces of finite energies are all regular.
\end{itemize}
As in the construction of $\wt{d}$, we then define
\begin{align*}
   & \lambda ( [\a]) := \sum_{z\in \pi_1(\B ([\a], (W,S), [\b]))  } \# {M}_z([\a], (W,S), [\b])_0 \cdot T^{\nu(A)} [\b]  \\
  &  \Delta_1 ( [\a]) := \sum_{z\in \pi_1(\B ([\a], (W,S), [\theta'])) } \# {M}_z([\a], (W,S), [\theta'])_0 \cdot T^{\nu(A)} [\theta']\\
   & \Delta_2 ( [\theta]) := \sum_{z\in \pi_1(\B ([\theta], (W,S), [\b])) } \# {M}_z([\theta], (W,S), [\b])_0 \cdot T^{\nu(A)} [\b]\\
   &\mu ([\a]) :=  \sum_{z\in \pi_1(\B ([\a], (W,S), [\b])) } \# \left({M}_z([\a], (W,S), [\b])_1 \cap H^{-1}(h) \right) \cdot T^{\nu(A)}[\b], 
\end{align*}
where: 
\begin{itemize}[noitemsep]
    \item  For any $[\a]$ and $[\b]$, we write $\pi_1(\B ([\a], (W,S), [\b]))$ to denote the set of homotopy classes of connections over $(W^+, S^+)$ limiting to $[\a]$ and $[\b]$.
    \item For any homotopy class $z$, we write ${M}_z([\a], (W,S), [\b])_i$ to denote the $i$-dimensional component of the moduli space of perturbed singular instantons in the class $z$. Here, we suppress writing the data of the metric and perturbation. These live over $(W^+, S^+)$.
    \item For any $[A] \in {M}_z([\a], (W,S), [\b])_i$, we write $\nu(A)$ for the integral
    \[
   \nu(A) = \frac{i}{\pi} \int_{S^+} \Omega_A - \frac{1}{2} S\cdot S ,
    \]
    where $\Omega_A$ is the 2-form on $S^+$ defined by
    \[
    F_A |_{S^+} = \begin{pmatrix}
        \Omega_A & 0 \\ 
        0 & -\Omega_A 
    \end{pmatrix} .
    \]
    Note that, by Stokes' theorem, we have 
    \begin{align}\label{eq:Stokes}
        \nu(A)=\mathrm{hol}_{(Y',K')}(\mf b)-\mathrm{hol}_{(Y,K)}(\mf a)
     \end{align}
    holds. The quantity $\nu(A)$ only depends on the homotopy class $z$. Note that the conventions of $\nu(A)$ in \cite{DS19} and \cite{DS20} differ; we follow \cite{DS20}. Once again, reversing the orientation of $S$ leaves $\nu(A)$ unchanged, as this sends $S$ to $-S$ but also interchanges the roles of the two summands of the above splitting.
    \item The map $H: M_z([\a], (W,S), [\b]) \to S^1$ is the modified holonomy along the path $\gamma$ and $h$ is a regular value of $H$. Note that this is the only place in which the data of $\gamma$ appears.
\end{itemize}

\begin{defn}
Let $(W, S)$ be a negative definite pair from $(Y, K)$ to $(Y', K')$. Choose basepoints $x_0$ and $x_0'$ on $K$ and $K'$, and let $\gamma$ be a properly embedded path in $S$ from $x_0$ to $x_0'$.
Define the cobordism map
    \[
    \wt{\lambda}:  \wt C(Y,K, x_0)\to \wt C(Y',K', x_0')
    \]
associated to this data by
    \[
\begin{pmatrix}
\lambda & 0 & 0 \\
\mu & \lambda & \Delta_2 \\
\Delta_1 & 0 &  \eta
    \end{pmatrix},
    \]
where $\lambda$, $\Delta_1$, $\Delta_2$, and $\mu$ are as above, and the component $\eta$ is defined to be multiplication by the quantity $
\eta= 
\eta(W, S)$ from \eqref{eta}. In \cite[Propositions 2.24, 4.17]{DS20}, it is shown that this is a morphism of $\cS$-complexes whose homotopy type is independent of the choice of orbifold metric and perturbation. Note that $\smash{\wt{\lambda}}$ is local if and only if our negative definite pair is strong.
\end{defn}

\begin{ex}\label{ex:homologyconcordance}
Let $W$ be a rational homology cobordism and $S$ be a concordance in $W$. In this situation, it is easily computed that $\kappa(W,S)=0$ and $\ind (W,S)=-1$. If $W$ is moreover a $\Z_2$-homology cobordism, then the number of minimizing classes is odd and it is straightforward to check
\[
\eta(W, S) = 1 \in R= \Z_2[T^{\pm 1}]. 
\]
Hence, a $\Z_2$-homology concordance induces a local morphism of $\cS$-complexes.
\end{ex}

For subsequent sections, it will be important to understand the relationship between $\wt{\lambda}_{(W,S)}$ and $\alpha$.

\begin{lem}\label{lem:dep_ori}
    Let $(W, S)$ be a negative definite pair from $(Y, K)$ to $(Y', K')$. Choose basepoints $x_0$ and $x_0'$ on $K$ and $K'$, and let $\gamma$ be a properly embedded path in $S$ from $x_0$ to $x_0'$. Then the induced cobordism map satisfies  
\[
\alpha \circ \wt{\lambda}_{(W,S)} = \wt{\lambda}_{(W,-S)} \circ \alpha . 
\]
\end{lem}
\begin{proof}
The claim for the $(3,3)$-component follows from the fact that $\eta(W, S)= \eta(W, -S)$. Consider the $(1,1)$-component. Expanding the right-hand side of the desired relation gives the equivalent identity
\begin{align*}
\alpha \circ \lambda_{(W, S)} ( [\a]) & = \sum_{z\in \pi_1(\B ([\iota\a], (W,-S), [\iota\b]))  } \# {M}_z([\iota \a], (W,-S), [\iota\b])_0 \cdot T^{\nu(A)} [\iota\b].    
\end{align*}
Expanding the left hand side, we see that it suffices to show
\[
{M}_z([\iota\a], (W,-S), [\iota\b])_0 \cong {M}_z([\a], (W,S), [\b])_0.
\]
This follows from the same argument as in the proof of Theorem~\ref{orientation_rev}. 

Explicitly, recall that the space of singular instantons on $(W, S)$ is constructed using a model connection defined as follows. We form the non-compact 4-manifold-surface $(W^+ ,S^+)$ obtained by attaching cylindrical ends. Let $E$ be a trivial bundle over $(W, S)$ with an identification $E|_{S} \cong_{\phi} L\oplus L^*$, and extend this to $(W^+ ,S^+)$.
Let $N(S)$ be a small tubular neighborhood of $S$, which is also a trivial bundle over $S$.
We may write $E|_{N(S)}=L\oplus L^{*}$ with \(L_S\cong \underline{\mathbb C}\). 
Let \(\theta_S\) denote the positive angular coordinate in the normal direction
to \(S\), determined by the orientations of \(W\) and \(S\). We define the model
singular connection
\[
        A^0_S
        =
        A^{\rm sm}
        +
        \frac{i}{4}\beta(r)d\theta_S
        \begin{pmatrix}
        1&0\\
        0&-1
        \end{pmatrix},
\]
where \(A^{\rm sm}\) is a smooth connection preserving the above splitting near
\(S\) which coincides with $\pr^* \a$ and $\pr^* \b$ over $(S^3, K) \times (-\infty, -1]$ and $(S^3, K') \times [1, \infty )$, and \(\beta\) is a cutoff function equal to \(1\) near \(S\). For limiting
critical points \([\a]\) and \([\b]\), we define
\[
        \mathcal A([\a],(W,S),[\b])
\]
to be the affine space of singular connections \(A\) on \(E|_{W\setminus S}\)
which are modeled on \(A^0_S\) near \(S\) and converge to representatives of
\([\a]\) and \([\b]\) on the two cylindrical ends. 
We then set
\[
        \mathcal B([\a],(W,S),[\b])
        =
        \mathcal A([\a],(W,S),[\b])/\mathcal G(W,S),
\]
where \(\mathcal G(W,S)\) is the gauge group preserving the singular structure.
The corresponding moduli space is the zero set of the perturbed ASD equation in
this configuration space. 

Now consider \(-S\). In this case the positive angular coordinate in the normal direction is given by \(d\theta_{-S}=-d\theta_S\). Thus the local singular data for \((W,-S)\),
$$A^0_{-S}=A^{\rm sm}+\frac{i}{4}\beta(r)d\theta_{-S}\begin{pmatrix}1&0\\0&-1\end{pmatrix},$$
is not the same as the singular data for $(W, S)$, but is rather taken to it by the bundle isomorphism \(\Psi\) exchanging \(L\) and \(L^*\). On the cylindrical ends, this is precisely the identification used in the definition of \(\alpha\), namely \([\a]\mapsto[\iota\a]\) and \([\b]\mapsto[\iota\b]\). Therefore the map \(A\mapsto \Psi(A)\) induces a bijection $$\mathcal A([\a],(W,S),[\b])\cong \mathcal A([\iota\a],(W,-S),[\iota\b]),$$ and this bijection identifies the corresponding gauge groups. Hence it descends to an identification of configuration spaces $$\mathcal B([\a],(W,S),[\b])\cong \mathcal B([\iota\a],(W,-S),[\iota\b]).$$ Since the perturbed ASD equation is preserved under this correspondence, we obtain an identification of moduli spaces $$M([\a],(W,S),[\b])_0\cong M([\iota\a],(W,-S),[\iota\b])_0$$
as desired.
\end{proof}

\subsubsection{Dependence on basepoints and paths}\label{sec:dobap}
It is important to understand that the instanton knot complex depends on a choice of basepoint and that each cobordism map depends on a choice of path. This is analogous to the situation in knot Floer homology, where $K$ is regarded as a doubly-basepointed knot and any knot cobordism is equipped with a set of dividing curves. (See the work of Zemke \cite{Zemkefunctoriality}.) As this will be extremely important in subsequent sections, we clarify the dependence here.

It is evident that the choice of basepoint does not affect the underlying chain group $C_*(Y, K)$. Indeed, the only data which changes is the component $v$ of the differential, which is now computed using the holonomy $H$ along a different parallel line. Let $x_0$ and $x_1$ be two choices of basepoint for $K$ and denote the $v$-maps computed with respect to $\R \times \{x_0\}$ and $\R \times \{x_1\}$ by $v_0$ and $v_1$, respectively. To relate these, let $\gamma$ be a path from $\{0\} \times x_0$ to $\{1\} \times x_1$ in $[0, 1] \times K$. Define 
\[
\mu_{\gamma} \colon C_*(Y, K, x_0) \rightarrow C_{*+1}(Y, K, x_1)
\]
analogously to $v$, by counting the degree of 
\[
\operatorname{hol}_\gamma: M([\a], [0,1]\times (Y,K), [\b])_1 \to S^1 , 
\]
where $\operatorname{hol}_\gamma$ is the holonomy along $\gamma$. Then a standard argument shows $\mu_\gamma$ satisfies $v_0 + v_1 = d\mu_\gamma + \mu_\gamma d$.
If we define
\[
\wt{\lambda}_{\gamma} : \wt{C}_*(Y, K, x_0) \to \wt{C}_*(Y, K, x_1)
\]
by
\[
\wt{\lambda}_\gamma = 
\begin{pmatrix}
\id  & 0 & 0\\
\mu_{\gamma } & \id   & 0 \\
0 & 0 & \id  
\end{pmatrix},
\]
we obtain a homotopy equivalence of $\cS$-complexes. Hence the homotopy type of our $\cS$-complex is independent of our choice of basepoint. Note, however, that the homotopy equivalence itself depends on the choice of $\gamma$. As we will see, in the involutive setting we are unable to establish basepoint-independence due to the fact that we cannot take $\gamma$ to be equivariant.

More generally, our construction of the cobordism map $\wt{\lambda}$ depends on a choice of $\gamma$. However, the following is shown in \cite[Theorem 3.34]{DS19}.
\begin{thm} 
Let $(W, S)$ be a negative definite pair from $(Y, K)$ to $(Y', K')$. Choose basepoints $x_0$ and $x_0'$ on $K$ and $K'$, and let $\gamma$ be a properly embedded path in $S$ from $x_0$ to $x_0'$. Then the $\mathcal{S}$-homotopy class of the map 
\[
\wt{\lambda}: \wt C(Y,K, x_0)\to \wt C(Y',K', x_0')
\]
depends only on the homotopy class of $\gamma$ rel boundary.     
\end{thm}
\begin{proof}
Let $\gamma_1$ and $\gamma_2$ be homotopic rel boundary and let $\mu_ 1$ and $\mu_2$ be defined as in the previous section, by considering ${M}_z([\a], (W,S), [\b])_1 \cap H^{-1}_1(h_1)$ and ${M}_z([\a], (W,S), [\b])_1 \cap H^{-1}_2(h_2)$. Here, 
\[
H_i: {M}_z([\a], (W,S), [\b])_1 \to S^1
\]
is the modified holonomy along $\gamma_i$ and $h_1$ and $h_2$ are regular values of $H_1$ and $H_2$. (See \cite[Appendix]{DS19} for the precise constructions of these  holonomy maps.) We can suppose $h_1=h_2$ since regular values of $H_1$ and $H_2$ are residual. 
Now consider a 1-parameter family $\gamma_t$ of paths interpolating between $\gamma_1$ and $\gamma_2$. Define the corresponding 1-parameter modified holonomy map
\[
H_t : [0,1]\times {M}_z([\a], (W,S), [\b])_1 \to S^1.
\]
This can be constructed by a natural 1-parameter family version of \cite[Appendix]{DS19}. Choosing a regular value of $H_t$ gives a 1-dimensional cobordism between the moduli spaces ${M}_z([\a], (W,S), [\b])_1 \cap H^{-1}_1(h_1)$ and ${M}_z([\a], (W,S), [\b])_1 \cap H^{-1}_2(h_2)$.
Forming a similar matrix as in the argument for basepoint-independence completes the proof.
\end{proof}

%    The only component in  $\wt{\lambda}_{(W,S,E, \gamma )}$ which depends on the path $\gamma$ is the map $\mu$. 
%    We choose orbifold metrics and Morse--Smale holonomy perturbations of Chern--Simons functionals for $(Y, K)$ and $(Y', K')$ to make the groups $\wt C(Y,K;R)$ and $\wt C(Y',K';R)$. The choices of such data change $\wt C(Y,K;R)$ $\wt C(Y',K';R)$ but $\mathcal{S}$-homotopy types are well defined.  
%We take homotopic paths $\gamma_1$ and $\gamma_2$ rel boundary. Corresponding to these paths, we have the maps $
%In the case of disconnected ends, $\gamma$ is replaced with an embedded graph going between the basepoints on each end; see \Cref{section:Two connected sum formula}. Once again, the only component of the cobordism map which changes is the $(2, 1)$-component. 

\begin{rem}
It will be useful to observe that the under-bar and upper-bar complexes from Definition~\ref{def:underbarupperbar} do \textit{not} depend on the choice of basepoint, nor do the induced cobordism maps depend on a choice of path, as the components $v$ and $\mu$ do not appear.
\end{rem}

\section{Singular instantons and strongly invertible knots}\label{sec:involutive}
In this section, we introduce an involutive version of singular instanton Floer theory for strongly invertible knots. We first define the notion of an \textit{involutive $\cS$-complex} and discuss the algebra of these complexes. We also introduce the critical notion of \textit{weak equivalence}, which is a variation of our usual formalism adapted to the fact that our theory has an \textit{a priori} dependence on our choice of basepoints and paths. We then give the geometric construction of involutive instanton theory and end by discussing some of the local equivalence sets from the introduction.

\subsection{Algebra for knots with involutions}

Our first goal is to introduce the notion of an involutive $\cS$-complex. As in the previous section, we continue to work over $R=\Z_2[T^{\pm 1}]$. 
%The conjugate of a polynomial $p\in R$ will be denoted $\ov{p}$; this replaces $T$ with $T^{-1}$.
We write $\Lambda = T^{2} + T^{-2}$.

\subsubsection{Involutive $\cS$-complexes} We begin with the basic definitions.
%In what follows, recall our discussion of $T$-linear maps from Definition~\ref{def:alpha0}.

\begin{defn}\label{def:involutiveScomplex}
An \textit{involutive $\cS$-complex is a pair} $(\wt{C}, \wt{\tau})$, where $\wt{C}$ is an $\cS$-complex and $\smash{\wt{\tau}}$ is a homotopy involution on $\smash{\wt{C}}$. Explicitly, this means $\smash{\wt{\tau}}$ is a $\cS$-morphism from $\smash{\wt{C}}$ to itself such that
        \[
        \wt\tau^2 + \id = \wt{H} \wt{d} + \wt{d} \wt{H}
        \]
for some $\wt{H}$ as in Definition~\ref{def:Shomotopy}. When the context is clear, we suppress $\smash{\wt{\tau}}$ and simply write $\smash{\wt{C}}$ for an involutive $\cS$-complex.
\end{defn}

\begin{defn}
For a fixed splitting of $\wt{C}$, we denote the components of $\wt{\tau}$ by
\begin{equation}\label{eq:tautildedef}
\wt{\tau} = 
\begin{pmatrix}
\tau & 0 & 0\\
\mu^\tau & \tau  & \Delta_2^\tau\\
\Delta_1^\tau & 0 &1 
\end{pmatrix}.
\end{equation}
Note that $\wt{\tau}$ induces an $R$-linear involution on $\ker \chi/\ima \chi \cong R$, which is necessarily the identity. This gives the $(3, 3)$-component of the above matrix.
%Each of these is $T$-linear. 
%As in Remark~\ref{rem:tauidentity}, we assume that the $(3,3)$-component is $\alpha$. 
%By abuse of notation, we denote this component by $1$, since it maps $1 \in R$ to $1 \in R$.
\end{defn}

\begin{rem}
A similar comment as Remark~\ref{rem:reducetoC} holds for the components of $\wt{\tau}$. We will sometimes regard the components of $\wt{\tau}$ as going from $C$ to itself or between $C$ and $R$, depending on the context. The condition that $\wt{\tau}$ is a homotopy involution may be translated into a collection of identities involving these components. 
\end{rem}

\begin{ex}\label{ex:trefoilinvolutive}
Recall from Example~\ref{ex:trefoil} that the $\cS$-complex for the trefoil is given by 
    \[
    \wt{C} (T_{2,3} ) = \langle \rho \rangle \oplus \langle \underline{\rho}  \rangle \oplus R, 
    \]
with the single nontrivial component of $\wt{d}$ given by $\delta_1 \rho = \Lambda \theta$. Let $\tau_0$ be the unique strong inversion on the trefoil. We claim that the action of $\smash{\wt{\tau}}_0$ is equal to the identity. It is clear for degree reasons that $\mu^\tau  = \Delta^\tau_1 = \Delta^\tau_2=0$. Thus, we only need to determine the component $\tau$. We have $\tau \rho = c \rho$ for some $c \in \Z_2[T^{\pm 1}]$. The fact that $\wt{\tau}$ commutes with $\smash{\wt{d}}$ implies
\[
\Lambda \theta  = 1 \circ \delta_1(\rho) = \delta_1 \circ \tau ( \rho) = \Lambda(c \theta).  
\]
This shows $c=1$, as desired.
\end{ex}

\begin{defn}\label{def:involutiveSmorphism}
An \textit{involutive $\cS$-morphism} $\wt{\lambda} \colon (\wt{C}, \wt{\tau}) \rightarrow (\wt{C}', \wt{\tau}')$ is an $\cS$-morphism that homotopy intertwines $\wt{\tau}$ and $\wt{\tau}'$; that is,
\[
\wt{\lambda} \wt{\tau} + \wt{\tau}' \wt{\lambda} = \wt{H} \wt{d}' + \wt{d} \wt{H}
\]
for some $\wt{H}$ as in Definition~\ref{def:Shomotopy}.
\end{defn}

Two involutive morphisms are homotopic if they  are homotopic as non-involutive morphisms; i.e., in the sense of Definition~\ref{def:Shomotopy}. Note that we do \textit{not} require the homotopy itself to have any specified behavior with respect to the $\wt{\tau}$-actions. A homotopy equivalence between two involutive $\cS$-complexes thus consists of a pair of $\cS$-morphisms that are (a) involutive and (b) homotopy inverse in the non-equivariant sense. Likewise, a local map of involutive $\cS$-complexes is an involutive $\cS$-morphism which is local in the sense of Definition~\ref{height i morphism}. Local equivalence is defined similarly.

\subsubsection{Upper-bar and under-bar complexes}

It is clear that $\wt{\tau}$ descends to homotopy involutions on the under-bar and upper-bar complexes.

\begin{defn}
Let $(\wt{C}, \wt{\tau})$ be an involutive $\cS$-complex. We write $\un{\tau} \colon \un{C} \to \un{C}$ and $\ov{\tau} \colon \ov{C} \to \ov{C}$ for the induced homotopy involutions on the under-bar and upper-bar complexes. We refer to $(\un{C}, \un{\tau})$ and $(\ov{C}, \ov{\tau})$ as the \textit{involutive under-bar and upper-bar complexes}. 
\end{defn}

We think of $\un{\tau}$ as capturing the components $\tau$ and $\Delta_2^\tau$ of $\wt\tau$, and $\ov{\tau}$ as capturing $\tau$ and $\Delta_1^\tau$:
\[
\un{\tau} = 
\begin{pmatrix}
\tau  & \Delta_2^\tau \\
0 &1 
\end{pmatrix}
\quad \text{and} \quad
\ov{\tau} = 
\begin{pmatrix}
\tau  & 0 \\
\Delta_1^\tau &1 
\end{pmatrix}.
\]
A morphism of involutive $\un{C}$-complexes is required to homotopy commute with $\un{\tau}$ up to some power of $T$ via a homotopy that preserves the two-step filtration of Section~\ref{sec:under-upper}, and likewise for a morphism of involutive $\ov{C}$-complexes. Clearly, an involutive $\cS$-morphism induces involutive morphisms of the associated $\un{C}$ and $\ov{C}$-complexes. Homotopy equivalences and local maps are defined in the obvious way.
%, and involutive $\theta$-supported cycles

In the case of an under-bar complex, we think about involutive local maps in terms of the following.

\begin{defn}\label{def:involutivethetasupported}
Let $(\un{C}, \un{\tau})$ be an involutive $\un{C}$-complex. We say that a cycle $x \in \un{C}$ is an \textit{involutive $\theta$-supported cycle} if it is a $\theta$-supported cycle in the sense of Definition~\ref{def:thetasupported} and its homology class is fixed by $\un{\tau}$; that is, 
\begin{equation}\label{eq:kfixed}
\un{\tau}_* [x] =  [x].
\end{equation}
\end{defn}

\begin{comment}

Once again, specifiying th\ue of $k$ in \eqref{eq:kfixed} give a specialization of Definition~\ref{def:kinvolutivethetasupported}, although this will not be used often in this paper.

\begin{defn}\label{def:kinvolutivethetasupported}
Let $(\un{C}, \un{\tau})$ be an involutive $\un{C}$-complex. For a fixed $k \in \Z$. we say that a cycle $x \in \un{C}$ is an \textit{$T^k$-involutive $\theta$-supported cycle} if it is a $\theta$-supported cycle in the sense of Definition~\ref{def:thetasupported} and
\[
\un{\tau}_* [x] = T^k [x].
\]
\end{defn}

A local morphism of involutive $\un{C}$-complexes maps involutive $\theta$-supported cycles to involutive $\theta$-supported cycles. It is straightforward to check that $\un{C}$ admits an involutive $\theta$-supported cycle if and only if there exists an local map from the trivial complex $R$ into $\un{C}$. Indeed, $\un{C}$ admits a $T^k$-involutive $\theta$-supported cycle if and only if there exists a $T^k$-involutive local map from the trivial complex into $\un{C}$. A similar comment as in Remark~\ref{rem:ktwist} applies: after re-scaling by units, every $T^k$-involutive cycle is either $T^0$-involutive or $T^1$-involutive.

\end{comment}

\subsubsection{Tensor products and duals} \label{sec:tensortaubasis}
It is straightforward to define tensor products and duals of involutive $\cS$-complexes. For future reference, we record how the tensor product involution behaves with respect to the splitting from Section~\ref{sec:tensorbasis}.

\begin{defn}
Given involutive $\mathcal{S}$-complexes $(\wt C,\wt \tau)$ and $(\wt C', \wt \tau')$, their tensor product $(\wt C^\otimes,\wt d^\otimes, \chi^\otimes)$ as an $\cS$-complex is defined as before. We equip this with the tensor product homotopy involution $\wt \tau^\otimes = \wt \tau\otimes \wt \tau'$.
\end{defn}

It is useful to explicitly understand the components of the tensor product $\wt\tau^\otimes$. Recall the tensor product basis given in \eqref{eq:tensorbasis}: 
\begin{align*}
\ima(\chi^\otimes)& = R\langle \underline{\mathfrak{a}}_i\otimes \mathfrak{b}_j + \mathfrak{a}_i\otimes \underline{\mathfrak{b}}_j,\quad  \underline{\mathfrak{a}}_i \otimes \underline{\mathfrak{b}}_j, \quad\underline{\mathfrak{a}}_i \otimes 1, \quad1\otimes \underline{\mathfrak{b}}_j \rangle \\
C^\otimes &=R\langle {\mathfrak{a}}_i\otimes \mathfrak{b}_j  ,\quad  {\mathfrak{a}}_i \otimes \underline{\mathfrak{b}}_j, \quad {\mathfrak{a}}_i \otimes 1, \quad1\otimes {\mathfrak{b}}_j\rangle
\end{align*}
A similar calculation as for $\wt{d}^\otimes$ shows that in this basis, $\wt{\tau}^\otimes$ is given by the block matrix 
\[
\wt{\tau}^\otimes \;=\;
\begin{pmatrix}
\tau^\otimes  & 0 & 0 \\[2pt]
\mu^\otimes & \tau^\otimes & \Delta_2^\otimes \\[2pt]
\Delta_1^\otimes  & 0 & 1
\end{pmatrix},
\]
where
\[
\tau ^\otimes =\begin{pmatrix}
\tau\otimes\tau' & 0 & 0 & 0 \\  \tau \otimes (\mu')^\tau + \mu^\tau \otimes \tau' & \tau\otimes\tau' & \tau \otimes (\Delta'_2)^\tau & \Delta_2^\tau \otimes \tau'  \\
\tau \otimes (\Delta'_1)^\tau & 0 & \tau \otimes 1 & 0 \\
\Delta_1^\tau \otimes \tau' & 0 & 0 & 1 \otimes \tau'
\end{pmatrix},
\]
and
\[
\mu^\otimes =\begin{pmatrix}
\mu^\tau\otimes \tau'  & 0 & 0 & \Delta_2^\tau \otimes \tau' \\
\mu^\tau\otimes (\mu')^\tau  & \mu^\tau\otimes \tau'  & \mu^\tau \otimes (\Delta_2')^\tau & \Delta_2^\tau \otimes (\mu')^\tau \\
\mu^\tau\otimes (\Delta_1')^\tau  & 0 & \mu^\tau\otimes 1 & \Delta_2^\tau \otimes (\Delta_1')^\tau \\
\Delta_1^\tau \otimes (\mu')^\tau  & \Delta_1^\tau \otimes \tau' & \Delta_1^\tau \otimes (\Delta_2')^\tau & 1 \otimes (\mu')^\tau
\end{pmatrix},
\]
and
\[
\Delta_1^\otimes =\begin{pmatrix}
\Delta_1^\tau \otimes (\Delta_1')^\tau & 0 & \Delta_1^\tau \otimes 1  & 1 \otimes (\Delta_1')^\tau 
\end{pmatrix}
\quad \text{and} \quad
\Delta_2^\otimes =\begin{pmatrix}
0 & \Delta_2^\tau \otimes (\Delta_2')^\tau & \Delta_2^\tau \otimes 1  & 1 \otimes (\Delta_2')^\tau 
\end{pmatrix}^T.
\]

As before, each of these matrices represents a map from $C^\otimes$ to itself or between $C^\otimes$ and $R$, but within each matrix, entries such as $\mu^\tau$ must be interpreted appropriately as to whether they go from $C$ to $\ima \chi$, $C$ to $C$, or $\ima \chi$ to $\ima \chi$. Each instance of $1$ represents the identity map on $C$, the identity map on $\ima \chi$, or the isomorphism between $C$ and $\ima \chi$.

 \begin{defn}
      Given an $\cS$-complex $(\wt C,\wt \tau)$, its dual $(\wt C^\dagger, \wt d^\dagger , \chi^\dagger)$ as an $\cS$-complex is defined as before. We equip this with the dual homotopy involution $\wt{\tau}^\dagger$ of $\wt{\tau}$.
 \end{defn}

\subsection{Weak morphisms}

As discussed in Definition~\ref{def:involutiveSmorphism}, a morphism of involutive $\cS$-complexes is simply an $\cS$-morphism in the usual sense which homotopy commutes with $\smash{\wt{\tau}}$. However, it turns out that this condition is overly stringent. As we will see, full homotopy commutation with $\wt{\tau}$ will only hold for a cobordism equipped with an $\tau$-invariant path. Since we do not usually wish to impose the existence of such a path, we introduce the following weaker notion.

\begin{defn}\label{def:weaklyinvolutive}
A \textit{weakly involutive $\cS$-morphism} $\wt{\lambda} \colon (\wt{C}, \wt{\tau}) \rightarrow (\wt{C}', \wt{\tau}')$ has the same definition as an involutive $\cS$-morphism, except that the homotopy commutation
\[
 \wt{\lambda} \wt{\tau} + \wt{\tau}' \wt{\lambda} = \wt{H} \wt{d}' + \wt{d}\wt{H}
\]
{\bf is not required to hold at the $(2,1)$-block}. 
\end{defn}

Explicitly, an $\cS$-morphism $\wt{\lambda}$ is weakly involutive if there exists a homotopy $\smash{\wt{H} \colon \wt{C} \rightarrow \wt{C}'}$ such that
\begin{enumerate}[noitemsep]
\item $\chi' \wt{H} + \wt{H} \chi = 0$, as in Definition~\ref{def:Shomotopy}; and,
\item When we write both sides of 
\[
\wt{\lambda} \wt{\tau} + \wt{\tau}' \wt{\lambda} = \wt{H} \wt{d}' + \wt{d}\wt{H}
\]
as $3 \times 3$ block matrices as in \eqref{eq:matrixsplitting}, we have equality of the corresponding blocks {\bf except for (possibly) the $(2, 1)$-block}.
\end{enumerate}
Note that due to Lemma~\ref{lem:matrixsplitting}, this condition is independent of the choice of splitting.

\begin{lem}
The composition of two weakly involutive $\cS$-morphisms is weakly involutive.
\end{lem}
\begin{proof}
Let $\wt\lambda:\wt C\to\wt C'$ and
$\wt\lambda':\wt C'\to\wt C''$ be weakly involutive.
Then there exist $\mathcal S$–homotopies $\smash{\wt H}$ and $\smash{\wt H'}$ and
$(2,1)$–block error terms $E_\lambda$ and $E_{\lambda'}$,
such that
\[
\wt\lambda\wt\tau+\wt\tau'\wt\lambda
=
\wt H\wt d+\wt d'\wt H+E_\lambda \quad \text{and} \quad
\wt\lambda'\wt\tau'+\wt\tau''\wt\lambda'
=
\wt H'\wt d'+\wt d''\wt H'+E_{\lambda'},
\]
where
\[
E_\lambda=
\begin{pmatrix}
0&0&0\\
e_\lambda&0&0\\
0&0&0
\end{pmatrix},
\qquad
E_{\lambda'}=
\begin{pmatrix}
0&0&0\\
e_{\lambda'}&0&0\\
0&0&0
\end{pmatrix}.
\]
Set $\wt\Lambda=\wt\lambda'\wt\lambda$.
Then
\[
\wt\Lambda\wt\tau+\wt\tau''\wt\Lambda
=
\wt\lambda' \wt\lambda\wt\tau
+
\wt\tau''\wt\lambda'\wt\lambda.
\]
Substituting the two weak relations and using the fact that $\wt{\lambda}$ and $\wt{\lambda}'$ are chain maps then shows
\begin{align*}
\wt\Lambda\wt\tau+\wt\tau''\wt\Lambda
&=
\wt\lambda'(\wt H\wt d+\wt d'\wt H+E_\lambda + \wt{\tau}' \wt{\lambda})
+
(\wt H'\wt d'+\wt d''\wt H'+E_{\lambda'} +\wt{\lambda}'\wt{\tau}')\wt\lambda \\
&= \wt\lambda'(\wt H\wt d+\wt d'\wt H+E_\lambda)
+
(\wt H'\wt d'+\wt d''\wt H'+E_{\lambda'})\wt\lambda \\
&= (\wt\lambda'\wt H+\wt H'\wt\lambda)\wt d
+
\wt d''(\wt\lambda'\wt H+\wt H'\wt\lambda)
+
\wt\lambda' E_\lambda
+
E_{\lambda'}\wt\lambda
\end{align*}
Thus the new error term is
\[
E_\Lambda
= \wt\lambda' E_\lambda
+
E_{\lambda'}\wt\lambda.
\]
Direct block multiplication shows that each of the two summands above is supported only in the $(2,1)$–block.
Hence $E_\Lambda$ has only $(2,1)$–component possibly nonzero. This gives the conclusion. 
\end{proof}

Given two involutive $\cS$-complexes, we define the notion of a weak homotopy equivalence or a weak local map to be a homotopy equivalence or local map in the usual sense, except that the morphisms between the complexes are only required to be weakly involutive. We caution the reader that we do \textit{not} relax any other identity with the possibility of a $(2, 1)$-block error term. Thus, for example, a weak homotopy equivalence consists of a pair of $\cS$-morphisms that are (a) weakly involutive and (b) homotopy inverse in the non-equivariant -- but usual -- sense.

The following claim gives some motivation for weak homotopy equivalence.

\begin{lem}\label{lem:21error}
Let $\smash{(\wt{C}_1, \wt{\tau}_1)}$ and $\smash{(\wt{C}_2, \wt{\tau}_2)}$ be involutive $\cS$-complexes. If $\smash{(\wt{C}_1, \wt{\tau}_1)}$ is weakly homotopy equivalent to $\smash{(\wt{C}_2, \wt{\tau}_2)}$, then it is homotopy equivalent to $\smash{(\wt{C}_2, \wt{\tau}_2')}$ for some $\smash{\wt{\tau}_2'}$ which differs from $\smash{\wt{\tau}_2}$ in the $(2, 1)$-block.
\end{lem}
\begin{proof}
Let $\wt{f} \colon \wt{C}_1 \rightarrow \wt{C}_2$ and $\wt{g} \colon \wt{C}_2 \rightarrow \wt{C}_1$ be the homotopy equivalences in question. Consider $\smash{\wt{f} \wt{\tau}_1 \wt{g}}$. It is immediate that this is a homotopy involution on $\smash{\wt{C}_2}$ and that $\smash{(\wt{C}_1, \wt{\tau}_1)}$ and $\smash{(\wt{C}_2, \wt{f} \wt{\tau}_1 \wt{g})}$ are homotopy equivalent $\cS$-complexes. On the other hand,
\begin{align*}
\wt{f} \wt{\tau}_1 \wt{g} &= (\wt{\tau}_2 \wt{f} + \wt{H} \wt{d} + \wt{d}\wt{H}  + E_{21})\wt{g} \\
& = \wt{\tau}_2(\id + \wt{H}'\wt{d} + \wt{d}\wt{H}') + (\wt{H} \wt{d} + \wt{d}\wt{H}  + E_{21})\wt{g} \\ &= \wt{\tau}_2 + \wt{H}'' \wt{d} + \wt{d} \wt{H}'' + E_{21} \wt{g}.
\end{align*}
Here, in the first like we have used the fact that $\wt{f}$ intertwines $\wt{\tau}_1$ and $\wt{\tau}_2$ up to a homotopy plus a $(2, 1)$-block error term $E_{21}$. We then use the fact that $\smash{\wt{f}}$ and $\smash{\wt{g}}$ are homotopy inverse in the second line, and collapse our homotopies into a single homotopy $\smash{\wt{H}''}$ in the third. It is straightforward to check that $E_{21}\wt{g}$ is supported in the $(2, 1)$-block. The identity map gives a homotopy equivalence between $\smash{(\wt{C}_2, \wt{f} \wt{\tau}_1 \wt{g})}$ and $\smash{(\wt{C}_2, \wt{\tau}_2 + E_{21}\wt{g})}$; setting $\smash{\wt{\tau}_2' = \wt{\tau}_2 + E_{21}\wt{g}}$ thus completes the proof. (Note that this is tautologically a homotopy involution, since it is chain homotopic to $\smash{\wt{f} \wt{\tau}_1 \wt{g}}$.)
\end{proof}

The upshot of Lemma~\ref{lem:21error} is that the weak homotopy type of $(\wt{C}, \wt{\tau})$ determines the components $\tau$, $\Delta_1^\tau$, and $\Delta_2^\tau$ of $\smash{\wt{\tau}}$. The reader should thus think of weak homotopy equivalence as throwing out only the information of $\mu^\tau$. As a sanity check, note that if $\smash{\wt{\tau_1}}$ and $\smash{\wt{\tau_2}}$ are two homotopy involutions on $\smash{\wt{C}}$ which differ only in their $\mu^\tau$-components, then the identity provides an obvious weak homotopy equivalence between $(\smash{\wt{C}}, \wt{\tau}_1)$ and $(\smash{\wt{C}}, \wt{\tau}_2)$. The following is unsurprising:

\begin{lem}\label{induced_map_bar}
A weakly involutive morphism induces a morphism of involutive under-bar and upper-bar complexes. Likewise, a weak homotopy equivalence or weak local map descends to a homotopy equivalence or local map of under-bar and upper-bar complexes.
\end{lem}

\begin{proof} 
Fix an $\mathcal S$-splitting and write all maps as $3\times3$ block matrices.
Let $\wt\lambda$ be weakly involutive. Then there exists $\wt H$ such that
\[
\wt\lambda\wt\tau+\wt\tau'\wt\lambda
=
\wt H\wt d+\wt d'\wt H
\]
holds in every block except possibly the $(2,1)$-block. The complexes $\underline C=\ker\chi$
and $\overline C=\wt C/\ima\chi$
are obtained by keeping blocks not including the $(2,1)$-entry.
Hence the possible error term disappears after restriction to $\underline C$
and after passing to the quotient $\overline C$.
Restricting $\wt H$ accordingly gives the homotopies
\[
\underline\lambda\,\underline\tau+\underline\tau'\,\underline\lambda
=
\underline H\,\underline d+\underline d'\,\underline H
\quad \text{and} \quad
\overline\lambda\,\overline\tau+\overline\tau'\,\overline\lambda
=
\overline H\,\overline d+\overline d'\,\overline H,
\]
as desired.
\end{proof}

A morphism in the weak setting is almost the same as simultaneously having morphisms of $\un{C}$- and $\ov{C}$-complexes. (Likewise for a weak homotopy equivalence or weak local map.) However, a weak morphism is slightly stronger: given a morphism of $\un{C}$-complexes and $\ov{C}$-complexes separately, there is no particular reason to believe these glue together to give a morphism on $\smash{\wt{C}}$, even if we only require homotopy commutation outside the $(2, 1)$-component.

\subsection{Geometric constructions for strongly invertible knots}\label{sec:geometricinvolutive}

We now show that given a based strongly invertible knot we obtain an involutive $\cS$-complex, and that an equivariant negative definite pair gives rise to a (weakly) involutive $\cS$-morphism. See \cite{ADMT} for the analogous construction in the setting of $3$-manifolds.

\subsubsection{The chain complex} Let $(Y, K, \tau)$ be a strongly invertible knot and $x_0$ be one of the two fixed points of $\tau$ on $K$. We construct a homotopy involution on $\smash{\wt{C}(Y, K, x_0)}$ as follows.

\begin{defn}
Let $(W_\tau, S_\tau)$ be the product cobordism $[0, 1] \times (Y, K)$, but with the identification of the outgoing end twisted by $\tau$. We fix the path $[0, 1] \times \{x_0\}$ on this cobordism and denote the corresponding cobordism map by
\[
\wt{\lambda}_{(W_\tau, S_\tau)} \colon \wt{C}(Y, K, x_0) \rightarrow \wt{C}(Y, \tau(K), x_0) = \wt{C}(Y, K^r, x_0).
\]
We construct $\wt{\tau}$ by postcomposing $\wt{\lambda}_{(W_\tau, S_\tau)}$ with $\alpha$, so that 
\[
\wt{\tau} = \alpha \circ \wt{\lambda}_{(W_\tau, S_\tau)} \colon \wt{C}(Y, K, x_0) \rightarrow \wt{C}(Y, K, x_0).
\]
\end{defn}

As the definition of $\wt{\tau}$ may be somewhat opaque, we describe it here with an explicit choice of perturbation data. Fix a holonomy perturbation $\pi$ for $(Y, K)$. Equip the usual identity cobordism $[0, 1] \times (Y, K)$ with a perturbation $\tilde{\pi}$ which interpolates between $\pi$ on the incoming end and the pushforward perturbation $\tau(\pi)$ on the outgoing end. This induces a cobordism map
\[
\Phi \colon \wt{C}(Y, K, x_0, \pi) \rightarrow \wt{C}(Y, K, x_0, \tau(\pi)).
\]
Compose this with the tautological pushforward isomorphism induced by $\tau$
\[
t \colon \wt{C}(Y, K, x_0, \tau(\pi)) \xrightarrow{\cong} \wt{C}(Y, \tau(K), x_0, \pi) = \wt{C}(Y, K^r, x_0, \pi).
\]
This gives the cobordism map $\wt{\lambda}_{(W_\tau, S_\tau)}$, which has been computed with respect to the perturbation $\pi$ on the incoming end $(Y, K)$ and the same perturbation $\pi$ on the outgoing end $(Y, K^r)$. Note that the perturbation $\pi$ on $(Y, K^r)$ becomes $\tau(\pi)$ after our re-identification of the outgoing end via $\tau$, which is why we have equipped the above cobordism with the interpolating data $\tilde{\pi}$. We then further compose with the isomorphism 
\[
\alpha \colon \wt{C}(Y, K^r, x_0, \pi) \xrightarrow{\cong} \wt{C}(Y, K, x_0, \pi)
\]
discussed in Section~\ref{sec:orientationreversal} to obtain a map from $\wt{C}(Y, K, x_0, \pi)$ to itself.

\begin{rem}
One can alternatively view $\wt{\lambda}_{(W_\tau, S_\tau)}$ by first taking the tautological pushforward
\[
t \colon \wt{C}(Y, K, x_0, \pi) \rightarrow \wt{C}(Y, \tau(K), x_0, \tau(\pi)) = \wt{C}(Y, K^r, x_0, \tau(\pi))
\]
induced by $\tau$, and then composing this with the continuation map 
\[
\Phi \colon \wt{C}(Y, K^r, x_0, \tau(\pi)) \rightarrow \smash{\wt{C}(Y, K^r, x_0, \pi)}
\]
induced by an interpolating choice of perturbation from $\tau(\pi)$ to $\pi$. From this point of view, $\wt{\lambda}_{(W_\tau, S_\tau)}$ is clearly the induced action of the diffeomorphism $\tau$ from $(Y, K)$ to $(Y, K^r)$. However, in subsequent proofs, it will be more convenient for us to think of $\smash{\wt{\lambda}_{(W_\tau, S_\tau)}}$ as a cobordism map with a particular identification of the outgoing end.
\end{rem}

The map $\wt{\tau}$ commutes with $\chi$ since this is true for each factor. Both $\Phi$ and $t$ have $1$ in their $(3, 3)$-entry (as follows from Example~\ref{ex:homologyconcordance}); composing this with $\alpha$ gives the $(3, 3)$-entry of $\wt{\tau}$. Note that the component $\mu^\tau$ depends on the choice of basepoint via our use of the path $[0, 1]\times \{x_0\}$ in $\Phi$.

We now verify that $\wt{\tau}$ is a homotopy involution. This follows from the fact that (up to homotopy) our cobordism maps depend only on a choice of path, but for completeness we give a proof here.

\begin{thm}\label{involutivecomplexproof}
The map 
\[
\wt{\tau} : \wt{C} (Y, K, x_0) \to \wt{C} (Y,K, x_0) 
\]
is a homotopy involution. The homotopy type of the pair $(\wt{C} (Y, K, x_0), \wt{\tau})$ is an invariant of $(Y, K, x_0)$ up to equivariant diffeomorphism.
\end{thm}

\begin{proof}
This is analogous to \cite[Lemma 5.8]{ADMT}. Since the proof is similar, we provide only the most salient points. Using Lemma~\ref{lem:dep_ori}, we have
\[
\wt{\tau}^2  = (\alpha \circ \wt{\lambda}_{({W_\tau}, S_\tau) }) \circ (\alpha \circ \wt{\lambda}_{({W_\tau}, S_\tau) }) = \wt{\lambda}_{({W_\tau}, -S_\tau) }\wt{\lambda}_{({W_\tau}, S_\tau) }.
\]
Hence it suffices to show $\wt{\lambda}_{({W_\tau}, -S_\tau) }\wt{\lambda}_{({W_\tau}, S_\tau) }$ is homotopic to the identity. For this, first observe that we have a diffeomorphism rel boundary
    \[
    (W_\tau, -S_\tau) \circ (W_\tau, S_\tau) \cong [0, 1] \times (Y, K).
    \]
Here, the composition of cobordisms is read from right-to-left, so that the first segment of the cobordism $(W_\tau, S_\tau)$ goes from $(Y, K)$ to $(Y, K^r)$, while the second segment of the cobordism $(W_\tau, -S_\tau)$ goes from $(Y, K^r)$ to $(Y, K)$. The diffeomorphism is given by the identity on $(W_\tau, S_\tau)$ and $\id \times \tau$ on $(W_\tau, -S_\tau)$. In order to see the relation
    \[
        \wt{\lambda}_{(W_\tau, -S_\tau)}\wt{\lambda}_{(W_\tau, S_\tau)} + \id = \wt{d} \wt{H} + \wt{H} \wt{d}
        \]
we fix the following geometric data: 
\begin{itemize}[noitemsep]
\item an orbifold metric $g$ on $(Y, K)$, 
\item a nondegenerate regular perturbation $\pi$,
\item an orbifold Riemannian metric $\tilde{g}$ on $(W^+_\tau, S^+_\tau )$ which concides $dt^2 + g$ on the ends,
\item a regular perturbation $\tilde{\pi} $ of instantons over $(W_\tau, \pm S_\tau)$. 
\end{itemize}
%This gives the moduli space $M_{z} ([\a], (W_\tau, S_\tau), [\b], \tilde{\pi})$ whose counts is $\mathcal{S}$-homotopy equivalent to the map $\wt{\lambda}_{(W_\tau, S_\tau)}$. 
The $\mathcal{S}$-homotopy class of the map $\wt{\lambda}_{(W_\tau, -S_\tau)}\wt{\lambda}_{(W_\tau, S_\tau)}$ can be computed from counting the moduli space 
\[
M_{z} ([\a], (W_\tau \circ W_\tau, -S_\tau\circ S_\tau ), [\b], \tilde{\pi} \circ \tilde{\pi}).
\]
Consider a 1-parameter family of perturbations and orbifold Riemannian metrics 
\[
\{\tilde{\tau}_s\}_{s \in [0,1]} \text{ and } \{\tilde{g}_s\}_{s \in [0,1]} \text{ for } (W_\tau, -S_\tau) \circ (W_\tau, S_\tau) \cong [0,1]\times (Y, K)
\]
satisfying the following conditions: 
\begin{itemize}[noitemsep]
    \item $\tilde{\pi}_0 = \tilde{\pi} \circ \tilde{\pi} $, while $\tilde{\tau}_1 = p^* \pi$ is the product perturbation on $(W_\tau, -S_\tau) \circ (W_\tau, S_\tau) \cong [0,1]\times (Y, K)$. 
    \item $\tilde{g}_0 = \tilde{g}$, while $\tilde{g}_1 = dt^2  + g $ is the product metric on $(W_\tau, -S_\tau) \circ (W_\tau, S_\tau) \cong [0,1]\times (Y, K)$. 
 
    \item The parametrized moduli space 
    \begin{align}\label{family moduli}
  {\bf M}_{z} ([\a], (W_\tau \circ W_\tau, -S_\tau\circ S_\tau ), [\b], \tilde{\pi}_s) :=   \bigcup_{s \in [0,1]} M_{z} ([\a], (W_\tau \circ W_\tau, -S_\tau\circ S_\tau ), [\b], \tilde{\pi}_s)
       \end{align}
    over $[0,1]$ is regular.
\end{itemize}
Note that the moduli space $M_{z} ([\a], (W_\tau \circ W_\tau, -S_\tau\circ S_\tau ), [\b], \tilde{\pi}_1)_0 $ consists of constant solutions. Then the desired homotopy can be obtained as a suitable count of \eqref{family moduli} and the homotopy relation is obtained by studying the 1-dimensional components of \eqref{family moduli}. 
\end{proof}

\subsubsection{Concordance inverses}
For the sake of completeness, we record the fact that Theorem~\ref{thm:concinv} also holds in the involutive category. The following is easy to see.

\begin{thm}\label{thm:concinveq}
For each Morse--Smale perturbation $\pi$ for $(Y,K)$, we have an isomorphism of pairs
\[
(\wt{C} (Y^r , K^r, x_0, \pi), \wt{\tau}) \cong (\wt{C} (Y , K, x_0, \pi ), \wt{\tau})^\dagger
\]
\end{thm}

In the case that $Y = S^3$, we of course identify $(Y^r, K^r, x_0)$ with $(S^3, -K, x_0)$. However, we remind the reader that the concordance inverse of a strongly invertible knot is only defined in the presence of a direction. Theorem~\ref{thm:concinveq} should thus be understood at face value: it computes the involutive complex of the concordance inverse \textit{with respect to a particular basepoint}. One can check that this basepoint does not have a consistent relation with the choice of direction under taking concordance inverses. (For example, if we choose $x_0$ to be the endpoint of a particular direction on $K$, then $x_0$ will \textit{not} be the endpoint of the corresponding direction on $-K$ obtained as the concordance inverse.) Up to weak homotopy equivalence, of course, there is no such distinction.

\subsubsection{Cobordism maps}
Next, we consider involutive cobordism maps.

\begin{defn}\label{def:eqndt}
Let $(W, S)$ be a negative definite pair from $(Y, K, \tau)$ to $(Y', K', \tau')$. We say that $(W, S)$ is \textit{equivariant} if there exists self-diffeomorphism $\tau_W \colon W \rightarrow W$ such that:
\begin{enumerate}[noitemsep]
    \item $\tau_W|_Y = \tau$ and $\tau_W|_{Y'} = \tau'$; and 
    \item $\tau_W(S)$ is isotopic to $S$ rel boundary. 
\end{enumerate}
\end{defn}

Note also that if $\tau_W$ is isotopic to $S$ rel boundary, then composing $\tau_W$ with the endpoint of this isotopy gives a self-diffeomorphism of $\tau_W$ that still satisfies both conditions above, but now fixes $S$ setwise. Thus, without loss of generality, we can assume in Definition~\ref{def:eqndt} that $\tau_W(S) = S$. 

\begin{thm} \label{nagative definte cobordism induces}
Let $(W, S)$ be an equivariant negative definite pair from $(Y, K, \tau)$ to $(Y', K', \tau')$. Choose basepoints $x_0$ and $x_0'$ for $K$ and $K'$, and let $\gamma$ be a properly embedded path in $S$ from $x_0$ to $x_0'$. Then the induced cobordism map 
\[
\wt{\lambda}_{(W,S)} \colon \wt{C}(Y, K, \tau, x_0) \rightarrow \wt{C}(Y', K', \tau', x_0')
\]
is weakly involutive.
If $\gamma$ and $\tau_W(\gamma)$ are homotopic rel boundary in $S$, then $\wt{\lambda}_{(W,S)}$ is involutive. 
\end{thm}

\begin{proof}
This is analogous to \cite[Lemma 5.13]{ADMT}. Since the proof is similar, we provide only the most salient points. We first suppose that $\gamma$ and $\tau_W(\gamma)$ are homotopic rel boundary. As in the proof of \cref{involutivecomplexproof}, after applying Lemma~\ref{lem:dep_ori} it suffices to prove
\[
\wt{\lambda}_{(W,-S)} \wt{\lambda}_{(W_\tau, S_\tau)} \simeq \wt{\lambda}_{(W_{\tau'}, S_{\tau'})} \wt{\lambda}_{(W,S)}.
\]
For this, observe that we have an orientation-preserving diffeomorphism rel boundary
    \[
    (W, -S) \circ (W_\tau, S_\tau) \cong (W_{\tau'}, S_{\tau'}) \circ (W, S).  
    \]
by using the extension $\tau_W$ on $(W, -S)$. We now construct a homotopy $\smash{\wt{H}}$ such that
\[
\wt{\lambda}_{(W,-S)} \wt{\lambda}_{(W_\tau, S_\tau)} + \wt{\lambda}_{(W_{\tau'}, S_{\tau'})} \wt{\lambda}_{(W,S)}= \wt{d}' \wt{H} + \wt{H} \wt{d}
\]
as follows. Fix the data:
    %data to describe the maps $\smash{\wt{\lambda}_{(W, S)}}$, $\smash{\wt{\lambda}_{(W_\tau, S_\tau)}}$, and $\smash{\wt{\lambda}_{(W_{\tau'}, S_{\tau'})}}$: 
    \begin{itemize}[noitemsep]
\item orbifold Riemannian metrics $g$ and $g'$ for $(Y, K)$ and $(Y', K')$, 
\item an orbifold Riemannian metric $\tilde{g}$ for $(W^+, S^+)$
which coincides with $dt^2 + g$ and $dt^2+ g'$ on the ends, 
\item Morse--Smale holonomy perturbations $\pi$ and $\pi'$ of the Chern--Simons functional of $(Y, K)$ and $(Y', K')$, 
\item a regular perturbation $\tilde{\pi}$ of instantons over $(W, \pm S)$, 
\item a regular perturbation $\tilde{\pi}_\tau$ of instantons over $(W_\tau, S_\tau)$,
\item a regular perturbation $\tilde{\pi}_{\tau'}$ of instantons over $(W_{\tau'}, S_{\tau'})$. 
\end{itemize}
Consider a 1-parameter family of perturbations and orbifold Riemannian metrics 
\[
\{\tilde{\pi}_s\}_{s\in [0,1]}  \text{ and } \{\tilde{g}_s\}_{s\in [0,1]}
\text{ for }
    (W, -S) \circ (W_\tau, S_\tau) \cong (W_{\tau'}, S_{\tau'}) \circ (W, S)
    \]
such that the following conditions are satisfied: 
\begin{itemize}[noitemsep]
    \item $\tilde{\pi}_0 = \tilde{\pi} \circ \tilde{\pi} _\tau$  and $\tilde{\pi}_1 = \tilde{\pi} _{\tau'} \circ \tilde{\pi}$
    \item 
    $\tilde{g}_0 = \tilde{g} \circ \tilde{g}_\tau$ and   $\tilde{g}_1 = \tilde{g}_{\tau'} \circ \tilde{g}$
    \item The families moduli space 
    \begin{align}\label{family_bordism}
   \bigcup_{s\in [0,1] } M_z( [\a],  (W_\tau, S_\tau )\circ (W, S), [\b],\tilde{\pi}_s, \tilde{g}_s )
        \end{align}
    is regular. 
\end{itemize}
Then the desired homotopy can be obtained as a suitable count of \eqref{family_bordism} and the homotopy relation is obtained by studying the 1-dimensional components of \eqref{family_bordism}. This completes the construction of $\smash{\wt{H}}$ and shows that the induced cobordism map is involutive. In the case that $\gamma$ is not necessarily isotopic to $\tau_W(\gamma)$, we simply fail to obtain the requisite identity in the $(2, 1)$-block; in this case, we conclude that the cobordism map is weakly involutive.
\end{proof}

We have the following special cases:

\begin{lem}
The weak homotopy equivalence class of $\wt{C}(Y, K, \tau)$ does not depend on the choice of basepoint.
\end{lem}
\begin{proof}
Let $x_0$ and $x_1$ be the two fixed points of $\tau$ on $K$. Consider the identity cobordism, equipped with a path $\gamma$ from $\{0\} \times x_0$ to $\{1\} \times x_1$. This cobordism has the obvious involution $\tau \times \id$, under which $\gamma$ is \textit{not} invariant. Nevertheless, Theorem~\ref{nagative definte cobordism induces} shows that the resulting cobordism map (which is a homotopy equivalence, as we saw in Section~\ref{sec:dobap}) is weakly involutive.
\end{proof}

\begin{lem}
Suppose $(Y, K, \tau)$ and $(Y', K', \tau')$ are isotopy-equivariantly concordant in a $\Z_2$-homology cylinder $(W, S, \tau_W)$. Then:
\begin{enumerate}[noitemsep]
    \item For any choice of basepoints, $\smash{(\wt{C}(Y, K), \wt{\tau})}$ and $\smash{(\wt{C}(Y', K'), \wt{\tau}')}$ are weakly locally equivalent.
    \item If we can find basepoints $x_0$ and $x_0'$ for $K$ and $K'$, together with a path $\gamma$ from $x_0$ to $x_0'$ which is isotopic rel boundary to $\tau_W(\gamma)$ in $S$, then $\smash{(\wt{C}(Y, K, x_0), \wt{\tau})}$ and $\smash{(\wt{C}(Y', K', x_0'), \wt{\tau}')}$ are locally equivalent.
\end{enumerate}
\end{lem}

\begin{proof}
From Example~\ref{ex:homologyconcordance}, we have seen that a concordance in a $\Z_2$-homology cylinder induces a local map. Applying Theorem~\ref{nagative definte cobordism induces} completes the proof.
\end{proof}

\subsection{The local equivalence set}\label{sec:localequivalenceset}

We now discuss the local equivalence sets from the introduction. We first recall the non-equivariant construction from \cite{DS19}.

\begin{defn}
Define the \textit{(non-equivariant) local equivalence group} by
\[
\Theta^\mathcal{S}  = \{\textrm{all } \mathcal{S}\text{-complexes}\}/\sim,
\]
where we write $\wt{C} \sim \wt{C}'$ if $\wt{C}$ and $\wt{C}'$ are locally equivalent in the sense of Definition~\ref{height i morphism}. 
We sometimes write $\Theta^\cS_R$ to emphasize the role of the coefficient ring $R$. This has a group structure, where the group operation is given by the tensor product of $\cS$-complexes, the identity is the trivial $\cS$-complex, and inverses are given by dualizing.
\end{defn}

In \cite{DS19} it is shown that the map $(Y,K)\mapsto h(Y, K) = [\widetilde C(Y,K)]$ induces a group homomorphism 
\begin{equation}\label{eq:basicloceqhom}
	\Theta^{3,1}_\Z \to \Theta^\mathcal{S}
\end{equation}
where $\Theta^{3,1}_\Z$ is the homology concordance group of knots in integer homology 3-spheres. In view of the path-dependence of cobordism maps, one may wonder whether the homology concordance group here is required to be based (e.g., come equipped with basepoints and a choice of path along each cobordism). However, since homotopy equivalence is a special case of local equivalence, we certainly have that the local equivalence class of $(Y, K)$ is an invariant of $(Y, K)$ independent of the choice of basepoint. Moreover, even though the choice of $\gamma$ alters the cobordism map, in the non-equivariant case we have that \textit{any} choice of path $\gamma$ on the homology concordance induces a local equivalence.

We now introduce the equivariant construction.

\begin{defn}\label{loc_eq_group}
Define the \textit{local} and \textit{weak local equivalence sets} by
\begin{align*}
\Theta^\mathcal{S}_{eq}  = \{\text{all involutive}\; \mathcal{S}\text{-complexes}\}/\sim_{eq} \quad \text{and} \quad \Theta^\mathcal{S}_{w.eq}  = \{\text{all involutive}\; \mathcal{S}\text{-complexes}\}/\sim_{w.eq},
\end{align*}
where we write $\wt{C} \sim_{eq} \wt{C}'$ or $\widetilde C\sim_{w.eq} \widetilde C'$ if $\wt{C}$ and $\wt{C}'$ are locally equivalent or weakly locally equivalent, respectively. It is clear that these are equivalence relations. 
\end{defn}

\begin{defn}\label{def:htau}
Define
\[
h_\tau \colon \wt{\mathcal{C}} \rightarrow \Theta^\mathcal{S}_{eq}
\]
by sending
\[
(K, \tau, x_0) \mapsto h_\tau(K, \tau, x_0) = [(\wt{C}(S^3, K, x_0), \wt{\tau})].
\]
Here, we write the domain of $h_\tau$ as $\wt{\mathcal{C}}$ for convenience, but this is rather misleading. Indeed, observe that the elements of $\smash{\wt{\mathcal{C}}}$ are \textit{directed} strongly invertible knots, while the proper input to forming an involutive $\cS$-complex is a \textit{based} strongly invertible knot. We may pass from the data of a direction to a basepoint by (for example) taking the endpoint of the direction. For concreteness, this is how we define $h_\tau$. We do \textit{not} assert that $h_\tau$ is a group homomorphism.\footnote{Indeed, note that taking $x_0$ to be the endpoint of the direction is not actually consistent with the placement of the basepoints in the tensor product formula for connected sums (see the next section) or the procedure for dualizing.} We likewise use the same notation $h_\tau(Y, K, \tau, x_0)$ for the local equivalence class $[(\wt{C}(Y, K, x_0), \wt{\tau})]$.
\end{defn}

We have a natural partial order on $\Theta^\mathcal{S}_{eq}$ and $\Theta^\mathcal{S}_{w.eq}$, where $\wt{C} \leq \wt{C}'$ if there exists an involutive local map (respectively, a weakly involutive local map) from $\smash{\wt{C}}$ to $\smash{\wt{C}'}$. Since every involutive map is weakly involutive, we have a forgetful map compatible with this partial order:
\[
\Theta^\mathcal{S}_{eq} \to \Theta^\mathcal{S}_{w.eq}.
\]
Note further that local equivalence of $\un{C}$- and $\ov{C}$-complexes factors through $\smash{\Theta^\mathcal{S}_{w.eq}}$. While it is possible to show that $\Theta^\mathcal{S}_{eq}$ is a partially ordered abelian group, this turns out to not be very useful for our purposes. Indeed, as we will see in the next section, the connected sum formula only allows us to calculate the weak local equivalence class of a connected sum, even if we have the usual local equivalence classes of each of the factors. We thus do not usually think of either $\smash{\Theta^\mathcal{S}_{eq}}$ or $\smash{\Theta^\mathcal{S}_{w.eq}}$ as groups, and refer to them as the (involutive) local equivalence sets.

\begin{rem}
It is important to observe that $\Theta^\mathcal{S}_{w.eq}$ does {\bf not} have a natural group structure. Indeed, the tensor product of involutive $\cS$-complexes does not even descend to a well-defined operation on the set of weak homotopy classes. While it is tempting to think that altering a $\mu$-component of one of the factors in a tensor product only alters the component $\mu^\otimes$ of the tensor product, we have seen in Section~\ref{sec:tensortaubasis} that this is not the case. Specifically, the $\mu$-components of each factor indeed appear in the calculation of the first column of $\tau^\otimes$, not just $\mu^\otimes$. Thus, the tensor product need not descend to the set of weak local equivalence classes.
\end{rem}

\section{Connected sums}\label{section:Two connected sum formula}

We now establish two partial connected sum formulas. The first is for the connected sum of two strong inversions, while the second is for the flipping involution of Section~\ref{sec:2.5}. In both cases, we closely follow Daemi and Scaduto's argument from the non-equivariant setting \cite{DS19}. This proceeds by showing that the pair of pants cobordism induces a homotopy equivalence between the instanton Floer complex of a disjoint union (which is identified with the tensor product) and the instanton Floer complex of the connected sum. See also the work of Donaldson \cite{Do02} and Fukaya \cite{fukaya1996floer}.

In the equivariant case, the pair of pants admits an obvious involution extending the target involution on the connected sum. This extension restricts to a model involution on the disjoint union whose induced action is straightforward to calculate. One would then hope to argue that the pants cobordism map gives a homotopy equivalence of involutive $\cS$-complexes. However, as in Section~\ref{cobordism map}, the components of this map depend on choosing collections of properly embedded paths.\footnote{Due to the fact that one of the ends is disconnected, this choice enters into the definition of several of the components, not just the $(2, 1)$-block. Indeed, strictly speaking, the pants cobordism map is not a cobordism map in the sense of Section~\ref{cobordism map}, since the ends are not connected.} For the components where these paths can be made equivariant, we are able to verify the required identities to establish homotopy commutation for the corresponding block. However, it turns out that the paths defining the pants cobordism map cannot be made equivariant in the case of the $(2, 1)$-block. As a result, in the involutive category, we are only able to establish weak homotopy equivalence.

\subsection{Connected sum of strong inversions}
Let $(Y, K , \tau)$ and $(Y', K' , \tau')$ be strongly invertible knots. Fix invariant basepoints $x_0$ and $x_0'$ along which to form the equivariant connected sum and let $x^\#$ be either one of the two invariant basepoints on $K \# K'$.

\begin{thm}\label{conn sum}
We have a weak homotopy equivalence  
    \[
    (\wt{C}(Y, K, x_0)\otimes \wt{C}(Y', K', x_0'), \wt\tau\otimes \wt\tau') \simeq_{w.eq.} (\wt{C}(Y\# Y', K\# K', x^\#), \wt{\tau\# \tau'})
    \]
between the tensor product involution and the connected sum involution.
\end{thm}

\begin{rem}
Strictly speaking, the equivariant connected sum requires each factor to be directed, but the involutive instanton complex only depends on a choice of basepoint. Our connected sum formula thus holds regardless of the directions of the factors, so long as the connected sum is taken at $x_0$ and $x_0'$. Note that since we only claim weak homotopy equivalence, the choice of basepoint for $K \# K'$ is irrelevant. Nevertheless, we fix $x^\#$ for concreteness.
\end{rem}

To prove Theorem~\ref{conn sum}, we first review the non-equivariant setup discussed in \cite[Section 6]{DS19}. Throughout, denote
\[
\wt{C} = \wt{C}(Y,K, x_0), \ \wt{C}' = \wt{C}(Y', K', x_0'), \ \wt{C}^\otimes = 
\wt{C} \otimes \wt{C}', \text{ and } \wt{C}^\# = \wt{C}(Y \# Y', K \# K', x^\#).
\]
We also write 
\[
\wt{C}\otimes \wt{C}' = C^\otimes_* \oplus C^\otimes_{*-1} \oplus R \quad \text{and} \quad \wt{C}^\# = {C}^\#_* \oplus {C}^\#_{*-1} \oplus R.
\]
Recall from Section~\ref{sec:tensorbasis} that we have described $C^\otimes_*$ as the direct sum of four kinds of basis elements, which may be thought of as an identification
\[
    C^\otimes_* =  (C\otimes C')_* \oplus (C\otimes C')_{*-1} \oplus C_*  \oplus C_{*}'. 
\]
Let $W$ be the pants cobordism
\[
W: Y \sqcup Y' \to Y\# Y'
\]
obtained by attaching a 4-dimensional 1-handle to $Y \sqcup Y'$. Let $S$ be the surface cobordism 
\[
S: K  \sqcup K' \to K \# K',
\]
obtained by attaching a 2-dimensional 1-handle lying in $W$ to $K  \sqcup K'$.
Denote the reversed cobordism by $W^\dagger:  Y\# Y' \to Y \sqcup Y'$ and reversed surface cobordism by $S^\dagger : K \# K' \to K  \sqcup K'$. In Figure~\ref{fig:paths}, we display a collection of paths on $S$ and $S^\dagger$ which will be used to define the pants cobordism maps. These go between $x_0$, $x_0'$, and $x^\#$, as discussed above.

\begin{figure}[h!]
  \centering
  \includegraphics[width=0.4\linewidth]{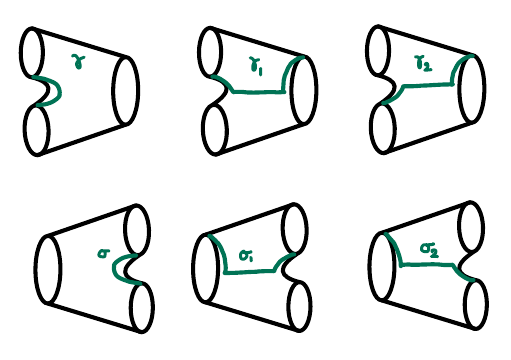}
  \caption{Paths on $(W, S)$ and $(W^\dagger, S^\dagger)$.}
  \label{fig:paths}
\end{figure}

With this in hand, we now define Daemi and Scaduto's connected sum map 
\[
\wt{\lambda} := \wt{\lambda}_{(W, S)}: \wt{C}(Y, K, x_0) \otimes \wt{C}(Y', K', x_0') \to \wt{C}(Y\# Y', K\# K', x^\#).
\]
As usual, it suffices to specify the components
\[
\wt\lambda_{(W, S)} = \left[
\begin{array}{ccc}
\lambda & 0 & 0\\
\mu & \lambda & \Delta_2\\
\Delta_1 & 0 & 1 
\end{array}
\right], 
\]
where 
\begin{itemize}[noitemsep]
    \item $\lambda = [\lambda_1, \lambda_2, \lambda_3, \lambda_4 ] :  C^\otimes_* \to C^\sharp _* $
    \item $\mu = [\mu_1, \mu_2, \mu_3, \mu_4 ]: C^\otimes_{*} \to C^\sharp _{*-1}$
    \item $\Delta_1= [\Delta_{1, 1},\Delta_{1, 2}, \Delta_{1, 3}, \Delta_{1, 4} ]: C^\otimes_{*} \to R$
    \item $\Delta_2: R \to C^\sharp_{*-1}$.
\end{itemize}
A homological computation of $\eta$ shows the $3\times 3$ component of $\wt{\lambda}$ is indeed $1$.  
The other components are defined by counting moduli spaces for $(W, S)$ with certain limiting behaviors and holonomy conditions. These are drawn in Figure~\ref{fig:lambda}. Explicitly, we consider the moduli spaces of the form
\begin{align*}
 \lambda: \quad   &M_\gamma (\alpha, \beta; \gamma),\quad M (\alpha, \beta; \gamma), \quad M (\alpha, \theta; \gamma), \quad M (\theta, \beta; \gamma), \\
\mu : \quad   &M_{\gamma_1\sqcup \gamma_2}  (\alpha, \beta; \gamma),\quad M_{\gamma_1} (\alpha, \beta; \gamma), \quad M_{\gamma_1} (\alpha, \theta; \gamma), \quad M_{\gamma_2} (\theta, \beta; \gamma), \\
  \Delta_1 : \quad   &M_\gamma (\alpha, \beta; \theta ),\quad M (\alpha, \beta; \theta), \quad M (\alpha, \theta ; \theta), \quad M (\theta, \beta; \theta), \\
   \Delta_2: \quad   &M(\theta, \theta ; \gamma  ). 
\end{align*}
Here, $\alpha$, $\beta$, and $\gamma$ are irreducible critical points of the (perturbed) Chern--Simons functionals for $(Y, K), (Y', K')$, and $(Y\#Y', K\# K')$, respectively. A subscript is meant to indicate that we cut down the moduli space with the modified holonomy condition corresponding to the listed curve(s). For instance, consider the components $\lambda_i$ of $\lambda$. We have that $\lambda_1$ is constructed by studying the moduli space on the pair of pants which is cut down by the holonomy along the curve $\gamma$; $\lambda_2$ is constructed by studying the usual pair of pants moduli space; and $\lambda_3$ and $\lambda_4$ are constructed via the moduli spaces with one reducible incoming limit. This corresponds to the first row of Figure~\ref{fig:lambda}. Note that the domains of these maps are consistent with our description of $C^\otimes$: the first two have domains identified (up to grading shift) with $C \otimes C'$, while the last two have domains with one irreducible input. See \cite[Section 6]{DS19}.

\begin{figure}[htbp]
  \centering
  \includegraphics[width=0.4\linewidth]{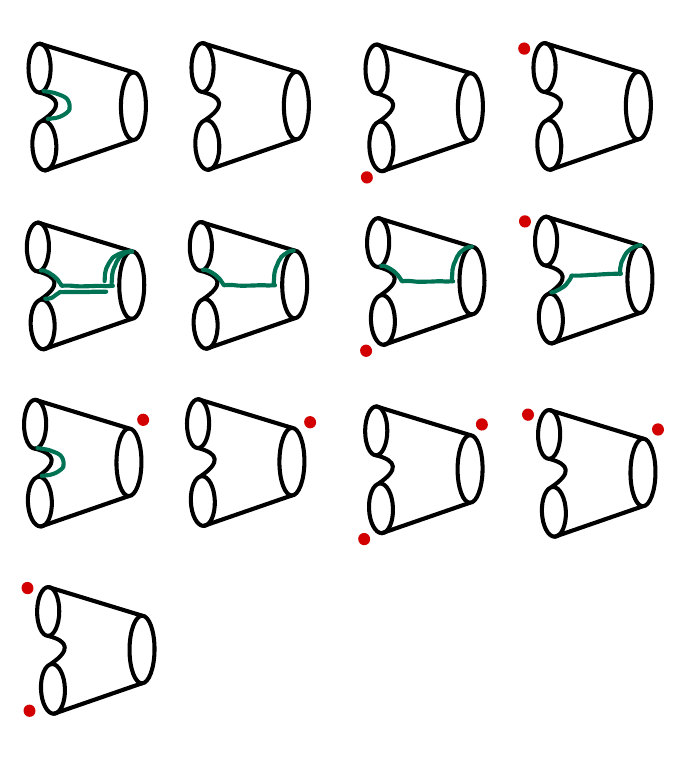}
  \caption{Construction of $\wt{\lambda}$. We cut down our moduli spaces by the holonomy conditions associated to the green curves. A red dot represents a reducible limit at the indicated end. The rows construct $\lambda$, $\mu$, $\Delta_1$, and $\Delta_2$, respectively.}
  \label{fig:lambda}
\end{figure}

%In \cite{DS19}, these maps were shown to give a $\mathcal{S}$-homotopy equivalence, in particular a local map. 
%In Figure \eqref{lambda}, the green paths correspond to paths which define modified holonomies. We use these holonomy maps to cut down the instanton moduli spaces over $(W, S)$. The red dots describe abelian flat critical points. 
%In other words, these maps are written as the counting of 

We likewise have the reversed cobordism map
\[
\wt{\lambda}' := \wt{\lambda}_{(W^\dagger, S^\dagger )}: \wt{C}(Y\# Y', K\# K', x^\#) \to \wt{C}(Y, K, x_0) \otimes \wt{C}(Y', K', x_0').
\]
Once again, it suffices to specify the components
\[
\wt{\lambda}'= \left[
\begin{array}{ccc}
\lambda' & 0 & 0\\
\mu' & \lambda' & \Delta'_2\\
\Delta_1' & 0 & 1 
\end{array}
\right], 
\]
where 
\begin{itemize}[noitemsep]
    \item $\lambda' = [\lambda_1', \lambda_2', \lambda_3', \lambda_4' ]^T :   C^\sharp _* \to C^\otimes_* $
    \item $\mu' = [\mu_1', \mu_2', \mu_3', \mu_4' ]^T : C^\sharp _{*} \to C^\otimes_{*-1} $
    \item $\Delta_1': C^\sharp_{*} \to R$
    \item $\Delta_2'= [\Delta_{2, 1}',\Delta_{2, 2}', \Delta_{2, 3}', \Delta_{2, 4}' ]^T: R \to C^\otimes _{*-1}$.
\end{itemize}
As in the definition of $\wt{\lambda}$, these terms are defined by counting points in the moduli spaces shown in Figure~\ref{fig:lambda'}. In \cite[Section 6]{DS19}, it is shown that $\smash{\wt{\lambda}}$ and $\smash{\wt{\lambda}'}$ are homotopy inverses and thus provide homotopy equivalences between $\smash{\wt{C}}^\otimes$ and $\smash{\wt{C}}^\#$. 

\begin{figure}[htbp]
  \centering
  \includegraphics[width=0.4\linewidth]{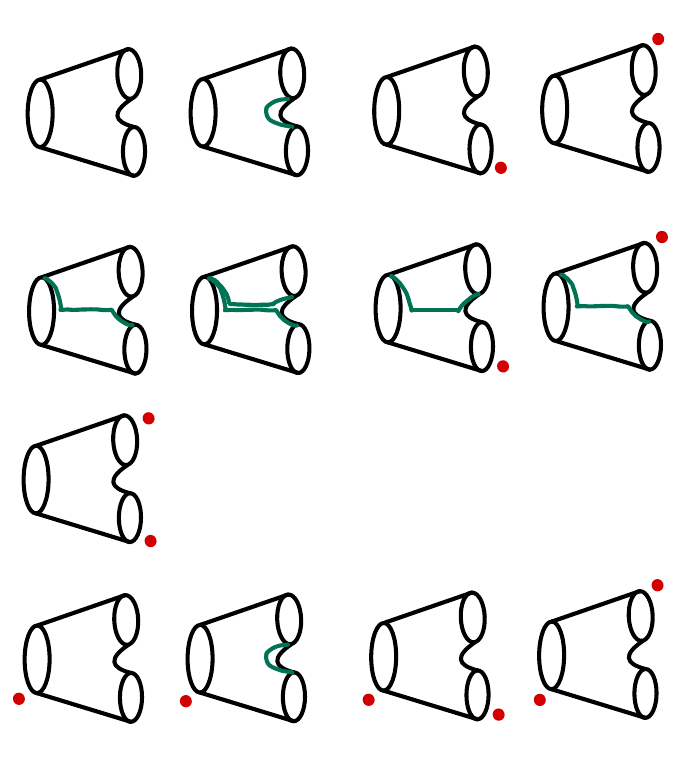}
  \caption{Construction of $\wt{\lambda}$. We cut down our moduli spaces by the holonomy conditions associated to the green curves. A red dot represents a reducible limit at the indicated end. The rows construct $\lambda'$, $\mu'$, $\Delta_1'$, and $\Delta_2'$, respectively.}
  \label{fig:lambda'}
\end{figure}

We now turn to the equivariant setting. The pants cobordism $(W, S)$ clearly admits an involution extending the involutions $\tau \sqcup \tau'$ and $\tau \# \tau'$ on the ends, and likewise for $(W^\dagger, S^\dagger)$. As indicated in Figure~\ref{Ext_inv_pants1}, this fixes the paths $\gamma$ and $\sigma$, but not $\gamma_1$, $\gamma_2$, $\sigma_1$, or $\sigma_2$.

\begin{figure}[htbp]
  \centering
  \includegraphics[
    width=0.3\linewidth
  ]{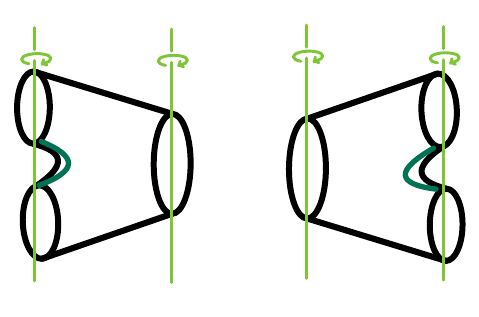}
  \caption{Extensions of $\tau \sqcup \tau'$ and $\tau \# \tau'$.}
  \label{Ext_inv_pants1}
\end{figure}

Consider the induced homotopy involutions 
\[
\wt{\tau}^\otimes := \wt\tau\otimes \wt\tau' \text{ and } \wt{\tau}^\sharp:= \wt{\tau \# \tau'}
\]
acting on $\wt{C} \otimes \wt{C}'$ and $\wt{C}^\#$. For reference, we remind the reader that we have decomposed the action of $\wt{\tau}^\otimes$ into components explicitly in Section~\ref{sec:tensortaubasis}. Letting
\[
\wt{\tau}^\otimes \;=\;
\begin{pmatrix}
\tau^\otimes  & 0 & 0 \\[2pt]
\mu^\otimes & \tau^\otimes & \Delta_2^\otimes \\[2pt]
\Delta_1^\otimes  & 0 & 1
\end{pmatrix},
\]
this is given by
\[
\tau ^\otimes =\begin{pmatrix}
\tau\otimes\tau' & 0 & 0 & 0 \\  \tau \otimes (\mu')^\tau + \mu^\tau \otimes \tau' & \tau\otimes\tau' & \tau \otimes (\Delta'_2)^\tau & \Delta_2^\tau \otimes \tau'  \\
\tau \otimes (\Delta'_1)^\tau & 0 & \tau \otimes 1 & 0 \\
\Delta_1^\tau \otimes \tau' & 0 & 0 & 1 \otimes \tau'
\end{pmatrix},
\]
and
\[
\mu^\otimes =\begin{pmatrix}
\mu^\tau\otimes \tau'  & 0 & 0 & \Delta_2^\tau \otimes \tau' \\
\mu^\tau\otimes (\mu')^\tau & \mu^\tau\otimes \tau'  & \mu^\tau \otimes (\Delta_2')^\tau & \Delta_2^\tau \otimes (\mu')^\tau \\
\mu^\tau\otimes (\Delta_1')^\tau  & 0 & \mu^\tau\otimes 1 & \Delta_2^\tau \otimes (\Delta_1')^\tau \\
 \Delta_1^\tau \otimes (\mu')^\tau  & \Delta_1^\tau \otimes \tau' & \Delta_1^\tau \otimes (\Delta_2')^\tau & 1 \otimes (\mu')^\tau
\end{pmatrix},
\]
and
\[
\Delta_1^\otimes =\begin{pmatrix}
\Delta_1^\tau \otimes (\Delta_1')^\tau & 0 & \Delta_1^\tau \otimes 1  & 1 \otimes (\Delta_1')^\tau 
\end{pmatrix}
\quad \text{and} \quad
\Delta_2^\otimes =\begin{pmatrix}
0 & \Delta_2^\tau \otimes (\Delta_2')^\tau & \Delta_2^\tau \otimes 1  & 1 \otimes (\Delta_2')^\tau 
\end{pmatrix}^T.
\]

\cref{conn sum} reduces to the following.

\begin{thm}\label{thm:conn_sum_key}
    The maps 
    \[
\wt{\lambda} := \wt{\lambda}_{(W, S )}: \wt{C}(Y, K, x_0) \otimes \wt{C}(Y', K', x_0') \to \wt{C}(Y\# Y', K\# K', x^\#) 
\]
and
\[
\wt{\lambda}' := \wt{\lambda}_{(W^\dagger, S^\dagger )}: \wt{C}(Y\# Y', K\# K', x^\#) \to \wt{C}(Y, K, x_0) \otimes \wt{C}(Y', K', x_0') 
 \]
are weakly involutive. 
\end{thm}

\begin{proof}%[Proof of \cref{thm:conn_sum_key}]
 
We first show $\wt{\lambda}$ is weakly involutive. 
We claim there is an $\mathcal{S}$-homotopy 
\[
\wt{H}: \wt{C}(Y, K, x_0) \otimes \wt{C}(Y', K', x_0') \to \wt{C}(Y\# Y', K\# K', x^\#)
\]
such that 
\[
\wt{d}^\# \wt{H} + \wt{H}\wt{d}^\otimes = \wt{\lambda} \circ \wt{\tau}^\otimes +  \wt{\tau}^\sharp \circ \wt{\lambda} + E_{21}, 
\]
where $E_{21}$ is a $(2, 1)$-block error term. 
As usual, it suffices to give the components of the homotopy
\[
\wt{H} = \left[
\begin{array}{ccc}
H & 0 & 0\\
\mu^H & H & \Delta_2^{H}\\
\Delta_1^{H} & 0 & 0
\end{array}
\right], 
\]
where 
\begin{itemize}[noitemsep]
    \item $H = [H_1, H_2, H_3, H_4 ] :  C^\otimes_* \to C^\sharp _* $
    \item $\mu^H = [\mu_1^H, \mu_2^H, \mu_3^H, \mu_4^H ]: C^\otimes_{*} \to C^\sharp _{*-1}$
    \item $\Delta^H_1= [\Delta^H_{1, 1},\Delta^H_{1, 2}, \Delta^H_{1, 3}, \Delta^H_{1, 4} ]: C^\otimes_{*} \to R$
    \item $\Delta_2^H: R \to C^\sharp_{*-1}$.
\end{itemize}
For this, consider the diffeomorphism 
\begin{align*}\label{key-diffeo}
(W, -S) \circ (W_{\tau \sqcup \tau'}, S_{\tau \sqcup \tau'}) \cong (W_{\tau \# \tau'}, S_{\tau \# \tau'}) \circ (W, S)
\end{align*}
constructed using the extension of $\tau \sqcup \tau'$ and $\tau \# \tau'$ over $W$. Note that the curve $\gamma$ is sent to itself by this diffeomorphism. Hence the diffeomorphism identifies the count of 1-dimensional instanton moduli spaces cut down by the modified holonomy map associated to $\gamma$. Denote the above composite cobordism by $(Z, T)$. We fix the following geometric data: 
\begin{itemize}[noitemsep]
\item orbifold Riemannian metrics $g$, $g'$ and $g''$ for $(Y, K)$, $(Y', K')$ and $(Y\# Y', K\# K')$, 
\item an orbifold Riemannian metric $\tilde{g}$ for $(Z^+, T^+)$
which coincides with $dt^2 + g$, $dt^2+ g'$ and $dt^2 + g''$ on the ends, 
\item Morse--Smale holonomy perturbations $\pi$, $\pi'$ and $\pi''$ of the (orbifold) Chern--Simons functional of $(Y, K)$, $(Y', K')$ and $(Y'', K'')$,
\item a regular perturbation $\tilde{\pi}$ of the instantons for $(W, \pm S)$, 
\item regular holonomy perturbations $\tilde{\pi}_\tau$, $\tilde{\pi}_{\tau'}$, and $\tilde{\pi}_{\tau \# \tau'}$ of the instantons over $(W_\tau, S_\tau)$, $(W_{\tau'}, S_{\tau'})$, and $(W_{\tau \# \tau'}, S_{\tau \# \tau'})$, respectively.
\end{itemize}
Consider a 1-parameter family $\{ \wt{\pi}_t \}$ of perturbations for $(Z^+, T^+)$ satisfying 
\begin{itemize}[noitemsep]
    \item $\wt{\pi}_0 = \tilde{\pi} \circ (\tilde{\pi}_{\tau} \sqcup \tilde{\pi}_{\tau'} )$ via $(Z, T)\cong (W, -S) \circ (W_{\tau \sqcup \tau'}, S_{\tau \sqcup \tau'})$ and $\wt{\pi}_1 =\tilde{\pi}_{\tau \# \tau} \circ \tilde{\pi}$ via $(Z, T)\cong (W_{\tau\# \tau'}, S_{\tau\# \tau'}) \circ (W, S)$,
    \item immersed $S^1\times D^3$'s defining holonomy perturbations of $\{ \wt{\pi}_t \}$ are away from the ends of $(Z^+, T^+)$,
    \item the parametrized moduli spaces on $(Z, T)$;  
    \[
    { \bf M}(\alpha, \beta; \gamma):=  \bigcup_{t \in [0,1] } M^{\wt{\pi}_t}(\alpha, \beta; \gamma) \text{ and }      { \bf M} (\alpha, \beta; \gamma):= \bigcup_{t \in [0,1] } M^{\wt{\pi}_t}_{\gamma} (\alpha, \beta; \gamma) 
    \]
    are regular in the parametrized sense
    for critical points $\alpha$, $\beta$ and $\gamma$ of perturbed Chern--Simons functionals.
\end{itemize}
The components of $\wt{H}$ are defined using counts of the moduli spaces
\begin{align*}
  H: \quad   &{\bf M}_\gamma (\alpha, \beta; \gamma),\quad {\bf M} (\alpha, \beta; \gamma), \quad {\bf M} (\alpha, \theta; \gamma), \quad {\bf M} (\theta, \beta; \gamma), \\
\mu^H : \quad   &{\bf M}_{\gamma_1\sqcup \gamma_2}  (\alpha, \beta; \gamma),\quad {\bf M}_{\gamma_1} (\alpha, \beta; \gamma), \quad {\bf M}_{\gamma_1} (\alpha, \theta; \gamma), \quad {\bf M}_{\gamma_2} (\theta, \beta; \gamma), \\
\Delta_1^H : \quad   &{\bf M}_\gamma (\alpha, \beta; \theta ),\quad {\bf M} (\alpha, \beta; \theta), \quad {\bf M} (\alpha, \theta ; \theta), \quad {\bf M} (\theta, \beta; \theta), \\
\Delta_2^H: \quad   &{\bf M}(\theta, \theta ; \gamma  ).
\end{align*}

We now verify the weak homotopy commutation. Expanding the equality
\[
\wt{d}^\# \wt{H} + \wt{H}\wt{d}^\otimes = \wt{\lambda} \circ \wt{\tau}^\otimes +  \wt{\tau}^\sharp \circ \wt{\lambda} + E_{21} 
\]
into components gives the four constituent relations
\begin{align}
  & \tau^\sharp \lambda + \lambda \tau^\otimes
    = H d^\otimes + d^\sharp H
    \label{eq:homotopy-1} \\
  & \mu^\sharp \lambda + \tau^\sharp \mu + \Delta_2^\sharp \Delta_1
    = \mu^H d^\otimes + H v^\otimes + \Delta_2^H \delta_1^\otimes
      + v^\sharp H + d^\sharp \mu^H + \delta_2^\# \Delta_1^H + E_{21} 
    \label{eq:homotopy-2} \\ 
  & \Delta_1^\sharp \lambda + \Delta_1 + \Delta_1 \tau^\otimes + \Delta_1^\otimes
    = \Delta_1^H d^\otimes + \delta_1^\sharp H
    \label{eq:homotopy-3} \\
  & \tau^\sharp + \Delta_2^\sharp + \lambda \Delta_2^\otimes + \Delta_2
    = d^\sharp \Delta_2^H + H \delta_2^\otimes
    \label{eq:homotopy-4}
\end{align}
where $E_{21}$ is a $(2, 1)$-block error term. 
We show each of these except for the second, which may be ignored in the case of a weakly involutive map. 
%We see how to prove these equations.
\begin{comment}
The equation \eqref{eq:homotopy-1} is equivalent to the followings: 
\begin{align}
   & \tau^\sharp \lambda_1
     + \lambda_1 \tau \otimes \tau'
     + \lambda_2 (\tau \otimes \mu' + \mu \otimes \tau')
     + \lambda_3 \tau \otimes \Delta_1'
     + \lambda_4 \Delta_1' \otimes \tau  \notag\\
   & \qquad
     = H_1 (d\otimes 1 + 1\otimes d')
     + H_2 (v\otimes 1 + 1\otimes v)
     + H_3 \delta_1 \otimes 1
     + H_4 1\otimes \delta_1'
     + d^\sharp H_1
     \label{eq:lambda1} \\
   & \tau^\sharp \lambda_2
     + \lambda_2 \tau \otimes \tau'
     = d^\sharp H_2 + H_2 (d\otimes 1 + 1\otimes d')
     \label{eq:lambda2} \\
   & \tau^\sharp \lambda_3
     + \lambda_2 \tau \otimes \Delta_2'
     + \lambda_3 \tau
     = d^\sharp H_3 + H_2 (1\otimes \delta_2') + H_3 d
     \label{eq:lambda3} \\
   & \tau^\sharp \lambda_4
     + \lambda_2 \Delta_2 \otimes \tau'
     + \lambda_4 \tau'
     = d^\sharp H_4 + H_2 (\delta_2 \otimes 1) + H_4 d'
     \label{eq:lambda4}
\end{align}
\end{comment}
Consider the first equation \eqref{eq:homotopy-1}. Expanding this further using our expressions for $d^\otimes$ and $\tau^\otimes$, we obtain
\begin{align}
   & \lambda_1(\tau \otimes \tau') 
   +\tau^\sharp \lambda_1 =  \lambda_2\bigl(\mu^\tau \otimes \tau' + \tau \otimes (\mu')^\tau  \bigr)
     + \lambda_3\bigl(\tau \otimes (\Delta_1')^\tau\bigr)
     + \lambda_4\bigl(\Delta_1^\tau \otimes \tau'\bigr)
      \notag \\
   & \qquad 
    +  H_1(d\otimes 1 + 1\otimes d')
     + H_2(v\otimes 1 + 1\otimes v')
     + H_3(1\otimes \delta_1')
     + H_4(\delta_1\otimes 1)
     + d^\sharp H_1
     \label{eq:lambda1} \\
   & \tau^\sharp \lambda_2
     + \lambda_2(\tau \otimes \tau')
     =
     d^\sharp H_2
     + H_2(d\otimes 1 + 1\otimes d')
     \label{eq:lambda2} \\
   & \tau^\sharp \lambda_3
     + \lambda_2\bigl(\tau \otimes (\Delta_2')^\tau\bigr)
     + \lambda_3(\tau \otimes 1)
     =
     d^\sharp H_3
     + H_2(1\otimes \delta_2')
     + H_3(d\otimes 1)
     \label{eq:lambda3} \\
   & \tau^\sharp \lambda_4
     + \lambda_2\bigl(\Delta_2^\tau \otimes \tau'\bigr)
     + \lambda_4(1\otimes \tau')
     =
     d^\sharp H_4
     + H_2(\delta_2\otimes 1)
     + H_4(1\otimes d')
     \label{eq:lambda4}
\end{align}
Take, for example, \eqref{eq:lambda2}. For irreducible critical points $\alpha, \beta, \gamma$ with $\ind(\alpha) + \ind (\beta) - \ind(\gamma) =0$, the boundary of the compactified 1-parameter family moduli space ${\bf M}^+(\alpha, \beta; \gamma)$ is the union of
\begin{align*}
&M(\alpha, \beta; \gamma') \times M_{\tau\# \tau'} (\gamma', \gamma), \quad (M_{\tau } (\alpha, \alpha') \sqcup M_{\tau' } (\beta, \beta') ) \times M(\alpha', \beta'; \gamma)\\
&\check{M}(\alpha, \alpha') \times {\bf M}(\alpha', \beta; \gamma), \quad \check{M}(\beta, \beta') \times {\bf M}(\alpha, \beta'; \gamma), \quad {\bf M}(\alpha, \beta; \gamma') \times \check{M}(\gamma', \gamma).
\end{align*}
Each of these contributes one of the terms to the desired equality \eqref{eq:lambda2} and thus establishes the identity via the usual argument. The relations \eqref{eq:lambda1}, \eqref{eq:lambda3}, and \eqref{eq:lambda4} are similarly obtained by considering the boundaries of the moduli spaces
\[
{\bf M}^+_\gamma (\alpha, \beta; \gamma), \quad {\bf M}^+ (\alpha, \theta; \gamma), \quad \text{and} \quad {\bf M}^+ (\theta, \beta; \gamma),
\]
respectively. The identification of these boundaries follows from a standard gluing argument. We sketch the boundary degenerations leading to \eqref{eq:lambda1} through \eqref{eq:lambda4} in Figure~\ref{fig:homotopy-1}. There: 
\begin{itemize}[noitemsep]
    \item A red dot represents a reducible limit at the indicated end.
    \item A cobordism with no markings represents the $0$-dimensional moduli space on that cobordism. In the case of the cylinder, this is the differential. (Likewise, $\delta_1$ or $\delta_2$ if one end is reducible.) 
    \item A cobordism with a green curve represents the $1$-dimensional moduli space on that cobordism, cut down by the holonomy condition associated to the curve. In the case of a cylinder, this is the $v$-map.
    \item A cobordism decorated by an $H$ represents the $0$-dimensional moduli space constructed using our $1$-parameter family of perturbations. This the moduli space used to define the homotopy.
    \item A cylinder decorated by $\tau$ represents the $0$-dimensional moduli space on the product cobordism with end twisted by $\tau$. Up to composing with $\alpha$, this is the component $\tau$ from the construction of $\wt{\tau}$. (Likwise, $\Delta_1^\tau$ or $\Delta_2^\tau$ if one end is reducible; or the identity component if both ends are reducible.)
    \item A cylinder with a purple line represents the $1$-dimensional moduli space on the product cobordism with end twisted by $\tau$, cut down by the holonomy condition associated to the line. Up to composing with $\alpha$, this is the $\mu^\tau$-map.
\end{itemize}
Strictly speaking, the relations \eqref{eq:lambda1} through \eqref{eq:lambda4} have factors of $\alpha$ coming from the definition of $\smash{\wt{\tau}}$. However, since $\alpha$ commutes with every cobordism map, we may reduce these relations to identities only involving compositions of cobordism maps.

\begin{figure}[h!]
  \centering
  \includegraphics[width=0.8\linewidth]{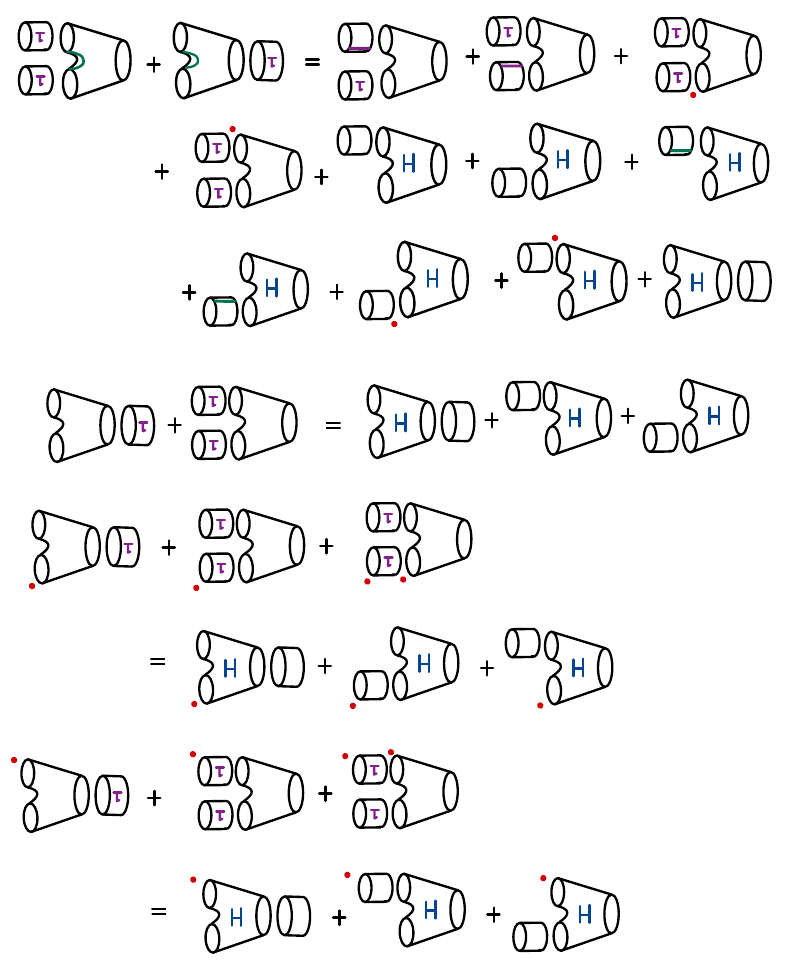}
  \caption{Boundary degenerations leading to \eqref{eq:lambda1} through \eqref{eq:lambda4}.}
  \label{fig:homotopy-1} 
\end{figure}

\begin{comment}
Next, we treat \eqref{eq:homotopy-3}, which is equivalent to 
\begin{align}
  & \Delta_1^\sharp \lambda_1
    + \Delta_{1,1}
    + \Delta_{1,1} \tau \otimes \tau'
    + \Delta_{1,2} (\tau \otimes \mu' + \mu \otimes \tau)
    + \Delta_{1,3} \tau \otimes \Delta_1'
    + \Delta_{1,4} \Delta_1' \otimes \tau
    \notag\\
  & \qquad
    = \Delta_{1,1}^H
    + \Delta_{1,2}^H (v\otimes 1 + 1\otimes v)
    + \Delta_{1,3} \delta_1 \otimes 1
    + \Delta_{1,4}^H 1 \otimes \delta_1'
    + \delta_1^\sharp H_1
    \label{eq:Delta1-1} \\
  & \Delta_1^\sharp \lambda_2
    + \Delta_{1,2}
    + \Delta_{1,2} \tau \otimes \tau'
    = \Delta_{1,2}^H (d\otimes 1 + 1\otimes d')
    + \delta_1^\sharp H_2
    \label{eq:Delta1-2} \\
  & \Delta_1^\sharp \lambda_3
    + \Delta_{1,3}
    + \Delta_{1,2} \tau \otimes \Delta_2'
    + \Delta_{1,3} \tau
    + \Delta_1^\tau \otimes 1
    = \Delta_{1,2}^H (1\otimes \delta_2')
    + \Delta_{1,3}^H d
    + \delta_1^\sharp H_3
    \label{eq:Delta1-3} \\
  & \Delta_1^\sharp \lambda_4
    + \Delta_{1,4}
    + \Delta_{1,2} \Delta_2 \otimes \tau'
    + \Delta_{1,4} \tau'
    + 1\otimes \Delta_1^{\tau'}
    = \Delta_{1,2}^H (\delta_2 \otimes 1)
    + \Delta_{1,4}^H d'
    + \delta_1^\sharp H_4
    \label{eq:Delta1-4}
\end{align}

\end{comment}
We now turn to the component equation \eqref{eq:homotopy-3}. Expanding this gives the following four identities:
\begin{align}
& 
 \Delta_{1,1}(\tau\otimes\tau') 
 +\Delta_1^\sharp \lambda_1
   + \Delta_{1,1}
  =  \Delta_{1,2}\bigl(\mu^\tau\otimes\tau'+ \tau\otimes(\mu')^\tau\bigr)
  + \Delta_{1,3}\bigl(\tau\otimes(\Delta_1')^\tau\bigr)
  + \Delta_{1,4}\bigl(\Delta_1^\tau\otimes\tau'\bigr)
\notag\\
  &\qquad
  + \Delta_1^\tau\otimes(\Delta_1')^\tau  +
  \Delta_{1,1}^H(d\otimes 1+1\otimes d')
  + \Delta_{1,2}^H(v\otimes 1+1\otimes v') 
+ \Delta_{1,3}^H(1\otimes\delta_1')
  \notag  \\
&\qquad
  + \Delta_{1,4}^H(\delta_1\otimes 1)
  + \delta_1^\sharp H_1
  \label{eq:Delta1-1} \\
& \Delta_1^\sharp \lambda_2
  + \Delta_{1,2}
  + \Delta_{1,2}(\tau\otimes\tau')
  =
  \Delta_{1,2}^H(d\otimes 1+1\otimes d')
  + \delta_1^\sharp H_2
  \label{eq:Delta1-2} \\
& \Delta_1^\sharp \lambda_3
  + \Delta_{1,3}
  + \Delta_{1,2}\bigl(\tau\otimes(\Delta_2')^\tau\bigr)
  + \Delta_{1,3}(\tau\otimes 1)
  + \Delta_1^\tau\otimes 1
  =
  \Delta_{1,2}^H(1\otimes\delta_2')
  + \Delta_{1,3}^H(d\otimes 1)
  + \delta_1^\sharp H_3
  \label{eq:Delta1-3} \\
& \Delta_1^\sharp \lambda_4
  + \Delta_{1,4}
  + \Delta_{1,2}\bigl(\Delta_2^\tau\otimes\tau'\bigr)
  + \Delta_{1,4}(1\otimes\tau')
  + 1\otimes(\Delta_1')^\tau
  =
  \Delta_{1,2}^H(\delta_2\otimes 1)
  + \Delta_{1,4}^H(1\otimes d')
  + \delta_1^\sharp H_4 \label{eq:Delta1-4}
\end{align}
Each of these comes from analyzing the boundaries of the compactified 1-dimensional moduli spaces 
\[
{\bf M}^+_\gamma (\alpha, \beta; \theta), \quad {\bf M}^+ (\alpha, \beta; \theta) ,\quad  {\bf M}^+ (\alpha, \theta; \theta), \quad \text{and} \quad {\bf M}^+ (\theta, \beta; \theta),
\]
respectively. We show the topological degenerations leading to \eqref{eq:Delta1-1} through \eqref{eq:Delta1-4} in Figure~\ref{fig:homotopy-3}.

\begin{figure}[h!]
  \centering
\includegraphics[width=0.8\linewidth]{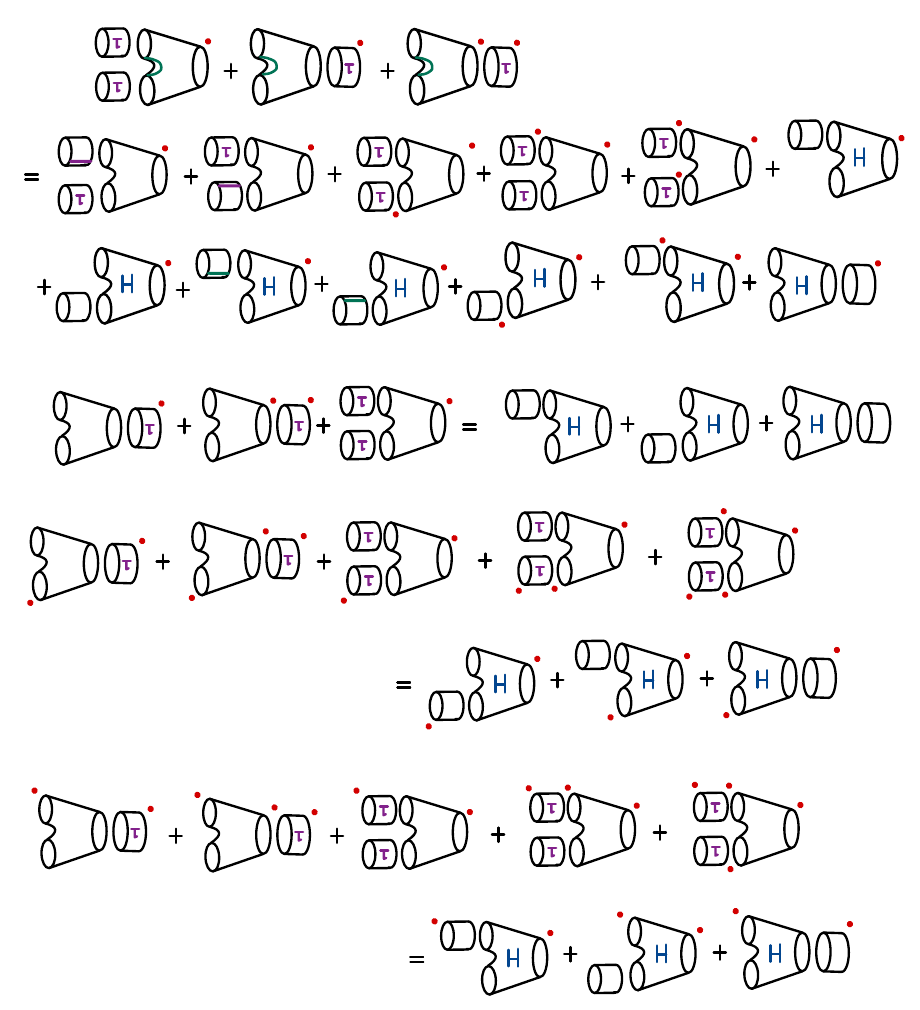}
  \caption{Boundary degenerations leading to \eqref{eq:Delta1-1} through \eqref{eq:Delta1-4}.}
  \label{fig:homotopy-3}
\end{figure}

Finally, consider the component equation \eqref{eq:homotopy-4}. This is equivalent to
\begin{align}
& \tau^\sharp \Delta_2
  + \Delta_2^\sharp
  + \Delta_2
  + \lambda_2\bigl(\Delta_2^\tau\otimes(\Delta_2')^\tau\bigr)
  + \lambda_3\bigl(\Delta_2^\tau\otimes 1\bigr)
  + \lambda_4\bigl(1\otimes(\Delta_2')^\tau\bigr)
  \notag\\
&\qquad
  =
  d^\sharp \Delta_2^H
  + H_3(\delta_2\otimes 1)
  + H_4(1\otimes\delta_2').
  \label{eq:Delta2}
\end{align}
The proof of this is similar to \eqref{eq:homotopy-3}, and we leave it to the reader. This shows $\smash{\wt{\lambda}}$ is weakly involutive.

We now prove that the map $\wt{\lambda}'$ in the other direction is weakly involutive. This means we claim there is an $\cS$-homotopy
\[
\wt{G}: \wt{C}(Y\# Y', K\# K', x^\#) \to \wt{C}(Y, K, x_0) \otimes \wt{C}(Y', K', x_0') 
\]
such that 
\[
\wt{d}^\otimes  \wt{G} + \wt{G}\wt{d}^\# = \wt{\lambda}' \circ \wt{\tau}^\sharp +  \wt{\tau}^\otimes  \circ \wt{\lambda}' + E_{21} . 
\]
As usual, it suffices to give the components of the homotopy
\[
\wt{G} = \left[
\begin{array}{ccc}
G & 0 & 0\\
\mu^G & G & \Delta_2^{G}\\
\Delta_1^{G} & 0 & 0
\end{array}
\right], 
\]
where 
\begin{itemize}[noitemsep]
    \item $G = [G_1, G_2, G_3, G_4 ]^T :  C^\sharp _* \to C^\otimes_*  $
    \item $\mu^G = [\mu_1^G, \mu_2^G, \mu_3^G, \mu_4^G ]^T: C^\sharp _{*} \to C^\otimes_{*-1}$
    \item $\Delta^G_1: C^\sharp_{*} \to R$
    \item $\Delta_2^G  = [\Delta^G_{2, 1},\Delta^G_{2, 2}, \Delta^G_{2, 3}, \Delta^G_{2, 4} ] ^T: R \to C^\otimes_{*-1}$.
\end{itemize}
The construction is similar to that of $\wt{H}$. For an analogous choice of 1-parameter family of perturbations, we define these maps by counting the moduli spaces
\begin{align*}
 G: \quad   &{\bf M} (\alpha; \beta,  \gamma),\quad {\bf M}_\sigma (\alpha; \beta, \gamma), \quad {\bf M} (\alpha; \theta, \gamma), \quad {\bf M} (\alpha; \beta, \theta), \\
\mu^G : \quad   &{\bf M}_{\sigma_1}  (\alpha; \beta,  \gamma),\quad {\bf M}_{\sigma_1\sqcup \sigma_2} (\alpha;  \beta,  \gamma), \quad {\bf M}_{\sigma_1} (\alpha; \theta, \gamma), \quad {\bf M}_{\sigma_2} (\alpha; \beta, \theta), \\
\Delta_1^G : \quad   &{\bf M}(\alpha;  \theta ; \theta  )  \\
\Delta_2^G: \quad   &{\bf M} (\theta; \beta,  \gamma ),\quad {\bf M}_\sigma (\theta; \beta,  \gamma  ), \quad {\bf M} (\theta; \theta,  \gamma ), \quad {\bf M} (\theta; \beta,  \theta ).
\end{align*}
Expanding the weak homotopy commutation into components gives the four constitutent relations
\begin{align}
  & \tau^\otimes \lambda' + \lambda' \tau^\sharp
    = G d^\sharp + d^\otimes G
    \label{eq:G-homotopy-1} \\
  & \mu^\otimes \lambda' + \tau^\otimes \mu + \Delta_2^\otimes \Delta'_1
    + \mu \tau^\sharp + \lambda' \mu^\sharp + \Delta_2' \Delta_1^\sharp
    \notag\\
  &\hspace{6.2em}
    = \mu^G d^\sharp + G v^\sharp + \Delta_2^{G}\delta_1^\sharp
      + v^\otimes G + d^\otimes \mu^G + \delta_2^\otimes \Delta_1^{G}
      \;+\; E_{21} 
    \label{eq:G-homotopy-2} \\ 
  &\Delta_1' + \Delta_1^\sharp +\Delta_1^\otimes \lambda' + \Delta_1 \tau^\sharp
    = \Delta_1^{G} d^\sharp + \delta_1^\otimes G
    \label{eq:G-homotopy-3} \\
  & \Delta_2' + \Delta_2^\otimes  + \tau^\otimes \Delta_2' + \lambda' \Delta_2^\sharp
    = d^\otimes \Delta_2^{G} + G \delta_2^\sharp.
    \label{eq:G-homotopy-4}
\end{align}
As before, we do not need to verify the identity \eqref{eq:G-homotopy-2}, as we only claim our map is weakly involutive. For the sake of completeness, we expand each of these further. The identity \eqref{eq:G-homotopy-1} is equivalent to
\begin{align*}
    & (\tau\otimes\tau')\lambda_1'
  + \lambda_1'\tau^\sharp
  =
  (d\otimes 1+1\otimes d')G_1
  + G_1d^\sharp,
   \\
& \bigl(\tau\otimes(\mu')^\tau+\mu^\tau\otimes\tau'\bigr)\lambda_1'
  +(\tau\otimes\tau')\lambda_2'
  +\bigl(\tau\otimes(\Delta_2')^\tau\bigr)\lambda_3'
  +\bigl(\Delta_2^\tau\otimes\tau'\bigr)\lambda_4'
  +\lambda_2'\tau^\sharp
  \notag\\
&\qquad
  =
  (v\otimes 1+1\otimes v')G_1
  +(d\otimes 1+1\otimes d')G_2
  +(1\otimes\delta_2')G_3
  +(\delta_2\otimes 1)G_4
  +G_2d^\sharp,
  \\
& \bigl(\tau\otimes(\Delta_1')^\tau\bigr)\lambda_1'
  +(\tau\otimes 1)\lambda_3'
  +\lambda_3'\tau^\sharp
  =
  (1\otimes\delta_1')G_1
  +(d\otimes 1)G_3
  +G_3d^\sharp,
  \\
& \bigl(\Delta_1^\tau\otimes\tau'\bigr)\lambda_1'
  +(1\otimes\tau')\lambda_4'
  +\lambda_4'\tau^\sharp
  =
  (\delta_1\otimes 1)G_1
  +(1\otimes d')G_4
  +G_4d^\sharp,
\end{align*}
while the identity \eqref{eq:G-homotopy-3} is equivalent to
\begin{align*}
& \Delta_1'
  +\Delta_1^\sharp
  +\bigl(\Delta_1^\tau\otimes(\Delta_1')^\tau\bigr)\lambda_1'
  +(\Delta_1^\tau\otimes 1)\lambda_3'
  +(1\otimes(\Delta_1')^\tau)\lambda_4'
  +\Delta_1'\tau^\sharp
  \notag\\
&\qquad
  =
  \Delta_1^G d^\sharp
  +(\delta_1\otimes 1)G_3
  +(1\otimes\delta_1')G_4,
\end{align*}
and the identity \eqref{eq:G-homotopy-4} is equivalent to
\begin{align*}
& \Delta_{2,1}'
  +(\tau\otimes\tau')\Delta_{2,1}'
  +\lambda_1'\Delta_2^\sharp
  =
  (d\otimes 1+1\otimes d')\Delta_{2,1}^G
  +G_1\delta_2^\sharp, \\
& \Delta_{2,2}'
  +\Delta_2^\tau\otimes(\Delta_2')^\tau
  +\bigl(\tau\otimes(\mu')^\tau+\mu^\tau\otimes\tau'\bigr)\Delta_{2,1}'
  +(\tau\otimes\tau')\Delta_{2,2}'
  \notag\\
&\qquad
  +\bigl(\tau\otimes(\Delta_2')^\tau\bigr)\Delta_{2,3}'
  +\bigl(\Delta_2^\tau\otimes\tau'\bigr)\Delta_{2,4}'
  +\lambda_2'\Delta_2^\sharp
  \notag\\
&\qquad
  =
  (v\otimes 1+1\otimes v')\Delta_{2,1}^G
  +(d\otimes 1+1\otimes d')\Delta_{2,2}^G
  +(1\otimes\delta_2')\Delta_{2,3}^G
  +(\delta_2\otimes 1)\Delta_{2,4}^G
  +G_2\delta_2^\sharp, \\
& \Delta_{2,3}'
  +\Delta_2^\tau\otimes 1
  +\bigl(\tau\otimes(\Delta_1')^\tau\bigr)\Delta_{2,1}'
  +(\tau\otimes 1)\Delta_{2,3}'
  +\lambda_3'\Delta_2^\sharp
  \notag\\
&\qquad
  =
  (1\otimes\delta_1')\Delta_{2,1}^G
  +(d\otimes 1)\Delta_{2,3}^G
  +G_3\delta_2^\sharp,\\
& \Delta_{2,4}'
  +1\otimes(\Delta_2')^\tau
  +(\Delta_1^\tau\otimes\tau')\Delta_{2,1}'
  +(1\otimes\tau')\Delta_{2,4}'
  +\lambda_4'\Delta_2^\sharp
  \notag\\
&\qquad
  =
  (\delta_1\otimes 1)\Delta_{2,1}^G
  +(1\otimes d')\Delta_{2,4}^G
  +G_4\delta_2^\sharp.
\end{align*}
These equations are proven in a similar way as in the case of $\wt{\lambda}$, by suitable countings of the boundaries of compactified 1-dimensional instanton moduli spaces for $(W^\dagger, S^\dagger)$.
This completes the proof. 
\end{proof}

\subsection{Flipping involution}\label{sec:flipping}
We now turn to the flipping involution $\tau_f$ discussed in Section~\ref{sec:2.5}. Fix a basepoint $x_0$ on $K$ and use the same basepoint on $K^r$. We perform the connected sum of $K$ and $K^r$ at $x_0$ and let $x^\#$ be either of the two fixed points of $\tau_f$ on $K \# K^r$. We prove that the action of $\tau_f$ is weakly homotopy equivalent to the \textit{algebraic flip map}, defined as follows.

\begin{defn}
Let $\wt{C}$ be an $\cS$-complex with a fixed basis. Define the algebraic flip map
\[
    \operatorname{alg}_{f}: \wt{C} \otimes \wt{C}\to \wt{C} \otimes \wt{C}
    \]
by transposing basis elements
    \[
    \zeta \otimes \zeta' \to \zeta' \otimes \zeta
    \]
    and extending $R$-linearly. It is clear that $\operatorname{alg}_f$ is a grading-preserving involution. 
\end{defn}

Recall from Theorem~\ref{orientation_rev} that we have an isomorphism $\alpha \colon \wt{C}(Y, K^r, x_0, \pi) \rightarrow \wt{C}(Y, K, x_0, \pi)$. Using this to identify $\smash{\wt{C}(Y, K^r, x_0, \pi)}$ with $\smash{\wt{C}(Y, K, x_0, \pi)}$, we think of the tensor product 
\[
\wt{C}(Y, K, x_0, \pi) \otimes \wt{C}(Y, K^r, x_0, \pi)
\]
by taking the tensor product of two copies of $\wt{C}(Y, K, x_0, \pi)$. We may then regard the algebraic flip map as an involution 
\[
\operatorname{alg}_{f}: \wt{C}(Y, K, x_0, \pi) \otimes \wt{C}(Y, K^r, x_0, \pi) \rightarrow \wt{C}(Y, K, x_0, \pi) \otimes \wt{C}(Y, K^r, x_0, \pi).
\]

\begin{thm}\label{conn sum2}
We have a weak homotopy equivalence
\[
(\wt{C}(Y, K, x_0) \otimes \wt{C}(Y, K^r, x_0), \operatorname{alg}_f) \simeq_{w.eq.} (\wt{C}(Y \# Y, K \# K^r, x^\#), \wt{\tau}_f)
\]
between the algebraic flip map and the flipping involution on the connected sum.
\end{thm}

The proof is almost the same as that of \cref{conn sum}. As before, we consider Daemi and Scaduto's pair of pants maps
    \begin{align*}
  &  \wt{\lambda} := \wt{\lambda}_{(W, S )}: \wt{C}(Y, K, x_0) \otimes \wt{C}(Y, K^r, x_0) \to \wt{C}( Y \# Y, K\# K^r, x^\#)   \\ 
   &     \wt{\lambda}' := \wt{\lambda}_{(W^\dagger, S^\dagger )}: \wt{C}(Y \# Y, K\# K^r, x^\#) \to \wt{C}(Y, K, x_0) \otimes \wt{C}(Y, K^r, x_0) 
    \end{align*}
 with respect to the choices of paths in Figures~\ref{fig:lambda} and \ref{fig:lambda'}. However, our choice of involution on the pair of pants is now slightly different. There is an obvious involution on the pants cobordism that restricts to the flipping involution on $K \# K^r$, as displayed below in Figure~\ref{Ext_inv_pants2}. Note that this again fixes $\gamma$ and $\sigma$, but not $\gamma_1$, $\gamma_2$, $\sigma_1$, or $\sigma_2$. 

\begin{figure}[htbp]
  \centering
  \includegraphics[
    width=0.2\linewidth
  ]{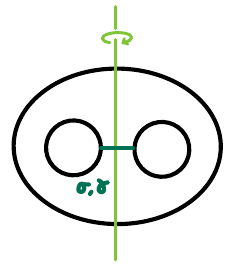}
  \caption{Extensions of $\tau_{\sqcup}$ and $\tau_f$.}
  \label{Ext_inv_pants2}
\end{figure}

The involution of Figure~\ref{Ext_inv_pants2} restricts to the involution $\tau_{\sqcup}$ on $(Y, K, x_0) \sqcup (Y, K^r, x_0)$ that interchanges the two factors and reverses orientation on $K$. It is straightforward to see that
\[
\wt{\tau}_\sqcup: \wt{C}(Y, K, x_0) \otimes \wt{C}(Y, K^r, x_0) \to \wt{C}(Y, K, x_0) \otimes \wt{C}(Y, K^r, x_0)
\]
can be identified with the algebraic flip map. Indeed, to compute $\wt{\tau}_\sqcup$ we first choose a perturbation for $(Y, K) \sqcup (Y, K^r)$, which we may take to be of the form $\pi \sqcup \pi$. We then form the pushforward perturbation under $\tau_{\sqcup}$ and take a perturbation on the product cobordism $[0, 1] \times (Y \sqcup Y, K \sqcup K^r)$ which interpolates between these. However, the pushforward perturbation is identical to $\pi \sqcup \pi$, so we may take the identity interpolation. This shows that the map $\Phi$ in the definition of $\wt{\tau}_\sqcup$ is the identity on the chain level and easily gives the claim.

\cref{conn sum2} reduces to the following. 

\begin{thm} The maps
\[
\wt{\lambda} := \wt{\lambda}_{(W, S )}: \wt{C}(Y, K, x_0) \otimes \wt{C}(Y, K^r, x_0) \to \wt{C}(Y \# Y, K\# K^r, x^\#)  \]
and
\[
\wt{\lambda}' := \wt{\lambda}_{(W^\dagger, S^\dagger )}: \wt{C}(Y \# Y, K\# K^r, x^\#) \to \wt{C}(Y, K, x_0) \otimes \wt{C}(Y, K^r, x_0) 
\]
    are weakly involutive. 
\end{thm}

\begin{proof}
The proof is analogous to that of
that of \cref{thm:conn_sum_key}, so we only provide a sketch. We must construct homotopies
\begin{align*}
\wt{H}: \wt{C}(Y, K, x_0) \otimes \wt{C}(Y, K^r, x_0) \to \wt{C}(Y \# Y, K\# K^r, x^\#) \\
\wt{G}:\wt{C}(Y \# Y, K\# K^r, x^\#) \to  \wt{C}(Y, K, x_0) \otimes \wt{C}(Y, K^r, x_0) 
\end{align*}
satisfying
\begin{align*}
  & \wt{d}^\# \circ \wt{H} + \wt{H}\circ \wt{d}^\otimes = \wt{\lambda} \circ \wt{\tau}^\otimes +  \wt{\tau}^\sharp \circ \wt{\lambda} + E_{21} \\ 
  &  \wt{d}^\otimes \circ  \wt{G} + \wt{G} \circ \wt{d}^\# = \wt{\lambda}' \circ \wt{\tau}^\sharp +  \wt{\tau}^\otimes  \circ \wt{\lambda}' + E'_{21},  
\end{align*}
where $E_{21}$ and $E'_{21}$ are $(2,1)$-block error terms.
The constructions of these homotopies are analogous to those given in the proof of \cref{thm:conn_sum_key}. 
Note that the paths $\gamma$ and $\sigma$ are setwise preserved under the involution of Figure~\ref{Ext_inv_pants2}, although their orientations are reversed. This means that the orientations of the instanton moduli spaces cut down by the associated holonomy maps are negated by the involution. However, since we are working over $\Z_2[T^{\pm 1}]$, this still gives the desired relation.
\end{proof}

\begin{comment}
    
\begin{cor} The involutive S-complex gives a homomorphism 
     \[
    h: \tilde{\mathcal{C}} \to \Theta^\mathcal{S}_{eq} \text{ defined as } h([(K,\tau)]) := [(\wt{C}(K), \tau)]
    \]
    which satisfies the following commutative diagram:
    \[
  \begin{CD}
     \tilde{\mathcal{C}} @>{h}>> \Theta^\mathcal{S}_{eq} \\
  @VVV    @VVV \\
     {\mathcal{C}}   @>>>  \Theta^\mathcal{S}
  \end{CD}, 
\]
where the vertical maps are the forgetting maps and the bottom horizontal arrow is the original S-complex homomorphism constructed in \cite{DS19}. 
\end{cor}
\end{comment}

\subsection{Computations and examples}
We now apply Theorems~\ref{conn sum} and \ref{conn sum2} to calculate the weak homotopy class of $T_{2,3} \# T_{2,3}$ for each of the two involutions $\tau_0 \# \tau_0$ and $\tau_f$. Here, $\tau_0$ is the unique strong inversion on the trefoil, which we showed in Example~\ref{ex:trefoilinvolutive} acts as the identity. We also conflate $T_{2,3}^r$ and $T_{2,3}$ when discussing $\tau_f$, as it is clear from Example~\ref{ex:trefoil} that there is no distinction between these complexes. 

We start by computing the underlying non-involutive complex. Following Section~\ref{sec:tensorbasis}, instead of writing down the full tensor product of $\cS$-complexes explicitly, we record the summand $C(T_{2,3} \# T_{2,3})$ and give the actions of $d$, $v$, $\delta_1$, and $\delta_2$. Recall that we have a basis for $C^\otimes$ given by
\[
C^\otimes = \langle {\mathfrak{a}}_i\otimes \mathfrak{b}_j  ,\quad  {\mathfrak{a}}_i \otimes \underline{\mathfrak{b}}_j, \quad {\mathfrak{a}}_i \otimes 1, \quad1\otimes {\mathfrak{b}}_j \rangle
\]
with respect to which
\[
d^\otimes =\begin{pmatrix}
d\otimes 1 + 1 \otimes d' &0&0&0\\
v \otimes 1 + 1 \otimes v' & d\otimes 1 + 1 \otimes d' & 1 \otimes \delta_2' & \delta_2 \otimes 1\\
1 \otimes \delta_1' &0&d \otimes 1& 0\\\delta_1\otimes 1& 0 & 0 & 1 \otimes d'\\
\end{pmatrix},
\]
while
\[
	v^\otimes =\begin{pmatrix}
v\otimes 1 &0& 0& \delta_2\otimes 1\\
0 & v\otimes 1 & 0 &0\\
0 & 0& v \otimes 1 & 0\\
0 & \delta_1\otimes 1& 0  & 1 \otimes v'\\
\end{pmatrix},
\] 
and
\[
\delta_1^\otimes= \begin{pmatrix} 0, 0, \delta_1 \otimes 1, 1 \otimes \delta_1'\end{pmatrix} \quad \text{and} \quad \delta_2^\otimes = \begin{pmatrix}0,0,\delta_2 \otimes 1,1 \otimes \delta_2'\end{pmatrix}^T.
\]
The complex $C(T_{2,3})$ has only one generator, which has a $\delta_1$-arrow to the reducible. Applying the above formulas gives the complex for $C(T_{2,3} \# T_{2,3})$ displayed on the right in Figure~\ref{fig:trefoil1}.

\begin{figure}[h!]
\includegraphics[scale = 1]{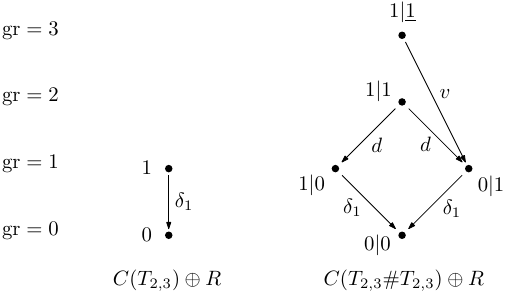}
	\caption{Left: the complex $C(T_{2,3}) \oplus R$. Right: the complex $C(T_{2,3} \otimes T_{2,3}) \oplus R$. In an unfortunate collision of notation, the generator $0$ denotes the reducible and $1$ denotes the irreducible, while in our usual notation for $C^\otimes$, the reducible of each factor is denoted $1$. We make this change for consistency with more general calculations in later sections. All arrows have a coefficient of $\Lambda$. Gradings are arranged by height.}\label{fig:trefoil1}
\end{figure}

There are two actions we wish to consider on this complex: $\tau_0 \# \tau_0$ and $\tau_f$. In each case, the action on the full $\cS$-complex is determined by the components $\tau$, $\mu^\tau$, $\Delta_1^\tau$, and $\Delta_2^\tau$ as maps on $C^\otimes$. Following Theorems~\ref{conn sum} and \ref{conn sum2}, we thus compute these components for the tensor product of $\tau_0$ with itself, as well as for the algebraic flip map. Since $\tau_0$ acts as the identity, the tensor product involution also acts as the identity. This means its $\tau$-component is the identity, while all other components vanish. For the algebraic flip map, the $\tau$-component interchanges $1|0$ and $0|1$ and fixes $0|0$ and $1|1$. The algebraic flip map also sends $1 | \un{1}$ to $\un{1} | 1 = 1 | \un{1} + (1 | \un{1} + \un{1} | 1) = 1 | \un{1} + \chi^\otimes(1 | 1)$; this shows the $\tau$-component on $1 | \un{1}$ is the identity, while the $\mu^\tau$-component (as a map from $C^\otimes$ to itself) sends $1 | \un{1}$ to $1 | 1$. The $\Delta_1^\tau$- and $\Delta^\tau_2$-components are zero. See Figure~\ref{fig:trefoil2}.

However, because Theorems~\ref{conn sum} and \ref{conn sum2} are weak homotopy equivalences, we do \textit{not} know the $\mu^\tau$-components of the \textit{actual} actions of $\tau_0 \# \tau_0$ and $\tau_f$. Nevertheless, we can constrain the possibilities for these $\mu^\tau$-components. As $\mu^\tau$ is degree $-1$ and goes between irreducible generators, there are three possible $\mu^\tau$-arrows: 
\[
\mu^\tau(1 | 1) = c_1 (1|0) + c_2(0 |1) \quad \text{and} \quad \mu^\tau(1 | \un{1}) = c_3 (1 | 1).
\]
Now, we know that
\[
v \tau + \tau v + d \mu^\tau + \mu^\tau d + \delta_2 \Delta^\tau_1 + \Delta^\tau_2 \delta_1 = 0 
\]
by examining the $(2, 1)$-component of $\wt{d} \wt{\tau} + \wt{\tau} \wt{d} = 0$. Plugging in the generator $1 | \un{1}$, we see that for $\tau_0 \# \tau_0$ we must have $c_3 = 0$, while for $\tau_f$, we must have $c_3 = 1$. Next, we know that $\mu^\tau$ satisfies
\[
\mu^\tau \tau + \tau \mu^\tau + \Delta_2\Delta_1 = Hv + vH + \mu^Hd + d\mu^H + \Delta_2^H \delta_1 + \delta_2 \Delta_1^H
\]
by examining the $(2, 1)$-component of $\smash{\wt{\tau}^2 + \id = \wt{H} \wt{d}  + \wt{d} \wt{H}}$. Plugging in the generator $1|1$ to the left-hand side gives zero in the case of $\tau_0 \# \tau_0$ but $(c_1 + c_2)(1|0) + (c_2 + c_1)(0|1)$ in the case of $\tau_f$. Since every arrow in Figure~\ref{fig:trefoil1} has a $\Lambda$-decoration, each term on the right-hand side is either zero or has image divisible by $\Lambda$. Hence we conclude that in the case of $\tau_f$, we must have $\Lambda | (c_1 + c_2)$. We summarize our computation as follows.

\begin{ex}\label{ex:trefoilconnsum}
The actions of $\tau_0 \# \tau_0$ and $\tau_f$ on $\wt{C}(T_{2,3} \# T_{2,3})$ are as indicated in Figure~\ref{fig:trefoil2}. Specifically:
\begin{itemize}[noitemsep]
 \item The $\tau$-component of $\tau_0 \# \tau_0$ is the identity (or rather, $\alpha$), while the $\tau$-component of $\tau_f$ interchanges the two indicated generators. 
 \item In both cases, $\Delta_1^\tau = \Delta_2^\tau = 0$. 
 \item There is a single guaranteed $\mu^\tau$-arrow for $\tau_f$, and two indeterminate $\mu^\tau$-arrows for both $\tau_0 \# \tau_0$ and $\tau_f$. In the case of $\tau_f$, we have $\Lambda | (c_1 + c_2)$.
\end{itemize}
Example~\ref{ex:trefoilconnsum} will play a significant role in Section~\ref{sec:trefoilcalculation}.
\end{ex}

\begin{figure}[h!]
\includegraphics[scale = 1]{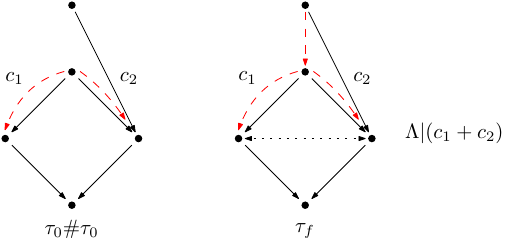}
	\caption{Left: the action of $\tau_0 \# \tau_0$. Right: the action of $\tau_f$. The only nontrivial $\tau$-component is represented by the horizontal dotted arrow. The $\mu^\tau$-components are given by the red dashed arrows. The vertical red dashed arrow has a coefficient of $1$; the other dashed arrows have coefficients of $c_1$ and $c_2$, respectively. In the case of $\tau_f$, we have $\Lambda | (c_1 + c_2)$.}\label{fig:trefoil2}
\end{figure}

\section{The Chern--Simons filtration}\label{sec:CS-filtration}

In this section, we introduce a refinement of our involutive instanton theory that takes into account the Chern--Simons filtration on the instanton knot complex. This will take the form of the \textit{enriched (involutive) complex} $\smash{\wt{\mf{E}}_\tau(Y, K, x_0)}$ associated to $(Y, K, \tau, x_0)$, which should be regarded as the strongest invariant provided in this paper. As usual, we introduce an appropriate notion of enriched local (and weak local) equivalence and define the corresponding enriched under-bar and upper-bar complexes. 

Our main goal will be to construct the concordance invariant $r_s$ discussed in the introduction. This associates to any strongly invertible knot $(Y, K, \tau)$ and any real parameter $s \in [-\infty, 0]$ a real number 
\[
r_s (Y, K, \tau)  \in [0, \infty ]
\]
which is invariant under isotopy-equivariant concordance. Our construction here is a straightforward modification of the analogous invariant for involutive $3$-manifolds introduced in \cite{ADMT}, which is itself adapted from the non-involutive $r_s$-invariant defined in \cite{NST19}. 
The $r_s$-invariant factors through (weak) local equivalence and has the following properties.

\begin{comment}
    \[
    [\wt{\mathfrak{E}}_\tau(S^3, \U)]_{\mathrm{loc}}
    \leq
    [\wt{\mathfrak{E}}_\tau(Y, K)]_{\mathrm{loc}},
\]
and
\end{comment}
\begin{thm} \label{thm:rsproperties}
The following hold:
\begin{enumerate}[noitemsep]
    \item\label{rsproperties:1} There exists a level $\delta$ involutive local map from the trivial complex into $\un{\mf E}_\tau(Y, K)$ if and only if
\[
    r_{-\delta}(Y, K, \tau)=\infty .
\]
    \item\label{rsproperties:2} Let $(W, S)$ be a strong equivariant negative definite pair from $(Y, K, \tau)$ to $(Y', K', \tau')$.
    Then
    \[
    r_s(Y, K, \tau) \leq r_{s-2\kappa(W, S)}(Y', K', \tau')+2\kappa(W, S), \quad \forall s \in [-\infty, 0],
    \]
    where \[
    \displaystyle \kappa(W, S) = \min_{z \in H^2(W; \Z) } - \left( z + \frac{1}{4} \mathrm{PD}[S]  \right)^2.
    \]
    %\item[(iii)] In the same situation in (iii), if we further suppose $W$ is $\Z_2$-homology cobordism, $r_s(Y', K', f|_{Y'})< \infty$ and $\pi_1 (W\ \setminus \ S)\cong \Z$, we have 
 %      \[
 %   r_s(Y, K, f|_Y) < r_s(Y', K', f|_{Y'}), \quad \forall s %\in [\infty, 0].
 %   \]
Moreover, if $\pi_1 (W\ \setminus \ S)\cong \Z$, then this inequality is strict whenever the right-hand side is finite.
    \item\label{rsproperties:3} We have a connected sum inequality 
    \[
    r_{s+s'}( Y\# Y', K \# K', \tau \# \tau') \geq \min \{ r_s (Y, K, \tau)+s', r_{s'} (Y', K', \tau')+s\}, \quad \forall s, s' \in [-\infty, 0].
    \]
    \item\label{rsproperties:4} Let $\U$ be a local strongly invertible unknot in an equivariant homology sphere $(Y, \tau)$. Then
    \[
    r_s(Y , \U, \tau) \leq 2r_{2s}(Y, \tau), \quad \forall s\in [-\infty, 0], 
    \]
    where the right-hand side is the involutive $r_s$-invariant introduced in \cite{ADMT}. 
\end{enumerate}
\end{thm}

In light of \eqref{rsproperties:1}, the reader should think of the value $\infty$ as the most trivial. Generally, we use \eqref{rsproperties:1} to obstruct the existence of a local map from the trivial complex in cases where it is known that $r_s(Y, K, \tau) < \infty$. The $r_s$-invariant satisfies the inequality
\[
r_s(Y, K, \tau) \leq r_{s'}(Y, K, \tau)
\]
whenever $s \geq s'$, and takes values in a discrete set as $s$ varies.

\subsection{Involutive instanton-type complexes}
For the remainder of this paper, we work over the ring
\[
\Z_2[T^{\pm 1}, y^{\pm 1}] = R[y^{\pm 1}].
\]
The formalism of the last several sections clearly extends to $\cS$-complexes defined over $R[y^{\pm 1}]$. In this context, all underlying modules are free and finitely generated over $\smash{R[y^{\pm 1}]}$ and all maps are $\smash{R[y^{\pm 1}]}$-linear. We promote the $\Z_4$-grading from Definition~\ref{def:involutiveScomplex} to a $\Z$-grading and require $\gr(y) = 4$. When discussing the reducible summand (which is now a copy of $R[y^{\pm 1}]$), we assume that we have chosen a generator $\theta$ with $\gr(\theta) = 0$. 

The only slight subtlety is in the definition of a local map: as usual, a local map of $\cS$-complexes is characterized by the fact that $\eta$ is a unit. However, since we consider only grading-preserving maps in this paper, note that $\gr(\eta) = 0$. Hence no powers of $y$ appear in $\eta$. Correspondingly, a $\theta$-supported cycle is a cycle in $\un{C}$ in which the coefficient of $\theta$ is a unit in $R$, rather than a unit in $R[y^{\pm 1}]$.

\subsubsection{Instanton-type complexes} We first discuss filtered complexes in the non-involutive setting.

\begin{defn}\label{def:instantontype}
An \textit{instanton-type complex} is a pair $(\wt{C}, \deg_I)$, where $\wt{C}$ is an $\cS$-complex over $R[y^{\pm 1}]$ and $\deg_I$ is a real-valued filtration on $\smash{\wt{C}}$. This means
\[
\deg_I (\wt{d} ( \zeta ) )\leq \deg_I ( \zeta ) \quad \text{and} \quad \deg_I (\chi ( \zeta ) )\leq \deg_I ( \zeta ). 
\]
The variables $T$ and $y$ are required to have the following behavior with respect to $\deg_I$ (and $\gr$):
\[
\gr(T) = 0, \quad \deg_I(T) = 0, \quad \gr(y) = 4, \quad \deg_I(y) = 1.
\]
By convention, $\deg_I(0) = -\infty$. We usually suppress $\deg_I$ and refer to $\smash{\wt{C}}$ as the instanton-type complex.
\end{defn}

The main use of $\deg_I$ will be to form the following subquotients.

\begin{defn}\label{def:3.4}
Let $\wt{C}$ be an instanton-type complex. For any real number $s$, define the subcomplex
\[
(\wt{C}^{[-\infty,s]}, \wt{d}^{[-\infty,s]}) = (\{ \zeta \in \wt{C} \mid \deg_I (\zeta ) \leq s\} , \wt{d}|_{\wt{C}^{[-\infty,s]}}). 
\]
For any real number $r<s$, define the quotient complex
\[
\wt{C}^{[r,s]} = \wt{C}^{[-\infty,s]}/ \wt{C}^{[-\infty,r]}.
\]
We usually assume $r < 0 \leq s$, so that $\wt{C}^{[r,s]}$ contains $\theta$. We refer to $\wt{C}^{[r,s]}$ as the \textit{$[r, s]$-subquotient}.
\end{defn}

When discussing morphisms between instanton-type complexes, the most natural requirement would be to require that our maps be filtered. However, it will be useful to relax this condition slightly and allow our morphisms to increase filtration by up to a fixed parameter.

\begin{defn}\label{def:3.6}
Let $\wt{C}$ and $\wt{C}'$ be two instanton-type complexes and fix any real number $\delta \geq 0$. We say that an $\cS$-morphism $\wt{\lambda} \colon \wt{C} \rightarrow \wt{C}'$ has \textit{level}-$\delta$ if 
\[
\deg_I(\wt{\lambda}(\zeta)) \leq \deg_I(\zeta) + \delta
\]
for all $\zeta \in \wt{C}$. We say that a homotopy $\wt{H}$ between two level $\delta$ morphisms has level $\delta$ if analogously
\[
\deg_I(\wt{H}(\zeta)) \leq \deg_I(\zeta) + \delta
\]
for all $\zeta \in \wt{C}$. We refer to a level zero morphism or homotopy as \textit{filtered}. For each $r < s$, a level $\delta$ morphism or homotopy induces a corresponding morphism or homotopy on the subquotient
\[
\wt{\lambda}^{[r, s]} \colon \wt{C}_*^{[r, s]} \rightarrow \wt{C}_*'^{[r + \delta, s + \delta]} \quad \text{or} \quad \wt{H}^{[r, s]} \colon \wt{C}_*^{[r, s]} \rightarrow \wt{C}_{*+1}'^{[r + \delta, s + \delta]}.
\]
\end{defn}

Having defined level $\delta$ morphisms and homotopies, we immediately obtain the notion of level $\delta$ homotopy equivalences, local maps, and local equivalences. In each case, we simply require all morphisms and homotopies to have level $\delta$. For example, a level $\delta$ homotopy equivalence consists of a pair of level $\delta$ morphisms which are homotopy inverse via level $\delta$ homotopies.\footnote{Note that the composition of a level $\delta_1$-morphism and a level $\delta_2$-morphism is level $(\delta_1 + \delta_2)$. Thus, the filtered case is the only setting in which homotopy equivalence and local equivalence constitute equivalence relations.} It is equally clear that we may define instanton-type $\smash{\un{C}}$- and $\smash{\ov{C}}$-complexes and form the subquotients
\[
\un{C}^{[r, s]} \quad \text{and} \quad \ov{C}^{[r, s]}.
\]
level $\delta$ morphisms of instanton-type $\smash{\un{C}}$- and $\smash{\ov{C}}$-complexes are defined similarly, as are homotopy equivalences, local maps, and local equivalences. 

\begin{defn}\label{def:3.9}
We say that a cycle $x \in \underline{C}^{[r, s]}$ is a \emph{$\theta$-supported $[r, s]$-cycle} if the coefficient of $\theta$ in $x$ is a unit in $R$. We require $r < 0 \leq s$, so that $\theta$ (or, strictly speaking, its image) indeed lies in $\smash{\un{C}^{[r, s]}}$.
\end{defn}

A level $\delta$ local morphism of instanton-type $\un{C}$-complexes takes $\theta$-supported $[r, s]$-cycles to $\theta$-supported $[r + \delta, s + \delta]$-cycles (so long as $\delta < |r|$).

%A level $\delta$ local morphisms maps $\theta$-supported $[r, s]$-cycles to $\theta$-supported $[r + \delta, s + \delta]$-cycles. (Here we assume $\delta < |r|$.)

\subsubsection{The involutive setting} Generalizing the filtration to the involutive setting is straightforward: once again, we simply require all morphisms and homotopies to have level $\delta$. There are only three additional maps we must consider: the map $\smash{\wt{\tau}}$, the homotopy between $\smash{\wt{\tau}^2}$ and the identity, and (when defining involutive morphisms) the homotopy which intertwines a given morphism with $\smash{\wt{\tau}}$.

\begin{defn}
A \textit{level $\delta$ involutive instanton-type complex} is a pair $\smash{(\wt{C}, \wt{\tau})}$, where $\smash{\wt{C}}$ is an instanton-type complex and $\smash{\wt{\tau}}$ is a level $\delta$ homotopy involution on $\smash{\wt{C}}$. Explicitly, this that means $\smash{\wt{\tau}}$ is a level $\delta$ morphism from $\smash{\wt{C}}$ to itself such that
\[
\wt{\tau}^2 + \id = \wt{d} \wt{H} + \wt{H} \wt{d}
\]
for some level $\delta$ homotopy $\wt{H}$. When $\delta = 0$, we refer to $\smash{(\wt{C}, \wt{\tau})}$ simply as an \textit{involutive instanton-type complex}. The filtered setting will be the most common case.
\end{defn}

\begin{defn}\label{def:involutivesubquotient}
Let $\smash{(\wt{C}, \wt{\tau})}$ be an level $\delta$ involutive instanton-type complex and fix any $r < s$. Suppose that $r$ and $s$ are at least distance $\delta$ away from the set of values attained by $\deg_I$; i.e.,
\begin{equation}\label{eq:keepaway}
\delta \leq \min\{ |\deg_I(\zeta) - r|, |\deg_I(\zeta) - s| \mid \zeta \in \wt{C} \}.
\end{equation}
Then it is clear that $\wt{\tau}$ descends to the $[r, s]$-subquotient to give a homotopy involution
\[
\wt{\tau}^{[r, s]} \colon \wt{C}^{[r, s]} \rightarrow \wt{C}^{[r, s]}.
\]
We refer to the pair 
\[
(\wt{C}^{[r, s]}, \wt{\tau}^{[r, s]}) 
\]
as the \textit{involutive $[r, s]$-subquotient}.
Note that if $\delta = 0$, then this exists for every $r < s$.
\end{defn}

\begin{defn}
A level $\delta$ morphism $\smash{\wt{\lambda} \colon (\wt{C}, \wt{\tau}) \rightarrow (\wt{C}', \wt{\tau}')}$ between involutive instanton-type complexes (each of unspecified level) is \textit{involutive} if it intertwines $\wt{\tau}$ and $\wt{\tau}'$ up to a level $\delta$ homotopy; i.e.,
\[
\wt{\lambda} \wt{\tau} + \wt{\tau}' \wt{\lambda} = \wt{H} \wt{d}' + \wt{d} \wt{H}
\]
for some level $\delta$ homotopy $\wt{H}$. If the above equality only holds outside of the $(2, 1)$-block, then we refer to $\smash{\wt{\lambda}}$ as \textit{weakly involutive}.
\end{defn}

Having defined level $\delta$ involutive morphisms, we immediately obtain the definition of level $\delta$ homotopy equivalences, local maps, and local equivalences in the involutive category. We may also repeat our definitions to construct involutive instanton-type $\un{C}$- and $\ov{C}$-complexes. Under the same conditions as in Definition~\ref{def:involutivesubquotient}, we have induced homotopy involutions on the subquotients
\[
(\un{C}^{[r, s]}, \un{\tau}^{[r,s]}) \quad \text{and} \quad (\ov{C}^{[r, s]}, \ov{\tau}^{[r,s]}).
\]
level $\delta$ involutive morphisms for $\un{C}$- and $\ov{C}$-complexes are defined similarly, as are homotopy equivalences, local maps, and local equivalences.

\begin{defn}
We say that a cycle $\smash{x \in \underline{C}^{[r, s]}}$ is an \emph{involutive $\theta$-supported $[r, s]$-cycle} if it is a $\theta$-supported cycle in the sense of Definition~\ref{def:3.9} and
 \[
 \un{\tau}^{[r, s]}_* [x] = [x] \in H_* (  \underline{C}^{[r,s]}) . 
 \]
\end{defn}

\subsubsection{Tensor products and duals} We now record how to form tensor products of filtered complexes.

\begin{defn}
Given involutive instanton-type complexes $(\wt{C}, \wt{\tau})$ and $(\wt{C}', \wt{\tau}')$, their tensor product as an $\cS$-complex is defined as before. We give this the homotopy involution $\wt{\tau} \otimes \wt{\tau}'$ and filtration characterized by
\[
\deg_I ( \zeta \otimes \zeta' ) = \deg_I (\zeta ) + \deg_I(\zeta ').
\]
If $(\wt{C}, \wt{\tau})$ and $(\wt{C}', \wt{\tau}')$ are instead level $\delta_1$ and level $\delta_2$, then their tensor product is level $(\delta_1 + \delta_2)$. 
\end{defn}

It will be helpful to consider the behavior of $\theta$-supported cycles with respect to tensor products.

\begin{lem}\label{lem:tensortheta}
Let $(\wt{C}, \wt{\tau})$ and $(\wt{C}', \wt{\tau}')$ be two involutive instanton-type complexes. Suppose $x$ is an involutive $\theta$-supported $[r, s]$-cycle in $\un{C} \subset \smash{\wt{C}}$ and $x'$ is an involutive $\theta$-supported $[r', s']$-cycle in $\un{C}' \subset \smash{\wt{C}'}$. Then $x \otimes x'$ is an involutive $\theta$-supported $[\max(r + s', r' + s), s + s']$-cycle.
\end{lem}
\begin{proof}
Firstly, it is easily checked that $x \otimes x'$ indeed lies in the under-bar subcomplex of $\smash{\wt{C} \otimes \wt{C}'}$ and that it is $\theta$-supported. To say that $x$ is an $[r, s]$-cycle is to say that $\deg_I(x) \leq s$ and that $\deg_I(\un{d} x) \leq r$, and likewise for $x'$. Hence $\deg_I(x \otimes x') \leq s + s'$. We know
\[
\wt{d}^\otimes (x \otimes x') = x \otimes (\wt{d}' x') + (\wt{d} x) \otimes x'.
\]
The first of these terms has filtration level at most $s + r'$, while the second has filtration level at most $s' + r$. Hence $x \otimes x'$ is a $[\max(r + s', r' + s), s + s']$-cycle. It is straightforward to check that the class of $x \otimes x'$ is invariant under $\un{\tau}^\otimes$ after passing to the subquotient. This gives the claim.
\end{proof}

Forming duals of filtered complexes is likewise straightforward.

\begin{defn}
Given an involutive instanton-type complex $(\wt{C}, \wt{\tau})$, its dual as an $\cS$-complex is defined as before. We give this the dual homotopy involution $\wt{\tau}^\dagger$ and filtration characterized by
\[
\deg_I(\zeta^\dagger ) = \sup \{ \deg_I(\zeta^\dagger(z)) - \deg_I(z) \mid 0 \neq z \in \wt{C} \}.
\]
If $(\wt{C}, \wt{\tau})$ instead has level $\delta$, then its dual also has level $\delta$.
\end{defn}

\subsubsection{The involutive $r_s$-invariant}
We now come to the central definition of the section. This is the knot Floer version of the involutive $r_s$-invariant introduced in \cite{ADMT}; we refer the reader to \cite[Section 4.2]{ADMT} for further discussion and examples.

\begin{defn}\label{def:4.4}
Let $(\un{C}, \un{\tau})$ be an involutive instanton-type $\un{C}$-complex and $s \in [- \infty, 0]$. Define
\[
r_s (\un{C}, \un{\tau})= -\inf_{r < 0} \{ r \mid
\text{there exists an involutive $\theta$-supported $[r, -s]$-cycle} \} \in [0,\infty],
\]
with the caveat that if the above set is empty, we set $r_s(\un{C}, \un{\tau}) = - \infty$. Here we assume that $(\un{C}, \un{\tau})$ has level zero so that the involutive subquotient exists for every $r$ and $s$, although this is largely a matter of convenience. For an involutive instanton-type complex $\smash{(\wt{C}, \wt{\tau})}$, we define $\smash{r_s( \wt{C}, \wt{\tau}) = r_s(\un{C}, \un{\tau})}$ by passing to the subcomplex $\un{C}$.
\end{defn}

The following is the algebraic analogue of Theorem~\ref{thm:rsproperties}. Here, the trivial involutive instanton-type complex is given by $R[y^{\pm 1}]$ equipped with the identity involution and filtration defined by $\deg_I(\theta) = 0$. 

\begin{lem}\label{bb}
Let $(\un{C}, \un{\tau})$ and $(\un{C}', \un{\tau}')$ be involutive instanton-type $\un{C}$-complexes. The following hold:
\begin{enumerate}
    \item\label{bb1}
    There exists a level $\delta$ involutive local map from the trivial complex into $(\un{C}, \un{\tau})$ if and only if 
\[
r_{-\delta}(\un{C}, \un{\tau})=\infty.
\]
    \item\label{bb2} If there exists a level $\delta$ involutive local map from $(\un{C}, \un{\tau})$ to $(\un{C}', \un{\tau}')$, then
    \begin{align*}
&r_s(\un{C} , \un{\tau}) \le r_{s-\delta} (\un{C}' , \un{\tau}') + \delta, \quad \forall s \in [-\infty, 0].
\end{align*}
  \item\label{bb3} We have a connected sum inequality
    \[
    r_{s+ s'}( \wt{C}\otimes \wt{C}', \wt{\tau}\otimes \wt{\tau}' ) \geq \min \{ r_s (\wt{C}, \wt{\tau})-s', r_{s'} (\wt{C}', \wt{\tau}')-s\}, \quad \forall s, s' \in [-\infty, 0].
    \]
    \end{enumerate}
\end{lem}

\begin{proof}
We prove \eqref{bb1}. Suppose $r_{-\delta} ( \un{C}, \un{\tau}) = \infty$. Then for each $r \rightarrow - \infty$, there exists an involutive $\theta$-supported $[r, \delta]$-cycle in $(\un{C}, \un{\tau})$. Since $\un{C}$ in grading zero is finitely generated over $R$, it is not difficult to check that for sufficiently large $r$, this gives an involutive $\theta$-supported $[-\infty, \delta]$-cycle $x$. We construct a local map from the trivial complex into $\un{C}$ by sending the generator of $R$ to this cycle $x$; this map has level $\delta$. To see our morphism homotopy commutes with $\un{\tau}$, note that $x + \un{\tau}(x) = \un{d} z$ for some $\smash{z \in \un{C}^{[-\infty, \delta]}}$. Define our homotopy by sending the generator of $R$ to $z$; this homotopy has level $\delta$. Conversely, if there exists a level $\delta$ local map from the trivial complex into $\un{C}$, then the image of the generator of $R$ provides the desired involutive $\theta$-supported $[-\infty, \delta]$-cycle.

Now we prove \eqref{bb2}. By the definition of $r_s(\un{C}, \un{\tau})$, for any $r$ approaching $- r_s(\un{C}, \un{\tau})$ from above, there exists an involutive $\theta$-supported $[r, -s]$-cycle $x$. The image of this cycle is an involutive $\theta$-supported $[r + \delta, -s + \delta]$-cycle. It follows that 
\[
-( r + \delta) \leq r_{s - \delta}(\un{C}', \un{\tau}).
\]
Taking $r \rightarrow -r_s(\un{C}, \un{\tau})$ and moving $\delta$ from the left to the right gives the desired inequality.

Finally, \eqref{bb3} follows from Lemma~\ref{lem:tensortheta} and the same limiting argument as in the previous paragraph. This completes the proof.
\end{proof}

\begin{rem}\label{rem:otherrsproperties}
It is also clear from the definition of the $r_s$-invariant that
\[
r_s(\un{C}, \un{\tau}) \leq r_{s'}(\un{C}, \un{\tau})
\]
whenever $s \geq s'$, since an involutive $\theta$-supported cycle in $\un{C}^{[r, -s]}$ then includes as an involutive $\theta$-supported cycle in $\smash{\un{C}^{[r, -s']}}$. The $r_s$-invariant takes values in the set of values attained by $\deg_I$, which is discrete due to the finite generation of $\un{C}$. Finally, it is useful to note that we may replace $\un{C}$ with $\smash{\wt{C}}$ everywhere in the statement of Lemma~\ref{bb}. This follows from the fact that a local map of instanton complexes induces a local map of under-bar complexes, and that a local map from the trivial complex into $\smash{\wt{C}}$ necessarily gives a local map into $\un{C}$.
\end{rem}

As a special case of the second part of Lemma~\ref{bb}, we have the following.

\begin{lem}
Let $(\un{C}, \un{\tau})$ and $(\un{C}', \un{\tau}')$ be two involutive instanton-type $\un{C}$-complexes. If $(\un{C}, \un{\tau})$ and $(\un{C}', \un{\tau}')$ are filtered locally equivalent, then $r_s(\un{C}, \un{\tau}) = r_s(\un{C}', \un{\tau}')$.
\end{lem}
\begin{proof}
This follows immediately from applying the second part of Lemma~\ref{bb} with $\delta = 0$.
\end{proof}

\begin{lem}
Let $\smash{(\wt{C}, \wt{\tau})}$ and $\smash{(\wt{C}', \wt{\tau}')}$ be two involutive instanton-type complexes. If $(\wt{C}, \wt{\tau})$ and $(\wt{C}', \wt{\tau}')$ are filtered weakly locally equivalent, then $\smash{r_s(\wt{C}, \wt{\tau}) = r_s(\wt{C}', \wt{\tau}')}$.
\end{lem}
\begin{proof}
This follows immediately from the fact that a weak local equivalence from $\smash{\wt{C}}$ to $\smash{\wt{C}'}$ induces a local equivalence from $\un{C}$ to $\un{C}'$.
\end{proof}

\begin{comment}
Note that we may replace each instance of $\smash{\un{C}}$ in the statement of Lemma~\ref{bb} with $\smash{\wt{C}}$. To see this for the second part of Lemma~\ref{bb}, note that a local map from $\smash{\wt{C}}$ to $\smash{\wt{C}}'$ gives a local map from $\smash{\un{C}}$ to $\smash{\un{C}'}$ via restriction. To see this for the first part of Lemma~\ref{bb}, note that there exists a local map from $R$ into $\smash{\wt{C}}$ if and only if there exists a local map from $R$ into $\smash{\un{C}}$. The forwards direction is a special case of the previous claim, while the background direction follows from postcomposing a local map into $\un{C}$ with the inclusion from $\un{C}$ to $\smash{\wt{C}}$. 
\end{comment}

%In fact, recall that a \textit{weak} local morphism from $\smash{\wt{C}}$ to $\smash{\wt{C}'}$ induces a local morphism from $\smash{\un{C}}$ to $\smash{\un{C}'}$. It follows that $r_s$ is in fact an invariant of filtered \textit{weak} local equivalence.

\subsection{Enriched complexes}
We have now achieved our main goal of algebraically constructing the involutive $r_s$-invariant for instanton-type complexes. Unfortunately, it turns out that the data of a strongly invertible knot does not actually define a single involutive instanton-type complex. Instead, for each choice of perturbation $\pi$, we obtain a level $\delta$ instanton-type complex, where $\delta$ depends on the norm of $\pi$. Taking a family of perturbations converging to zero, we thus obtain a \textit{sequence} of level $\delta$ instanton-type complexes with $\delta \rightarrow 0$.

To capture this, we introduce the notion of an \textit{enriched involutive complex}, as in \cite[Section 4.4]{ADMT}. In practice, however, no intuition will be lost by continuing to think of $(Y, K)$ as giving a single involutive instanton-type complex. Indeed, in most of our explicit examples, it turns out that no perturbation data is needed, in which case the formalism of this section is superfluous. 

\subsubsection{Basic constructions}
The following is taken from \cite[Section 4.4]{ADMT}.

\begin{defn}\label{def:4.14}
An {\it enriched involutive complex} ${\wt{\mf E}}_\tau$ consists of a sequence of level $\delta_i$ involutive instanton complexes
\[
(\wt{C}_i,\wt{\tau}_i) \text{ of level } \delta_i
\]
for $i \in \Z^{\geq 0}$, together with a sequence of level $\delta_{i,j}$ involutive morphisms between them 
\[
\psi^j_i : (\wt{C}_i, \wt{\tau}_i) \to (\wt{C}_j, \wt{\tau}_j) \text{ of level } \delta_{i,j}
\]
for $i, j \in \Z^{\geq 0}$. These are required to satisfy the following conditions: 

\begin{itemize}[noitemsep]

\item (Clustering condition): The $\deg_I$-levels of homogenous chains in the $\wt{C}_i$ cluster around a discrete subset of $\R$. More precisely, we require that there exists a discrete subset $\mathfrak{K} \subset \R$ such that for any $\delta>0$, there exists $N$ such that $\forall i >N$ and $\zeta \in \wt{C}_i$, 
    \[
    \deg_I(\zeta) \in B_\delta(\mathfrak{K}).
    \]
Note that necessarily $0 \in \mathfrak{K}$.
\item (Composition of maps): Each $\psi_i^i$ is the identity and each $\psi_j^k \circ \psi_i^j$ is homotopic to $\psi^k_i$ via a homotopy of level $\smash{\delta_{i,j,k}}$.

%\item (Compatibility with involutive structures): Commutation homotopies connecting $\wt{\tau}_j \circ \psi^j_i$ and $\psi^j_i\circ \wt{\tau}_i$ are level $\delta_{i,j}$.

\item (Perturbations converging to zero): We have 
\[
\delta_i \rightarrow 0, \quad \delta_{i,j} \rightarrow 0, \quad \text{and} \quad \delta_{i,j,k} \rightarrow 0
\]
as $i, j, k \rightarrow \infty$. More precisely, for any $\delta > 0$, there exists $N$ such that $\forall i, j, k > N$, we have $\delta_i, \delta_{i,j}, \delta_{i,j,k} < \delta$.
\end{itemize}
\end{defn}

We now consider to what extent we can pass to subquotients of an enriched involutive complex.

\begin{defn}\label{def:enrichedsubquotient}
Let $\wt{\mf E}_\tau$ be an enriched involutive complex and fix any $r, s \in [-\infty, \infty] \ \setminus \ \mathfrak{K}$. Due to the clustering condition, it is clear that for $i$ sufficiently large, \eqref{eq:keepaway} is satisfied and $\wt{\tau}_i$ descends to the $[r, s]$-subquotient of $\smash{\wt{C}_i}$. A similar argument shows that the homotopy type of this pair stabilizes as $i \rightarrow \infty$. We thus define 
\[
\wt{\mathfrak{C}}_\tau^{[r, s]} = \lim_{i \rightarrow \infty} (\wt{C}^{[r, s]}_i, \wt{\tau}^{[r, s]}_i).
\]
Note that if $[r, r']$ and $[s, s']$ are completely contained in $[-\infty, \infty] \ \setminus \ \mathfrak{K}$, then clearly
\[
\wt{\mathfrak{C}}_\tau^{[r, s]} \simeq \wt{\mathfrak{C}}_\tau^{[r', s']}.
\]
\end{defn}

% It is also clear that $r', s'$ are such that $[r, r']$ and $[s, s']$ are subsets of $[-\infty, \infty] \ \setminus \ \mathfrak{K}$, then 
%\[
%\wt{\mathfrak{E}}_\tau^{[r, s]} = %\]

Morphisms between enriched involutive complexes are defined as follows.

\begin{defn}
A \textit{level $\epsilon$ involutive morphism} from $\wt{\mathfrak{E}}_\tau$ to $\wt{\mathfrak{E}}'_\tau$ consists of a sequence of morphisms of involutive instanton-type complexes
\[
\wt{\lambda}_i :\wt{C}_i \to \wt{C}'_i \text{ of level } \epsilon_i
\]
for $i \in \Z$, such that $\epsilon_i \rightarrow \epsilon$ as $i \rightarrow \infty$. A \textit{level $\epsilon$ weakly involutive morphism} consists of a similar sequence of weakly involutive morphisms. In the case that $\epsilon = 0$, we refer to these simply as \textit{involutive morphisms} and \textit{weakly involutive morphisms}, respectively. We sometimes require that 
\[
\smash{\wt{\lambda}_j \psi_i^j \simeq (\psi')_i^j \wt{\lambda}_i}
\]
via a chain homotopy whose level goes to $\epsilon$ as $i, j \rightarrow \infty$. We call such morphisms \textit{coherent}. 
\end{defn}

\begin{comment}
\begin{defn} \label{def:4.15}
A \textit{homotopy equivalence} between two enriched involutive complexes $\mathfrak{E}_\tau$ and $\mathfrak{E}'_\tau$ is a sequence of equivariant homotopy equivalences $f_i$ and $g_i$ making
\[
\wt{C}_i \simeq \wt{C}'_i \text{ of level } \epsilon_i
\]
for $i \in \Z$, such that $\epsilon_i \rightarrow 0$ as $i \rightarrow \infty$. We require $f_j \psi_i^j \simeq (\psi')_i^j f_i$ via a chain homotopy whose level goes to zero as $i, j \rightarrow \infty$, and similarly for the $g_i$.
\end{defn}

\begin{defn} \label{def:4.16}
An {\it (enriched) local map} between two enriched involutive complexes ${\mathfrak{E}}_\tau$ to ${\mathfrak{E}}'_\tau$ is a sequence of weakly involutive local maps 
\[
\lambda_i :\wt{C}_i \to \wt{C}'_i \text{ of level } \epsilon_i
\]
for $i \in \Z$, such that $\epsilon_i \rightarrow 0$ as $i \rightarrow \infty$. 
%{\color{red}which commute with $\psi_i^j$ up to almost filtered chain homotopies}.  
If enriched local maps in both directions exist, then we say $\mathfrak{E}_\tau$ and $\mathfrak{E}'_\tau$ are \textit{(enriched) locally equivalent}. Note that we do not require any commutation requirement with the $\psi_i^j$.
\end{defn}
\end{comment}

level $\epsilon$ homotopy equivalences, local maps, and local equivalences between enriched involutive complexes are defined similarly. For example, a level $\epsilon$ involutive homotopy equivalence consists of a sequence of involutive homotopy equivalences of level $\epsilon_i$, with $\smash{\epsilon_i \rightarrow \epsilon}$. In this case, we require the maps constituting these homotopy equivalences to be coherent. We do not usually require coherence for local maps.\footnote{This is largely a matter of convenience: for the sake of computing invariants such as $r_s$, the coherence condition is unnecessary, and we prefer not to verify it when constructing local maps.} As usual, if we wish to stress the case $\epsilon = 0$, we describe these equivalences as filtered. We may also repeat our definitions to construct enriched involutive under-bar and upper-bar complexes, which we denote by $\un{\mf E}_\tau$ and $\ov{\mf E}_\tau$. Under the same conditions as in Definition~\ref{def:enrichedsubquotient}, we may form
\[
\un{\mf E}_\tau^{[r, s]} \quad \text{and} \quad \ov{\mf E}_\tau^{[r, s]}.
\]
level $\epsilon$ involutive morphisms of enriched under-bar and upper-bar complexes are defined similarly, as are homotopy equivalences, local maps, and local equivalences.

\subsubsection{Tensor products and duals}
We record the tensor product and dualization operations for enriched complexes. The following are easily checked to be enriched complexes:

\begin{defn} \label{def:4.17}
Let $\wt{\mathfrak{E}}_\tau$ and $\wt{\mathfrak{E}'}_\tau$ be two enriched involutive complexes. Then we define
\[
\wt{\mathfrak{E}}_\tau \otimes \wt{\mathfrak{E}}_\tau' = \{ \wt{C}_i\otimes  \wt{C}'_i, \wt{\tau}_i\otimes \wt{\tau}_i'  \}
\]
with the following data: 
\begin{itemize}[noitemsep]
\item the maps $\psi^j_i \otimes (\psi')^j_i$
\item the clustering set is given by $\mathfrak{K}^\otimes = \{ r_1 + r_2 \mid r_1 \in \mathfrak{K}, r_2 \in \mathfrak{K}'\}$. 
\item various homotopies obtained as tensor products of the homotopies from ${\mathfrak{E}}_\tau$ and ${\mathfrak{E}'}_\tau$. 
\end{itemize}
%Naturaly, ${\mathfrak{E}}_\tau \otimes {\mathfrak{E}}_\tau'$ becomes an enriched involutive complex. 
\end{defn}

\begin{defn} \label{def:4.18}
Let $\wt{\mathfrak{E}}_\tau$ be an enriched involutive complex. Then we define 
\[
\wt{\mathfrak{E}}_\tau^* = \{ (\wt{C}_i)^* , \tau_i^* \} 
\]
with the following data: 
\begin{itemize}[noitemsep]
\item the maps $(\psi^j_i)^*$
\item the clustering set is given by $\mathfrak{K}^* = \{ -r \mid r \in \mathfrak{K}\}$. 
\item various homotopies obtained as duals of the homotopies from ${\mathfrak{E}}_\tau$. 
\end{itemize}
\end{defn}

\subsubsection{The enriched $r_s$-invariant}

We now extend the definition of the $r_s$-invariant to the enriched setting. Given an enriched involutive complex $\un{\mathfrak{E}}_\tau$, this essentially consists of taking the limit of $r_s(\un{C}_i, \un{\tau}_i)$ as $i \rightarrow \infty$. (This is not technically correct, since $(\un{C}_i, \un{\tau}_i)$ is not a level zero involutive instanton-type complex. We have not defined $r_s$ for level $\delta$ complexes, since in this case $\un{\tau}$ does not descend to the subquotient for every pair of values.) The following is taken from \cite[Definition 4.22]{ADMT}.

\begin{defn}\label{def:4.22}
Let $\un{\mathfrak{E}}_\tau$ be an enriched involutive complex and $s \in [- \infty, 0]$. If $-s \notin \mathfrak{K}$, define 
\[
r_s( \un{\mathfrak{E}}_\tau) = -\inf_{r < 0 \text{ and } r \notin \mathfrak{K}} \{ r \mid
\text{there exists an equivariant $\theta$-supported cycle in }\mathfrak{E}_\tau^{[r, -s]} \} \in [0,\infty],
\]
with the caveat that if the above set is empty, we set $r_s(\un{\mathfrak{E}}_\tau) = - \infty$. For $-s \in \mathfrak{K}$, we define 
\[
r_s( \un{\mathfrak{E}}_\tau) = \lim_{t\to s^-}r_t( \un{\mathfrak{E}}_\tau).
\]
Note that the right-hand side is eventually constant due the discussion in Definition~\ref{def:enrichedsubquotient}.
\end{defn}

The properties listed in Lemma~\ref{bb} also hold for the $r_s$-invariant in the enriched setting. Here, the trivial enriched involutive complex consists of the constant sequence of trivial involutive complexes, equipped with the structure maps $\psi^j_i = \id$.

\begin{lem}
\label{lem:properties} 
Let $\un{\mf E}_\tau$ and $\un{\mf E}'_\tau$ be enriched involutive under-bar complexes. The following hold:
\begin{enumerate}[noitemsep]
    \item\label{properties1}
    There exists a level $\delta$ involutive local map from the trivial complex into $\un{\mf E}_\tau$ if and only if 
\[
r_{-\delta}(\un{\mf E}_\tau)=\infty.
\]
    \item\label{properties2} If there exists a level $\delta$ involutive local map from $\un{\mf E}_\tau$ to $\un{\mf E}'_\tau$, then
    \begin{align*}
&r_s(\un{\mf E}_\tau) \le r_{s-\delta} (\un{\mf E}'_\tau) + \delta, \quad \forall s \in [-\infty, 0].
\end{align*}
  \item\label{properties3} We have a connected sum inequality
    \[
    r_{s+ s'}( \wt{\mf E}_\tau\otimes \wt{\mf E}'_\tau ) \geq \min \{ r_s (\wt{\mf E}_\tau)-s', r_{s'} (\wt{\mf E}'_\tau)-s\}, \quad \forall s, s' \in [-\infty, 0].
    \]
\end{enumerate}

\end{lem}

\begin{proof}
These follow from \cref{bb} combined with a standard limiting argument. See for example \cite[Section 4.5]{ADMT}.   
\end{proof}

\begin{rem}
The analogous statements as in Remark~\ref{rem:otherrsproperties} hold for $r_s(\wt{\mf E}_\tau)$.
\end{rem}

\subsection{Geometric construction of enriched involutive complexes}

We now explain how to obtain an enriched involutive complex from the data of a based strongly invertible knot $(Y, K, \tau, x_0)$. We adopt the same notation as in Section~\ref{sec:geometricconstruction}. To begin with, note that the gauge group $\G(Y, K)$ has connected components
\[
\pi_0(\G(Y, K)) \cong \Z\oplus \Z.
\]
This bijection is given by sending a gauge tranformation $g$ to $(\deg(g), \deg_K(g))$, where $\deg (g) \in \Z$ is the degree of $g : Y \to \SU(2)=S^3$ and $\deg_K(g) \in \Z$ is the degree of $g|_K : K \to U(1) \subset \SU(2)$.
Applying $g$ changes the value of the (lifted) Chern--Simons functional and (lifted) Floer grading by
\begin{align*}
\operatorname{cs} (g^* a )& = \operatorname{cs}(a) +2 \deg (g) + \deg_K (g) \\
\ind_\Z (g^* a )&  = \ind_\Z(a) + 8\deg (g) + 4\deg_K (g).
\end{align*}
We thus consider the normal subgroup of gauge transformations
\[
\G_*(Y, K) := \{ g \in \G(Y, K) \mid 2\deg (g) + \deg_K(g) =0  \} \subset \G(Y, K) .
\]
Quotienting out by this restricted gauge group gives $\Z$-covers 
\[
\wt{\B}(Y, K) = \A(Y, K)/ \G_*(Y, K) \quad \text{and} \quad \wt{\B}^\ast(Y, K) = \A^\ast(Y, K)/ \G_*(Y, K)
\]
of the full quotients ${\B}(Y, K)$ and ${\B}^\ast(Y, K)$, respectively. 

\subsubsection{Instanton-type complexes}
We now define a singular instanton knot complex using the restricted gauge group $\G_*$. Fix a $\pi$ for the Chern--Simons functional on $\smash{\wt{\B}(Y, K)}$ which makes it Morse--Smale 
\[
\mathrm{cs} + f_{\pi}: \wt{\mathcal{B}}^* (Y, K) \to \R.
\] 
We assume this to be a lift of a perturbation for the Chern--Simons functional on $\B(Y, K)$. Repeating the construction of Section~\ref{sec:geometricconstruction}, we form the module  
\[
	C(Y,K, \pi)  = \bigoplus_{[\mf a]\in {\wt{\mathfrak{C}}}^{*}_{\pi}(Y,K)} R\cdot T^{\mathrm{hol}_{(Y, K)}(\mf a)}[\mf a],
\]
where $\wt{\mathfrak{C}}^{*}_{\pi}(Y,K)$ is now the critical set of the perturbed functional on $\wt{\B}(Y, K)$. (We abuse notation and use the same notation for this module as in Section~\ref{sec:geometricconstruction}.) The covering $\Z$-action on ${\mathfrak{C}}^{*}_{\pi}(Y,K)$ gives $C(Y,K, \pi)$ the structure of a $R[y^{\pm 1}]$-module; indeed, choosing any lift of $\mathfrak{C}^{*}_{\pi}(Y,K)$ to $\smash{\wt{\mathfrak{C}}^{*}_{\pi}(Y,K)}$ identifies $C(Y,K, \pi)$ as the free $R[y^{\pm 1}]$-module
\[
	C(Y,K, \pi)  = \bigoplus_{[\mf a]\in {\mathfrak{C}}^{*}_{\pi}(Y,K)} R[y^{\pm 1}]\cdot T^{\mathrm{hol}_{(Y, K)}(\mf a)}[\mf a].
\]
On this cover, the original absolute $\Z_4$-grading lifts to the absolute $\Z$-grading. Note that $y^{\pm 1}$ changes the Chern--Simons grading by $\pm 1$ and the $\Z$-grading by $\pm 4$. The construction of the differential (as well as the dependence on a choice of basepoint $x_0$) are as before. It is straightforward to see that the differential is filtered with respect to the values of the perturbed Chern--Simons functional. 

\begin{defn}
Let $K$ be a knot in a homology sphere $Y$ and fix a basepoint $x_0$ on $K$. For any perturbation $\pi$ as above, define
\[
\wt{C}_*(Y,K, \pi, x_0):= {C}_*(Y,K, \pi, x_0) \oplus {C}_{*-1} (Y,K, \pi, x_0) \oplus R[y^{\pm 1}].
\]
Together with the perturbed Chern--Simons functional, this defines an instanton-type complex. 
\end{defn}

The \textit{filtered} homotopy type of this complex is \textit{not} independent of $\pi$. Setting $y=1$ and forgetting the Chern--Simons filtration gives the usual $\mathcal{S}$-complex of Section~\ref{sec:geometricconstruction}.

\begin{rem}
In \cite{DS20}, the variable $U^{\pm 1}$ is used instead of $y^{\pm 1}$; here, we follow the notation of \cite{ADMT}. 
\end{rem}

\subsubsection{Enriched complexes} Now let $(Y, K, \tau, x_0)$ be a based strongly invertible knot. As before, consider the cobordism
\[
(W_\tau, S_\tau) : (Y, K) \to (Y, K^r).
\]
To define an enriched involutive complex $\wt{\mathfrak{C}}_\tau(Y,K)$, we take 
\begin{itemize}[noitemsep]
\item a sequence of nondegenerate regular perturbation $\pi_i$ such that 
\[
\| \pi _i \|\to 0 \text{ as } i \to \infty  ,
\]
\item an orbifold Riemannian metric $\tilde{g}$ on $(W^+_\tau, S^+_\tau )$ which concides $dt^2 + g$ on the ends for a choice of orbifold Riemannian metric for $(Y, K)$,
\item a sequence of regular perturbation $\tilde{\pi}_i $ of instantons over $W_\tau$ such that 
\[
\| \tilde{\pi }_i \|\to 0 \text{ as } i \to \infty.
\]
\end{itemize}

For each fixed $i$, we form the instanton-type complex $\wt{C}(Y, K, \pi_i, x_0)$. On this complex, we define the action of $\wt{\tau}$ using the above data, just as in Section~\ref{sec:geometricinvolutive}. This gives a sequence of morphisms
\[
\wt{\tau}_i : \wt{C} (Y, K, \pi_i, x_0) \to \wt{C} (Y, K, \pi_i, x_0)
\]
In general, $\wt{\tau}_i$ will not be filtered, but will rather have level $\delta$ with $\delta$ depending on the norm of the perturbation $\tilde{\pi}_i$ over $W_\tau$. Moreover, $\wt{\tau}_i$ will be a homotopy involution, with the homotopy having level $\delta$ with $\delta$ depending on the size of the $1$-parameter family of perturbations used in the proof of Theorem~\ref{involutivecomplexproof}. Hence
\[
(\wt{C}_i = \wt{C}(Y, K, \pi_i, x_0), \wt{\tau}_i)
\]
constitutes an involutive instanton complex of some level $\delta_i$.
Homotopies 
\[
\psi_i^j : \wt{C} (Y, K, \pi_i) \to \wt{C} (Y, K, \pi_j)
\]
are defined by counting 1-parameter instanton moduli spaces for families of perturbations connecting $\pi_i$ and $\pi_j$; each homotopy has level $\delta_{i,j}$ depending on the norm of the corresponding perturbation.
By taking suitable 2-parameter families of perturbations and using the associated parametrized singular instanton moduli spaces, one can verify that $ \psi^k_j\circ \psi_i^j$ is chain homotopic to $\psi_i^k$ via a chain homotopy of some level $\delta_{i,j,k}$.

\begin{defn}
Define an enriched involutive complex by taking our sequence of level $\delta_i$ involutive instanton complexes
\[
\wt{\mathfrak{E}}_\tau(Y, K, x_0) = \{(\wt{C}_i, \wt{\tau}_i)\}_{i \in \N}, 
\]
together with the structure maps $\psi^j_i$ and chain homotopies between $ \psi^k_j\circ \psi_i^j$ and $\psi_i^k$ discussed above. We define $\mathfrak{K}$ as the set of critical values of the unperturbed singular Chern--Simons functional 
\[
\operatorname{cs} : \wt{\B}^* (Y, K) \to \R. 
\]
A standard argument as in \cite[Section 5.3]{ADMT} shows that this is indeed an enriched involutive complex whose homotopy class is an invariant of $(Y, K, \tau, x_0)$. As usual, the weak homotopy class of $\smash{\wt{\mathfrak{E}}_\tau ( Y, K, x_0)}$ is independent of the choice of basepoint. 
\end{defn}

A similar construction defines our cobordism maps in the enriched category.

\begin{thm}\label{nagative definte cobordism induces_enriched}
Let $(W, S)$ be an equivariant negative definite pair from $(Y, K, \tau)$ to $(Y', K', \tau')$. Choose basepoints $x_0$ and $x_0'$ for $K$ and $K'$, and let $\gamma$ be a properly embedded path in $S$ from $x_0$ to $x_0'$. Then we obtain a level $2\kappa_{\rm{min}}(W, S)$ weakly involutive cobordism map 
\[
\wt{\lambda}_{(W,S)} \colon \wt{\mf E}_\tau(Y, K, x_0) \rightarrow \wt{\mf E}_\tau(Y', K', x_0').
\]
If $\gamma$ and $\tau_W(\gamma)$ are homotopic rel boundary in $S$, then $\wt{\lambda}_{(W,S)}$ is involutive. If $(W, S)$ is strong, then the induced morphism is local.
\end{thm}

\begin{proof}
We construct a sequence of morphisms by choosing a sequence of perturbations on $(W, S)$.
Let $\epsilon_i\to 0$ be the error coming from the $i$-th perturbation.
If a singular instanton $A$ for $(W, S)$ counted by the $i$-th cobordism map goes from $\mf a$
to $\mf b$, then Stokes' theorem gives
\[
        \operatorname{cs}_{\pi_i'}(\mf b)
        - \operatorname{cs}_{\pi_i}(\mf a)
        \leq
        2\kappa_{\rm min}(W,S)+\epsilon_i , 
\]
where $\operatorname{cs}_{\pi} $ denotes the perturbed orbifold Chern--Simons functional with respect to the perturbation $\pi$.
Here the factor \(2\) is due to our normalization of the orbifold
Chern--Simons filtration. This implies that the \(i\)-th cobordism map
has filtration shift at most \(2\kappa_{\rm min}(W,S)+\epsilon_i\).
Passing to the limiting enriched morphism therefore gives level
\(2\kappa_{\rm min}(W,S)\).
Each of these homotopy commutes with the relevant $\wt{\tau}_i$ by the
argument of \cref{nagative definte cobordism induces}. Details are left to the reader; see \cite[Section 2.3]{NST19}.
\end{proof}

Applying our algebraic machinery to $\wt{\mathfrak{E}}_\tau (Y, K)$, we have:

\begin{defn}
Let $(Y, K, \tau)$ be a strongly invertible knot. Define
\[
r_s (Y, K, \tau) := r_s(\un{\mathfrak{E}}_\tau (Y, K)) \in [0, \infty]. 
\]
In the case that $Y = S^3$, we simply write $r_s(K, \tau)$.
\end{defn}

In the case of an equivariant $\Z_2$-homology concordance, Theorem~\ref{nagative definte cobordism induces_enriched} induces a filtered weak local map. It follows from Lemma~\ref{lem:properties} that $r_s (Y, K, \tau)$ is invariant under equivariant $\Z_2$-homology concordance.

\subsection{Enriched local equivalence}
We now construct the local equivalence and weak local equivalence sets in the (involutive) enriched category. These should be thought of as containing the maximal amount of equivariant concordance information derived from our involutive singular instanton theory.

\begin{defn}\label{rloc_eq_grp}
Define the \textit{(enriched) local} and \textit{weak local equivalence sets} by
\[
\Theta^{\mathfrak{E}} _{eq} = \left\{ \text{enriched involutive complexes} \right\} / \sim_{eq} \quad \text{and} \quad
\Theta^{\mathfrak{E}} _{w.eq} = \left\{ \text{enriched involutive complexes} \right\} / \sim_{w.eq}
\]
where we write $\wt{\mf{E}}_\tau \sim_{eq} \wt{\mf{E}}_\tau'$ or $\wt{\mf{E}}_\tau \sim_{w.eq} \wt{\mf{E}}_\tau'$ if $\wt{\mf{E}}_\tau$ and $\wt{\mf{E}}_\tau'$ are (filtered) locally equivalent or (filtered) weakly locally equivalent, respectively. It is clear that these are equivalence relations.
\end{defn}

\begin{defn}
Define
\[
h_\tau^f \colon \wt{\mathcal{C}} \rightarrow \Theta^{\mathfrak{E}} _{eq}
\]
by sending
\[
(K, \tau, x_0) \mapsto h^f_\tau(K, \tau, x_0) = [\wt{\mathfrak{E}}_\tau  ( S^3, K,x_0)].
\]
As in Definition~\ref{def:htau}, we write the domain of $h_\tau^f$ as $\wt{\mathcal{C}}$ for convenience, but this is misleading. We pass from the data of a direction to a basepoint by (for example) taking the endpoint of the direction. We do \textit{not} assert that $h_\tau^f$ is a group homomorphism. We likewise use the same notation $h^f_\tau(Y, K, \tau, x_0)$ for the local equivalence class $[\wt{\mf E}_\tau(Y, K, x_0)]$.
\end{defn}

As before, there is a natural partial order on $\Theta^{\mathfrak{E}} _{eq}$ and $\Theta^{\mathfrak{E}}_{w.eq}$, where we write $\smash{\wt{\mf E}_\tau \leq \wt{\mf E}_\tau'}$ if there exists a filtered local map (respectively, a filtered weakly local map) from $\smash{\wt{\mf E}_\tau}$ to $\smash{\wt{\mf E}_\tau'}$. We have a forgetful map compatible with this partial order:
\[
\Theta^{\mathfrak{E}} _{eq} \rightarrow \Theta^{\mathfrak{E}} _{w.eq}.
\]
Local equivalence of under-bar and upper-bar enriched complexes factors through $\Theta^{\mathfrak{E}} _{w.eq}$. Just as in the case for $\mathcal{S}$-complexes, it is possible to put an abelian group structure on $\smash{\Theta^{\mathfrak{E}}_{eq}}$. The identity element is given by the sequence of trivial involutive instanton-type complexes (as in Lemma~\ref{lem:properties}), while the group operation is given by tensor product and inverses are given by dualizing. However, we will not use this group structure.

The map $h_\tau^f$ fits into the commutative diagram
\[\begin{tikzcd}
	{\widetilde{\mathcal{C}}} & {\Theta^{\mathfrak{E}}_{eq}} & {\Theta^S_{eq}} \\
	{\mathcal{C}} & {\Theta^{\mathfrak E}} & {\Theta^S}
	\arrow["{h^f_\tau}"', from=1-1, to=1-2]
	\arrow["{h_\tau}", curve={height=-18pt}, from=1-1, to=1-3]
	\arrow[from=1-1, to=2-1]
	\arrow["{\mathrm{forget}}"',from=1-2, to=1-3]
	\arrow[from=1-2, to=2-2]
	\arrow[from=1-3, to=2-3]
	\arrow["{h^f}", from=2-1, to=2-2]
	\arrow["h"', curve={height=18pt}, from=2-1, to=2-3]
	\arrow["{\mathrm{forget}}",from=2-2, to=2-3]
\end{tikzcd}\]
Here, the map $h^f$ is the usual homomorphism to the non-involutive enriched local equivalence group $\Theta^{\mathfrak{E}}$ constructed by Daemi and Scaduto in \cite[Section 7.4]{DS19}. The two unmarked horizontal maps are the forgetful maps obtained by ignoring the data of the Chern--Simons filtration. The vertical maps are obtained by passing from the involutive to the non-involutive category. 

Note also that we may replace the local equivalence sets along the top row with their weak local equivalence counterparts; since weak involutive local equivalence only differs from involutive local equivalence via a condition involving $\wt{\tau}$, the vertical forgetful maps in fact factor through the weak local equivalence sets. As we have seen, the $r_s$-invariant factors through $h^f_\tau$, and in fact further factors through the forgetful map:
\[\begin{tikzcd}
	{\widetilde{\mathcal{C}}} & {\Theta^{\mathfrak{E}}_{eq}} & {\Theta^{\mathfrak{E}}_{w.eq}} & {[0, \infty]}
	\arrow["{h^f_\tau}", from=1-1, to=1-2]
	\arrow[from=1-2, to=1-3]
	\arrow["{r_s}", from=1-3, to=1-4]
\end{tikzcd}\]

\subsection{Properties of $r_s$}
We close this section by giving the enriched analogues of our partial connected sum formulas and proving Theorem~\ref{thm:rsproperties}.

\subsubsection{Tensor products}
To extend our partial connected sum formulas to the enriched setting, we simply keep track of the filtration levels of the pants cobordisms used in the proofs.

\begin{thm}\label{enrich_conn sum}
Let $(Y, K , \tau)$ and $(Y', K' , \tau')$ be strongly invertible knots. Fix invariant basepoints $x_0$ and $x_0'$ along which to form the equivariant connected sum and let $x^\#$ be either one of the two invariant basepoints on $K \# K'$. Then we have a (filtered) weak homotopy equivalence
    \[
    \wt{\mathfrak{E}}_\tau(Y, K, x_0)\otimes \wt{\mathfrak{E}}_{\tau'}(Y', K', x_0') \simeq_{w.eq} \wt{\mathfrak{E}}_{\tau\# \tau'}(Y\# Y', K\# K', x^{\#}).
    \]
\end{thm}

\begin{proof}
In the enriched setting, we upgrade the pants cobordism maps used in the proof of \cref{conn sum}
       \begin{align*}
  &  \wt{\lambda} := \wt{\lambda}_{(W, S )}: \wt{C}(Y, K, x_0) \otimes \wt{C}(Y', K', x_0') \to \wt{C}(Y\# Y', K\# K', x^{\#})   \\ 
   &     \wt{\lambda}' := \wt{\lambda}_{(W^\dagger, S^\dagger )}: \wt{C}(Y\# Y', K\# K', x^{\#}) \to \wt{C}(Y, K, x_0) \otimes \wt{C}(Y', K', x_0') 
    \end{align*}
to a pair of sequences of cobordism maps. Each of these is defined by choosing perturbations on the ends and a perturbation over the cobordism. In addition, we choose perturbations of various $1$-parameter families in order to define the homotopies $\smash{\wt{H}}$ and $\smash{\wt{G}}$ used in the proof, as well as the fact that our cobordism maps are homotopy inverse. Taking a sequence such that all perturbations go to zero gives the claim. Indeed, for the unperturbed pants cobordism, the minimal reducible has minimizing class $z = 0$ and \(S\cdot S=0\). Hence Stokes' theorem gives
\[
        \operatorname{cs}_{Y\# Y', K\# K'}(\mf c)
        -
        \operatorname{cs}_{Y, K}(\mf a)
        -
        \operatorname{cs}_{Y', K'}(\mf b)
        \leq 0
\]
for any unperturbed solution to instanton equations counted by \(\wt{\lambda}\), and similarly for
\(\wt{\lambda}'\), where $\operatorname{cs}_{Y, K}$ denotes the orbifold Chern--Simons functional for $(Y, K)$.
  After perturbing, the corresponding shifts are bounded by
\(\epsilon_i\), with \(\epsilon_i\to 0\), and hence the resulting enriched
morphisms are level zero.
\end{proof}

For the tensor product formula involving the flipping involution, we first note that the algebraic flip map can clearly be defined in the enriched setting. As before, fix a basepoint $x_0$ on $K$ and use the same basepoint on $K^r$. We perform the connected sum of $K$ and $K^r$ at $x_0$ and let $x^\#$ be either of the two fixed points of the flipping involution $\tau_f$ on $K \# K^r$.

\begin{defn}
Define an enriched involutive complex $\wt{\mf E}_{\tau,\mathrm{flip}}$ as follows. Choose a sequence of perturbations $\pi_i$ converging to zero and consider the sequence of instanton-type complexes
\[
\wt{C}_i \otimes \wt{C}_i = \wt{C}(Y, K, x_0, \pi_i) \otimes \wt{C}(Y, K, x_0, \pi_i).
\]
On each complex in this sequence, consider the algebraic flip map 
    \[
    \operatorname{alg}_{\tau_i, f}: \wt{C}_i \otimes \wt{C}_i \to \wt{C}_i \otimes \wt{C}_i.
    \]
It is clear that this constitutes a level zero involution and makes the tensor product into an involutive instanton-type complex. We equip this sequence of complexes with the structure maps 
    \[
    \psi^j_i \otimes \psi^j_i: \wt{C}_i \otimes \wt{C}_i \to \wt{C}_j \otimes \wt{C}_j,
    \]
where $\psi^j_i \colon \wt{C}_i \rightarrow \wt{C}_j$ is defined as usual by counting $1$-parameter instanton moduli spaces for families of perturbations connecting $\pi_i$ and $\pi_j$.
\end{defn}

\begin{thm}\label{enrich_conn sum2}
We have a (filtered) weak homotopy equivalence
\[
\wt{\mf E}_{\tau,\mathrm{flip}} \simeq_{w.eq.} \wt{\mf E}_{\tau_f}(Y \# Y^r, K\#K^r, x^{\#}).
\]
\end{thm}
\begin{proof}
The proof is the same as that of Theorem~\ref{enrich_conn sum}.
\end{proof}

\subsubsection{Proof of Theorem~\ref{thm:rsproperties}}
We now finally prove Theorem~\ref{thm:rsproperties}. In fact, most of the parts of Theorem~\ref{thm:rsproperties} follow immediately from our work so far; the only remaining part is \eqref{rsproperties:4}, together with the strict inequality in \eqref{rsproperties:2}. We explain the first of these.

%We compare our $r_s$-invariant with the invariant introduced in \cite{ADMT}.  
\begin{lem}\label{relation_with_old_def}
Let $\U$ be a local strongly invertible unknot in an equivariant homology sphere $(Y, \tau)$. Then we have
    \[
    r_s(Y , \U, \tau) \leq 2r_{2s}(Y, \tau), \quad \forall s\in [-\infty, 0], 
    \]
    where the right-hand side is the involutive $r_s$-invariant introduced in \cite{ADMT}. 
\end{lem}

\begin{proof}
Suppose $S$ is an unknotted disc in $W\cong I\times Y$ that fills the unknot $\U$ in $\{t_0\}\times Y$.
Then $(W,S)\colon (Y,\U)\to (Y,\emptyset)$ is a cobordism of pairs, which admits a reducible singular flat connection with index $-1$. 
Note that $\tau$ extends to $[0,1]\times Y$ and one can assume $\tau(S)=S$.
The moduli space of (perturbed) singular ASD connections on $(W,S)$ which are asymptotic to irreducible (perturbed) flat connections and have dimension zero gives rise to a chain map
\[
\phi \colon C_*(Y,\U;\, R )\longrightarrow C_*(Y;\,\Z_2),
\]
where $(C_*(Y;\,\Z_2),d_Y)$ is Floer’s instanton complex over $\Z_2$.
Considering the $0$-dimensional moduli spaces on $(W,S)$ which are asymptotic to an irreducible (perturbed) flat connection on the end $(Y,\U)$ and the trivial connection on the end $Y$
determines a map 
\[
D_1 \colon C^*(Y,\U;\,R)\longrightarrow R
\]
which satisfies the following property:
\begin{align}\label{phi equivalence between YU and Y1}
\mathfrak{D} d + \delta_1 + D_1 \phi=0    
\end{align}
as is proven in \cite[Proof of Proposition 15]{DS20}. 
Moreover, since $\tau$ extends to $[0,1]\times Y$, certain counting of moduli spaces implies $\phi $ is $\tau$-equivariant up to filtered homotopy.
 As the map $\phi$ preserves the Chern--Simons filtrations,  this shows an inequality
\begin{equation}\label{inequality of rs}
	r_s(Y,\U, \tau) \leq 2 r_{2s}(Y,\tau ).
\end{equation}
    The factors of $2$ come from the conventions used to define the Chern--Simons functional.
\end{proof}

\begin{proof}[Proof of Theorem~\ref{thm:rsproperties}]
The first claim is just a special case of Lemma~\ref{lem:properties}\eqref{properties1}. The inequality of the second claim follows from Theorem~\ref{nagative definte cobordism induces_enriched} combined with Lemma~\ref{lem:properties}\eqref{properties2}. The strict inequality in the claim is a standard argument in filtered instanton theory; see for example \cite[Proof of Theorem 1.1 and Lemma 6.1]{ADMT}. The third claim follows from Theorem~\ref{enrich_conn sum} combined with Lemma~\ref{lem:properties}\eqref{properties3}. Finally, the fourth claim is Lemma~\ref{relation_with_old_def}.
\end{proof}

\begin{rem}
Theorem~\ref{thm:rsproperties}\eqref{rsproperties:2} is the involutive analogue of \cite[Corollary 5.54]{DISST22}. There is a typo in \cite[Corollary 5.54]{DISST22}, which claims 
    \[
    r_{s-2\kappa(W,S)}(Y, K) \leq r_{s}(Y', K')+2\kappa(W,S). 
    \]
    This should be 
     \[
    r_{s}(Y, K) \leq r_{s-2\kappa(W,S)}(Y', K')+2\kappa(W,S)
    \]
    as it is derived from the inequality \cite[Theorem 5.50]{DISST22}. 
\end{rem}

\begin{comment}
Finally, we state the connected sum inequality for the $r_s$-invariant. 
\begin{thm}\label{general_conn_sum_r_s}
 Let $(Y, K, \tau)$ and $(Y', K', \tau')$ be two strongly invertible knots in oriented homology 3-spheres $Y$ and $Y'$ with choices of base points. 
For any $s$ and $s' \in [-\infty,0]$, we have
\[
r_{s+s'}(Y\# Y', K\# K', \tau\# \tau')
\ge \min \{ r_s(Y, K, \tau) + s',\ r_{s'}(Y', K', \tau') + s \}.
\]
\end{thm}
\begin{proof}
From \cref{enrich_conn sum}, we see 
two enriched involutive $\mathcal{S}$-complexes 
    \[
    \wt{\mathfrak{E}}_\tau(Y, K)\otimes \wt{\mathfrak{E}}_{\tau'}(Y', K') \text{ and } \wt{\mathfrak{E}}_{\tau\# \tau'}(Y\# Y', K\# K')
    \]  
    are enriched weakly involutive local equivalent. 
    Therefore, we have 
        \[
   r_s( \wt{\mathfrak{E}}_\tau(Y, K)\otimes \wt{\mathfrak{E}}_{\tau'}(Y', K') ) = r_s(\wt{\mathfrak{E}}_{\tau\# \tau'}(Y\# Y', K\# K'))
    \] 
    for any $s \in [-\infty, 0]$ since the $r_s$-invariant is invariant under enriched weakly involutive local equivalece, as it is stated in \cref{weak_equivalence_determine}.  Now we apply the second item of \cref{lem:properties} to get 
    \[
     r_s( \wt{\mathfrak{E}}_\tau(Y, K)\otimes \wt{\mathfrak{E}}_{\tau'}(Y', K') )\ge \min \{ r_s( \wt{\mathfrak{E}}_\tau(Y, K)) + s',\ r_{s'}(\wt{\mathfrak{E}}_{\tau'}(Y', K')) + s \}
    \]
    which gives the conclusion. 
\end{proof}
\end{comment}

\section{Applications}\label{sec:applications}

We now prove several of the applications described in the introduction. These will follow quickly from the formal properties listed in Theorem~\ref{thm:rsproperties}, together with a surgery cobordism relating the $r_s$-invariant of a strongly invertible knot to the $r_s$-invariants of its surgeries. We then provide an explicit example of a knot whose $r_s$-invariants satisfy the hypotheses of Theorem~\ref{thm:intro-whitehead}. This computation will rely on the singular instanton knot formalism developed over the last several sections and will make crucial use of our partial connected sum formulas. 

\subsection{Corks and exotic disks} We begin with the most straightforward application of the $r_s$-invariant. As we have seen, the $r_s$-invariant obstructs isotopy-equivariant sliceness. Theorem~\ref{thm:intro-onesign} is a simple generalization of this principle.

\introonesign*

\begin{proof}
    Let $D \subset W$ be an $H$-slice disk for $K$. It is straightforward to check that if $\sigma(K) = 0$, then $(W ,D)$ is a strong negative definite pair. If $D$ were isotopy-equivariantly slice, then viewing this as an equivariant cobordism from the unknot to $K$ and using Theorem~\ref{thm:rsproperties}\eqref{rsproperties:2} gives  
    \[
    \infty \leq r_s (K, \tau) 
    \]
    for all $s\in [0, \infty ]$. Here, we use the fact that the $r_s$-invariant of the unknot is easily seen to be $\infty$. The result on positive definite 4-manifolds follows from reversing orientation on $W$.
\end{proof}

We now show that the $r_s$-invariant obstructs surgeries along $K$ from bounding corks. For this, we have the following lemma, which may be viewed as a partial surgery formula. 

\begin{lem}\label{prop:surgery_formula}
Let $(Y, K, \tau, x_0)$ be a based strongly invertible knot with $\sigma(Y, K)=0$. Let $Y_1(K)$ be $(+1)$-surgery on $K$, equipped with the involution induced by $\tau$. Then there is a level $1/8$ local map 
\[
\wt{\mathfrak{E}}_\tau(Y_1(K), \U, y_1) \to \wt{\mathfrak{E}}_\tau (Y, K,x_0). 
\]
Here, $\U$ is a local strongly invertible unknot in $Y_1(K)$ and $y_1$ is a fixed point of the involution on $\U$.
\end{lem}
\begin{proof}
Consider the 4-manifold with boundary obtained by adding a 1-framed 2-handle to 
\([0,1] \times Y\) along \(\{1\} \times K\). We turn this around to obtain a negative definite cobordism $W$ from $Y_1(K)$ to $K$. The strong inversion $\tau$ extends over $W$ such that the action on the second homology is multiplication by $-1$ and preserves the 2-handle core setwise; see for example \cite[Section 5]{DHM20}. Let \(S\) be the annulus formed by taking the interior connected sum of the core of the 2-handle with a small 2-disk whose boundary is a strongly invertible unknot 
\(\U \subset Y_1(K)\) near the fixed point set. It is clear that $S$ may be assumed to be fixed by the extension. We thus obtain an equivariant cobordism
\[
(W,S) \colon (Y_1(K),\U) \to (Y,K).
\]
Moreover, we can clearly take an invariant properly embedded path $\gamma: [0,1]\to S \subset W$ connecting $x_0$ with one of the two fixed points on $\U$. A brief calculation shows that in this case 
\[
\kappa(W, S) = 1/16, \quad \ind(W, S) =  - \sigma(Y,K)/2-1, \quad \text{and} \quad \eta(W,S) = 1.
\]
Thus, if $\sigma(Y, K) = 0$, then $(W, S)$ is a strong equivariant negative definite pair. By \cref{nagative definte cobordism induces_enriched}, we obtain the desired local map. 
\end{proof}

This immediately gives an inequality of $r_s$-invariants.

\begin{lem}\label{r_0_surgery_ineq}
Let $(Y, K, \tau)$ be a strongly invertible knot with $\sigma(Y, K)=0$. Let $Y_1(K)$ be $(+1)$-surgery on $K$, equipped with the involution induced by $\tau$. Then
\[
r_s(Y_1(K), \U, \tau) \leq  r_{s-1/8} (Y, K, \tau)+ 1/8, 
\]
where $\U$ is a local strongly invertible unknot in $Y_1(K)$.
\end{lem}
\begin{proof}
This follows immediately from Lemma~\ref{prop:surgery_formula} and Theorem~\ref{thm:rsproperties}\eqref{rsproperties:2}.
\end{proof}

%\subsection{Finding linearly independent sequences of strong corks}

Our claim regarding corks quickly follows. 

\corkindependent*

\begin{proof}
    Write $\smash{Y_n = S^3_{1/n}(K)}$. By \cref{r_0_surgery_ineq}, we have
    \[
    r_0(Y_1, \U, \tau) \leq r_{-1/8}(Y, K, \tau) + 1/8 < \infty.
    \]
    As in \cite[Lemma 1.5]{ADMT}, we construct a simply connected, negative definite equivariant cobordism $W_n$ from $Y_n$ to $Y_{n-1}$.

        \begin{figure}[h!]
  \centering
  \includegraphics[width=0.5\linewidth]{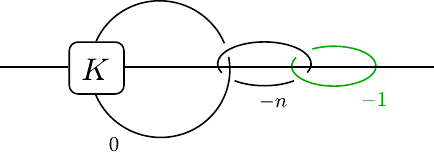}
  \caption{Equivariant cobordism from $1/n$-surgery to $1/(n-1)$. The symmetry is with respect to the horizontal axis.}
  \label{equi_cobor_raional}
\end{figure}

It is straightforward to equip this with an invariant annulus $S_n$ running from a local unknot $\U \subset Y_n$ to a local unknot $\U \subset Y_{n-1}$ and check that this is a strong negative definite pair. Clearly, $\pi_1(W_n \ \setminus \ S_n) = \Z$. Hence by Theorem~\ref{thm:rsproperties}\eqref{rsproperties:2},
    \[
    \cdots < r_0(Y_3, \U, \tau) < r_0(Y_2, \U, \tau) < r_0(Y_1, \U, \tau) < \infty.
    \]
    Moreover, observe that we have an equivariant positive definite cobordism from the unknot in $S^3$ to $(Y_n, \U)$, living inside the $n$-framed $2$-handle attachment along $K$. Reversing orientations gives an equivariant negative definite cobordism from the unknot in $S^3$ to $(-Y_n, \U)$, which is easily checked to be a strong negative definite pair. Thus
    \[
    \infty \leq r_0(-Y_n, \U, \tau).
    \]
The same argument as in \cite[Theorem 7.7]{ADMT} using the connected sum inequality of Theorem~\ref{thm:rsproperties}\eqref{rsproperties:3} then shows that any nontrivial linear combination of $(Y_n, \U, \tau)$ has a nontrivial $r_0$-invariant. In particular, any nontrivial linear combination of the $Y_n$ constitutes a strong cork. 
 \end{proof}

We now turn to the proof of Theorem~\ref{thm:intro-whitehead}. For this, we require a useful lemma regarding the $r_s$-invariant of $K \# K$.

\begin{lem}\label{aaa}
Suppose $r_{-1/8}(K, \tau)< \infty$ and $r_{0}(-K, -\tau)= \infty$.
Then $r_{-1/8} (K\# K, \tau \# \tau) < \infty$.
\end{lem}

\begin{proof}
Apply the connected sum inequality of Theorem~\ref{thm:rsproperties}\eqref{rsproperties:3} with the pairs $(K \# K, \tau \# \tau)$ and $(-K, -\tau)$ and parameters $s = -1/8$ and $s' = 0$. This gives
\[
r_{-1/8}((K \# K) \# -K, (\tau \# \tau)\# -\tau) \geq \min \{r_{-1/8}(K\# K, \tau\# \tau ), r_{0}(-K, -\tau)+1/8\}.
\]
Since $((K \# K) \# -K, (\tau \# \tau)\# -\tau)$ is equivariantly concordant to $(K, \tau)$, this shows that 
\[
\infty > r_{-1/8}(K, \tau) \geq \min \{r_{-1/8}(K\# K, \tau\# \tau ), r_{0}(-K, -\tau)+1/8\}.
\]
Since the second term in the minimum is $\infty$, we conclude that $r_{-1/8}(K\# K, \tau\# \tau ) < \infty$, as desired.
\end{proof}

We now finally obtain:  
\introwhitehead*

\begin{proof}
Consider the branched cover $\Sigma_2(\Wh_i(K))$. We have an involution on this manifold obtained by lifting the strong inversion on $\Wh_i(K)$ induced by $\tau$. As discussed in \cref{lem:obvious}, to show that $\Wh_i(K)$ is not isotopy-equivariantly slice, it suffices to show that $\Sigma_2(\Wh_i(K))$ is a strong cork. By \cref{lem:liftbranched1}, we have
    \[
    \Sigma_2(\Wh_i(K)) = S^3_{1/(2i)}(K \# K^r).
    \]
This identification intertwines the involution on $\Sigma_2(\Wh_i(K))$ with the involution on the surgery $S^3_{1/(2i)}(K \# K^r)$ induced by $\tau \# \tau^r$. 
Using Lemma~\ref{aaa} and the fact that the instanton complex is unchanged by orientation reversal, we have
\[
r_{-1/8}(K \# K^r, \tau \# \tau^r) = r_{-1/8}(K \# K, \tau \# \tau) < \infty.
\]
By Theorem~\ref{cork_independent}, any nontrivial linear combination of $1/n$-surgeries along $(K \# K^r, \tau \# \tau^r)$ thus forms a strong cork. In particular, each $1/(2i)$-surgery is a strong cork, as desired.
\end{proof}

\begin{rem}\label{rem:heegaardfloer}
It is not difficult to establish an analogous version of Theorem~\ref{thm:intro-whitehead} for the $\underline{V}^{\iota\tau}_{0}$-invariant defined in \cite{DMS} using Heegaard Floer homology. Indeed, we claim that if $\underline{V}^{\iota\tau}_{0}(K \# K^r) > 0$, then any symmetric pair of disks bounded by $Wh_i(K)$ are not isotopic. The argument is similar to the proof of Theorem~\ref{thm:intro-whitehead}: consider the double branched cover of $\Wh_i(K)$. Once again, it suffices to show that this constitutes a strong cork. This can be done by showing that the $(\iota \circ \tau)$-local class of $\Sigma_2(\Wh_i(K))$ is nontrivial, as defined in \cite{DHM20}. As in Lemma~\ref{prop:surgery_formula}, the idea is to bound this class by the $(\iota \circ \tau)$-local class of the large surgery complex $A_0(K \# K^r)$. This is done using the same sequence of equivariant cobordisms from $\smash{S^3_{1/(2i)}(K \# K^r)}$ to $\smash{S^3_{1}(K \# K^r)}$ as in the proof of Theorem~\ref{cork_independent}, applying a similar equivariant cobordism to from $(+1)$-surgery to large surgery, and then utilizing the equivariant large surgery formula of \cite{mallick2022knot}.
\end{rem}

\subsection{Applying the partial connected sum formula}\label{sec:trefoilcalculation}
We now give an explicit example of a knot satisfying the hypotheses of Theorem~\ref{thm:intro-whitehead}. Note that due to the inequality
\[
r_s(Y, K, \tau) \leq r_{s'}(Y, K, \tau)
\]
whenever $s \geq s'$, this also satisfies the hypotheses of Theorem~\ref{cork_independent}. As discussed in Example~\ref{ex:sumfour}, let 
\[
K = 2T_{2,3} \# - 2T_{2,3} \quad \text{and} \quad \tau = \tau_f \# -2\tau_0.
\]
The factors for this tensor product are displayed on the left in Figure~\ref{fig:trefoil3}. To understand the tensor product, we start by understading the underlying non-involutive complex. We begin by enumerating the $\theta$-supported cycles in the tensor product. For this, we list all of the elements of $C(K)$ with grading zero and calculate their images under $d = d^\otimes$. These are:
\[
\begin{array}{ll}
d(p|x) = \Lambda(p | z + \theta | x) \quad & d(p|y) = \Lambda(p|z + p|\un{t} + \theta|y) \\
d(q|x) = \Lambda(q | z + \theta | x) \quad & d(q|y) = \Lambda(q|z + q|\un{t} + \theta|y) \\
d(r|z) = \Lambda(p | z + q | z) \quad & d(s|t) = \Lambda(q | \un{t}) \\
d(p|\un{z}) = 0 \quad &d(q|\un{z}) = 0 \\
d(r|\un{t}) = \Lambda(p| \un{t} + q| \un{t})
\end{array}
\]
Note that the first three rows are basis elements of the first kind, while the last two rows are basis elements of the second kind. We also calculate the filtration values of these generators. The Chern--Simons filtration values for the two generators of the trefoil are $\mathrm{cs}(0) = 0$ and $\mathrm{cs}(1) = 1/3$; see \cite[Section 9.2.3]{DS19}. Comparing with Figure~\ref{fig:trefoil1} gives the Chern--Simons filtrations displayed in Figure~\ref{fig:trefoil3}.

\begin{figure}[h!]
\includegraphics[scale = 1]{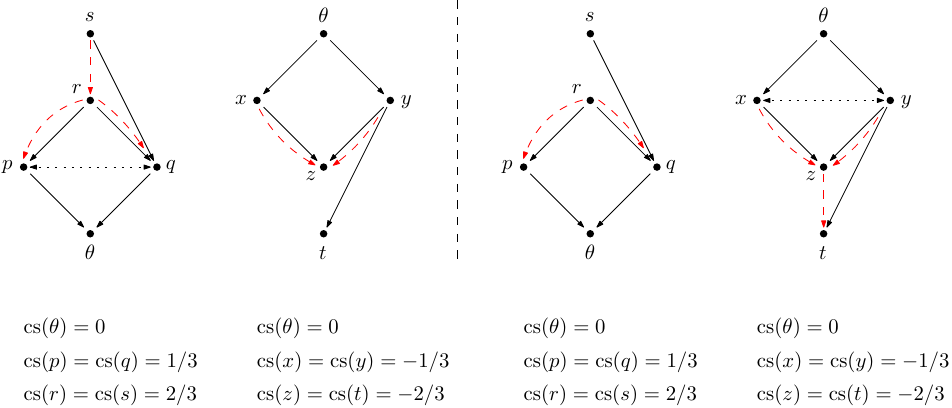}
	\caption{Left: the complex for $(2T_{2,3}, \tau_f)$ and the complex for $(-2T_{2,3}, -2\tau_0)$. Right: the complex for $(2T_{2,3}, 2\tau_0)$ and the complex for $(-2T_{2,3}, -\tau_f)$.}\label{fig:trefoil3}
\end{figure}

We now investigate the $r_s$-invariant. Since the $r_s$-invariant can be defined in terms of the under-bar complex, we do not have to calculate the $\mu^\otimes$-component of the tensor product. (Indeed, our partial connected sum formula does not allow us to calculate this component anyway.) Recall that 
\[
\tau ^\otimes =\begin{pmatrix}
\tau\otimes\tau' & 0 & 0 & 0 \\  \tau \otimes (\mu')^\tau + \mu^\tau \otimes \tau' & \tau\otimes\tau' & \tau \otimes (\Delta'_2)^\tau & \Delta_2^\tau \otimes \tau'  \\
\tau \otimes (\Delta'_1)^\tau & 0 & \tau \otimes 1 & 0 \\
\Delta_1^\tau \otimes \tau' & 0 & 0 & 1 \otimes \tau'
\end{pmatrix}.
\]
and
\[
\Delta_1^\otimes =\begin{pmatrix}
\Delta_1^\tau \otimes (\Delta_1')^\tau & 0 & \Delta_1^\tau \otimes 1  & 1 \otimes (\Delta_1')^\tau 
\end{pmatrix}
\]
and
\[
\Delta_2^\otimes =\begin{pmatrix}
0 & \Delta_2^\tau \otimes (\Delta_2')^\tau & \Delta_2^\tau \otimes 1  & 1 \otimes (\Delta_2')^\tau 
\end{pmatrix}^T.
\]
We form the under-bar complex by taking a copy of $C^\otimes$, together with $\theta | \theta$. This has differential with components $d^\otimes$ and $\delta_2^\otimes(\theta | \theta) = \Lambda(\theta | x + \theta | y)$, and involution with components $\tau^\otimes$ and $\Delta^\otimes_2 = 0$.

\begin{lem}
We have $r_{-1/8}(K, \tau) < \infty$.
\end{lem}
\begin{proof}
The factors for this tensor product are displayed on the left in Figure~\ref{fig:trefoil3}. By Lemma~\ref{bb}\eqref{bb1}, it suffices to show we have no $\tau$-invariant $\theta$-supported cycle in filtration level $[-\infty, 1/8]$. Write such a cycle as
\begin{equation}\label{eq:supposedinvariant}
\theta | \theta + a (p|x) + b (q | x) + \text{other terms}.
\end{equation}
Since $\delta_2^\otimes(\theta | \theta) = \Lambda(\theta | x + \theta | y)$ and $p|x$ and $q|x$ are the only other elements which have differential supported by $\theta|x$, we have
$a + b = 1$. Now consider the action of $\tau_f \# -2\tau_0$ on the above cycle. We claim that this is
\begin{equation}\label{eq:supposedinvarianttau}
\theta | \theta  + a (q|x) + b (p|x) + \text{other terms}.
\end{equation}
Indeed, since $p|x$ and $q|x$ are basis elements of the first kind, this follows from the first column of $\tau^\otimes$. Note that the first row of $\tau^\otimes$ shows no other terms contribute to the coefficients of $p|x$ or $q|x$.

Now, we know \eqref{eq:supposedinvariant} and \eqref{eq:supposedinvarianttau} are homologous as filtered cycles and thus differ by a boundary. It is easily checked that the only chains which have differentials supported by $p|x$ or $q|x$ have Chern--Simons filtration value $1/3$. Hence the coefficients of $p|x$ in \eqref{eq:supposedinvariant} and \eqref{eq:supposedinvarianttau} must be exactly equal (and likewise for the coefficients of $q|x$). Thus $a = b$, which clearly contradicts $a + b = 1$.
\end{proof}

\begin{lem}
We have $r_0(-K, -\tau) = \infty$.
\end{lem}
\begin{proof}
The factors for this tensor product are displayed on the right in Figure~\ref{fig:trefoil3}. By Lemma~\ref{bb}\eqref{bb1}, it suffices to produce a $\tau$-invariant $\theta$-supported cycle in filtration level $[-\infty, 0]$. Note that for the action of $- \tau_f$, we have $\mu^\tau(x) = c_1 z$ and $\mu^\tau(y) = c_2 z$ with $\Lambda | (c_1 + c_2)$, as we derived in the discussion preceding Figure~\ref{fig:trefoil2}. Construct a $\tau$-invariant $\theta$-supported cycle by considering
\[
\theta | \theta + q | x + q | y + s | t + c_2 (q | \un{z}).
\]
The intuition for this is that $q|x$ and $q|y$ are necessary to cancel $\delta_2^\otimes(\theta | \theta) = \Lambda(\theta | x + \theta | y)$; the addition of $s | t$ completes this to a cycle. All four of these terms have Chern-Simons filtration zero. (We could equally well have considered $\theta | \theta + p | x + p | y + s | t + r | \un{t}$.) The final term $c_2 (q | \un{z})$ will aid in making our cycle equivariant. Note that $q | \un{z}$ is indeed a cycle with Chern-Simons filtration $- 1/3$.

We now simply check that our cycle is invariant under $2\tau_0 \# - \tau_f$. The middle three terms are basis elements of the first kind, while the last is a basis element of the second kind. Consulting the first two columns of $\tau^\otimes$, we thus obtain
\[
\theta | \theta + (q | y + q | (c_1 \un{z})) + (q | x + q | (c_2 \un {z})) + s|t + c_2 (q | \un{z}).
\]
Rearranging gives
\[
\theta | \theta + q | x + q | y + s | t + c_2 (q | \un{z}) + (c_1 + c_2) (q | \un{z}).
\]
This is just the original cycle, plus $(c_1 + c_2)(q | \un{z})$. Now, observe that $d(q | \un{x}) = \Lambda(q | \un{z})$. Crucially, $\Lambda | (c_1 + c_2)$ and $q | \un{x}$ has Chern-Simons filtration zero. Hence the error term is homologically zero, and the homology class of our $\theta$-supported cycle is indeed $\tau$-invariant. This completes the proof.
\end{proof}
\subsection{Further consequences}
Having provided an explicit example of Theorem~\ref{thm:intro-whitehead}, we now give further topological consequences. We first give our applications involving non-isotopy of $\Z$-slice disks.

\introZ*

\begin{proof}
Take any strongly invertible slice knot $(K, \tau)$ satisfying the assumptions of \cref{thm:intro-whitehead}. Set
\[
K_0 = \Wh_1(K) \# -\Wh_1(K) \quad \text{and} \quad \tau_0 = \tau \# - \tau,
\]
where (as usual) we also write $\tau$ for the strong inversion on $\Wh_1(K)$ induced by the strong inversion on $K$. Then $(K_0, \tau_0)$ is equivariantly slice, being the connected sum of a strongly invertible knot with its concordance inverse. It is also $\Z$-slice, as each summand $\Wh_1(K)$ is $\Z$-slice. However, we claim that it is not equivariantly (or even isotopy-equivariantly) $\Z$-slice. 

Suppose this were the case. Let $D$ be a $\Z$-disk for $K_0$ which is smoothly isotopic to $\tau_{B^4}(D)$ for some extension $\tau_{B^4}$ of $\tau_0$ over $B^4$. As usual, taking the double branched cover $\Sigma_2(D)$ gives a homology ball with boundary
\[
\Sigma_2(\partial D)= \Sigma_2(K_0) = \Sigma_2(\Wh_1(K)) \# -  \Sigma_2(\Wh_1(K))
\]
such that the involution on $\Sigma_2(K_0)$ obtained by lifting $\tau_0$ extends over $\Sigma_2(D)$. This involution is the connected sum of the involutions obtained by lifting $\tau$ to $\Sigma_2(\Wh_1(K))$ and $-\Sigma_2(\Wh_1(K))$. Crucially, because $D$ is a $\Z$-disk, we have that $\Sigma_2(D)$ is simply connected. Attach an equivariant 3-handle to $\Sigma_2(D)$ along the connected sum sphere to obtain a simply connected, equivariant homology cobordism 
\[
W \colon \Sigma_2(\Wh_1(K)) \rightarrow \Sigma_2(\Wh_1(K)).
\]
Inside the cobordism we may choose an equivariant annulus $S$ from a local strongly invertible unknot $\U$ in $\Sigma_2(\Wh_1(K))$ to itself; this has $\pi_1(W\ \setminus \ S) \cong \Z$. The strict inequality of Theorem~\ref{thm:rsproperties}\eqref{rsproperties:2} then implies we must have
\[
r_0 ( \Sigma_2(\Wh_1(K)), \U, \tau) = \infty.
\]
However, we know from the proofs of Theorems~\ref{thm:intro-whitehead} and \ref{cork_independent} that 
\[
r_0 ( \Sigma_2(\Wh_1(K)), \U, \tau) \leq r_{-1/8} ( K \# K, \tau \# \tau) < \infty,
\]
a contradiction.
\end{proof}

\introD*
\begin{proof}
Consider $K$ and $K_0$ as in the proof of \cref{thm:intro-Z} and proceed by contradiction. Let $D$ and $\tau_{B^4}(D)$ be $\Z$-disks for $K_0$. Suppose these are isotopic after external stabilization by some closed, simply connected definite $4$-manifold, or isotopic after internal stabilization by some number of real projective planes of the same sign and with external fundamental group $\Z_2$. Taking branched covers, we again obtain an equivariant cobordism
%    We show that any symmetric pair of $\Z$-slice disks is exotic after stabilization by any simply connected definite closed 4-manifold $X$ or any $n\mathbb{RP}^2$-knot stabilization with $\pi_1(S^4 \setminus \mathbb{RP}^2) \cong \Z_2$. We suppose so, then by repeating the proof of \cref{thm:intro-Z}, we get a definite equivariant simply connected cobordism 
    \[
W' \colon \Sigma_2(\Wh_1(K)) \rightarrow \Sigma_2(\Wh_1(K)).
\]
This is definite and simply connected by the restrictions placed on the stabilization. By reversing the cobordism if necessary, we can suppose $W'$ is a negative definite cobordism from $\Sigma_2(\Wh_1(K)) $ to itself. Let $S \subset W'$ be an equivariant annulus from a local strongly invertible unknot in $\Sigma_2(\Wh_1(K))$ to itself; this has $\pi_1(W' \smallsetminus S) \cong \Z$. Since $\Wh_1(K)$ is slice, its branched cover has zero $d$-invariant and $W'$ has standard intersection form. It is then easily checked that $(W', S)$ is a strong negative definite pair. The strict inequality of Theorem~\ref{thm:rsproperties}\eqref{rsproperties:2} gives the same contradiction as before. 
\end{proof}

Our final claim is essentially an extension of Theorem~\ref{thm:intro-whitehead}:

\introkernel*
\begin{proof}
Take any strongly invertible slice knot $(K, \tau)$ satisfying the assumptions of \cref{thm:intro-whitehead} and consider $\{\Wh_i(K)\}_{i \geq 1}$. We claim that $\{\Wh_i(K)\}_{i \geq 1}$ lies in $\ker \mathfrak{F} \cap \ker (F)$. Indeed, since $K$ is slice, we have that $\Wh_i(K)$ is slice; moreover, we have a smooth $\Z$-slice disk as shown in \cite[Proposition 2.1]{GHKP}. By \cite{conway2021characterisation}, any two $\Z$-slice disks are topologically isotopic rel boundary. Hence this disk is necessarily topologically isotopy-equivariant, showing that $\Wh_i(K)$ lies in $\ker(\mathfrak F)$. 

The proof of Theorem~\ref{thm:intro-whitehead} actually shows that any nontrivial linear combination of branched covers $\{\Sigma_2(\Wh_i(K))\}_{i \geq 1}$ yields a strong cork. If we have a linear combination of the $\Wh_i(K)$ which satisfies the word condition described in Theorem~\ref{thm:intro-kernel}, then the corresponding linear combination of branched covers $\Sigma_2(\Wh_i(K))$ is nontrivial in the usual sense of linear algebra, as the equivariant connected sum operation on these strong corks is commutative. Hence any linear combination of the $\Wh_i(K)$ as in Theorem~\ref{thm:intro-kernel} is not isotopy-equivariantly slice, as desired.
\end{proof}

\subsection{Other examples}
In this paper, we have attempted to develop a framework for studying strongly invertible knots using singular instanton Floer theory. Our hope is that developing this formalism directly for knots will lead to a broader class of computable examples than using the $3$-manifold invariants of \cite{ADMT} alone. For completeness, however, we now record some further examples of exotic disks that can be established only using the $3$-manifold $r_s$-invariant of \cite{ADMT}. 

\begin{rem}
Note that by replacing $r_s(Y, \tau)$ with $r_s(Y, \U, \tau)$, one can view our knot-theoretic framework as a generalization of the equivariant $3$-manifold formalism of \cite{ADMT}. Indeed, all results in \cite{ADMT} (and several examples in \cite{Muk26}) derived from instanton theory can be proven using our new formulation in this manner. Conversely, the inequality
\[
    r_s(Y , \U, \tau) \leq 2r_{2s}(Y, \tau), \quad \forall s \in [-\infty, 0]
    \]
can be used to give bounds on $r_s(Y , \U, \tau)$ if the corresponding $3$-manifold invariant is known.
\end{rem}
\begin{comment}
    where the right-hand side invariant $r^*_s(Y,\tau)$ is the invariant introduced in \cite{ADMT} for any orientation preserving diffeomorphism on $Y$. In \cite{ADMT}, it has been verified that 
    \[
    r_{0}^*(S^3_{1} ( \overline{9}_{46} ),\tau)< \infty
    \]
    for the strongly inversion on $\overline{9}_{46}$ which corresponds to Akbulut's involution on his cork. 
    From the inequality \eqref{comparison}, one can see all  computations showing nontrivialities of $r_s^*$ for certain type of involutions given in \cite{ADMT} can give the same computations for our new formulation by putting the unknot. 
    Indeed, all results in \cite{ADMT} derived from instanton theory can be proven using our new formulation. The new computations of $r_0^*$-invariant given in \cite{Muk26} can be used to give more computations of our invariants $r_0(Y , \U, \tau)$.

    Along this line, we can give a pair of symmetric exotic disks via the computation given in \cite{ADMT}: 
    \end{comment}
    % The invariant $r_0(S^3_{1/m} ( \overline{9}_{46} \# \overline{9}_{46}), \U, 2\tau )$ is finite for any $m >0$.
    %  From \eqref{comparison} combined with $r_{0}^*(S^3_{1} ( \overline{9}_{46} ),\tau)< \infty$ proven in \cite{ADMT}, we see $r_0(S^3_{1} ( \overline{9}_{46} , \U, \tau )< \infty $. [I will move this to a statement outside the theorem, since I think the result of the theorem doesn't actually have anything to do with this inequality.]
    
\begin{thm} \label{thm:3manifold}
   For any $i \geq 1$, $\Wh^i( \overline{9}_{46})$ is not isotopy-equivariantly slice and thus bounds an exotic pair of disks.
\end{thm}
\begin{proof}
In \cite{ADMT}, it is shown that $r_0 (S^3_{1} ( \overline{9}_{46} ),\tau  )< \infty$. In Figure~\ref{equi_moves}, we give an equivariant negative definite cobordism from $S^3_{1} ( \overline{9}_{46} \# \overline{9}_{46})$ to $S^3_{1} ( \overline{9}_{46})$. By the monotonicity of the $r_0$-invariant established in \cite[Theorem 1.1]{ADMT}, we thus have 
   \[
   r_0^* (S^3_{1} ( \overline{9}_{46} \# \overline{9}_{46}),\tau\# \tau ) \leq   r_0^* (S^3_{1} ( \overline{9}_{46} ),\tau  )< \infty.  
   \]
The cobordism in the proof of Theorem~\ref{cork_independent} then shows
   \[
r_0(S^3_{1/m} ( \overline{9}_{46} \# \overline{9}_{46}),\tau \# \tau)) < r_0( S^3_{1} ( \overline{9}_{46} \# \overline{9}_{46}),\tau  )< \infty. 
   \]
This implies $\smash{S^3_{1/m} ( \overline{9}_{46} \# \overline{9}_{46})}$ is a strong cork for each $m \geq 1$. But then by Lemmas~\ref{lem:obvious} and \ref{lem:liftbranched1}, we conclude that $\smash{\Wh^i( \overline{9}_{46})}$ is not isotopy-equivariantly slice. 

    \begin{figure}[h!]
  \centering
  \includegraphics[width=0.95\linewidth]{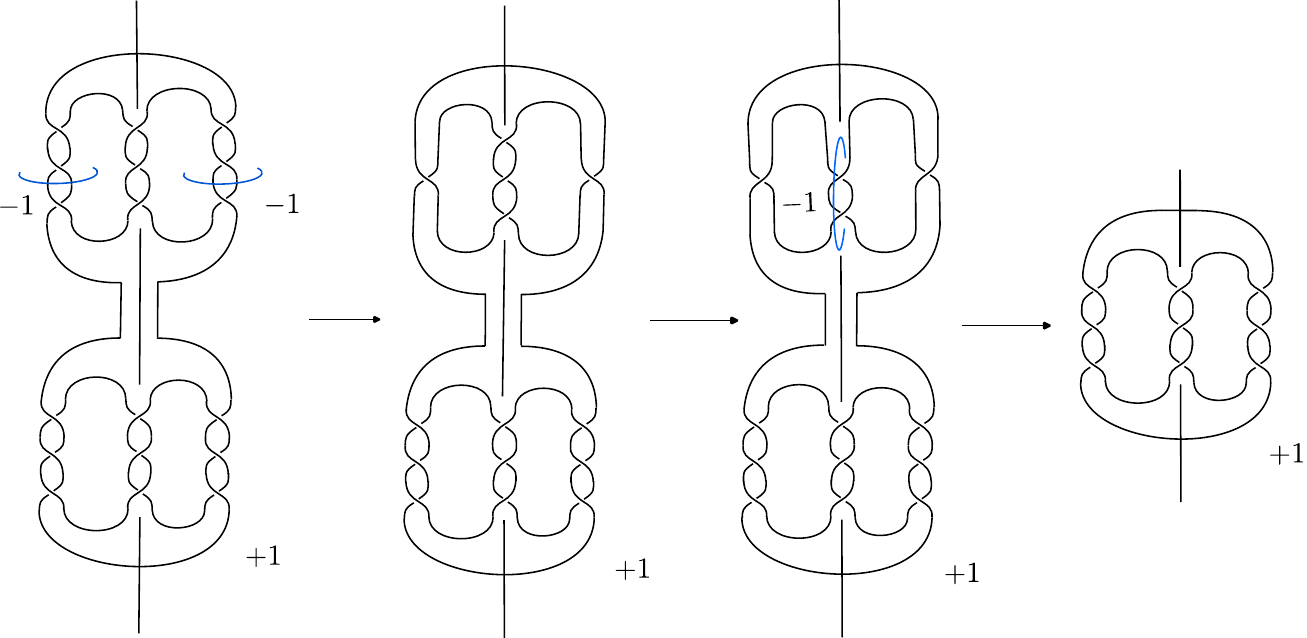}
  \caption{Construction of an equivariant negative definite cobordism from $S^3_{1} ( \overline{9}_{46} \# \overline{9}_{46})$ to $S^3_{1} ( \overline{9}_{46})$. The symmetry is with respect to rotation along the vertical line, in the second step we have applied an equivariant isotopy.}
  \label{equi_moves}
\end{figure}
\end{proof}

\section{Two-bridge knots}\label{section:Computations and Examples}

In this section, we give some sample (partial) computations of $
\wt{\tau}$ for two-bridge knots.  Our calculations are based on work of Daemi and Scaduto, who computed $C_*(K)$ for all two-bridge knots in \cite[Section 9]{DS19}. Using our partial connected sum formula, we then give another example of a knot satisfying the hypotheses of Theorem~\ref{cork_independent}. This will be conceptually different than the example of Section~\ref{sec:trefoilcalculation} and will utilize an explicit understanding of the action of our involution on the representation variety of the branched double cover.

%Our calculations are based on work of Daemi and Scaduto, who computed $C_*(K)$ for all two-bridge knots in \cite[Section 9]{DS19}. 

\subsection{The algorithm of Daemi and Scaduto}

Given a two-bridge knot $K = K_{p,q}$ with branched cover $L(p, q)$, with respect the unperturbed Chern--Simons functional and the orbifold metric obtained as the quotient of the standard spherical metric of the lens space, Daemi and Scaduto give an explicit combinatorial algorithm for calculating:
\begin{enumerate}[noitemsep]
\item The generators of $C_*(K)$;
\item The grading and Chern-Simons degree of each generator;
\item The differential $d$ on $C_*(K)$; and, 
\item The maps $\delta_1$ and $\delta_2$. 
\end{enumerate}
Since all critical points of the Chern-Simons functional are known to be nondegenerate for two-bridge knots, the instanton gradings are just given as the unperturbed critical values of the Chern-Simons functionals. Moreover, all trajectories connecting these critical points are regular because of our choice of metric. 
This means that our enriched complex 
\[
\wt{\mathfrak{E}}(K_{p,q})  = \{ \wt{C}(K_{p,q}, \pi_i = 0), \psi_{i}^j = \id\} 
\]
can be taken as a constant complex $\wt{C}(K_{p,q})$ with the identity structure maps. 

%We regard $\wt{C}_*(K_{p,q})$ as $\Z$-graded $R[y^{\pm 1}]$-modules and $C_*(K_{p,q})$ as $\Z$-graded subcomplex over $R[y^{\pm 1}]$.
%If we put $y=1$ and regard it as a complex over $\Z_4$-grading, we recover the $\mathcal{S}$-complex $\tilde{C}(K_{p,q})$.

It is shown in \cite[Section 9]{DS19} that the generators of $C_*(K)$ are in correspondence with conjugacy classes of representations
\[
\pi_1(L(p, q)) = \Z_p \rightarrow \SU(2).
\]
We enumerate the generators of $C_*(K)$ from $1$ to $(p-1)/2$, where generator $i$ corresponds to the nontrivial representation sending $\zeta$ to $\zeta^i \oplus \zeta^{-i}$. Note that this is conjugate to the representation sending $\zeta$ to $\zeta^{-i} \oplus \zeta^{i}$. We thus sometimes think of the index set as running from $1$ to $p-1$, but with $i$ and $p - i$ identified.

We obtain a set of generators for $\smash{\wt{C}(K)}$ by taking two copies of $C_*(K)$, together with the reducible connection (which corresponds to the trivial representation and we denote by $0$). We denote corresponding generators in the two copies of $C_*(K)$ by $i$ and $\underline{i}$, so that $\chi(i) = \underline{i}$. The above data then gives a partial computation of $\smash{\wt{d}}$, with the only missing component being the $v$-map. In fact, Daemi and Scaduto combinatorially generate a set of restricted possibilities for the $v$-map, although in general this is not fully determined. An example of their algorithm is given later in this section.

\subsection{Involutions on $K_{p,q}$ and $L(p, q)$}
We now consider two-bridge knots in the equivariant setting. As shown in \cite[Section 3]{Sa86}, $K_{p, q}$ has exactly one strong inversion (up to Sakuma equivalence) if $K_{p, q}$ if a torus knot and two strong inversions otherwise. However, the two inversions on $K_{p, q}$ when $q^2 \equiv 1$ mod $p$ are constructed in a slightly manner than the two inversions on $K_{p, q}$ when $q^2 \notequiv 1$ mod $p$; this distinction will be important presently. As discussed in Lemma~\ref{lem:liftbranched1}, each strong inversion lifts to exactly two involutions on $L(p, q)$, which differ by the branching action. Our first goal will be to determine how these involutions act on the representation variety of $L(p, q)$.

It will be helpful to note that the branching action $\rho$ acts as multiplication by $-1$ on $\pi_1(L(p, q)) = \Z_p$. To see this, let $S^3 = \{(z, w) \in \mathbb{C}^2 \colon z^2 + w^2 = 1\}$ and form $L(p, q)$ as its quotient. Then $\rho$ is the map on $L(p, q)$ induced by complex conjugation in $S^3$. (See for example the proof of \cite[Corollary 4.12]{HR85}.) Let $H_1 = \{(z, w) \in S^3 \colon |z|^2 \leq 1/2\}$ and $H_2 = \{(z, w) \in S^3 \colon |w|^2 \leq 1/2\}$ be the standard genus-one Heegaard splitting of $S^3$; this descends to a genus-one Heegaard splitting of $L(p, q)$ which is preserved by $\rho$. Clearly, $\rho$ negates the core of each solid torus, giving the desired computation. Since the representations $\zeta \mapsto \zeta^i \oplus \zeta^{-i}$ and $\zeta \mapsto \zeta^{-i} \oplus \zeta^{i}$ are conjugate, this shows that $\rho$ induces the identity map on the representation variety of $L(p, q)$. In particular, both lifts of a given involution on $K_{p, q}$ induce the same map on the representation variety.

We have the following casework based on the number of strong inversions on $K_{p, q}$:

\subsubsection*{The case $q \equiv \pm 1$ mod $p$}
This corresponds to the case when $K_{p, q}$ is a torus knot. The unique strong inversion $\tau_0$ on $K_{p, q}$ is displayed in Figure~\ref{fig:inverting1}.

\begin{figure}[h!]
\includegraphics[scale = 0.6]{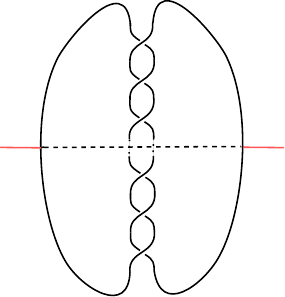}
	\caption{The unique strong inversion $\tau_0$ on the torus knot $T_{2, 2n+1}$. The arc $\gamma$ is given by the red segment (passing over the point at $\infty$). There are $2n+1$-many half crossings.}\label{fig:inverting1}
\end{figure}

We claim that the lift (or, more precisely, both lifts) of $\tau_0$ act as the identity on the representation variety. Indeed, consider a neighborhood of the arc $\gamma$ indicated in Figure~\ref{fig:inverting1}. It is easily checked that $\partial N(\gamma)$ is an equivariant bridge sphere for $K_{p, q}$; that is, it separates $K_{p, q}$ into two rational tangles. Taking the branched cover of each of these tangles separately, we see that the interior and exterior of $\partial N(\gamma)$ lift to a genus-one Heegaard splitting of $L(p, q)$. Since $\tau_0$ preserves both rational tangles, the lift of $\tau_0$ preserves the solid tori of this splitting. This implies that $\tau_0$ must act as $\pm 1$ on $\pi_1(L(p, q))$, since $\tau_0$ necessarily acts by $\pm 1$ on the core of either solid torus and these generate $\pi_1(L(p, q))$. In both cases, $\tau_0$ acts as the identity on the representation variety.

\subsubsection*{The case $q^2 \neq 1$ mod $p$}
In this case, there are two strong inversions on $K_{p, q}$, which we display in Figure~\ref{fig:inverting2}. (See \cite[Figure 3.2]{Sa86}.) As we will see, these two inversions behave quite similarly for our purposes, so we denote them by $\tau_0$ and $\tau_0'$.

\begin{figure}[h!]
\includegraphics[scale = 0.7]{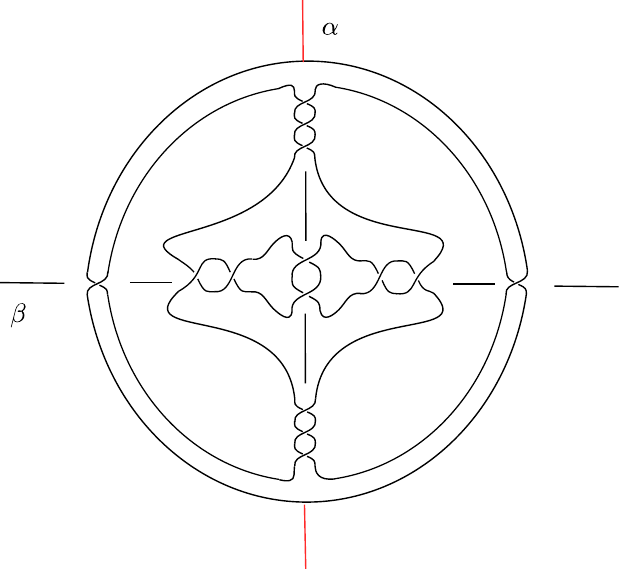}
	\caption{The two strong inversion $\tau_0$ and $\tau_0'$ on $K_{p,q}$ when $q^2 \neq 1$ mod $p$. The arc $\gamma$ is given by the red segment (passing over the point at $\infty$). Here $q/p=[2,4,6,2]$.}\label{fig:inverting2}
\end{figure}

We claim that the lifts of $\tau_0$ and $\tau_0'$ both act as the identity on the representation variety. The argument is the same as before: the arc $\gamma$ in Figure~\ref{fig:inverting2} gives an equivariant bridge sphere $\partial N(\gamma)$ for $K_{p, q}$. Once again, both $\tau_0$ and $\tau_0'$ preserve each rational tangle, so their lifts preserve the solid tori in some genus-one Heegaard splitting of $L(p, q)$. Thus $\tau_0$ and $\tau_0'$ act as $\pm 1$ on $\pi_1(L(p, q))$ and hence the identity on the representation variety.

\subsubsection*{The case $q^2 \equiv 1$ mod $p$}
We now treat the most interesting situation, which is when $q^2 \equiv 1$ mod $p$ (but $q \notequiv 1$ mod $p$). In this case there are once again two strong inversions on $K_{p, q}$, which we display in Figure~\ref{fig:inverting3}. However, the reader should be cautioned that rotation about the horizontal axis is \textit{not} the second of the two involutions specified in \cite[Section 3]{Sa86}. Indeed, in this case, it turns out that rotation about the vertical and horizontal axes are Sakuma equivalent. (This is a consequence of the fact that if $q^2 \equiv 1$ mod $p$, the continued fraction expansion of $p/q$ satisfies a certain symmetry.) Instead, the second involution guaranteed by \cite[Proposition 3.6]{Sa86} is more subtle: in the diagram of Figure~\ref{fig:inverting3}, it consists of rotation about the circle $\beta$ which (roughly speaking) sends the origin to the point at $\infty$. We denote the vertical rotation by $\tau_0$ and the second involution by $\tau_1$.

\begin{figure}[h!]
\includegraphics[scale = .7]{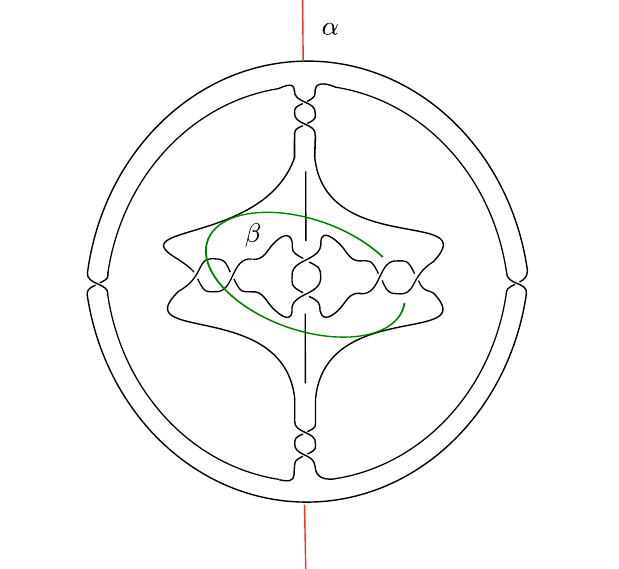}
	\caption{The two strong inversion $\tau_0$ and $\tau_1$ on $K_{p,q}$ when $q^2 \neq 1$ mod $p$. The arc $\gamma$ is given by the red segment. The $\beta$-circle is shown in green. Here $q/p=[2,4,-4,-2]$.}\label{fig:inverting3}
\end{figure}

As before, $\tau_0$ acts as the identity on the representation variety since it preserves $N(\gamma)$. However, note that $\tau_1$ does \textit{not} preserve $N(\gamma)$. Instead, to understand the action of $\tau_1$ on $\pi_1(L(p, q))$, consider the diagram for $K_{p, q}$ given in Figure~\ref{fig:inverting4}. (See \cite[Figure 2.5]{Sa86}.) Consider the sheet which intersects the page in the vertical axis and contains the axis of symmetry. Together with the point at infinity, this constitutes an equivariant bridge sphere for $\tau_1$. 

\begin{figure}[h!]
\includegraphics[scale = 0.7]{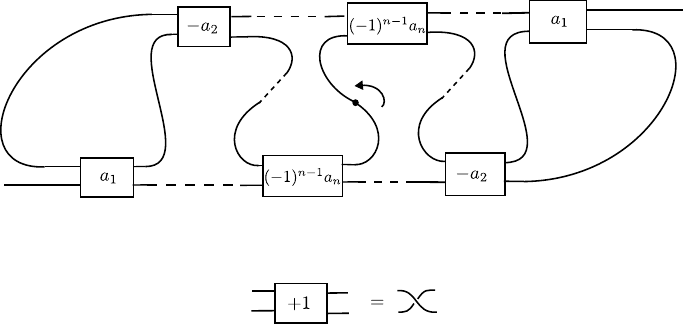}
	\caption{The involution $\tau_1$ drawn in a different diagram.}\label{fig:inverting4}
\end{figure}

Now, in this case $\tau_1$ \textit{interchanges} the two rational tangles on either side of the bridge sphere. Hence the lift of $\tau_1$ interchanges the two solid tori of some genus-one splitting of $L(p, q)$. In particular, $\tau_1$ sends the core of one solid torus to plus or minus the core of the other solid torus. It is a quick calculation to show that one core represents $\pm q$ times the other in $\pi_1(L(p, q))$; hence $\tau_1$ acts as multiplication by $\pm q$ on $\pi_1(L(p, q)) = \Z_p$. Enumerating the representation variety by
\[
\{1, 2, \cdots, (p-1)/2\} \subset \Z_p
\]
as described previously, we have that $\tau_1$ acts via multiplication by $q$ on the representation variety (with the understanding that $i$ and $-i$ are identified in our enumeration). \\

We now use our understanding of these actions on the representation variety to perform a partial computation of $\tau$ on the instanton knot complex. It will first be useful to verify that the lift of each symmetry of $K_{p, q}$ is covered by an involution on $S^3$ which preserves the standard round metric:

\begin{lem}\label{lem:roundmetric}
Let $\tau$ be an involution on $L(p, q)$ obtained by lifting a symmetry of $K_{p, q}$ to the branched double cover. Then $\tau$ is covered by an involution on $S^3$ preserving the standard round metric.
\end{lem}
\begin{proof}
It will convenient to work backwards by starting with various involutions on $S^3$. Following \cite[Section 4.8]{HR85}, consider the involutions on $S^3 = \{(z, w) \in \mathbb{C}^2 \colon z^2 + w^2 = 1\}$ given by
\begin{align*}
&\wt{A} \colon (z, w) \rightarrow (\overline{z}, \overline{w}) \\
&\wt{B} \colon (z, w) \rightarrow (z, -w) \\
&\wt{B}' \colon (z, w) \rightarrow (-z, w) \\
&\wt{C} \colon (z, w) \rightarrow (w, z),
\end{align*}
These obviously preserve the round metric. The first three are evidently $\Z_p$-equivariant with respect to the action $\zeta(z, w) = (\zeta z, \zeta^q w)$ and hence cover involutions $A$, $B$, and $B'$ on $L(p, q)$. The final involution $\smash{\wt{C}}$ is not quite $\Z_p$-equivariant, but if $q^2 \equiv 1$ mod $p$ then we have $\smash{\zeta(\wt{C}(z, w)) = \wt{C}(\zeta^q(z, w))}$. Hence in this situation $\smash{\wt{C}}$ also descends to an involution on $L(p, q)$, which we denote by $C$. As discussed in the proof of \cite[Corollary 4.12]{HR85}, the quotient of $L(p, q)$ by $A$ is $S^3$ and the image in $S^3$ of the fixed-point set of $A$ is precisely $K_{p, q}$.

Each of $B$, $B'$, and $C$ commutes with $A$, as their corresponding lifts commute. Hence $B$, $B'$, and $C$ further descend to involutions $
\tau_B$, $\tau_{B'}$ and $\tau_{C}$ of $S^3$ which setwise preserve $K_{p, q}$. Since the fixed-point sets of $\smash{\wt{B}}$, $\smash{\wt{B}'}$, and $\smash{\wt{C}}$ each intersect the fixed-point set of $\smash{\wt{A}}$ in two points, we see that $
\tau_B$, $\tau_{B'}$ and $\tau_{C}$ are strong inversions on $K_{p, q}$. (Note that each of these strong inversions has two lifts to $L(p, q)$; e.g., $\tau_B$ is covered by both $B$ and $AB$.) The involutions $B$ and $B'$ preserve each solid torus in the Heegaard splitting of $L(p, q) = S^3/\Z_p$, while $C$ interchanges the two solid tori. Hence $\tau_B$, $\tau_{B'}$, and $\tau_C$ exhaust the strong inversions described in the discussion preceding the lemma.
\end{proof}

We summarize our findings as follows:

\begin{thm}\label{thm:kpqinvolutions}
Let $K_{p,q}$ be a two-bridge knot. Then:
\begin{enumerate}[noitemsep]
    \item If $q \equiv \pm 1$ mod $p$, then $K_{p, q}$ has a unique strong inversion $\tau_0$. The $\deg_I$-$0$ part of the $\tau$-action on $C_*(K_{p,q})$ is identity. 
    \item If $q^2 \notequiv 1$ mod $p$, then $K_{p, q}$ has two strong inversions $\tau_0$ and $\tau_0'$. The $\deg_I$-$0$ parts of both act as the identity map on $C_*(K_{p,q})$.
    \item If $q^2 \equiv 1$ mod $p$ (but $q \notequiv \pm 1$ mod $p$), then $K_{p, q}$ has two strong inversions $\tau_0$ and $\tau_1$. The $\deg_I$-$0$ part of the former acts as the identity on $C_*(K)$, while the $\deg_I$-$0$ part of the latter acts as multiplication by $q$  with respect to the natural enumeration of the representation variety of $L(p, q)$.
\end{enumerate}
\end{thm}
\begin{proof}
In each situation, $\tau$ can be regarded as the map induced from a smooth involution $\tilde{\tau} : S^3 \to S^3$ which commutes with $\Z_p$-action whose quotient is the desired lens space and with the complex conjugation.  
We see $\tilde{\tau}$ preserves the standard Riemannian metric on $S^3$.

Take two irreducible critical points $[\a]$, $[\b]$ for $K(p,q)$.
The map $\wt{\tau}$ on the irreducible components is defined by counting the moduli space 
\[
M_z([\a], (W_\tau, S_\tau), [\b]), 
\]
where $W_\tau$ and $S_\tau$ are the mapping cylinders with respect to $\tau$ of $[0,1]\times S^3$ and $[0,1]\times K(p,q)$. If we write the orbifold metric $\check{g}$ on $(S^3, K(p,q))$ obtained as the quotient of the standard metric on $S^3$ with respect to the $\Z_p$-action, the metric $dt^2 +  \check{g}$ defines a well-defined metric on $W_\tau$ since $\tau$ preserves $\check{g}$. 
We define the moduli space $M_z([\a], (W_\tau, S_\tau), [\b])$ using this canonical metric without perturbation. 
Since our metric has positive constant scalar curvature and zero-self-dual-Weyl curvature in the orbifold sense, we see the moduli space $M_z([\a], (W_\tau, S_\tau), [\b])$ is regular. 

By taking the double branched covering correspondence, we have 
\[
M_z([\a], (W_\tau, S_\tau), [\b])_0 \cong M_{w}^{c^*} ([\tilde{\a}], \Sigma_2(S_\tau), [\tilde{\b}])_0, 
\]
where 
\begin{itemize}[noitemsep]
    \item $\Sigma_2(S_\tau)$ is the double branched covering space of $W_\tau$ along $S_\tau$ which is diffeomorphic to $[0,1]\times L(p,q)$ twisted by the lift of diffeomorphism $\tilde{\tau}$,
    \item $\tilde{\a}$ and $\tilde{\b}$ are flat connections on $L(p,q)$ appeared as lifts of $\a$ and $\b$, 
    \item $M_w([\tilde{\a}], \Sigma_2(S_\tau), [\tilde{\b}])_0 $ denotes the $0$-dimensional component of the instanton moduli space over 
    \[
   (-\infty, -1]\times L(p,q)  \cup \Sigma_2(S_\tau) \cup [1,\infty) \times L(p,q)
    \]
    with respect to the metric 
    \[
    dt^2 + \pi^* \check{g}, 
    \]
    where $\pi$ is the projection $L(p,q) \to S^3$ obtained as the quotient with respect to the complex conjugation $c: L(p,q) \to L(p,q)$.
    The instantons are assumed to converge to representatives $\tilde{\a}$ and $\tilde{\b}$.
    \item Since the complex conjugation $c: L(p,q) \to L(p,q)$ preserves the standard metric on $L(p,q)$, we naturallly have an order 2 smooth action 
    \[
    c^* : M_w([\tilde{\a}], \Sigma_2(S_\tau), [\tilde{\b}]) \to M_w([\tilde{\a}], \Sigma_2(S_\tau), [\tilde{\b}]). 
    \]
The space $M_w^{c^*}([\tilde{\a}], \Sigma_2(S_\tau), [\tilde{\b}])$ denotes the fixed point part of $M_w([\tilde{\a}], \Sigma_2(S_\tau), [\tilde{\b}])$ with respect to the $c^*$-action. The notation $w$ denotes a homotopy class of a path corresponds to $z$.

\end{itemize}

Next, we suppose 
\[
cs ({\a}) -cs ({\b}) =0
\]
in order to compute the degree $0$-part of $\tau$. Note that we have $cs ({\a})= cs ({\tilde{\a}})$ and $cs ({\b})= cs ({\tilde{\b}})$ from their convensions. 
Then, the unperturbed ASD-equation reduces to the flat equation. Therefore, in this situation, we have the following identifications: 
\[
M_{w}^{c^*} ([\tilde{\a}], \Sigma_2(S_\tau), [\tilde{\b}])_0 \cong \left\{ [\rho] \in  \operatorname{Hom} (\Z_p, \SU(2))/\SU(2) \middle| [\tilde{\a}]= [\rho], \tau^*[\rho] = [\tilde{\b}] \right\} . 
\]
This shows the desired computations by putting $T=1$. For computations with local coefficients, we need to calculate 
\[
\nu(A) = \frac{i}{\pi} \int_{\R \times K } \Omega_A \text{ with }F_A |_{\R \times K} = \begin{pmatrix}
        \Omega_A & 0 \\ 
        0 & -\Omega_A 
    \end{pmatrix}.
\]
for $A \in M_z([\a], (W_\tau, S_\tau), [\b])_0$, however, one can see that $F_A=0$ over $\R \times S^3 \ \setminus \ \R \times K$ since any element in $M_{w}^{c^*} ([\tilde{\a}], \Sigma_2(S_\tau), [\tilde{\b}])_0$ can be represented as a singular flat connection which is described as an element in 
\[
\left\{ [\rho] \in  \operatorname{Hom} (\Z_p, \SU(2))/\SU(2) \middle| [\tilde{\a}]= [\rho], \tau^*[\rho] = [\tilde{\b}]\right\} . 
\]
From the density, we can observe $\nu(A)=0$.
This shows the desired result. 
\end{proof}

\subsection{Examples}\label{sec:Ex}
We now turn to some examples. The simplest knots whose (non-involutive) instanton complexes are fully calculated are the $(2, 2n+1)$ torus knots; see \cite[Section 3.2]{DS20}. Unfortunately, these are not very interesting from the involutive point of view: they have only one strong inversion, whose $\deg_I$-$0$ part acts as the identity. (We caution the reader, however, that in all but the simplest cases, such as the trefoil, this does not necessarily determine $\wt{\tau}$.) We thus illustrate Theorem~\ref{thm:kpqinvolutions} in the more interesting case that we have two strong inversions $\tau_0$ and $\tau_1$ that do not coincide. 

\subsubsection{A simple example: $K_{15, 4}$}
The smallest pair of parameters where $q^2 \equiv 1$ mod $p$ but $q \notequiv \pm 1$ mod $p$ is $K_{15, 4}$. We display the generators of $C(K_{15, 4})$ in Figure~\ref{fig:154a}, together with their Chern-Simons degrees, the reducible generator, the $\delta_1$- and $\delta_2$- maps, and the possibilities for $v$. Applying Theorem~\ref{thm:kpqinvolutions}, we have that the $\deg_I$-$0$ part of the $\tau$-component of $\tau_1$ interchanges the three pairs of generators in Figure~\ref{fig:154a}. In fact, it is straightforward to check that this is the entire $\tau$-component: the only possible $\tau$-arrow which is not captured by the $\deg_I$-$0$ part would be from $3$ to $6$. However, since $d(3) = 0$ while $d(5) = 6$, the commutation $\tau d = d\tau$ shows such an arrow is impossible. For the sake of completeness, we also determine $\Delta_1^\tau$ and $\Delta_2^\tau$. It is easily checked that $\Delta_1^\tau = 0$ due to the identity
\[
\delta_1 \tau = \Delta_1 d + 1 \delta_1.
\]
Indeed, the only possible $\Delta_1^\tau$-arrow is from $5$ to $0$. However, applying the above identity to $6$ rules this out. Similarly, it is easily checked that $\Delta_2^\tau = 0$ due to the identity
\[
d\Delta_2 + \delta_2 1 = \tau \delta_2.
\]
Indeed, the only possible $\Delta_2^\tau$-arrows are from $0$ to $2$ and $0$ to $7$. Applying the above identity to $0$ shows that $\Delta_2(0)$ must be a cycle, but no nontrivial linear combination of $2$ and $7$ forms a cycle. We summarize our computation as follows.

\begin{figure}[h!]
\includegraphics[scale = 0.75]{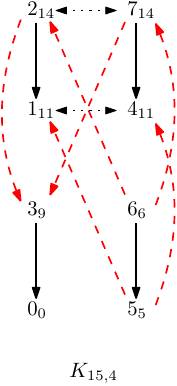}
	\caption{The results of \cite[Section 9]{DS19} and Theorem~\ref{thm:kpqinvolutions} for $K = K_{15, 4}$. Generators are enumerated from $0$ to $(p-1)/2 = 7$ with a subscript of $n$ representing Chern-Simons degree $n/15$. Solid arrows represent the differential and the $\delta_1$- and $\delta_2$-maps. Red dashed lines are possible $v$-maps; the actual $v$-map is some subset of the red arrows. Dotted horizontal arrows represent the action of $\tau_1$; if no dotted arrow is drawn out of a generator, the $\tau$-component acts as the identity. Relative gradings in $\Z_4$ correspond to heights on the page.}
	\label{fig:154a}
\end{figure}

\begin{ex}
The actions of $\tau_0$ and $\tau_1$ on $\wt{C}(K_{15,4})$ are as indicated in Figure ~\ref{fig:154a}. Specifically:
\begin{itemize}[noitemsep]
\item The $\tau$-component of $\tau_0$ is the identity (or rather, $\alpha$), while the $\tau$-component of $\tau_1$ interchanges the three pairs of generators indicated.
\item In both cases, $\Delta_1^\tau = \Delta_2^\tau = 0$.
\item In general, there are multiple possibilities for $\mu^\tau$. We leave it to the reader to enumerate these.
\end{itemize}
\end{ex}
%\subsection{A more involved two-bridge example}
\subsubsection{A more substantial example: $K_{45, 26}$} \label{subsection:45}
As we have seen, instanton Floer theory is capable of distinguishing $\tau_0$ and $\tau_1$ via their actions on the instanton knot complex. However, it is natural to ask whether the pairs $(\smash{\wt{C}(K_{p,q}), \wt{\tau}_0})$ and $(\smash{\wt{C}(K_{p,q}), \wt{\tau}_1})$ are distinct \textit{up to local equivalence}; that is, whether instanton Floer theory distinguishes $\tau_0$ and $\tau_1$ up to isotopy-equivariant concordance. Note this is not (necessarily) the case for the example $K_{15, 4}$ provided above: in principle, we could have that all $v$- and $\mu^\tau$-components vanish for both $\tau_0$ and $\tau_1$. Then $\wt{\tau}_0$ and $\wt{\tau}_1$ would both be locally equivalent to the involutive subcomplex spanned by $3$ and $0$, and hence locally equivalent to each other.

Indeed, it is not so easy to find an example where $\tau_0$ and $\tau_1$ are provably distinguished up to local equivalence. The reason is that in order to find a $K_{p, q}$ with this property, we intuitively require at least one of $\tau_0$ or $\tau_1$ to act nontrivially on $\smash{\wt{C}(K_{p,q}})$ in a way which involves the reducible. Thus, the most obvious thing to do would be to search for a $K_{p, q}$ in which one of $\tau_0$ and $\tau_1$ has a nontrivial $\Delta_1^\tau$- or $\Delta_2^\tau$-map, while the other does not. However, since $\Delta_1^\tau$ and $\Delta_2^\tau$ cannot be computed directly, finding such an example is difficult. One can often constrain $\Delta_1^\tau$ and $\Delta_2^\tau$, but generally this only determines $\Delta^\tau_1$ and $\Delta_2^\tau$ in the case that they vanish.

Instead, our idea is to search for a $K_{p, q}$ in which (the $\deg_I$-$0$ part of) $\tau_1$ interchanges two irreducible generators that both have a $\delta_1$-map to the reducible. Roughly speaking, this means that $\tau_1$ does in fact interact with the reducible, but does so indirectly via $\delta_1$. (Our discussion is only motivational; the actual argument is given in Lemma~\ref{lem:4526}.) The $(p, q)$ with $p, q < 200$ for which this occurs are
\[
(45, 26),\ (55, 34),\ (91, 64),\ (95, 56),\ (105, 64),\ (105, 76),\ (153, 118),\ (171, 134).
\]
This list was constructed using a straightforward implementation of \cite[Section 9]{DS19} in a computer program. (Daemi and Scaduto's algorithm involves counting solutions to certain linear congruences for pairs of integers in a bounded domain.) The simplest example is $K_{45, 26}$, whose irreducible chain complex has $22$ generators. We display the calculation of $C(K_{45, 26})$ in Figure~\ref{fig:4526a}.
 
Applying Theorem~\ref{thm:kpqinvolutions}, we have that the $\deg_I$-$0$ part of the $\tau$-component of $\tau_1$ acts as in Figure~\ref{fig:4526a}. Although we cannot determine the full $\tau$-component, in what follows we will only need to understand how the $\tau$-component interacts with the generators $1$ and $19$. Using the fact that $\tau d = d\tau$, it is straightforward to check that there are no $\tau$-arrows out of $1$ and $19$ other than the transposition interchanging them (or, the in the case of $\tau_0$, the identity self-arrows). Indeed, the only possible $\tau$-arrows out of $1$ or $19$ that strictly decrease filtration would be to $13$, $22$, and $10$. However, $d(1) = d(19) = 0$, while no nontrivial linear combination of $13$, $22$, and $10$ forms a cycle. Conversely, it is clear that the only possible $\tau$-arrow into $1$ or $19$ that strictly decreases filtration is from $5$.

Now consider $\Delta_1^\tau$. We have the identity
\[
\delta_1 \tau = \Delta_1 d + 1 \delta_1.
\]
Evaluate this on the generators $13$, $22$, and $10$. Because $\tau$ is non-decreasing with respect to the Chern-Simons filtration and the only generators on which $\delta_1$ is nonvanishing are $1$, $19$, and $5$, it is clear that the terms involving $\delta_1$ vanish. Hence $\Delta_1^\tau(6) = \Delta_1^\tau(21) = \Delta_1^\tau(9)$ is zero. Thus, the only possibilities for arrows coming from $\Delta_1^\tau$ are from $16$ to $0$ and $11$ to $0$. (One can show there are no other restrictions on $\Delta^\tau_1$, although this will not be important.) 

Finally, consider $\Delta_2^\tau$. We have the identity
\[
v \tau + d \mu^\tau + \delta_2 \Delta_1^\tau  = \tau v + \mu^\tau d + \Delta_2^\tau \delta_1.
\]
Evaluate this on the generator $1$. We have shown that $\tau(1) = 1$ for $\tau_0$ and $\tau(1) = 19$ for $\tau_1$. Since are no $v$-arrows out of $1$ or $19$, the above equation gives $d \mu^\tau(1) = \Delta^\tau_2(0)$. Hence $\Delta^\tau_2(0)$ must lie in the image of $d$. The only such chains in the correct grading are $y^{-1}(12 + 8)$ and $y^{-1}(17 + 3)$ (and linear combinations thereof), but $y^{-1}8$ and $y^{-1}17$ both have Chern-Simons degree 1/45. This is not allowed since $\Delta_2^\tau$ is filtered; hence $\Delta_2^\tau$ vanishes. We summarize our computation as follows.

\begin{figure}[h!]
\includegraphics[scale = 0.75]{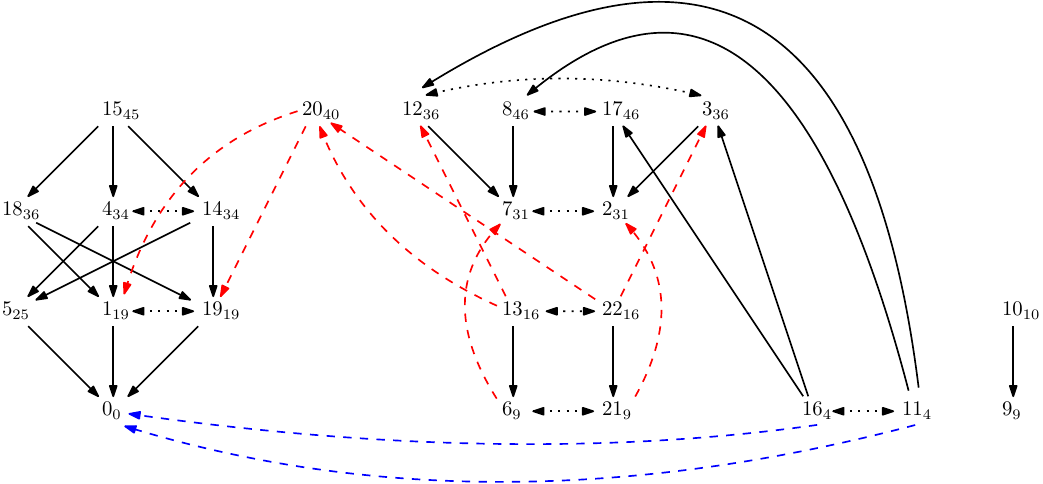}
	\caption{The results of \cite[Section 9]{DS19} and Theorem~\ref{thm:kpqinvolutions} for $K_{45, 26}$. Conventions are similar to Figure~\ref{fig:154a}; a subscript of $n$ represents Chern-Simons degree $n/45$. Blue dashed arrows come from possible $\Delta_1^\tau$-maps, while the map $\Delta_2^\tau$ vanishes.}
	\label{fig:4526a}
\end{figure}

\begin{ex}\label{ex:4526}
The actions of $\tau_0$ and $\tau_1$ on $\wt{C}(K_{45,26})$ are as indicated in Figure ~\ref{fig:4526a}. Specifically:
\begin{itemize}[noitemsep]
\item The $\deg_I$-$0$ part of the $\tau$-component of $\tau_0$ is the identity, while the $\deg_I$-$0$ part of the $\tau$-component of $\tau_1$ interchanges acts as in Figure ~\ref{fig:4526a}.
\item In both cases, there are no further $\tau$-arrows out of $1$ or $19$. The only possible $\tau$-arrows into $1$ or $19$ are from $5$.
\item In both cases, the only possible $\Delta_1$-arrows are from $16$ to $0$ and from $11$ to $0$.
\item In both cases, $\Delta_2^\tau = 0$.
\item In general, there are multiple possibilities for $\mu^\tau$. We leave it to the reader to enumerate these.
\end{itemize}
\end{ex}

We now claim that $\tau_0$ and $\tau_1$ are distinguished up to filtered local equivalence. While the reader can prove this directly, here we instead prove the more interesting fact that $(K_{45, 26} \# - K_{45, 26}, \tau_1 \# - \tau_0)$ satisfies the hypotheses of Theorem~\ref{cork_independent}. %Thus, not only does $(K, \tau)$ not bound any isotopy-equivariant pair of slice disks, we in fact can use $1/n$-surgeries on 

\begin{lem}\label{lem:4526}
We have $r_{-1/8}(K_{45, 26} \# - K_{45, 26}, \tau_1 \# - \tau_0) < \infty$.
\end{lem}
\begin{proof}
Form the $\un{C}$-complex of the tensor product by taking a copy of $C^\otimes$ together with $R$, and recording the maps $d^\otimes$, $\delta_2^\otimes$, $\tau^\otimes$, and $\Delta_2^\otimes$. By Lemma~\ref{bb}\eqref{bb1}, it suffices to show we have no $\tau$-invariant $\theta$-supported cycles in filtration level $[-\infty, 1/8]$. Write such a cycle as
\begin{equation}\label{eq:cycle1}
0 | 0^\dagger + a(1 | 1^\dagger) + b(19 | 1^\dagger) + \text{ other terms}.
\end{equation}
Now, we know $\delta_2^\otimes(0|0^\dagger) = \Lambda(0|1^\dagger + 0|19^\dagger + 0|5^\dagger)$. The only other elements which have differential supported by $0|1^\dagger$ are $1|1^\dagger$, $19|1^\dagger$, and $5|1^\dagger$. The last of these has Chern--Simons filtration $6/45 > 1/8$ and thus does not appear in our filtered cycle. Hence $a + b = 1$. Now consider the action of $\tau_1 \# - \tau_0$ on the above cycle. We claim that this is 
\begin{equation}\label{eq:cycle2}
0 | 0^\dagger + a(19 | 1^\dagger) + b(1 | 1^\dagger) + \text{ other terms}.
\end{equation}
To establish the computation, we must understand the possible ways in which $1 | 1^\dagger$ and $19 | 1^\dagger$ can appear in the image of $\wt{\tau}^\otimes$. Since $1 | 1^\dagger$ and $19 | 1^\dagger$ are basis elements of the first kind, we only need to check incoming arrows arising from the map $\tau | \tau$. Clearly, $1 | 1^\dagger$ appears in the image of $19 | 1^\dagger$, and vice versa. To see these are the only incoming arrows, suppose (for example) that $1 | 1^\dagger$ appears in $(\tau | \tau)(\mf a | \mf b)$. From our computation in Example~\ref{ex:4526}, this means that $\mf a$ is either $19$ or $5$, while $\mf b$ must be $1^\dagger$. (Dualizing Example~\ref{ex:4526} shows that there are no $\tau$-arrows into $1^\dagger$ other than the identity arrow.) The Chern--Simons filtration value of $5 | 1^\dagger$ is $6/45 > 1/8$, so this does not appear in our filtered cycle. A similar argument holds for $19 | 1^\dagger$, establishing the claim.

Now, we know that \eqref{eq:cycle1} and \eqref{eq:cycle2} are homologous as filtered cycles and thus differ by a boundary. It is easily checked that the only chains which have differentials supported by $1 | 1^\dagger$ or $1 | 19^\dagger$ have Chern-Simons filtration value greater than $1/8$. Hence the coefficients of $1|1^\dagger$ in \eqref{eq:cycle1} and \eqref{eq:cycle2} must be exactly equal (and likewise for the coefficients of $19|1^\dagger$). Thus $a = b$, which clearly contradicts $a + b = 1$.
\end{proof}

Applying Theorem~\ref{cork_independent} thus provides a new family of corks. We note the interesting fact that an explicit understanding of the symmetry action on the representation variety was crucial.

\begin{rem}
Unfortunately, because the computation of Example~\ref{ex:4526} is incomplete, the authors do not have enough information to establish the second hypothesis $r_0(K_{45, 26} \# - K_{45, 26}, \tau_0 \# - \tau_1) = \infty$ of Theorem~\ref{thm:intro-whitehead}. Hence we cannot use the present example to establish further examples of exotic pairs of disks via Whitehead doubling. Nevertheless, the authors conjecture that both hypotheses of Theorem~\ref{thm:intro-whitehead} hold in this case.
\end{rem}

\section{Stabilization by $\CP$ and $\RP$}\label{sec:proof-applications}

We now establish our final application. 

\stabilization*

\subsection{Construction of the example}\label{sec:stabilizationexample}
We construct the example asserted in Theorem~\ref{CP2RP2stabilization} as follows. First, throughout this section we let $K$ be our usual connected sum of trefoils
\[
K = 2T_{2,3} \# -2T_{2,3} \quad \text{with} \quad \tau = \tau_f \# -2\tau_0.
\]
Recall that the Whitehead doubles $\Wh_1(K)$ and $\Wh_2(K)$ are then strongly invertible, with involutions induced by $\tau$. For $m \in \N$, consider the equivariant connected sum
\[
L_m = -m \Wh_1(K) \# \Wh_2(K).
\]
By abuse of notation, we denote the involution on $L_m$ also by $\tau$. We claim that for $m$ sufficiently large, $(L_m, \tau)$ provides the desired example. Importantly, note that
\[
\Sigma_2(L_m) = -m S^3_{1/2}(K \# K^r) \# S^3_{1/4}(K \# K^r).
\]
As usual, we also use $\tau$ to denote the lift of the strong inversion on $L_m$ to the branched double cover. Theorem~\ref{CP2RP2stabilization} follows from the following equivariant bordism claim regarding $(\Sigma_2(L_m), \tau)$.

\begin{thm}\label{thm:twosigns}
For $m$ sufficiently large, the involution $\tau$ on $\Sigma_2(L_m)$:
\begin{enumerate}[noitemsep]
\item\label{twosigns1} Does not extend as a diffeomorphism over any negative definite 4-manifold $W^-$ which is bounded by $\Sigma_2(L_m)$ and has $H_1(W^-, \Z_2) = 0$; and,
\item\label{twosigns2} Does not extend as a homology-reversing diffeomorphism over any positive definite 4-manifold $W^+$ which is bounded by $\Sigma_2(L_m)$ and has $H_1(W^+, \Z_2) = 0$.
\end{enumerate}
\end{thm}

We establish Theorem~\ref{thm:twosigns} presently. Let us first show how it implies Theorem~\ref{CP2RP2stabilization}.

\begin{proof}[Proof of Theorem~\ref{CP2RP2stabilization}]
We begin with the first claim. Let $D$ and $\tau_{B^4}(D)$ be a pair of symmetric slice disks for $L_m$ in $B^4$. Suppose for the sake of contradiction that $D$ and $\tau_{B^4}(D)$ were isotopic rel boundary in $B^4 \# n \CP$ for some $n \in \Z$. 

Isotope $\tau_{B^4}$ in the interior of $B^4$ so that it fixes each connected sum region, which we may assume to be disjoint from $D$. The resulting self-diffeomorphism then extends to a self-diffeomorphism $f$ from $(B^4 \# n \CP, D)$ to $(B^4 \# n \CP, \tau_{B^4}(D))$ restricting to $\tau$ on the boundary. By supposition, we have an ambient isotopy moving $\tau_{B^4}(D)$ to $D$ in $B^4 \# n \CP$. The endpoint of this isotopy gives a diffeomorphism $g$ from $(B^4 \# n \CP, \tau_{B^4}(D))$ to $(B^4 \# n \CP, D)$ that restricts to the identity on the boundary. We thus obtain a self-diffeomorphism
\[
g \circ f \colon (B^4 \# n \CP, D) \rightarrow (B^4 \# n \CP, D)
\]
which acts as $\tau$ on the boundary. 

Now take the double cover of $B^4 \# n \CP$ branched along $D$. This gives a $4$-manifold $\Sigma_2(D)$ with
\[
\Sigma_2(L_m) = \partial \Sigma_2(D).
\]
Clearly, $\Sigma_2(D)$ is formed by starting with the branched cover of $B^4$ over $D$ and taking the connected sum with $2n\CP$. The argument of Lemma~\ref{lem:liftbranched2} shows that we may lift $g \circ f$ to a self-diffeomorphism of $\Sigma_2(D)$ which restricts to $\tau$ on $\Sigma_2(L_m)$. If $n \leq 0$, this immediately contradicts the first part of \cref{thm:twosigns}. If $n > 0$, then we observe that every automorphism of the intersection form on $\Sigma_2(D) \# 2n\CP$ can be geometrically realized by a self-diffeomorphism which fixes the boundary. We can thus compose $g \circ f$ with such a self-diffeomorphism so that the result is homology-reversing. This contradicts the second part of \cref{thm:twosigns}.

The proof for stabilization by $\RP$ is similar. Suppose that $D \# n\RP$ and $\tau_{B^4}(D) \# n\RP$ were isotopic. By the same argument as before, we obtain a self-diffeomorphism $g \circ f$ of $B^4$ which preserves $D \# n\RP$ setwise and acts as $\tau$ on the boundary. The branched cover of $B^4$ over $D \# n\RP$ is given by $\Sigma_2(D) \# (- n\CP)$. Since $H_1(B^4 \smallsetminus D \# n\RP) \cong \Z_2$, the same argument as in Lemma~\ref{lem:liftbranched2} shows that we may lift $g \circ f$ to a self-diffeomorphism of this branched cover that restricts to $\tau$ on $\Sigma_2(L_m)$. Once again, any automorphism of the intersection form of $\Sigma_2(D) \# (- n\CP)$ can be realized by a self-diffeomorphism which fixes the boundary. Composing with such a self-diffeomorphism if necessary, we obtain a contradiction with \cref{thm:twosigns}.
\end{proof}

\begin{rem}
The Whitehead doubling parameters $\Wh_1$ and $\Wh_2$ in the definition of $L_m$ are chosen for concreteness; they may be replaced with $\Wh_i$ and $\Wh_j$ (respectively) for any $0 < i < j$. 
\end{rem}

\begin{rem}\label{rem:alternative}
As discussed in the introduction, the authors believe it is possible to give examples of specific disks that survive stabilization by $n \RP$ for any $n \in \Z$ by other methods. Roughly speaking, let $J$ bound an exotic pair of disks $D$ and $D'$, and consider the disks $D \natural (-D)$ and $D' \natural (-D')$ for $J \# - J$. (Note that this is similar to the construction of Theorem~\ref{thm:intro_definite}.) If the branched covers of these are of the form described in \cite[Remark 7.5]{ADMT}, then $D \natural (-D)$ and $D' \natural (-D')$ will survive stabilization by $\RP$ of either sign.
\end{rem}

\begin{rem}
It may be possible to replicate Theorem~\ref{CP2RP2stabilization} by taking branched covers and using involutive Heegaard Floer techniques, but the authors do not have sufficient computational tools to verify this. See \cite[Remark 7.4]{ADMT}.
\end{rem}

\subsection{Inequality from instanton Floer theory}
We prove Theorem~\ref{thm:twosigns} in two parts. The first part will follow from our work on the involutive $r_s$-invariant and will hold for every $m$.

\begin{proof}[Proof of Theorem~\ref{thm:twosigns}\eqref{twosigns1}]
We claim that $r_0(\Sigma_2(L_m), U, \tau) < \infty$. Given this, the first part of Theorem~\ref{thm:twosigns} follows immediately from Theorem~\ref{thm:rsproperties}\eqref{rsproperties:2}. Indeed, since $L_m$ is slice, the branched cover of $\Sigma_2(L_m)$ bounds a homology ball and thus has zero $d$-invariant. It follows that any $W^-$ as in the statement of Theorem~\ref{thm:twosigns}\eqref{twosigns1} has standard intersection form. Puncturing $W^-$ and equipping it with a trivial annulus then gives a strong negative definite pair. Hence Theorem~\ref{thm:rsproperties}\eqref{rsproperties:2} would then imply $r_0(\Sigma_2(L_m), U, \tau) = \infty$.

We thus prove the claim. First consider the connected sum inequality
\begin{equation}\label{r_0 ineq}
  r_0(S^3_{1/4}(K \# K^r), \U, \tau) \geq \text{min} \{r_0(\Sigma_2(L_m), U, \tau), r_0(m S^3_{1/2}(K \# K^r), \U, \tau)\}. 
\end{equation}
Here to avoid cluttering the notation we write $\tau$ everywhere, as the relevant involution is clear in each case. Observe that
\[
r_0(S^3_{1/4}(K \# K^r), \U, \tau) < r_0(S^3_{1/2}(K \# K^r), \U, \tau) \leq r_0(mS^3_{1/2}(K \# K^r), \U, \tau),
\]
where the first inequality follows from our usual simply connected, negative definite cobordism, and the second inequality holds by repeated use of the connected sum inequality. 
This shows the second term on the right-hand side of \eqref{r_0 ineq} is strictly greater than the left-hand side, and thus that
\[
r_0(\Sigma_2(L_m), U, \tau) \leq r_0(S^3_{1/4}(K \# K^r), \U, \tau).
\]
On the other hand, we have 
\[
r_0(S^3_{1/4}(K \# K^r), \U, \tau) < r_0(S^3_{1}(K \# K^r), \U, \tau) \leq r_{-1/8}(K \# K^r, \tau) + 1/8 < \infty.
\]
Here, the first inequality is due to our usual negative definite equivariant cobordism, while the second and third inequalities we have established in Lemmas~\ref{r_0_surgery_ineq} and \ref{aaa}. This completes the proof.
\end{proof}

\subsection{Inequality from Heegaard Floer theory}

The proof of the second part of Theorem~\ref{thm:twosigns} will proceed using some rather involved techniques from Heegaard Floer theory developed in \cite{DHM20} and \cite{DMS}. We review these here, assuming throughout that the reader is broadly familiar with Heegaard Floer homology and the local equivalence formalism of \cite{HMZ}. 

\subsubsection{Local equivalence groups}\label{sec:loceqgp}
There are two local equivalence groups we will use. The first is the local equivalence group $\Inv$ of ($3$-manifold) $\iota$-complexes defined in \cite[Section 8]{HMZ}. In \cite{DHM20}, it was shown that an involution $\tau$ on a homology sphere $Y$ induces a homotopy involution on $\CFm(Y)$ (also denoted by $\tau$), and that the pairs $(\CFm(Y), \tau)$ and $(\CFm(Y), \iota \tau)$ are $\iota$-complexes. Here, $\iota \tau$ is the action of $\tau$ followed by the involutive Heegaard Floer map $\iota$. Importantly, $\Inv$ is a partially ordered group, where
\[
(C_1, \iota_1) \leq (C_2, \iota_2)
\]
whenever there is a local map from $(C_1, \iota_1)$ to $(C_2, \iota_2)$ in the sense of \cite[Definition 8.5]{HMZ}. We write strict inequality if there exists a local map from $(C_1, \iota_1)$ to $(C_2, \iota_2)$ but no local map in the other direction. It follows from \cite[Theorem 1.5]{DHM20} that if $(W, \tau_W)$ is a negative definite equivariant cobordism between two homology spheres $(Y_1, \tau_1)$ and $(Y_2, \tau_2)$, each with zero $d$-invariant, then
\begin{enumerate}[noitemsep]
\item $(\CFm(Y_1), \tau_1) \leq (\CFm(Y_2), \tau_2)$ if $W$ is homology-fixing
\item $(\CFm(Y_1), \iota_1\tau_1) \leq (\CFm(Y_2), \iota_2\tau_2)$ if $W$ is homology-reversing 
\end{enumerate}
Here, note that the intersection form of $W$ is standard since $Y_1$ and $Y_2$ are assumed to have zero $d$-invariant. We will have more use for the $\spinc$-reversing case; hence we work with $\iota\tau$.

The second local equivalence group we use is the local equivalence group $\K_{\iota\tau}$ of knot Floer $\iota_K\tau_K$-complexes. These are pairs $(C, \iota_K\tau_K)$ where $C$ is a knot Floer complex and $\iota_K\tau_K$ is a grading-preserving, $\Z_2[U, V]$-equivariant homotopy involution on $C$. The notation is chosen to match our geometric application: in \cite{DMS}, it is shown that a strong inversion $\tau$ on $K$ induces a skew-graded, skew-$\Z_2[U, V]$-equivariant homotopy involution $\tau_K$ on $\CFK(K)$. The composition $\iota_K \tau_K$ is then a grading-preserving, $\Z_2[U, V]$-equivariant homotopy involution on $\CFK(K)$, where $\iota_K$ is the involutive knot Floer map. Once again, $\K_{\iota\tau}$ is a partially ordered group, where 
\[
(C_1, \iota_1\tau_1) \leq (C_2, \iota_2\tau_2)
\]
whenever there exists a $\iota_K\tau_K$-homotopy equivariant local map from the left to the right; we write strict inequality if there is no map in the other direction. Despite the notation, we do not mean to suggest that in the abstract setting $\iota_K \tau_K$ should be thought of as the composition of two maps; we treat $\iota_K \tau_K$ as a single homotopy involution on $C$.\footnote{Note that $\K_{\iota\tau}$ is thus distinct from the group $\K_{\tau, \iota}$ of $(\tau_K, \iota_K)$-complexes defined in \cite[Section 2.2]{DMS}. In this latter group, we remember the information of $\tau_K$ and $\iota_K$ simulateously, whereas here we work with the composition.} The group $\K_{\iota\tau}$ is a straightforward analogue of the local equivalence group $\K_\tau$ of $\tau_K$-complexes given in \cite[Definition 2.17]{DMS}. Using \cite[Definition 2.14]{DMS}, the group operation on $\K_{\iota\tau}$ is given by the twisted tensor product
\[
(C_1, \iota_1\tau_1) \otimes (C_2, \iota_2\tau_2) = (C_1 \otimes C_2, (\id \otimes \id + \Phi \otimes \Psi)(\iota_1\tau_1 \otimes \iota_2\tau_2)).
\]
It follows from \cite[Theorem 1.8]{DMS} that taking $\iota_K\tau_K$-local equivalence classes gives a homomorphism
\[
\wt{\mathcal{C}} \rightarrow \K_{\iota\tau}.
\]

The final ingredient is an equivariant large surgery formula. Given a strongly invertible knot $K$, we have that $\iota_K \tau_K$ restricts to a homotopy involution on the large surgery subcomplex $A_0(K)$, which we denote by $\iota\tau$. It follows from the equivariant large surgery formula of \cite[Corollary 1.3]{mallick2022knot} that for sufficiently large $n$, we have a homotopy equivalence of $\iota$-complexes 
\[
(CF^-(S^3_n(K), [0]), \iota \tau) \simeq (A_0(K), \iota \tau). 
\]
This homotopy equivalence is relatively graded with a grading shift of $-(n-1)/4$. 

\subsubsection{Almost $\iota$-complexes} 
Our proof will rely heavily on analyzing the action of $\iota \tau$ on the various homology spheres arising in the discussion of Section~\ref{sec:stabilizationexample}. However, it will be difficult for us to calculate these actions in their entirety. Instead, our argument will proceed based on bounds for these calculations coming from the structure of the local equivalence group $\Inv$. However, since $\Inv$ is very difficult to understand, we will need an auxiliary simplification of $\Inv$ called the \textit{almost $\iota$-complex local equivalence group}. This was introduced in \cite{dai2018infinite}; see the discussion prior to \cite[Lemma 7.11]{DHM20} for a similarly abbreviated overview. 

The formalism of almost $\iota$-complexes is a variant of theory of $\iota$-complexes in which all identities involving $\iota$ are only required to hold modulo $U$. More precisely, we say that a pair $(C, \iota)$ is an \textit{almost $\iota$-complex} if $C$ is an $\CFm$-complex in the usual sense and $\iota: C \rightarrow C$ is a grading-preserving endomorphism such that:
\begin{enumerate}[noitemsep]
\item $\iota \partial + \partial \iota \equiv 0$ mod $U$; and,
\item $\iota^2 + \id \equiv \partial H + H \partial$ mod $U$ for some homotopy $H$. \label{eq:homotopy-commute-mod-U}
\end{enumerate}
A \textit{morphism of almost $\iota$-complexes} is a chain map which commutes with $\iota$ up to homotopy modulo $U$, as in \eqref{eq:homotopy-commute-mod-U} above. Isomorphisms, homotopy equivalences, and local maps between almost $\iota$-complexes are thus defined exactly as for $\iota$-complexes, except that each map is only required to homotopy commute with $\iota$ modulo $U$. In fact, note that up to homotopy equivalence we may assume our complex to be \textit{reduced}, as in \cite[Lemma 3.21]{dai2018infinite}. This means that $\ima \partial \subset \ima U$, in which case homotopy commutation modulo $U$ is the same as commutation modulo $U$.

Clearly, any $\iota$-complex can be regarded as an almost $\iota$-complex. Indeed, in general multiple $\iota$-complexes may correspond to the same almost $\iota$-complex: if $(C, \iota_1)$ and $(C, \iota_2)$ are two $\iota$-complexes such that $\iota_1 \equiv \iota_2$ mod $U$, then the identity map from $(C, \iota_1)$ to $(C, \iota_2)$ provides an isomorphism of almost $\iota$-complexes. However, not every almost $\iota$-complex arises from an underlying $\iota$-complex. Roughly speaking, the problem is that a prescribed value of $\iota$ mod $U$ may not have a lift which is an actual chain map and/or homotopy involution.

We denote the local equivalence group of almost $\iota$-complexes by $\widehat{\Inv}$. This group is in fact \textit{totally} ordered (that is, any two elements are comparable), as we will see in a moment. We have an obvious forgetful map from $\Inv$ to $\smash{\widehat{\Inv}}$ which is compatible with the orderings on both groups. We use the notation $\leq_a$ for the order on $\smash{\widehat{\Inv}}$, in order to distinguish it from the order on $\Inv$. If $(C_1, \iota_1)$ and $(C_2, \iota_2)$ are two $\iota$-complexes, we thus write
\[
(C_1, \iota_1) \leq_a (C_2, \iota_2)
\]
when we wish to consider $(C_1, \iota_1)$ and $(C_2, \iota_2)$ as almost $\iota$-complexes. Note that $\leq$ then implies $\leq_a$, although not (necessarily) conversely.

The central advantage of the almost $\iota$-complex formalism is that it is possible to explicitly enumerate the set of almost $\iota$-complexes up to local equivalence (of almost $\iota$-complexes). Indeed, as shown in \cite{dai2018infinite}, each local equivalence class has a preferred representative, called a \textit{standard complex}. The set of standard complexes is parameterized by certain finite sequences of symbols, which (among other things) record $U$-torsion tower lengths in the usual $\Z_2[U]$-homology of each standard complex. See \cite[Definition 4.1]{dai2018infinite}. 

Explicitly, let
\[
(a_1, b_1, a_2, b_2, \ldots, a_n, b_n)
\]
be a sequence of length $2n$, where each $a_i$ is either the symbol $-$ or $+$ and each $b_i$ is a nonzero integer. The standard complex $C(a_1, b_1, \ldots, a_n, b_n)$ corresponding to this sequence is constructed as follows. Introduce $2n+1$ generators $\{x_i\}_{i = 0}^{2n}$ over $\Z_2[U]$, which we think of as spanning $2n + 1$ towers being arranged from left-to-right across the page. Successive towers are connected by arrows representing the action of $\omega = 1 + \iota$ and/or the action of the differential. Two towers spanned by $x_i$ and $x_{i+1}$ are connected by an $\omega$-arrow exactly when $i$ is even and a differential arrow exactly when $i$ is odd. The $\omega$-arrows are specified by the $a_i$, while the differential arrows are specified by the $b_i$; the arrow goes from left-to-right if the corresponding symbol is negative and from right-to-left if it is positive. The differential arrow corresponding to $b_i$ is decorated by $U^{|b_i|}$, so that the $|b_i|$ are precisely the lengths of the $U$-torsion towers in homology. Some illustrative examples are given in Figure~\ref{fig:examplecomplexes}. See \cite[Section 4.1]{dai2018infinite} for further discussion. 
\begin{figure}[h!]
\includegraphics[scale = 0.9]{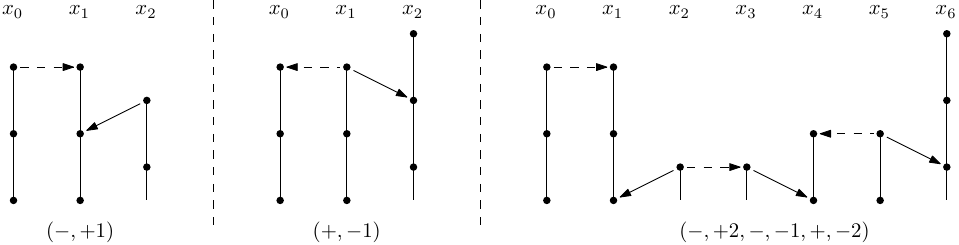}
	\caption{Examples of standard complexes. The action of $U$ is given by vertical translation. Dashed arrows depict the action of $\omega = 1 + \iota$; solid arrows denote the differential. Note that the first two complexes are genuine $\iota$-complexes. The third is only an almost $\iota$-complex, since $(\partial \circ \omega)$ and $(\omega \circ \partial)$ are only equal modulo $U$ on (for example) $x_2$.}
	\label{fig:examplecomplexes}
\end{figure}

In \cite[Section 6]{dai2018infinite}, it is shown that every almost $\iota$-complex is locally equivalent to a (unique) standard complex. In fact, more is true: in the proof of \cite[Theorem 6.3]{dai2018infinite}, it is shown that every almost $\iota$-complex is homotopy equivalent to the direct sum of a standard complex and another complex with an almost $\iota$-action. Moreover, it is shown that the parameter sequence determines exactly when there is a local map between two standard complexes. Explicitly, there exists a local map
\[
C(a_1, b_1, \ldots, a_m, b_m) \leq_a C(a_1', b_1', \ldots, a_m', b_m')
\]
if and only if we have the lexicographic inequality of sequences
\[
(a_1, b_1, \ldots, a_m, b_m) \leq^{!} (a_1', b_1', \ldots, a_m', b_m').
\]
Here, the order $\leq^{!}$ is defined by declaring 
\[
- <^{!} 0 <^{!}+,
\]
declaring
\[
-1 <^{!} -2 <^{!} -3 <^{!} \cdots <^{!} 0 <^{!} \cdots <^{!} 3 <^{!} 2 <^{!} 1,
\]
and extending this to sequences of symbols lexicographically. When comparing two strings of different lengths, we extend the shorter one by appending trailing zeroes, even though usually zero is not a valid symbol. See \cite[Section 4.2]{dai2018infinite}. The fact that the lexicographic order is total of course implies $\smash{\widehat{\Inv}}$ is totally ordered.

\subsubsection{Completion of Theorem~\ref{thm:twosigns}}
We now complete the proof of Theorem~\ref{thm:twosigns}. We will need the following lemmas. Define $Y_1$ to be the almost $\iota$-complex $(-, +1)$ given in Figure~\ref{fig:examplecomplexes}.

\begin{lem}\label{lem:HF1}
We have
\[
(A_0(K \# K^r), \iota \tau) \leq_a Y_1.
\]
\end{lem}

\begin{lem}\label{lem:HF2}
For all $m$ sufficiently large, we have
\[
m Y_1 <_a (CF^-(S^3_{1/4}(K \# K^r)), \iota \tau).
\]
\end{lem}

Lemmas~\ref{lem:HF1} and \ref{lem:HF2} should be thought of as bounding the (almost $\iota$-complex) local equivalence classes of surgeries along $K \# K^r$. Roughly speaking, Lemma~\ref{lem:HF1} says the local equivalence class of large surgery on $K \# K^r$ is negative; and, in fact, is bounded above by $Y_1$. (Note that $Y_1$ is strictly negative by the lexicographic order.) For the second lemma, the reader should think of $mY_1$ as a sequence of almost $\iota$-complexes which is becoming more and more negative. Lemma~\ref{lem:HF2} says the local equivalence class of $\smash{S^3_{1/4}(K \# K^r)}$ is eventually bounded below by some $mY_1$. In fact, Lemma~\ref{lem:HF2} holds for all $1/n$-surgeries along $K \# K^r$ with $n > 0$, but we focus on $1/4$-surgery for concreteness.

We temporarily delay the proofs of Lemmas~\ref{lem:HF1} and \ref{lem:HF2} and instead first explain how these bounds complete the proof of Theorem~\ref{thm:twosigns}.

\begin{proof}[Proof of Theorem~\ref{thm:twosigns}\eqref{twosigns2}]
Suppose the second claim of Theorem~\ref{thm:twosigns} were false. Reversing orientations, we obtain a negative definite manifold $-W^+$ with boundary
\[
m S^3_{1/2}(K \# K^r) \# - S^3_{1/4}(K \# K^r).
\]
The fact that the involution $\tau$ admits a homology-reversing extension over $-W^+$ gives
\[
\Z_2[U] \leq m (CF^-(S^3_{1/2}(K \# K^r)), \iota \tau) - (CF^-(S^3_{1/4}(K \# K^r)), \iota \tau),
\]
where $\Z_2[U]$ is the trivial $\iota$-complex. Note that $\Sigma_2(L_m)$ bounds an integer homology ball, so it has $d$-invariant zero. This is equivalent to the inequality
\begin{equation}\label{eq:hf1}
(CF^-(S^3_{1/4}(K \# K^r)), \iota \tau) \leq m (CF^-(S^3_{1/2}(K \# K^r)), \iota \tau).
\end{equation}
Now, we have a homology-reversing, equivariant negative definite cobordism from $1/2$-surgery to $(+1)$-surgery on $K \# K^r$. This gives the inequality
\begin{equation}\label{eq:hf2}
(CF^-(S^3_{1/2}(K \# K^r)), \iota \tau) \leq (CF^-(S^3_{+1}(K \# K^r)), \iota \tau).
\end{equation}
Furthermore, a similar argument as in \cite[Lemma 4.2]{dai20222} shows that there is a local map with grading shift
\begin{equation}\label{eq:hf3}
(CF^-(S^3_{+1}(K \# K^r)), \iota \tau) \rightarrow (CF^-(S^3_{n}(K \# K^r), [0]), \iota \tau).
\end{equation}
This is similar to \eqref{eq:hf2}, except with an additional grading shift taking into account the fact that the outgoing end is not an integer homology sphere. However, by postcomposing with the equivariant large surgery isomorphism (which cancels the grading shift of the previous map), we get that there is a local map 
\[
(CF^-(S^3_{+1}(K \# K^r)), \iota \tau) \rightarrow (A_0(K \# K^r), \iota \tau).
\]
Composing \eqref{eq:hf1}, \eqref{eq:hf2}, and \eqref{eq:hf3}, we thus have
\[
(CF^-(S^3_{1/4}(K \# K^r)), \iota \tau) \leq m (A_0(K \# K^r), \iota \tau).
\]
Now pass to almost $\iota$-complexes. Then Lemma~\ref{lem:HF1} shows $\smash{(CF^-(S^3_{1/4}(K \# K^r)), \iota \tau) \leq_a m Y_1}$. But if $m$ is sufficiently large, this contradicts Lemma~\ref{lem:HF2}, completing the proof.
\end{proof}

\subsubsection{Non-equivariant calculation}\label{sec:torsionorder}

It thus remains to prove Lemmas~\ref{lem:HF1} and \ref{lem:HF2}. As discussed previously, these should be viewed as bounds on the the local equivalence classes of surgeries along $K \# K^r$. Before proceeding, it will be useful to know that the regular Heegaard Floer homologies of these surgeries are small. Here, we prove that both large and small surgeries have $U$-torsion order one.

\begin{lem}\label{lem:torsion1}
All $U$-torsion in $H_*(A_0(K \# K^r))$ is of order one. Moreover, all $U$-torsion generators have Maslov grading zero.
\end{lem}

\begin{proof}
First note that $K \#K^r$ is Floer thin and slice, so that every generator $[p, i , j]$ in the chain complex of $\mathit{CFK}(K \# K^r)$ has Maslov grading $i+j$. Here we think of the generators of $\mathit{CFK}(K \# K^r)$ over $\mathbb{F}_2[U, U^{-1}]$ in the $\mathbb{Z} \oplus \mathbb{Z}$-lattice; see for example \cite[Section 8.1]{HM17}. It is then a standard fact that $\mathit{CFK}(K \# K^r)$ is a direct sum of boxes with differential length one (placed symmetrically with respect to the main diagonal $y=x$), together with a single $U$-nontorsion generator with Maslov grading zero at the origin. The claim follows.
\end{proof}

\begin{lem}
All $U$-torsion in $\HFm(S^3_{1/4}(K \# K^r))$ is of order one.
\end{lem}

\begin{proof}
We use the work of Ni and Zhang \cite{ni2014characterizing} and adopt the notation of \cite[Section 3.3]{ni2014characterizing}. Following their discussion of the rational surgery formula, we claim that $\smash{\HF_{\text{red}}(S^3_{1/4}(K \# K^r))}$ will only come from
\[
\bigoplus_{s \in \mathbb{Z}}A_{ \text{red}, \lfloor \frac{s}{4} \rfloor}(K \# K^r).
\]
Note that in \cite[Proposition 3.6]{ni2014characterizing}, the expression for the reduced group contains truncated towers of length $V_k$ or $H_j$. However, these all vanish since $V_k$ and $H_j$ are concordance invariants and $K \# K^r$ is slice. Now, as observed in the proof of Lemma~\ref{lem:torsion1}, $\mathit{CFK}(K \# K^r)$ is a direct sum of a single generator at the origin and boxes of differential length one. Note that away from the boundary of hook region, there are no nontrivial generators of the reduced group. These only occur along the boundary region, and for all such generators $a$, $U^{-1}a$ is again trivial in homology. It follows that the only elements that can contribute nontrivially to $A_{ \text{red},i}(K \# K^r)$ for any $i \in \mathbb{Z}$ are necessarily of $U$-torsion order one.
\end{proof}

\subsubsection{Proof of Lemmas~\ref{lem:HF1} and \ref{lem:HF2}} We are finally ready to give the proofs of Lemmas~\ref{lem:HF1} and \ref{lem:HF2}. We start with Lemma~\ref{lem:HF1}. For this, we first establish non-positivity of the large surgery. 

\begin{lem}\label{lem:cal1}
    For all $n$ sufficiently large, there does not exist an $\iota$-complex local map from the trivial complex into $(A_0(K \# K^r), \iota \tau)$. 
\end{lem}

\begin{proof}
Suppose there was. We then have a local map of $\iota$-complexes
\[
\Z_2[U] \leq (A_0(K \# K^r), \iota \tau).
\]
By the same argument as in \cite[Lemma 4.4]{DMZ24}, this implies the existence of a local map
\[
\Z_2[U,V] \leq (\CFK(K \# K^r), \iota_K \tau_K)
\]
where $\iota_K$ and $\tau_K$ are the actions on the knot Floer complex of $K \# K^r$. Now, recall that
\[
(K, \tau) = (2T_{2,3} \# -2T_{2,3}, \tau_f \# - 2\tau_0),
\]
suppressing orientation reversals for brevity. Examining Figure~\ref{fig:trefoils3}, we have
\[
 (K \# K^r, \tau) = ( (2T_{2,3} \# -2T_{2,3}) \# (-2T_{2,3} \# 2T_{2,3}), (\tau_f \# - 2\tau_0) \# (-2\tau_0 \# \tau_f)),
\]
where we have again suppressed writing orientation reversals.
(Note also that $(2T_{2,3}, \tau_f)$ lies in the center of $\smash{\wt{\mathcal{C}}}$, as discussed in \cite[Section 1]{Sa86}.) Using the fact that taking the $\iota_K\tau_K$-local equivalence class constitutes a homomorphism, we obtain the inequality of local equivalence classes
\[
2(\CFK(2T_{2,3}), \iota_K(2\tau_0)) \leq 2(\CFK(2T_{2,3}), \iota_K\tau_f).
\]
Denote 
\[
A = (\CFK(2T_{2,3}), \iota_K(2\tau_0)) \quad \text{and} \quad B = (\CFK(2T_{2,3}), \iota_K\tau_f)
\]
so that $2A \leq 2B$. These complexes are small enough to be computed directly. For grading reasons, we have $\iota_K = \tau_0$ on $\CFK(T_{2,3})$. Thus $(\CFK(T_{2,3}), \iota_K\tau_0)$ is the usual complex for the trefoil, equipped with the identity map. We then use the tensor product formula discussed in Section~\ref{sec:loceqgp} to calculate $A = 2(\CFK(T_{2,3}), \iota_K \tau_0)$. Finally, the complex $B$ is computed from \cite[Theorem 4.3]{DMS} (see \cite[Proposition 8.12]{mallick2022knot} for an explicit computation). After an appropriate change of basis, these calculations are displayed in Figure~\ref{fig:itcomputations}.

\begin{figure}[h!]
\includegraphics[scale = 0.9]{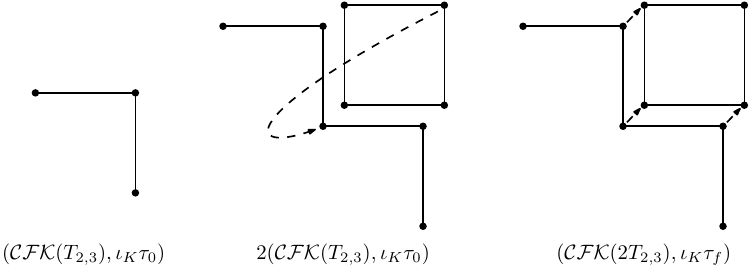}
	\caption{Left: the complex $(\CFK(T_{2,3}, \iota_K\tau_0)$. Middle: the complex $A = 2(\CFK(T_{2,3}, \iota_K\tau_0)$. Right: the complex $B = (\CFK(2T_{2,3}, \iota_K\tau_f)$. We have drawn the action of $1 + \iota_K \tau_K$ using the dotted arrows. If no dotted arrow is drawn, the action of $1 + \iota_K \tau_K$ is zero.}
	\label{fig:itcomputations}
\end{figure}

The following are not difficult to see:
\begin{enumerate}[noitemsep]
\item We have $B \leq A$. Indeed, a local map from $B$ to $A$ is provided by mapping the staircase isomorphically onto the staircase, and sending the box of four generators to zero. 
\item We in fact have $B < A$. Indeed, it is not difficult to see that any local map from $A$ to $B$ must map the three central staircase generators to the three central staircase generators, plus possibly the three box generators adjacent to them. Neither possibility is $\iota_K\tau_K$-homotopy equivariant.
\end{enumerate}
From the first observation, we have $-A \leq -B$. Tensoring this with $2A \leq 2B$ and using the fact that $\leq$ is compatible with the group operation gives $A \leq B$. But this contradicts the second observation. 
\end{proof}

We now upgrade this non-existence result to the setting of almost $\iota$-complexes.

\begin{lem}\label{lem:cal2}
Let $C$ be an $\iota$-complex such that all $U$-torsion in $H_*(C)$ is of order one. Then there exists an $\iota$-complex local map from the trivial complex into $C$ if and only if there exists an almost $\iota$-complex local map from the trivial complex into $C$.
\end{lem}
\begin{proof}
Since every $\iota$-complex morphism induces an almost $\iota$-complex morphism, the forwards direction is trivial. We prove the converse direction. We have an $\iota$-complex local map from the trivial complex into $C$ if and only if there exists a cycle $x \in C$ whose homology class $[x]$ is $U$-nontorsion and 
\[
\iota_*([x]) = [x].
\]
On the other hand, we have an almost $\iota$-complex local map from the trivial complex into $C$ if and only if there exists a cycle $x \in C$ whose homology class $[x]$ is $U$-nontorsion and
\[
\iota x \equiv x \bmod U.
\]
Thus, assume we have the latter. For the sake of concreteness, fix a $U$-paired basis for $C$, so that $C$ is spanned by a generator $t$ whose class is $U$-nontorsion, together with generators $\{y_i, z_i\}_{i = 1}^n$ such that $\partial y_i = Uz_i$. Here, we use the fact that $H_*(C)$ has $U$-torsion order one. We claim that
\[
\iota x + x = \sum_i U^{n_i} z_i
\]
with the $n_i \geq 1$. Indeed, the left-hand side is a cycle, so the right-hand side is spanned by $t$ and $z_i$. Since both $\iota x$ and $x$ have $U$-nontorsion class, we see that $\iota x + x$ has $U$-torsion class, which means $t$ does not appear on the right-hand side. Finally, $\iota x \equiv x \bmod U$ implies all exponents on the right-hand side are at least one. But then the right-hand side is in the image of $\partial$, so $\iota_*[x] = [x]$, as desired.
\end{proof}

We are now ready for our first serious claim regarding the structure of $\smash{\widehat{\Inv}}$. The following states that if $C$ is an $\iota$-complex whose $U$-torsion order is at most one, then whenever $C <_a 0$, we in fact have $C \leq_a Y_1$. This is not true if the hypotheses on $C$ are removed, as will be clear from the proof.

\begin{lem}\label{lem:cal3}
Let $C$ be an $\iota$-complex such that all $U$-torsion in $H_*(C)$ is of order one. Suppose there does not exist any almost $\iota$-complex local map from the trivial complex into $C$. Then $\smash{C \leq_a Y_1}$.
\end{lem}
\begin{proof}
Let $C$ be an $\iota$-complex as in the statement of the lemma. Then $C$ is locally equivalent to standard complex $S$; as shown in the proof of \cite[Theorem 6.3]{dai2018infinite}, we may assume that we have a splitting of almost $\iota$-complexes $C = S \oplus T$. What this means is that:
\begin{enumerate}[noitemsep]
    \item we have a splitting $C = S \oplus T$ as complexes over $\f[U]$;
    \item modulo $U$, the map $\iota \colon C \rightarrow C$ preserves each of $S$ and $T$; and,
    \item modulo $U$, the map $\iota \colon C \rightarrow C$ restricted to $S$ yields the standard complex structure on $S$.
\end{enumerate}
We consider the parameter sequence of $S$.
To say that there does not exist any almost-$\iota$ complex morphism from the trivial complex into $C$ is to say that the first parameter of $S$ is a $-$. Since all $U$-torsion is of order one, the second parameter is then $+1$ or $-1$; as shown in \cite[Theorem 8.3]{dai2018infinite}, the fact that $C$ is an actual $\iota$-complex implies the second parameter is $+1$. Suppose for the sake of contradiction that $C > Y_1$ as almost-$\iota$ complexes. Since $Y_1$ corresponds to the sequence $(-, +1)$, the next parameter must be a $+$. This leads to the generators $x_0$, $x_1$, $x_2$, and $x_3$ of $S$ in Figure~\ref{fig:standardcomplex2}.

\begin{figure}[h!]
\includegraphics[scale = 0.9]{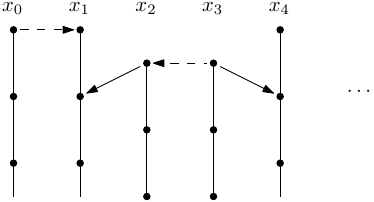}
	\caption{The standard complex $S$ in the proof of Lemma~\ref{lem:cal3}.}
	\label{fig:standardcomplex2}
\end{figure}

We consider whether it is possible to lift the given $\iota$ mod $U$ action to a \textit{bona fide} chain map. Consider the action of $\omega = 1 + \iota$ on $x_3$. We have that $(\partial \circ \omega) x_3$ is supported by $Ux_1$, so the same must be true for $(\omega \circ \partial)x_3$. In particular, we must have $\partial x_3 \neq 0$; hence $\partial x_3 = Ux_4$ for the next generator $x_4$. The only way for $(\omega \circ \partial)x_3$ to be supported by $Ux_1$ is then for $\omega x_4$ to be supported by $x_1$. But $x_1$ does not lie in the image of $U$, and $\omega$ mod $U$ applied to $x_4$ either vanishes or is $x_5$. This gives a contradiction and completes the proof.
\end{proof}

Putting these claims together gives the first of our lemmas.

\begin{proof}[Proof of Lemma~\ref{lem:HF1}]
By Lemma~\ref{lem:cal1}, we have that there is no $\iota$-complex local map from the trivial complex into $(A_0(K \# K^r), \iota\tau)$. By our work in Section~\ref{sec:torsionorder}, the homology of this complex has $U$-torsion order one. Hence by Lemma~\ref{lem:cal2}, we may upgrade the above statement to hold in the setting of almost $\iota$-complexes. Applying Lemma~\ref{lem:cal3} then completes the proof.
\end{proof}

The second of our lemmas is a consequence of the next serious claim regarding $\smash{\widehat{\Inv}}$.

\begin{lem}\label{lem:structure-lemma}
Let $C$ be an $\iota$-complex such that all $U$-torsion in $H_*(C)$ is of order one. Then for all $m$ sufficiently large, we have $mY_1 <_a C$.
\end{lem}

\begin{proof}
Suppose not, so that $C \leq_a mY_1$ for all $m$. As in the proof of Lemma~\ref{lem:structure-lemma}, we may assume that we have a splitting of almost $\iota$-complexes $C = S \oplus T$, where $S$ is the standard complex representative of $C$. We attempt to constrain the parameter sequence 
\[
S = C(a_1, b_1, \ldots, a_n, b_n). 
\]
It follows from \cite[Theorem 8.1]{dai2018infinite} that the parameter sequence corresponding to $mY_1$ (as an almost-$\iota$ complex) is the string of length $2m$ given by
\[
mY_1 = C(-, +1, -, +1, -, +1, \cdots, -, +1).
\]
We know that there is a local map from $C$ into $mY_1$ for all $m$. Since $C$ is locally equivalent to $S$, this means that there is a local map from $S$ into $mY_1$. 

Choose $m > n$ so that $mY_1 = C(-, +1, -, +1, \ldots, -, +1)$ is longer than the length of $S$. Consider the first symbol at which $S$ and $nY_1$ differ. There are several possibilities:
\begin{enumerate}[noitemsep]
    \item $S$ could be a prefix of $mY_1$: This means that $S$ and $mY_1$ first differ (in a degenerate sense) at the symbol $a_{n+1} = 0$. This is impossible, since the corresponding symbol in $mY_1$ is $-$ and we do not have $0 \leq^{!} -$.
    \item $S$ could first differ from $mY_1$ at some symbol $a_k$: This is impossible, since the corresponding symbol in $mY_1$ is $-$ and there is no symbol which is strictly less than $-$.
    \item $S$ could first differ from $mY_1$ at some symbol $b_k$: In this case, we know that $b_k$ must be an integer other than $+1$. However, because $S$ is a summand of $C$, the homology of $S$ also has $U$-torsion of at most order one. It follows that we must have $b_k = -1$.
\end{enumerate}
We thus conclude that the parameter sequence of $S$ is of the form
\[
(-, +1, -, +1, \ldots, -, +1, -, -1, a_{k+1}, b_{k+1}, \ldots, a_n, b_n).
\]
We claim that this contradicts the fact that $C = S \oplus T$ is an actual $\iota$-complex. For this, we simply consider whether it is possible to lift the given $\iota$ mod $U$ action to a \textit{bona fide} chain map. Consider the generators $x_{2k-3}, x_{2k-2}$, $x_{2k - 1}$, and $x_{2k}$ in $S \subset C$, as indicated in Figure~\ref{fig:standardcomplex}. 

\begin{figure}[h!]
\includegraphics[scale = 0.9]{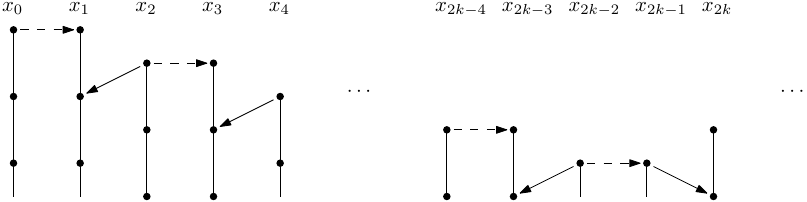}
	\caption{The standard complex $S$ in the proof of Lemma~\ref{lem:structure-lemma}. Note the direction of the arrow from $x_{2k-1}$ to $x_{2k}$.}
	\label{fig:standardcomplex}
\end{figure}

Suppose we had a lift of the given $\omega$ mod $U$. This would have to satisfy
\[
\omega x_{2k-3} = Up \quad \text{and} \quad \omega x_{2k-2} = x_{2k-1} + Uq
\]
for some generators $p$ and $q$ in $C$. We then compute
\[
(\omega \circ \partial) x_{2k - 2} = \omega (Ux_{2k-3}) = U^2p \quad \text{and} \quad (\partial \circ \omega) x_{2k - 2} = \partial (x_{2k-1} + Uq) = Ux_{2k} + U\partial q.
\]
Now, $x_{2k}$ does not appear in the image of $\partial$ except with a power of $U$, so there is no term from $U \partial q$ which can cancel with $Ux_{2k}$. In particular, the left-hand expression is in the image of $U^2$, while the right-hand expression is not. Hence it is impossible to find a lift of $\omega$ mod $U$ such that $\omega \circ \partial = \partial \circ \omega$. This provides the desired contradiction.
\end{proof}

%\begin{rem}
%    It is important to note that the hypothesis of the lemma is necessary. For example, if $C$ is the standard complex $Y_2$, then the reader may check that in fact $Y_2 \leq n Y_1$ for all $n$. 
%\end{rem}

The proof of Lemma~\ref{lem:HF2} is then immediate.

\begin{proof}[Proof of Lemma~\ref{lem:HF2}]
By our work in Section~\ref{sec:torsionorder}, the homology of $\smash{\CFm(S^3_{1/4}(K \# K^r))}$ has $U$-torsion order one. The claim then follows from Lemma~\ref{lem:structure-lemma}.
\end{proof}

\bibliographystyle{plain}
\bibliography{tex}

\end{document}